\documentclass[a4paper,10pt]{book}

\usepackage[english]{babel}
\usepackage[latin1]{inputenc}
\usepackage[T1]{fontenc}
\usepackage{amsmath}
\usepackage{amssymb}
\usepackage{multirow}
\usepackage{alltt}
\usepackage{graphicx}
\usepackage{bbm}
\usepackage{amsthm}
\usepackage{mathrsfs}
\usepackage{textcomp}
\usepackage{fullpage}
\usepackage[usenames,dvipsnames]{color}

\definecolor{darkblue}{RGB}{0,0,170}
\definecolor{brickred}{RGB}{200,0,0}
\usepackage[pdfauthor={Nicola Abatangelo}, hyperindex, linktocpage]{hyperref}
\hypersetup{colorlinks=true, linkcolor=darkblue, citecolor=brickred}

\newcommand{\R}{\mathbb{R}}
\newcommand{\N}{\mathbb{N}}
\newcommand{\C}{\mathcal{C}}
\newcommand{\T}{\mathcal{T}}
\newcommand{\J}{\mathcal{J}}

\newcommand{\HH}{\mathcal{H}}
\newcommand{\eps}{\varepsilon}
\newcommand{\ep}{\varepsilon}
\newcommand{\dist}{\hbox{dist}}
\newcommand{\Ds}{{\left(-\lapl\right)}^s}
\newcommand{\As}{{\left(-\lapl\right|_\Omega)}^s}
\newcommand{\lapl}{\triangle}
\newcommand{\grad}{\triangledown\!}

\renewcommand{\le}{\leq} 
\renewcommand{\ge}{\geq}
\renewcommand{\lessgtr}{\asymp} 

\newcommand{\Dums}{{\left(\left.-\lapl\right|_\Omega\right)}^{1-s}}
\newcommand{\Dms}{{\left(\left.-\lapl\right|_\Omega\right)}^{-s}}
\newcommand{\Asmu}{{\left(\left.-\lapl\right|_\Omega\right)}^{s-1}}
\newcommand{\rest}{{\left(\left.-\lapl\right|_\Omega\right)}}
\newcommand{\trest}{{\left.-t\lapl\right|_\Omega}}
\newcommand{\super}{\overline}

\usepackage{imakeidx}
\makeindex[title={Index of definitions},intoc]

\newtheoremstyle{mytheoremstyle} 
    {\topsep}                    
    {\topsep}                    
    {\itshape}                   
    {}                           
    {\bf\scshape}                   
    { \bf | }                          
    {2pt}                       
    {\thmname{#1}\thmnumber{ #2}\thmnote{: #3}}  

\theoremstyle{mytheoremstyle}
\newtheorem{defi}{definition}[chapter]
\newtheorem{ex}[defi]{example}
\newtheorem{theo}[defi]{theorem}
\newtheorem{prop}[defi]{proposition}
\newtheorem{lem}[defi]{lemma}
\newtheorem{rmk}[defi]{remark}
\newtheorem{cor}[defi]{corollary}

\newtheorem*{defi*}{d\'efinition}
\newtheorem*{ex*}{example}
\newtheorem*{theo*}{th\'eor\`eme}
\newtheorem*{prop*}{proposition}
\newtheorem*{lem*}{lemma}
\newtheorem*{rmk*}{remarque}
\newtheorem*{cor*}{corollaire}

\usepackage{titlesec}

\addto\captionsenglish{}

\titleformat{\part}[frame]{}{{\Large\bf\ \partname\ \thepart\ \filright}}{15pt}{\thispagestyle{empty}\Huge\itshape\hfill
}[]

\titleformat{\chapter}[display]{}{\hrulefill\hspace{10pt}\Large\filleft{\it
\chaptertitlename}
\hspace{3pt}
{\Huge\textbf{
\oldstylenums{\thechapter}}}}{10pt}{
\LARGE\bfseries
\filleft}[\rule{\textwidth}{.4pt}]

\titlespacing{\section}{0pt}{30pt}{10pt}

\titleformat{\section}{}{\Large\thesection}{.6em}{\centering\Large\bf\scshape}[]

\titleformat{\subsection}{}{\thesubsection}{.6em}{\normalfont\large\bfseries}[]

\usepackage{fancyhdr}
\renewcommand{\chaptermark}[1]{\markboth{#1}{}}

\fancypagestyle{plain}{%
\fancyhead{} 
}

\usepackage{tocloft}

\usepackage{titletoc}
\titlecontents*{subsection}[1.5em]{\small}{ \thecontentslabel\it~}{}{,~\thecontentspage\ }[~--- 	][]

\title{\bf\Huge Large solutions for fractional Laplacian operators}
\author{\textsc{\Large nicola abatangelo}}
\date{\vspace*{6\bigskipamount}
PhD thesis defended on September 28$^{\hbox{\small th}}$,\\
Universit\'e de Picardie Jules Verne (Amiens, France)}

\begin{document}

\frontmatter

\maketitle

\cleardoublepage
\vspace*{\stretch{1}}

\paragraph{R\'esum\'e.} {\it La th\`ese \'etudie les probl\`emes de Dirichlet lin\'eaires et semilin\'eaires pour diff\'erents op\'erateurs du type Laplacien fractionnaire. Les donn\'ees peuvent \^etre des fonctions r\'eguli\`eres ou plus g\'en\'eralement des mesures de Radon. Le but est de classifier les solutions qui pr\'esentent une singularit\'e au bord du domaine prescrit.

Nous remarquons d'abord l'existence de toute une gamme de fonctions harmoniques explosants au bord et nous les caract\'erisons selon une nouvelle notion de trace au bord. \`A l'aide d'une nouvelle formule d'int\'egration par parties, nous \'elaborons ensuite une th\'eorie faible de type Stampacchia pour \'etendre la th\'eorie lin\'eaire \`a un cadre qui comprend ces fonctions : nous \'etudions les questions classiques d'existence, d'unicit\'e, de d\'ependance \`a l'\'egard des donn\'ees, la r\'egularit\'e et le comportement asymptotique au bord. 

Puis, nous d\'eveloppons la th\'eorie des probl\`emes s\'emilin\'eaires, en g\'en\'eralisant la m\'ethode des sous- et sursolutions. Cela nous permet de construire l'analogue fractionnaire des grandes solutions dans la th\'eorie des EDPs elliptiques nonlin\'eaires, en donnant des conditions suffisantes pour l'existence. 

La th\`ese se termine par la d\'efinition et l'\'etude d'une notion de courbures directionnelles nonlocales.}
\vspace*{.3cm}

\noindent {\it Mots cl\'es:} Laplacien fractionnaire, op\'erateurs nonlocaux, grandes solutions, solutions faibles $L^1$,
\'equations elliptiques nonlin\'eaires, probl\`eme de Dirichlet, singularit\'e au bord,
courbures nonlocales.

\vspace*{\stretch{1}}

\paragraph{Abstract.} {\it The thesis studies linear and semilinear Dirichlet problems driven by different fractional Laplacians. 
The boundary data can be smooth functions or also Radon measures. 
The goal is to classify the solutions which have a singularity on the boundary of the prescribed domain.

We first remark the existence of a large class of harmonic functions with a boundary blow-up 
and we characterize them in terms of a new notion of degenerate boundary trace. 
Via some integration by parts formula, we then provide a weak theory of Stampacchia's sort 
to extend the linear theory to a setting including these functions:
we study the classical questions of existence, uniqueness, continuous dependence on the data, 
regularity and asymptotic behaviour at the boundary.

Afterwards we develop the theory of semilinear problems, by adapting and generalizing some sub- and supersolution methods. 
This allows us to build the fractional counterpart of large solutions in the elliptic PDE theory of nonlinear equations, 
giving sufficient conditions for the existence.

The thesis is concluded with the definition and the study of a notion of nonlocal directional curvatures.}
\vspace*{-0.1cm}


\vspace*{\stretch{2}}

\noindent\begin{tabular}{l|l}
{\it Laboratoire Ami\'enois de Mathematique} $\qquad\qquad\ \ $ &  \\ 
\hfill {\it Fondamentale et Appliqu\'ee, {\sc cnrs umr 7352}} & 
{\it Dipartimento di Matematica Federigo Enriques}\\
{\small 33 rue Saint-Leu} & {\small via Saldini, 50} \\ 
{\small 80039 Amiens Cedex 1} & {\small 20133 Milano (MI)} \\ 
{\small France} & {\small Italy} \\ 
\end{tabular}

\chapter*{Contents}\chaptermark{Contents}\markright{contents}
\tableofcontents
\mainmatter

\chapter*{Introduction}\chaptermark{Introduction}\markright{introduction}
\addcontentsline{toc}{chapter}{Introduction}

This manuscript presents the study of boundary blow-up solutions for two different integro-differential operators,
both called {\it fractional Laplacian}. Operators of this sort have been intensively studied in recent years,
since they show many resemblances with the classical theory of elliptic operators, whereas they are often defined 
as integral operators with singular kernels and long-range interactions, making them {\it nonlocal} operators\footnote{Note 
that all differential operators, such as the classical Laplacian $-\lapl$, are purely local.}.
At the same time, they naturally appear in concrete applications, 
most of all to model phenomena where long-range interactions are present.
We have for example: phase transition models \cite{sire-valdinoci}, crystal dislocation \cite{toland}, 
the obstacle problem \cite{silvestre}, flame propagation \cite{slepcev},
minimal surfaces \cite{crs}, finance \cite{cont-tankov}, and materials science \cite{bates}.

\paragraph*{The fractional Laplacian.} We call the first one of these operators simply\footnote{It is sometimes called 
{\it restricted fractional Laplacian} when restricted only to functions who are supported in a fixed bounded region, 
see for example \cite{vaz}.}
{\it fractional Laplacian} and we denote it by $\Ds$.
Fixed $s\in (0,1)$, for a function\footnote{It is not not our purpose in this Introduction to 
provide thorough regularity or integrability assumptions on $u$ 
in order to make sense of the definitions:
let us work only formally at this stage.} $u:\R^N\to\R$ 
this operator can be defined in a number of different equivalent ways:
\begin{itemize}
\item as an integral operator with singular kernel \cite{hitchhiker}
\begin{equation}\label{def1}
\Ds u(x)\ =\ C_{N,s}\;p.v.\int_{\R^N}\frac{u(x)-u(y)}{|x-y|^{N+2s}}\:dy
\end{equation}
where ``$p.v.$'' 
stands for the integral in the principal value sense, $C_{N,s}$ is a positive renormalizing constant and
it equals (e.g. \cite[formula (A.1)]{fracpoho} and \cite[formula (1.3) and following]{pqvv})
\begin{equation}\label{Cns}
C_{N,s}\ =\ \frac{4^s\,\Gamma\!\left(\frac N2+s\right)}{\pi^{N/2}}\cdot\frac{s}{\Gamma(1-s)};
\end{equation}
\item as a pseudodifferential operator \cite{hitchhiker,grubb}
\begin{equation}\label{def2-intro}
\mathcal{F}\left[\Ds u\right](\xi)\ =\ |\xi|^{2s}\,\mathcal{F}u(\xi)
\end{equation}
where $\mathcal{F}$ denotes the Fourier transform;
\item as a Dirichlet-to-Neumann map \cite{extension}: let $U:\R^N\times[0,+\infty)$ solve
\[
\left\lbrace\begin{aligned}
\hbox{div}\left(t^{1-2s}\,U(x,t)\right) &=0\qquad x\in\R^N,\ t>0 \\
U(x,0) &=u(x)\qquad x\in\R^N
\end{aligned}\right.
\]
then
\begin{equation}\label{def3}
\Ds u(x)\ =\ -c(N,s)\lim_{t\downarrow 0} t^{1-2s}\,\partial_t U(x,t)
\end{equation}
where $c(N,s)$ is the positive renormalizing constant
\begin{equation}\label{cns-intro}
c(N,s)=\frac{4}{(N+2-2s)\,\pi^{N/2+s}\,\Gamma\left(N/2+s\right)};
\end{equation}
\item as the infinitesimal generator of the
heat semigroup $e^{t\lapl}$ in $\R^N$ subordinated in the sense of Bochner with
corresponding Bernstein function $\tau\mapsto\tau^s$ \cite[Section 13.1]{schilling} 
\begin{equation}\label{def4}
\Ds u(x)\ =\ \frac{s}{\Gamma(1-s)}\int_0^\infty\left(\frac{u(x)-e^{t\lapl}u(x)}{t}\right)\frac{dt}{t^s};
\end{equation}
\item as the infinitesimal generator of a symmetric $\alpha$-stable stochastic process ${(X_t)}_{t\geq0}$ for $\alpha=2s$
\begin{equation}\label{def5}
\Ds u(x)\ =\ \lim_{t\downarrow 0}\frac{u(x)-\mathbb{E}_x[u(X_t)]}{t}.
\end{equation}
\end{itemize}
The operator turns out to be ``nonlocal'' in the sense that to compute it (on a given function and at a fixed point)
it is not sufficient to know the values of the function in a neighbourhood of the point; conversely,
modifying the function far away from the point affects the value at the point itself 
(this feature is clear when using the representation in \eqref{def1});
this is not at all the case, for instance, for the classical Laplacian $-\lapl$.
The different, yet equivalent, definitions for $\Ds$ allow to take very different approaches in solving related problems,
creating a fruitful interplay between e.g. variational techniques, the theory of pseudodifferential operators,
functional analysis and potential theory.

\paragraph*{The spectral fractional Laplacian.} The second operator we deal with
is another nonlocal operator whose definition strongly depends
on the domain taken into account. So, fix a bounded domain $\Omega\subset\R^N$, $s\in(0,1)$
and denote by $\delta(x)=\dist(x,\partial\Omega)$.
For a function\footnote{As before, let us work only formally for the moment.} 
$u:\Omega\to\R$ the operator $\As$ can be defined
\begin{itemize}
\item as an integral operator with singular kernel \cite{song-vondra}
\begin{equation}\label{def6}
\As u(x)\ =\ p.v.\int_\Omega[u(x)-u(y)]\,J(x,y)\;dy+\kappa(x)u(x),\qquad x\in\Omega
\end{equation}
where $J(x,y)$ and $\kappa(x)$ are the functions given by
\begin{eqnarray*}
J(x,y) & = & \frac{s}{\Gamma(1-s)}\int_0^\infty\frac{p_\Omega(t,x,y)}{t^{1+s}}\;dt
\qquad \asymp\ \frac{\delta(x)\delta(y)}{\left|x-y\right|^{N+2s}}\cdot\frac{1}{|x-y|^2+\delta(x)\delta(y)} \\
\kappa(x) & = & \frac{s}{\Gamma(1-s)}\int_0^\infty\left(1-\int_\Omega p_\Omega(t,x,y)\right)\frac{dt}{t^{1+s}}
\qquad \asymp\ \frac{1}{{\delta(x)}^{2s}}
\end{eqnarray*}
and $p_\Omega(t,x,y)$ is the heat kernel of $\Omega$ associated to homogeneous boundary conditions;
\item via the spectral decomposition of $u$: given ${\{\varphi_j\}}_{j\in\N}$ a
Hilbert basis of $L^2(\Omega)$ consisting of the eigenfunctions of the Dirichlet Laplacian
$\left.-\lapl\right|_\Omega$ associated to the eigenvalues $\lambda_j,\ j\in\N$
\begin{equation}\label{def7}
\As u(x)\ =\ \sum_{j=1}^\infty \lambda_j^s\widehat u_j\:\varphi_j(x),\qquad \widehat u_j=\int_\Omega u_j\varphi_j;
\end{equation}
\item as a Dirichlet-to-Neumann map \cite{cabre-tan,cdds}: let $U:\Omega\times[0,+\infty)$ solve
\[
\left\lbrace\begin{aligned}
\hbox{div}\left(t^{1-2s}\,U(x,t)\right) &=0\qquad & & x\in\Omega,\ t>0 \\
U(x,t) &= 0 & & x\in\partial\Omega,\ t>0 \\
U(x,0) &= u(x) & & x\in\Omega
\end{aligned}\right.
\]
then
\begin{equation}\label{def8}
\As u(x)\ =\ -c(N,s)\lim_{t\downarrow 0} t^{1-2s}\,\partial_t U(x,t)
\end{equation}
where $c(N,s)$ is the positive renormalizing constant defined in \eqref{cns-intro};
\item as the infinitesimal generator of the
heat semigroup $e^{\left.t\lapl\right|_\Omega}$ in $\Omega$ 
subordinated in the sense of Bochner with
corresponding Bernstein function $\tau\mapsto\tau^s$ \cite[Section 13.1]{schilling} 
(and also \cite{gprssv})
\begin{equation}\label{def9}
\As u(x)\ =\ \frac{s}{\Gamma(1-s)}\int_0^\infty\left(\frac{u(x)-e^{\left.t\lapl\right|_\Omega}u(x)}{t}\right)\frac{dt}{t^s};
\end{equation}
\item as the infinitesimal generator of a subordinate killed Brownian motion ${(\widetilde X_t)}_{t\geq0}$
\begin{equation}\label{def10}
\As u(x)\ =\ \lim_{t\downarrow 0}\frac{u(x)-\mathbb{E}_x[u(\widetilde X_t)]}{t}.
\end{equation}
\end{itemize}
To our knowledge, definition \eqref{def7} is the most commonly used in the PDE framework.
Still it bears some restrictions on the domain of the operator, and 
no theory for nonhomogeneous boundary conditions for $\As$ was available until this work:
this is one of the contributions of this thesis, see Chapter \ref{spectral}. This is
why we will mostly exploit definitions \eqref{def6} and \eqref{def9}.

\paragraph*{Some stochastic interpretation and the names of the operators.}
As mentioned above, both operators reveal a deep connection with the theory of
stochastic processes. Indeed they generate L\'evy processes with jumps\footnote{This type 
of processes are sometimes called L\'evy flights in physics.}, i.e. whose paths are discontinuous.
The operator $\Ds$ generates a so called $(2s)$-stable process in $\R^N$,
and in particular the distribution law at a fixed time is homogeneous of order $2s$.
This process can be built also by subordinating a Brownian motion in time:
given a Brownian motion (a Wiener process) $(B_t)_{t\geq 0}$, 
one evaluates it at random times $(T_t)_{t\geq 0}$ rather than at the deterministic time,
which are almost surely increasing according to a linear $(2s)$-stable process. This way we
provide a new process $(B_{T_t})_{t\geq 0}$. The random time $T_t$ has not 
continuous trajectories in general. In this approach, we finally obtain a 
L\'evy process whose jumps are due to jumps taken in time, as we were recording
the random movement of a particle but we stop our camera randomly and then turn it on again
(and we might also speed up and slow down the time flow).
Recall that the generator of the Brownian motion is the classical (negative) Laplacian $-\lapl$.
If the subordinator is well chosen, then the stochastic operation of subordinating has as analytical counterpart
taking a fractional power of the generator: at this point $\Ds$ comes forth.
\medskip

The same kind of construction works for $\As$ but with an interesting difference. Before subordinating,
one kills the process upon leaving the domain $\Omega$. For this reason
we have written that it generates a {\it subordinate killed Brownian motion}.
The killed process is generated by the Dirichlet Laplacian $\left.-\lapl\right|_\Omega$
and the subordination provides its fractional power $\As$.
For this reason we have chosen this notation for the spectral fractional Laplacian
rather than the common one $A_s$.

Let us finally mention that we call it {\it spectral fractional Laplacian}
to comply with the usual way it is referred to in the bibliography. 
We might call it also {\it fractional Dirichlet Laplacian},
and this would take more into account the way it is constructed. The word {\it spectral}
might be misleading, since it recalls the definition via the spectral decomposition
which has some natural restrictions and does not apply
to the full class of functions we consider here.
In this spirit we might also call {\it Dirichlet fractional Laplacian} 
the operator obtained by restricting the fractional Laplacian to function
with null boundary conditions. Indeed, such an operator generates
a {\it killed subordinate Brownian motion}, i.e. the process 
obtained by killing the $(2s)$-stable L\'evy process upon leaving a fixed domain.
Mind that ``killed subordinate'' and ``subordinate killed'' are different in nature
and in the same way ``Dirichlet fractional'' and ``fractional Dirichlet'' would not be 
the same thing.

\paragraph*{Boundary blow-up solutions.} In the theory of semilinear elliptic equations,
functions solving
\begin{equation}\label{eq-intro}
-\lapl u+ f(u)=0,\qquad\hbox{in some }\Omega\subseteq\R^N\hbox{ open and bounded}
\end{equation}
coupled with the boundary condition
\[
\lim_{x\rightarrow\partial\Omega}u(x)=+\infty
\]
are known as {\it boundary blow-up solutions} or {\it large solutions}.
There is a huge amount of bibliography dealing with this problem
which dates back to the seminal work of Bieberbach \cite{bieberbach}, for $N=2$ and $f(u)=e^u$.
Keller \cite{keller} and Osserman \cite{osserman} independently established a sufficient and necessary
condition on a nondecreasing nonlinearity $f$ for the existence of a boundary blow-up solution:
it takes the form
\begin{equation}\label{ko-intro}
\int^{+\infty}\frac{dt}{\sqrt{F(t)}}\ <\ +\infty,\qquad\hbox{where }F'=f\geq 0
\end{equation}
and it is known as {\it Keller-Osserman condition}.
One can find these solutions with singular behaviour at the boundary in a number of applications:
for example, Loewner and Nirenberg \cite{ln} studied the case $f(u)=u^{(N+2)/(N-2)},\;N\geq3$,
which is strictly related to the {\it singular Yamabe problem} in conformal geometry,
while Labutin \cite{labutin} completely characterized the class of sets $\Omega$
that admit a large solution for $f(u)=u^q,\,q>1,$ with capacitary methods 
inspired by the theory of {\it spatial branching processes}, that are particular stochastic processes;
see also the purely probabilistic works
by Le Gall \cite{legall} and Dhersin and Le Gall \cite{dherslegall} dealing with $q=2$
and the delicate case of nonsmooth domains. Mselati \cite{mselati} completely classified
positive solutions in terms of their boundary trace,
which can be $+\infty$ on one part of the boundary and a measure that does not
charge sets of zero boundary capacity on the remaining part;
see also Dumont, Dupaigne, Goubet and R{\u{a}}dulescu \cite{KO-dupaigne} for the case of oscillating nonlinearity.
We refer to the very recent book by Marcus and V\'eron \cite[Chapters 5 and 6]{marcus-veron-book}
for further readings and developments in this direction.

\paragraph*{Large $s$-harmonic functions, integration by parts formulas 
and weak $L^1$ theories.} 
In the fractional context, our starting point is that
boundary blow-up solutions arise even in linear problems.
In particular, it is possible to provide
{\it large $s$-harmonic functions}, i.e. functions satisfying
\begin{equation}\label{problem-intro}
\left\lbrace\begin{aligned}
\Ds u &= 0 & & \hbox{in }\Omega \\
u &= 0 & & \hbox{in }\R^N\setminus\overline{\Omega}\\
\lim_{\stackrel{\hbox{\scriptsize $x\!\rightarrow\!\partial\Omega$}}{x\in\Omega}}u(x)
&= +\infty 
\end{aligned}\right.
\qquad\hbox{and}\qquad
\left\lbrace\begin{aligned}
\As u &=0 & & \hbox{in }\Omega \\
\lim_{\stackrel{\hbox{\scriptsize $x\!\rightarrow\!\partial\Omega$}}{x\in\Omega}}u(x) &=+\infty
\end{aligned}\right.
\end{equation}
see e.g. Lemmas \ref{expl-sol}, \ref{harm} and \eqref{h1-behav} below.
\medskip

These examples show on the one hand how problems for $\Ds$ where only outer values are prescribed
are ill-posed in the classical sense and different kinds of boundary conditions have
to be taken into account, while for $\As$ they show the need for 
nonzero boundary conditions.
This means in turn that we need notions of weak solutions
that can deal with these new boundary data.
We provide answers to these questions respectively in Section \ref{liner-sec} and Theorem \ref{point}:
the trace is a limit at the boundary weighted with 
powers of the distance to the boundary,
$\delta(x)^{1-s}u(x)$ in the case of the fractional Laplacian
and $\delta(x)^{2-2s}u(x)$ in the spectral fractional Laplacian case.
We refer also to Grubb \cite[e.g. Theorems 4 and 5]{grubb}
for very recent results on the fractional Laplacian, while we mention the 
work by Dhifli, M\^aagli and Zribi \cite{dhifli} for a first attempt on
the spectral fractional Laplacian.
\medskip

In both cases we needed to develop a weak $L^1$-theory of Stampacchia's sort with measure data,
able to treat this class of $s$-harmonic functions as well. 
We consider the following nonhomogeneous problems
\[
\left\lbrace\begin{aligned}
\Ds u &= \lambda & & \hbox{in }\Omega \\
u &= \mu & & \hbox{in }\R^N\setminus\overline\Omega \\
\frac{u}{w_1} &= \nu & & \hbox{on }\partial\Omega
\end{aligned}\right.
\qquad\hbox{and}\qquad
\left\lbrace\begin{aligned}
\As u &= \xi & & \hbox{in }\Omega \\
\frac{u}{w_2} &= \zeta & & \hbox{on }\partial\Omega
\end{aligned}\right.
\]
and we show how are well-posed when $\lambda,\xi\in\mathcal{M}(\Omega),\ \mu\in\mathcal{M}(\R^N\setminus\overline\Omega),
\ \nu,\zeta\in\mathcal{M}(\partial\Omega)$ are Radon measures
satisfying assumptions
\begin{align*}
& \int_\Omega \delta^s\;d|\lambda|<+\infty,\qquad \int_\Omega\delta\;d|\xi|<+\infty \\
& \int_{\R^N\setminus\overline\Omega}\delta^{-s}\min\{1,\delta^{-N-s}\}\;d|\mu|<+\infty \\
& |\nu|(\partial\Omega)<+\infty,\qquad |\zeta|(\partial\Omega)<+\infty
\end{align*}
and $w_1$ and $w_2$ are reference functions\footnote{Here
the boundary datum has to be taken as the limit as approaching a point on the boundary.} defined in $\Omega$
whose behaviour is comparable respectively with
$\delta^{-(1-s)}$ and $\delta^{-(2-2s)}$.
The solutions are functions satisfying respectively
\[
\begin{aligned}
& \int_\Omega u\Ds\phi\ =\ \int_\Omega \phi\;d\lambda
-\int_{\R^N\setminus\Omega}\Ds\phi\;d\mu+\int_{\partial\Omega}\left.\frac{\phi}{\delta^s}\right|_{\partial\Omega}d\nu & 
u\in L^1(\Omega),\ \phi\in\T_1(\Omega), \\
& \int_\Omega u\As\phi\ =\ \int_\Omega \phi\;d\xi
-\int_{\partial\Omega}\frac{\partial\phi}{\partial\nu}\;d\zeta & u\in L^1(\Omega,\delta(x)dx),
\ \phi\in\T_2(\Omega)
\end{aligned}
\]
and the test function spaces are defined as
\begin{align*}
& \T_1(\Omega)\ =\ \left\lbrace\phi\in C^s(\R^N):
\left\lbrace\begin{aligned}
\Ds\phi &= \psi & &\hbox{in }\Omega,\ \psi\in C^\infty_c(\Omega) \\
\phi &= 0 & &\hbox{in }\R^N\setminus\overline{\Omega} \\
\delta^{1-s}\phi &= 0 & & \hbox{on }\partial\Omega
\end{aligned}\right.\right\rbrace, \\
& \T_2(\Omega)\ =\ \left\lbrace \phi\in C^1_0(\overline{\Omega}):
\As\phi = \psi \hbox{ in }\Omega,\ \psi\in C^\infty_c(\Omega)\right\rbrace.
\end{align*}
These results are contained respectively in Sections \ref{stampacchia-frac} and \ref{dir-sect}.
These definitions are inspired by and, in a sense, extend some
integration by parts formulas.
The techniques are mainly based on potential analysis,
exploiting the representation via the Green functions, Poisson and Martin kernels,
whose constructions are included in the treatment.
A more precise presentation of the results and the rigorous statements of the main findings 
are included in the introductory Chapters \ref{intropartone} and \ref{introparttwo}.
\medskip

\paragraph*{Nonlinear problems.} So we have seen how boundary singularities can
show up even in presence of quite smooth data.
When turning to nonlinear problems, this lack of smoothness at the boundary
provide a wide class of solutions with a {\it prescribed} singularity at the boundary,
presented in Chapter \ref{nonlin-sec} and Section \ref{nonlin-sect}. There, we deal with problems
\[
\left\lbrace\begin{aligned}
\Ds u &= f(x,u) & & \hbox{in }\Omega \\
u &= g & & \hbox{in }\R^N\setminus\overline{\Omega} \\
\frac{u}{w_1} &= h & & \hbox{on }\partial\Omega
\end{aligned}\right.
\qquad\hbox{and}\qquad
\left\lbrace\begin{aligned}
\As u &= f(x,u) & & \hbox{in }\Omega \\
\frac{u}{w_2} &= h & & \hbox{on }\partial\Omega
\end{aligned}\right.
\]
In these cases, large $s$-harmonic functions can be used as sub- or supersolutions
whenever the sign of the nonlinearity $f$ allows it.
Indeed we solved these nonlinear problems using some sub- and supersolution methods
by adapting the arguments in Cl\'ement and Sweers \cite{clem-sweers} and
Montenegro and Ponce \cite{ponce}.
\medskip

Taking, for example, datum $h\equiv 1$ will 
provide solutions with a boundary explosion 
of the rate $\delta^{s-1}$ and $\delta^{2s-2}$ respectively.
These solutions are not {\it large} though, in the following sense.
In the classical theory for the Laplacian,
large solutions are generated by a compensation between the boundary condition
- which tries to push the solution to infinity - and the nonlinearity
- which pulls it down, in view of the negative sign. The nonlinearity must provide 
a sufficiently strong answer to the boundary condition, and in this 
spirit the Keller-Osserman condition quantifies the necessary rate 
in order for this phenomenon to take place.
Yet so far, in the fractional context, 
we have actually introduced solutions whose boundary singularity
is just due to linear features of the operators.
So, at this stage one question is still left unanswered: is it possible to find 
solutions with a singularity at the boundary generated 
not by the features of the operator,
but rather due to a compensation phenomenon?

\paragraph*{Large solutions for fractional Laplacian operators.}
We produce initial answers in Chapters \ref{fracKO} and \ref{specKO},
stating sufficient conditions on the nonlinearity in order 
for solutions to $\Ds u+f(u)=0$ and $\As u+f(u)=0$
to exist, with a singularity at the boundary
which is higher order with respect to the ones described 
in the previous paragraph, so that 
the singular boundary conditions read as
\[
\frac{u}{w_1}=+\infty
\hbox{ on }\partial\Omega,
\qquad\frac{u}{w_2}=+\infty
\hbox{ on }\partial\Omega
\] 
which are equivalent to
\[
\delta^{1-s}u=+\infty
\hbox{ on }\partial\Omega,
\qquad\delta^{2-2s}u=+\infty
\hbox{ on }\partial\Omega.
\]

The case of the power nonlinearity $f(u)=u^p$ for $\Ds$
has been studied by Felmer and Quaas \cite{felmer-quaas} and
Chen, Felmer and Quaas \cite{chen-felmer}.
We extend those results by providing explicit bounds on $p$ for the existence,
which are
\begin{equation}\label{prange-frac}
p\in\left(1+2s,1+\frac{2s}{1-s}\right),
\end{equation}
allowing more general nonlinearities
and a clear framework to work in,
especially when speaking of boundary conditions.
We also prove the following pointwise bound holding throughout the domain
\[
\int_{u(x)}^\infty\frac{dt}{\sqrt{F(t)}}\ \geq\ c\,\delta(x)^s,
\qquad F'=f;
\]
this upper bound is computed via an explicit supersolution for the problem.
In the particular case where $f(u)=u^p$ this yields
\[
u\leq C\,\delta^{-2s/(p-1)}.
\]

In the case of the spectral fractional Laplacian,
we study power-like nonlinearities.
We deal with
\[
\left\lbrace\begin{aligned}
\As u &= -u^p & \hbox{in }\Omega \\
\delta^{2-2s}u &=+\infty & \hbox{on }\partial\Omega
\end{aligned}\right.
\]
proving the existence of a solution $u\in L^1(\Omega,\delta(x)dx)$ for
\begin{equation}\label{prange-spectral}
p\in\left(1+s,1+\frac{s}{1-s}\right).
\end{equation}
Also in this case we provide the following upper bound, coming
from the construction of a supersolution, on the solution 
\[
u\ \leq\ C\,\delta^{-2s/(p-1)}\qquad \hbox{in }\Omega.
\]
This rate coincides with the case of the fractional Laplacian
and is sharp in the case $s=1$.
\medskip

Let us mention that we suspect \eqref{prange-frac} and \eqref{prange-spectral}
to be optimal, but at present we cannot provide a proof. 
When $p$ is too large (i.e. it exits the ranges from above), then the candidate
supersolution will not fulfil the boundary condition. Moreover,
in this case no solution with bounded boundary datum can exist (see Theorem \ref{-SIGN} 
and Lemma \ref{nonexist} below)
so we do not expect the large problem to admit a solution either. 
Conversely, when $p$ is too small, then we
lose the necessary integrability on the candidate supersolution
in order to compute its (spectral or not) fractional Laplacian.
In both cases we lose some tools necessary for a sub- and supersolution's type method,
such as comparison principles, uniqueness of solutions, precise boundary
behaviour and the ability to compute sharp subsolutions. 
These lacks indicate that the theory is just at the beginning 
and much work can be still done.
\medskip

When looking at $f(u)=u^p$ both for $\Ds$ and $\As$,
a natural question would be if the ranges \eqref{prange-frac} and \eqref{prange-spectral}
of admissible exponents $p$'s
asymptotically converge as $s\uparrow 1$ to the 
set of admissible exponents for $-\lapl$, which is given
by \eqref{ko-intro} and simply reduces to $p\in(1,+\infty)$.
The answer is clearly {\it no}. While both
\[
1+\frac{2s}{1-s}\rightarrow+\infty\qquad\hbox{and}
\qquad 1+\frac{s}{1-s}\rightarrow+\infty\qquad\qquad
\hbox{as }s\uparrow 1,
\]
we only have
\[
1+2s\rightarrow 3\qquad\hbox{and}\qquad
1+s\rightarrow 2\qquad\qquad
\hbox{as }s\uparrow 1.
\]
This is not discouraging, though.
In both fractional settings we need respectively that
$u\in L^1(\Omega)$ and $u\in L^1(\Omega,\delta(x)dx)$
to make sense of the operator, see for example definitions
\eqref{def1} and \eqref{def6}. This is an additional assumption 
we do not have in the classical problem, so it is reasonable 
to get smaller ranges for $p$.
Moreover, the classical solution
to the large problem is known to behave like
\[
u\ \asymp\ \delta^{-2/(p-1)}
\]
and such a $u$ is $L^1(\Omega)$ when $p>3$,
while it belongs to $L^1(\Omega,\delta(x)dx)$ for $p>2$;
so, in this sense, we actually have the asymptotic
convergence of the admissible ranges of $p$.

\paragraph*{Nonlocal directional curvatures.}
In the final chapter of the manuscript we introduce a fractional
notion of directional curvature for $(N-1)$-dimensional smooth manifolds in $\R^N$.
This interest is motivated by the fractional Laplacian context in 
which several classical problems have been recently rephrased
by attracting the attention of a large number of researchers.
In particular, in the differential geometry and geometric analysis framework,
several new results have been recently obtained
for the {\it diffusion by mean curvature} and the closely related problem of {\it minimal surfaces}
in a nonlocal setting.

As we shall see, the
concept of nonlocal mean curvature is naturally associated to a suitable average of
appropriate nonlocal directional curvatures, which asymptotically
approach their classical counterpart. On the other hand,
the nonlocal curvatures seem to have a more messy and fanciful behaviour than the classical ones.
In particular we show how, 
differently from the classical case, minimal and maximal nonlocal directional curvatures
are not in general attained at orthogonal directions and 
the set of extremal directions for nonlocal directional curvatures can be prescribed
somehow arbitrarily\footnote{The precise statements of these results are listed
in Paragraph~\ref{99}.}.
Some motivations for introducing these new objects can be found in Paragraph \ref{JJ}.

\part{Boundary blow-up solutions for \texorpdfstring{$\Ds$}{Ds}}\label{partone}

\chapter{Overview and main results}\label{intropartone}

In this first Part we study a suitable notion of weak solution
to semilinear problems driven by the fractional Laplacian $\Ds$, i.e.
the integral operator defined as (see e.g. Di Nezza, Palatucci and Valdinoci \cite{hitchhiker} for an introduction)
\begin{equation}
\Ds f(x)=C_{N,s}\,PV\int_{\R^N}\frac{f(x)-f(y)}{{|x-y|}^{N+2s}}\;dy.
\end{equation}
where $C_{N,s}$ is the normalizing constant defined in \eqref{Cns}.

In order to do this, we will need to develop a theory
for the Dirichlet problem for the fractional Laplacian with measure data
(see Karlsen, Petitta and Ulusoy \cite{karlsen} and Chen and V\'eron \cite{chen-veron}, for earlier results
in this direction; in particular, look at Paragraph \ref{chenv} for a comparison with \cite{chen-veron}).
We pay particular attention to those solutions having an explosive behaviour at 
the boundary of the prescribed domain,
known in the literature as \it large solutions \rm or also
\it boundary blow-up solutions\rm.

Let us recall that in the classical setting (see Axler, Bourdon and Ramey \cite[Theorem 6.9]{axler}),
to any nonnegative Borel measure $\mu\in\mathcal{M}(\partial B)$ on $\partial B$ it is possible to associate,
via the representation through the Poisson kernel,
a harmonic function in $B$ with $\mu$ as its trace on the boundary.
Conversely,
any positive harmonic function on the ball $B$
has a trace on $\partial B$ that is a nonnegative Borel measure
(see \cite[Theorem 6.15]{axler}).

When studying the semilinear problem for the Laplacian,
solutions can achieve the boundary datum
$+\infty$ on the whole boundary. More precisely,  take $\Omega$ a bounded smooth domain
and $f$ nondecreasing such that $f(0)=0$. According to the works of
Keller \cite{keller} and Osserman \cite{osserman},
the equation 
$$
\left\lbrace\begin{array}{l}
\displaystyle
\lapl u=f(u) \ \hbox{ in }\Omega\subseteq\R^N \\
\displaystyle
\lim_{\stackrel{\hbox{\scriptsize $x\!\rightarrow\!\partial\Omega$}}{x\in\Omega}}u(x)=+\infty
\end{array}\right.
$$
has a solution if and only if $f$ satisfies the so called \it Keller-Osserman condition \rm
\eqref{ko-intro}.

In the fractional context, our starting point is that
it is possible to provide
functions satisfying
\begin{equation}\label{problem}
\left\lbrace\begin{array}{l}
\displaystyle
\Ds u=0 \ \hbox{ in some open bounded region }\Omega\subseteq\R^N, \\
\displaystyle
\lim_{\stackrel{\hbox{\scriptsize $x\!\rightarrow\!\partial\Omega$}}{x\in\Omega}}u(x)=+\infty.
\end{array}\right.
\end{equation}
An example of a large $s$-harmonic function on the unit ball $B$ 
is
$$
u_\sigma(x)=\left\lbrace\begin{array}{ll}
\displaystyle
\frac{c(N,s)}{\left(1-|x|^2\right)^\sigma} & \hbox{ in }B \\
\displaystyle
\frac{c(N,s+\sigma)}{\left(|x|^2-1\right)^\sigma} & \hbox{ in }\R^N\setminus\overline{B}
\end{array}\right.\quad\sigma\in(0,1-s),\ \ c(N,s)=\frac{\Gamma(N/2)\,\sin(\pi s)}{\pi^{1+N/2}}.
$$
See Lemma \ref{expl-sol} below.
Moreover, letting $\sigma\rightarrow1-s$ we recover
the following example found in Bogdan, Byczkowski, Kulczycki,
Ryznar, Song, and Vondra{\v{c}}ek \cite{sharm},
qualitatively different from the previous one:
$$
u_{1-s}(x)=\left\lbrace\begin{array}{ll}
\displaystyle
\frac{c(N,s)}{\left(1-|x|^2\right)^{1-s}} & \hbox{ in }B, \\
\displaystyle
0 & \hbox{ in }\R^N\setminus\overline{B}.
\end{array}\right.
$$
The function $u_{1-s}$ so defined satisfies
$$
\left\lbrace\begin{aligned}
\Ds u &=0 & &\hbox{in }B \\
u &=0 & & \hbox{in }\R^N\setminus\overline{B}
\end{aligned}\right.
$$
and shows how problems where only outer values are prescribed
are ill-posed in the classical sense. Different kinds of boundary conditions have
to be taken into account:
indeed, in the first case we have an $s$-harmonic function
associated to the prescribed data of $u_\sigma$ outside $B$;
in the second case all the mass of the boundary datum 
concentrates on $\partial B$.
This means that we need a notion of weak solution
that can deal at the same time with these two different boundary data,
one on the complement of the domain
and the other one on its boundary.
We refer to Grubb \cite[e.g. Theorems 4 and 5]{grubb}
for very recent results on the regularity theory for the fractional Laplacian,
dealing with this kind of problems. 


\paragraph{The notion of $s$-harmonic function.}
We start with the definition of $s$-harmonicity, found in Landkof \cite[\textsection 6.20]{landkof}.
Denote by
$$
\eta(x)=\left\lbrace\begin{array}{ll}
\displaystyle
\frac{c(N,s)}{{|x|}^N\,\left(|x|^2-1\right)^s} & |x|>1 \\
0 & |x|\leq 1
\end{array}\right.
$$
and by
$$
\eta_r(x)=\frac{1}{r^N}\;\eta\left(\frac{x}{r}\right)=\left\lbrace\begin{array}{ll}
\displaystyle
\frac{c(N,s)\,r^{2s}}{{|x|}^N\,\left(|x|^2-r^2\right)^s} & |x|>r, \\
0 & |x|\leq r.
\end{array}\right.
$$
The constant $c(N,s)$ above is chosen in such a way that
\begin{equation}\label{cns}
\int_{\R^N}\eta(x)\;dx=\int_{\R^N}\eta_r(x)\;dx=1
\end{equation}
and therefore, see \cite[\textsection 6.19]{landkof}
\begin{equation}\label{cns2}
c(N,s)\ =\ \frac{\Gamma(N/2)\,\sin(\pi s)}{\pi^{1+N/2}}.
\end{equation}

The definition of $s$-harmonicity is given via a
mean value property, namely

\begin{defi}[$s$-harmonic function\index{$s$-harmonic function}]\label{sharm-def} 
We say that a measurable nonnegative function 
$u:\R^N\rightarrow[0,+\infty]$ is
$s$-harmonic on an open set $\Omega\subseteq\R^N$ if $u\in C(\Omega)$ and
for any $x\in\Omega$ and $0<r<\dist(x,\partial\Omega)$
\begin{equation}\label{sharm-eq}
u(x)\ =\ \int_{\R^N\setminus B_r(x)}\frac{c(N,s)\,r^{2s}}{|y-x|^N\,\left(|y-x|^2-r^2\right)^s}\:u(y)\;dy
\ =\ \left(\eta_r*u\right)(x).
\end{equation}
\end{defi}

\section{an integration by parts formula}\label{intparts-sec}
For any two functions $u,\ v$ in the Schwartz class $\mathcal{S}$,
the self-adjointness of the operator $\Ds$ entails
$$
\int_{\R^N}u\,\Ds v\ =\ \int_{\R^N}v\,\Ds u,
$$
this follows from the representation of the fractional Laplacian
via the Fourier transform,
see \cite[Paragraph 3.1]{hitchhiker}. By splitting $\R^N$ into two domains of integration 
\begin{equation}\label{intparts}
\int_{\Omega}u\,\Ds v-\int_{\Omega}v\,\Ds u\ =\ 
\int_{\R^N\setminus\Omega}v\,\Ds u-\int_{\R^N\setminus\Omega}u\,\Ds v.
\end{equation}

Denote by 
\begin{multline*}
C^{2s+\ep}(\Omega)=\{v:\R^N\rightarrow\R \hbox{ such that } v\in C(\Omega)\hbox{ and for any }K\subset\subset\Omega\\
\hbox{ there exists }\alpha=\alpha(K,v)\hbox{ such that }v\in C^{2s+\alpha}(K) \}.
\end{multline*}

\begin{prop}\label{integrbypartsform}
Let $\Omega\subseteq\R^N$ open and bounded. 
If $u\in C^{2s+\ep}(\Omega)\cap L^\infty(\R^N)$ and 
$v=0$ in $\R^N\setminus\Omega$, $v\in C^{2s+\ep}(\Omega)\cap C^s(\R^N)$
and $\Ds v\in L^1(\R^N)$, then 
\begin{equation}\label{intpartsprop}
\int_{\Omega}u\,\Ds v-\int_{\Omega}v\,\Ds u\ =\ -\int_{\R^N\setminus\Omega}u\,\Ds v.
\end{equation}
\end{prop}
The proof can be found in Paragraph \ref{byparts-proof}.

From now on the set $\Omega\subseteq\R^N$ will be an
open bounded domain with $C^{1,1}$ boundary.
More generally, we will prove in Section \ref{liner-sec} the following
\begin{prop}\label{intparts-prop} Let 
$\delta(x)=\dist(x,\partial\Omega)$ for any $x\in\R^N$,
and $u\in C^{2s+\ep}(\Omega)$ such that
\begin{equation}\label{gg}
{\delta(x)}^{1-s}u(x)\in C(\overline{\Omega})
\qquad\hbox{ and }\qquad
\int_{\C\Omega}\frac{|u(x)|}{1+{|x|}^{N+2s}}\cdot\frac{dx}{{\delta(x)}^s}<+\infty.
\end{equation}
Let
\begin{description}
\item[\it a)] $G_\Omega(x,y)$, $x,y\in\Omega,\,x\neq y$, be the Green function 
of the fractional Laplacian on $\Omega$, that is
$$
G_\Omega(x,y)\ =\ \frac{-C_{N,-s}}{{|x-y|}^{N-2s}}-H_\Omega(x,y)
$$ 
and $H_\Omega$ is the unique function in $C^s(\R^N)$ solving
$$
\left\lbrace\begin{aligned}
\Ds H_\Omega(x,\cdot) &=0 & & \hbox{in }\Omega \\
H_\Omega(x,y) &=\frac{-C_{N,-s}}{{|x-y|}^{N-2s}} & & \hbox{in }\R^N\setminus\overline{\Omega}
\end{aligned}\right.
$$
pointwisely,\footnote{The construction of $H$ can be found in the proof of Theorem \ref{meanvalue} below.}
\item[\it b)] $P_\Omega(x,y)=-\Ds G_\Omega(x,y)$, $x\in\Omega,y\in\R^N\setminus\overline{\Omega}$,
be the corresponding Poisson kernel (where the fractional Laplacian is computed in the variable $y\in\C\Omega$ 
and treating $x\in\Omega$ as a fixed parameter),
\item[\it c)] for $x\in\Omega,\theta\in\partial\Omega$,\footnote{This is a readaptation of the Martin kernel of $\Omega$.} 
$$
M_\Omega(x,\theta)\ =\ 
\lim_{\stackrel{\hbox{\scriptsize $y\!\rightarrow\!\theta$}}{y\in\Omega}}\frac{G_\Omega(x,y)}{{\delta(y)}^s},
$$
\item[\it d)] for $\theta\in\partial\Omega$, 
\[
Eu(\theta)\ =\ 
\lim_{\stackrel{\hbox{\scriptsize $x\!\rightarrow\!\theta$}}{x\in\Omega}}
\frac{u(x)}{\int_{\partial\Omega}M_\Omega(x,\theta')\;d\mathcal{H}(\theta')},
\] 
where the limit is well-defined in view of Lemma \ref{Eu} below.
\end{description}
Then the integration by parts formula
\begin{equation}\label{byparts}
\int_{\Omega}u\,\Ds v\ =\ 
\int_{\Omega}\Ds u\,v-\int_{\R^N\setminus\overline{\Omega}}u\,\Ds v
+\int_{\partial\Omega}Eu(\theta)\,D_sv(\theta)\;d\mathcal{H}(\theta)
\end{equation}
holds, where
$$
D_sv(\theta)\ =\ 
\lim_{\stackrel{\hbox{\scriptsize $x\!\rightarrow\!\theta$}}{x\in\Omega}}\frac{v(x)}{{\delta(x)}^s},
\quad\theta\in\partial\Omega
$$
for any $v\in C^s(\R^N)$ such that $\Ds v|_\Omega\in C^\infty_c(\Omega)$ 
and $v\equiv 0$ in $\R^N\setminus\Omega$.
Such a limit exists and is continuous in $\theta$
in view of Lemma \ref{testspace}.
In addition, we have the representation formula
\begin{equation}\label{representation}
	u(x)=\int_\Omega G_\Omega(x,y)\Ds u(y)\; dy
	-\int_{\R^N\setminus\overline{\Omega}}\Ds G_\Omega(x,y)\,u(y)\;dy+
	\int_{\partial\Omega}D_sG_\Omega(x,\theta)\, Eu(\theta)\;d\mathcal{H}(\theta).
\end{equation}
\end{prop}

The present Part is organized as follows. 
Section \ref{meanvalue-sec} relates Definition \ref{sharm-def}
with the fractional Laplacian $\Ds$.
Section \ref{liner-sec} recalls some facts on Green functions and
Poisson kernels and it studies the linear Dirichlet problem
both in the pointwise and in a weak sense.
Chapter \ref{nonlin-sec} deals with the nonlinear problem.
Now, let us outline the main results in Section \ref{liner-sec}
and Chapter \ref{nonlin-sec}.

\section{the dirichlet problem}

For a fixed $x\in\Omega$ the Poisson kernel satisfies 
$$
0\leq-\Ds G_\Omega(x,y)\leq c\,\min\left\lbrace{\delta(y)}^{-s},{\delta(y)}^{-N-2s}\right\rbrace,
\qquad y\in\R^N\setminus\overline{\Omega}
$$
for some constant $c>0$ independent of $x$ and $y$,
as a consequence of \eqref{est-pois} and \eqref{est-pois2} below.
In particular any Dirichlet condition $u=g$ in $\R^N\setminus\overline{\Omega}$
satisfying \eqref{gintro} below is admissible
in the representation formula \eqref{representation}. 
We prove the following

\begin{theo}\label{pointwise} Let $f\in C^\alpha(\overline{\Omega})$
for some $\alpha\in(0,1)$, 
$g:\R^N\setminus\overline{\Omega}\rightarrow\R$ be any measurable function 
satisfying\footnote{Compare also with equation \eqref{g} below.}
\begin{equation}\label{gintro}
\int_{\R^N\setminus\overline{\Omega}}|g(y)|\min\left\lbrace{\delta(y)}^{-s},{\delta(y)}^{-N-2s}\right\rbrace\;dy<+\infty
\end{equation}
and $h\in C(\partial\Omega)$.
Then, the function defined by setting
\begin{equation}\label{linearsol}
u(x)=
\left\lbrace\begin{array}{ll}
\begin{aligned} 
& \int_\Omega f(y)\,G_\Omega(y,x)\;dy
-\int_{\R^N\setminus\overline{\Omega}}g(y)\,\Ds G_\Omega(x,y)\;dy\ + \\
&\qquad\qquad\qquad\qquad\qquad\qquad\qquad
+\int_{\partial\Omega}h(\theta)\,M_\Omega(x,\theta)\;d\mathcal{H}(\theta),
\end{aligned} & \quad x\in\Omega \\
g(x), & \quad x\in\R^N\setminus\overline\Omega 
\end{array}\right.
\end{equation}
belongs to $C^{2s+\ep}(\Omega)$, fulfils \eqref{gg}, and $u$ is the only pointwise solution of 
$$
\left\lbrace\begin{aligned}
\Ds u &= f & &\hbox{in }\Omega, \\
u &= g & &\hbox{in }\R^N\setminus\overline{\Omega}, \\
Eu &= h & & \hbox{on }\partial\Omega.
\end{aligned}\right.
$$
Here $Eu$ has been defined in Proposition \ref{intparts-prop}.
Moreover, if $g\in C(\overline{V_\ep})$ for some $\ep>0$,
where $V_\ep=\{x\in\R^N\setminus\overline{\Omega}:\delta(x)<\ep\}$
and $h=0$, then $u\in C(\overline{\Omega})$.
\end{theo}

\begin{rmk}\rm Even if it is irrelevant to write $\int_{\R^N\setminus\Omega}$
or $\int_{\R^N\setminus\overline{\Omega}}$ in formula \eqref{linearsol}, with the latter notation
we would like to stress on the fact that the boundary $\partial\Omega$
plays an important role in this setting.
\end{rmk}

\begin{rmk}\rm Since $\partial\Omega\in C^{1,1}$, we will exploit the behaviour
of the Green function and the Poisson kernel
described by \cite[Theorem 2.10 and equation (2.13) resp.]{chen}:
there exist $c_1=c_1(\Omega,s)>0$, $c_2=c_2(\Omega,s)$ such that
\begin{equation}\label{est-poisson}
\begin{split}
\frac{{\delta(x)}^s}{c_1{\delta(y)}^s\left(1+\delta(y)\right)^s{|x-y|}^N}
& \leq\ -\Ds G_\Omega(x,y) \ \leq\\
& \leq\ \frac{c_1{\delta(x)}^s}{{\delta(y)}^s\left(1+\delta(y)\right)^s{|x-y|}^N},
\end{split}
\qquad x\in\Omega,\ y\in\R^N\setminus\overline\Omega,
\end{equation}
and
\begin{equation}\label{est-green}
\begin{split}
\frac{1}{c_2\,|x-y|^N}\left(|x-y|^2\,\wedge\,\delta(x)\delta(y)\right)^s
&\leq\ G_\Omega(x,y)\ \leq \\
&\leq\ \frac{c_2}{|x-y|^N}\left(|x-y|^2\,\wedge\,\delta(x)\delta(y)\right)^s
\end{split}
\qquad x,y\in\Omega.
\end{equation}
\end{rmk}

\begin{rmk}[construction of large $s$-harmonic functions]\label{largesharm}\rm 
The case when $f=0$ in Theorem \ref{pointwise} corresponds to 
$s$-harmonic functions: when $h\not\equiv 0$ then $u$ automatically 
explodes somewhere on the boundary (by definition of $E$), while
if $h\equiv 0$ then large $s$-harmonic functions can be built
as follows.
Take any positive $g$ satisfying \eqref{gintro} with
$$
\lim_{\delta(x)\downarrow 0}g(x)\ =\ +\infty
$$
and let $g_n=\min\{g,n\}$, $n\in\N$.
By Theorem \ref{pointwise}, the corresponding solutions $u_n\in C(\overline{\Omega})$.
In particular, $u_n=n$ on $\partial\Omega$ and by the Maximum Principle
${\{u_n\}}_{n\in\N}$ is increasing. Hypothesis \eqref{gintro} guarantees a uniform pointwise bound far from $\partial\Omega$
on ${\{u_n\}}_{n\in\N}$. Then 
$$
u_n(x)\ =\ -\int_{\R^N\setminus\Omega}g_n(y)\Ds G_\Omega(x,y)\;dy
$$
increases to a $s$-harmonic function $u$ such that
$$
\liminf_{\stackrel{\hbox{\scriptsize $x\in\Omega$}}{x\rightarrow\partial\Omega}}u(x)\geq
\liminf_{\stackrel{\hbox{\scriptsize $x\in\Omega$}}{x\rightarrow\partial\Omega}}u_n(x)=n,
\qquad \hbox{ for any }n\in\N.
$$
\end{rmk}

Next, in view of Theorem \ref{pointwise},
we introduce the test function space
\begin{multline*}
\T(\Omega)=\lbrace\phi\in C^\infty(\Omega):
\hbox{ there exists }\psi\in C^\infty_c(\Omega)
\hbox{ such that } \\
\Ds \phi=\psi \hbox{ in }\Omega,\ 
\phi=0 \hbox{ in }\R^N\setminus\overline\Omega,\ 
E\phi=0 \hbox{ on }\partial\Omega\rbrace.
\end{multline*}
Note that $\T(\Omega)\subseteq C^s(\R^N)$, by \cite[Proposition 1.1]{rosserra}.
Starting from the integration by parts formula \eqref{byparts},
we introduce the following notion of weak solution  

\begin{defi}[weak $L^1$ solution]\index{weak $L^1$ solution}\label{weakdefiintro} 
Given three Radon measures $\lambda\in\mathcal{M}(\Omega)$, 
$\mu\in\mathcal{M}(\R^N\setminus\Omega)$ and $\nu\in\mathcal{M}(\partial\Omega)$, such that
\[
\int_\Omega\delta(x)^s\;d|\lambda|(x)<+\infty,\ 
\int_{\R^N\setminus\overline\Omega}\min\{\delta(x)^{-s},\delta(x)^{-N-2s}\}\;d|\mu|(x)<+\infty,
\ |\nu|(\Omega)<+\infty
\]
we say that a function $u\in L^1(\Omega)$
is a solution of 
$$
\left\lbrace\begin{aligned}
\Ds u &=\lambda & & \hbox{in }\Omega \\
u &=\mu & & \hbox{in }\R^N\setminus\overline{\Omega} \\
Eu &=\nu & & \hbox{on }\partial\Omega
\end{aligned}\right.
$$
if for every $\phi\in\T(\Omega)$ it holds
\begin{equation}\label{weakdefi-eq}
\int_\Omega u(x)\Ds \phi(x)\;dx\ =\ 
\int_\Omega\phi(x)\;d\lambda(x)
-\int_{\R^N\setminus\overline{\Omega}}\Ds \phi(x)\;d\mu(x)+
\int_{\partial\Omega}D_s\phi(\theta)\;d\nu(\theta).
\end{equation}
\end{defi}
The integrals in the definition are finite for any $\phi\in\T(\Omega)$
in view of Lemma \ref{testspace}.

Thanks to the representation formula \eqref{representation},
we prove

\begin{theo}\label{existence-weak2}
Given two Radon measures $\lambda\in\mathcal{M}(\Omega)$
and 
$\mu\in\mathcal{M}(\R^N\setminus\overline{\Omega})$ such that
$$
\int_\Omega G_\Omega(x,y)\;d|\lambda|(y)<+\infty,
\qquad\hbox{and}\qquad
-\int_{\R^N\setminus\overline{\Omega}}\Ds G_\Omega(x,y)\;d|\mu|(y)<+\infty,
$$
for a.e. $x\in\Omega$,
and a Radon measure $\nu\in\mathcal{M}(\partial\Omega)$ such that $|\nu|(\partial\Omega)<+\infty$,
the problem
$$
\left\lbrace\begin{aligned}
\Ds u &=\lambda & & \hbox{in }\Omega \\
u &=\mu & & \hbox{in }\R^N\setminus\overline{\Omega} \\
Eu &=\nu & & \hbox{on }\partial\Omega
\end{aligned}\right.
$$
admits a unique solution $u\in L^1(\Omega)$ in the weak sense.
In addition, we have the representation formula 
$$
u(x)=\int_\Omega G_\Omega(x,y)\;d\lambda(y)-\int_{\C\Omega}\Ds G_\Omega(x,y)\;d\mu(y)
+\int_{\partial\Omega}M_\Omega(x,y)\;d\nu(y)
$$
and
$$
\Arrowvert u\Arrowvert_{L^1(\Omega)}\leq C\left(
\Arrowvert{\delta(x)}^s\lambda\Arrowvert_{\mathcal{M}(\Omega)}+
\Arrowvert\min\{{\delta(x)}^{-s},{\delta(x)}^{-N-2s}\}\,\mu\Arrowvert_{\mathcal{M}(\R^N\setminus\overline{\Omega})}+
\Arrowvert\nu\Arrowvert_{\mathcal{M}(\partial\Omega)}\right)
$$
for some constant $C=C(n,s,\Omega)>0$.
\end{theo}

We will conclude Section \ref{liner-sec} by showing that
\begin{prop} The weak solution of
$$
\left\lbrace\begin{aligned}
\Ds u(x) &=\frac{1}{{\delta(x)}^\beta} & & \hbox{in }\Omega,\ \beta\in(0,1+s) \\
u &=0 & & \hbox{in }\R^N\setminus\overline{\Omega} \\
Eu &=0 & & \hbox{on }\partial\Omega
\end{aligned}\right.
$$
satisfies 
\begin{eqnarray}
\displaystyle
c_1{\delta(x)}^s
\leq u(x)\leq
c_2{\delta(x)}^s & &
\hbox{ for }0<\beta<s, \nonumber \\ 
\displaystyle
c_3{\delta(x)}^s\log\frac{1}{\delta(x)}
\leq u(x)\leq
c_4{\delta(x)}^s\log\frac{1}{\delta(x)} & &
\hbox{ for }\beta=s, \label{udelta} \\
\displaystyle
c_5{\delta(x)}^{-\beta+2s}
\leq u(x) \leq
c_6{\delta(x)}^{-\beta+2s} & &
\hbox{ for }s<\beta<1+s. \nonumber
\end{eqnarray}
Moreover, there exist constants $\underline{c}=\underline{c}(n,s,\Omega)>0$
and $\overline{c}=\overline{c}(n,s,\Omega)>0$
such that the solution of 
$$
\left\lbrace\begin{aligned}
\Ds u &=0 & & \hbox{in }\Omega \\
u &=g & & \hbox{in }\R^N\setminus\overline{\Omega} \\
Eu &=0 & & \hbox{on }\partial\Omega
\end{aligned}\right.
$$
satisfies
\begin{equation}\label{uleqg}
\underline{c}\,\underline{g}(\delta(x)):=
\underline{c}\,\inf_{\delta(y)=\delta(x)}g(y)
\leq u(x)\leq 
\overline{c}\sup_{\delta(y)=\delta(x)}g(y)
=:\overline{c}\,\overline{g}(\delta(x))
\qquad x\in\Omega,
\end{equation}
for any $g$ satisfying \eqref{gintro} such that $\underline{g},\,\overline{g}$
are decreasing functions in $0^+$ with $\underline{g}\leq\overline{g}$ and
$$
\lim_{t\downarrow 0}\underline{g}(t)\ =\ +\infty.
$$
\end{prop}

\section{semilinear dirichlet problems}

In the following, $\Omega\subset\R^N$ will always be a bounded open
region with $C^{1,1}$ boundary. We consider nonlinearities $f:\Omega\times\R\rightarrow\R$ 
satisfying hypotheses 
\begin{description}
\item[\it f.1)] $f\in C(\Omega\times\R)$, $f\in L^\infty(\Omega\times I)$ for any bounded $I\subseteq\R$
\item[\it f.2)] $f(x,0)=0$ for any $x\in\Omega$, and
$f(x,t)\geq 0$ for any $x\in\Omega,\,t>0$,
\end{description}
and all positive boundary data $g$ that satisfy \eqref{gintro}.

After having constructed large $s$-harmonic functions,
we first prove the following preliminary

\begin{theo}\label{NL-CS}
Let $f:\Omega\times\R\rightarrow\R$ be a function 
satisfying f.1).
Let $g:\R^N\setminus\overline{\Omega}\rightarrow\R$ be a measurable bounded function.
Assume the nonlinear problem
$$
\left\lbrace\begin{aligned}
\Ds u &=-f(x,u) & & \hbox{in }\Omega \\
u &=g & & \hbox{in }\R^N\setminus\overline{\Omega} \\
Eu &=0 & & \hbox{on }\partial\Omega
\end{aligned}\right.
$$
admits a subsolution $\underline{u}\in L^1(\Omega)$ 
and a supersolution $\overline{u}\in L^1(\Omega)$ in the weak sense
$$
\left\lbrace\begin{aligned}
\Ds\underline{u} &\leq -f(x,\underline{u}) & & \hbox{in }\Omega \\
\underline{u} &\leq g & & \hbox{in }\R^N\setminus\overline{\Omega}\\
E\underline{u} &=0 & & \hbox{on }\partial\Omega
\end{aligned}\right.
\quad\hbox{and}\quad
\left\lbrace\begin{aligned}
\Ds\overline{u} &\geq -f(x,\overline{u}) & & \hbox{in }\Omega \\
\overline{u} &\geq g & & \hbox{in }\R^N\setminus\overline{\Omega}\\
E\overline{u} &=0 & & \hbox{ on }\partial\Omega
\end{aligned}\right.
$$
i.e. for any $\phi\in\T(\Omega)$, $\Ds\phi|_\Omega\geq 0$, it holds
\[
\int_\Omega\underline{u}\,\Ds\phi\leq
-\int_\Omega f(x,\underline{u})\,\phi
-\int_{\R^N\setminus\overline{\Omega}} g\,\Ds\phi,\ 
\int_\Omega\overline{u}\,\Ds\phi\geq
-\int_\Omega f(x,\overline{u})\,\phi
-\int_{\R^N\setminus\overline{\Omega}} g\,\Ds\phi.
\]
Assume also $\underline{u}\leq\overline{u}$ in $\Omega$, 
and $\underline{u},\overline{u}\in L^\infty(\Omega)\cap C(\Omega)$.
Then the above nonlinear problem has a weak solution $u$ 
in the sense of Definition \ref{weakdefiintro}
satisfying
$$
\underline{u}\leq u\leq\overline{u}.
$$
In addition,
\begin{itemize}
\item if $f$ is increasing in the second variable,
i.e. $f(x,s)\leq f(x,t)$ whenever $s\leq t$, for all $x\in\Omega$,
then there is a unique solution,
\item if not, there is a unique minimal solution $u_1$, that is
a solution $u_1$ such that $\underline{u}\leq u_1\leq v$ for any other supersolution $v\geq\underline{u}$.
\end{itemize} 
\end{theo}

In case our boundary datum $g$ is a nonnegative bounded function,
then Theorem \ref{NL-CS} (with {\it f.2)} as an additional assumption) 
provides a unique solution
since we may consider $\overline{u}=\sup g$
and $\underline{u}=0$.
Then we attack directly the problem with unbounded boundary values,
and we are especially interested in those data
exploding on $\partial\Omega$.
The existence of large $s$-harmonic functions
turns out to be the key ingredient to prove all the following theorems,
that is,

\begin{theo}[construction of large solutions]\label{LARGE-BUILDING}
Let $f:\Omega\times\R\rightarrow\R$ be a function 
satisfying f.1) and f.2). Then there exist $u,\,v:\R^N\rightarrow[0,+\infty]$
solving
$$
\left\lbrace\begin{aligned}
\Ds u &= -f(x,u) \quad \hbox{ in }\Omega \\ \displaystyle
\lim_{\stackrel{\hbox{\scriptsize $x\!\rightarrow\!\partial\Omega$}}{x\in\Omega}}u(x)
&=+\infty
\end{aligned}\right.
\qquad\hbox{ and }\qquad
\left\lbrace\begin{aligned}
\Ds v &= f(x,v) \quad \hbox{ in }\Omega \\ \displaystyle
\lim_{\stackrel{\hbox{\scriptsize $x\!\rightarrow\!\partial\Omega$}}{x\in\Omega}}v(x)
&=+\infty,
\end{aligned}\right.
$$
in a weak $L^1$ sense.
\end{theo}

Depending on the nature of the nonlinearity $f$
one can be more precise about the Dirichlet values of $u$.
Namely,

\begin{theo}
\label{-SIGN}
Let $f:\Omega\times\R\rightarrow\R$ be a function 
satisfying f.1) and f.2), and
$g:\R^N\setminus\overline{\Omega}\rightarrow[0,+\infty]$ a measurable function
satisfying \eqref{gintro};
let also be $h\in C(\partial\Omega)$, $h\geq 0$.
The semilinear problem 
$$
\left\lbrace\begin{aligned}
\Ds u(x) &=-f(x,u(x)) & & \hbox{in }\Omega \\
u &=g & & \hbox{in }\R^N\setminus\overline{\Omega}\\
Eu &=h & & \hbox{on }\partial\Omega
\end{aligned}\right.
$$
satisfies the following:
\begin{description}
\item[\it i)] if $h\equiv 0$, the equation has a weak solution for any $g$ satisfying \eqref{gintro},
\item[\it ii)] if $h\not\equiv 0$ then
\begin{itemize}
\item the problem has a weak solution if there exist
a nondecreasing function $\Lambda:(0,+\infty)\to(0,\infty)$
and $a>0$
such that 
$$
f(x,t)\leq \Lambda(at),\quad \hbox{ for }t>0
$$
and $\Lambda(\delta^{s-1})\delta^s\in L^1(\Omega)$,
\item the problem does not admit any weak solution if there exist $b_1,T>0$
such that 
$$
f(x,t)\geq b_1{t}^{\frac{1+s}{1-s}},\quad \hbox{ for }t>T.
$$
\end{itemize}
\end{description}
If, in addition, $f$ is increasing in the second variable then the problem admits 
only one positive solution.
\end{theo}

\begin{rmk}\label{rmkintro}\rm To our knowledge, the only previous results dealing with 
semilinear problems driven by the fractional Laplacian are \cite{felmer-quaas} and \cite{chen-felmer},
as mentioned in the Introduction.
On the one hand the above theorem improves them in the following ways:
\begin{enumerate}
\item we have quite general assumptions on the nonlinearity, while in \cite{felmer-quaas} and \cite{chen-felmer}
only power-like nonlinearities are taken into account,
\item we give the exact value of the threshold growth (i.e. $f(t)\sim t^{(1+s)/(1-s)}$)
distinguishing the existence and the nonexistence of solutions,
\item we specify in what sense solutions blow up at the boundary via the $E$ trace operator,
clarifying why the authors in \cite{felmer-quaas} and \cite{chen-felmer} find infinitely many solutions.
\end{enumerate}
On the other hand we do not find {\it all} boundary blow-up solutions
which are proved to exist in \cite{felmer-quaas} and \cite{chen-felmer}:
these correspond to boundary data $h\equiv+\infty$ or $g\in L^1(\R^N\setminus\Omega)$
not satisfying \eqref{gintro}.
This last question will be studied in the following.
\end{rmk}

\begin{theo}[sublinear source]\label{SUBLINEAR}
Let $f:\Omega\times\R\rightarrow\R$ be a function 
satisfying f.1) and f.2), and
$g:\R^N\setminus\overline{\Omega}\rightarrow[0,+\infty]$ a measurable function
satisfying \eqref{gintro};
let also be $h\in C(\partial\Omega)$, $h\geq 0$.
Suppose also that 
$$
f(x,t)\leq\Lambda(t),\quad\hbox{ for all }x\in\overline{\Omega},\,t\geq 0
$$ 
where $\Lambda\in C^1(0,+\infty)$ is concave and $\Lambda'(t)\xrightarrow[]{t\uparrow\infty}0$.
Then there exists a positive weak solution $u$
to the semilinear problem 
$$
\left\lbrace\begin{aligned}
\Ds u(x) &=f(x,u(x)) & & \hbox{in }\Omega \\
u &=g & & \hbox{in }\R^N\setminus\overline{\Omega}\\
Eu &=h & & \hbox{ on }\partial\Omega.
\end{aligned}\right.
$$
\end{theo}

\begin{theo}[superlinear source]\label{SUPERLINEAR}
Let $f:\Omega\times\R\rightarrow\R$ be a function 
satisfying f.1) and f.2).
For $0<\beta<1-s$, consider problems
\begin{align*}
(\star)\qquad\left\lbrace\begin{aligned}
\Ds u(x) &=\lambda f(x,u(x)) & & \hbox{in }\Omega \\
u(x) &={\delta(x)}^{-\beta} & & \hbox{in }\R^N\setminus\overline{\Omega}\\
Eu &=0 & & \hbox{on }\partial\Omega
\end{aligned}\right.\\
(\star\star)\qquad\left\lbrace\begin{aligned}
\Ds u(x) &=\lambda f(x,u(x)) & & \hbox{in }\Omega \\
u(x) &=0 & & \hbox{in }\R^N\setminus\overline{\Omega}\\
Eu &=1 & & \hbox{on }\partial\Omega.
\end{aligned}\right.
\end{align*}

\underline{Existence}. If
there exist $a_1,\,a_2,\,T>0$ and $p\geq 1$, such that
$$
f(x,t)\leq a_1+a_2\,{t}^p,\qquad x\in\Omega,\ t>T
$$
and $p\beta<1+s$, then there exists $L_1>0$ depending on $\beta$ and $p$ such that
problem $(\star)$ admits a weak solution $u\in L^1(\Omega)$ for any $\lambda\in[0,L_1]$.
Similarly, if $p(1-s)<1+s$, then there exists $L_2>0$ depending on $p$ such that
problem $(\star\star)$ admits a weak solution $u\in L^1(\Omega)$ for any $\lambda\in[0,L_2]$.

\underline{Nonexistence}. If
there exist $b,\,T>0$ and $q>0$, such that
$$
b\,{t}^q\leq f(x,t),\qquad x\in\Omega,\ t>T
$$
and $q\beta\geq 1+s$, then problem $(\star)$ admits a weak solution only for $\lambda=0$.
Similarly, if $q(1-s)\geq 1+s$, then problem $(\star\star)$ admits a weak solution only for $\lambda=0$.
\end{theo}

We finally note that, with the definition of weak solution we are dealing with, 
the nonexistence of a weak solution implies \it complete blow-up\rm, meaning that:

\begin{defi}[complete blow-up]\index{complete blow-up}\label{complete-blowup-defi} 
If for any nondecreasing sequence ${\{f_k\}}_{k\in\N}$
of bounded functions such that $f_k\uparrow f$ pointwisely as $k\uparrow+\infty$,
and any sequence ${\{u_k\}}_{k\in\N}$ of positive solutions to 
$$
\left\lbrace\begin{aligned}
\Ds u_k &=f_k(x,u_k) & & \hbox{in }\Omega \\
u_k &=g & & \hbox{in }\R^N\setminus\overline{\Omega} \\
Eu_k &=h & &\hbox{on }\partial\Omega, 
\end{aligned}\right.
$$
there holds 
$$
\lim_{k\uparrow+\infty}\frac{u_k(x)}{{\delta(x)}^s}\ =\ +\infty,\qquad
\hbox{ uniformly in }\,x\in\Omega,
$$
then we say there is complete blow-up.
\end{defi}

\begin{theo}\label{COMPLETE-BLOWUP} Let $f:\Omega\times\R\rightarrow\R$ be a function 
satisfying f.1) and f.2)
and $g:\R^N\setminus\overline{\Omega}\rightarrow[0,+\infty]$ a measurable function
satisfying \eqref{gintro}; let also be 
$h\in C(\partial\Omega)$, $h\geq 0$.
If there is no weak solution to 
$$
\left\lbrace\begin{aligned}
\Ds u &=f(x,u) & & \hbox{in }\Omega \\
u &=g & & \hbox{in }\R^N\setminus\overline{\Omega} \\
Eu &=h & & \hbox{on }\partial\Omega, 
\end{aligned}\right.
$$
then there is complete blow-up.
\end{theo}

\subsection{The interest in larger solutions}

So, we have achieved the existence of boundary blow-up solutions to nonlinear problems
with negative right-hand side,
provided by Theorem \ref{-SIGN}; anyhow this singular behaviour is driven by 
a linear phenomenon rather than a compensation between the nonlinearity and the explosion
(as in the classical case), indeed no growth condition on $f$ arises except when $h\not\equiv 0$,
where one essentially needs
\[
\int_\Omega f(x,\delta(x)^{s-1})\,\delta(x)^s\;dx\ <\ \infty
\]
in order to make sense of the weak $L^1$ definition.

For this reason in what follows we address the question of the existence of solutions to problems of the form
\[
\left\lbrace\begin{aligned}
\Ds u &= -f(u) & \hbox{in }\Omega \\
u &= g & \hbox{in } \R^N\setminus\Omega \\
\delta^{1-s}u &= +\infty & \hbox{on }\partial\Omega
\end{aligned}\right.
\quad g\geq 0,\ \ \int_{\R^N\setminus\Omega}g\,\delta^{-s}\min\{1,\delta^{-N-s}\}=+\infty.
\]
providing sufficient conditions for the their solvability.
In doing so, we start answering the question left opened 
in Remark \ref{rmkintro}.
The results listed in Theorems \ref{main1} and \ref{main2} below
can be applied for a particular case of the {\it fractional singular Yamabe problem},
see e.g. Gonz\'alez, Mazzeo and Sire \cite{mar-sire}.

\section{towards a fractional keller-osserman condition}\label{fracKO-intro}

In Chapter \ref{fracKO}
we will work in the following set of assumptions:
\begin{itemize}
\item $\Omega$ is a bounded open domain of class $C^2$,
\item $f:\R\to\R$ satisfies
\begin{equation}\label{f1}
f \hbox{ is an increasing $C^1$ function with $f(0)=0$}
\end{equation}
\item $F$ is the antiderivative of $f$ vanishing in 0:
\begin{equation}\label{F}
F(t)\ :=\ \int_0^tf(\tau)\;d\tau,
\end{equation}
\item there exist $0<m<M$, such that 
\begin{equation}\label{tech}
1+m\leq \frac{tf'(t)}{f(t)}\leq 1+M,
\end{equation}
and thus $f$ satisfies \eqref{ko-intro} because, integrating the lower inequality,
one gets
\[
f(t)\geq f(1)t^{1+m}\qquad\hbox{and}\qquad
F(t)\geq\frac{f(1)}{2+m}\,t^{2+m};
\]
we can therefore define the function
\begin{equation}\label{phi}
\phi(u)\ :=\ \int_u^{+\infty}\frac{dt}{\sqrt{F(t)}},
\end{equation}
\item $f$ satisfies 
\begin{equation}\label{L1}
\int_1^{+\infty}\phi(t)^{1/s}\;dt\ <\ +\infty.
\end{equation}
\end{itemize}

In what follows we will use the expression $g\lessgtr h$ where $g,h:(0,+\infty)\to(0,+\infty)$ to shorten
\[
\hbox{there exists }C>0 \hbox{ such that } \frac{h(t)}{C}\leq g(t)\leq Ch(t),\hbox{ for any }t>0.
\]

\begin{rmk}\label{phi-rmk}\rm The function $\phi:(0,+\infty)\rightarrow(0,+\infty)$ is monotone decreasing and
\[
\lim_{t\downarrow0}\phi(t)=+\infty,\qquad\lim_{t\uparrow+\infty}\phi(t)=0.
\]
Moreover
\[
\phi'(u)=-\frac{1}{\sqrt{F(u)}}.
\]
is of the same order as $-(u\,f(u))^{-1/2}$ since for $t>0$ and some $\tau\in(0,t)$
\[
\frac{F(t)}{t\,f(t)}=\frac{f(\tau)}{f(\tau)+\tau\,f'(\tau)}\ 
\left\lbrace\begin{aligned}
& \geq\frac{1}{2+M} \\
& \leq\frac{1}{2+m}
\end{aligned}\right.\qquad \hbox{ by the Cauchy Theorem.}
\]
This entails that the order of $\phi(u)$ is the same as $(u/f(u))^{1/2}$ indeed
for $u>0$ and some $t\in(u,+\infty)$
\[
\frac{\sqrt{\frac{u}{f(u)}}}{\phi(u)}=\frac{\frac{1}{2}\sqrt{\frac{f(t)}{t}}\cdot\frac{f(t)-tf'(t)}{f(t)^2}}{\phi'(t)}
\lessgtr \frac{f(t)-tf'(t)}{-f(t)}= \frac{tf'(t)}{f(t)}-1
\]
which belongs to $(m,M)$ by hypothesis \eqref{tech}.
Note that hypothesis \eqref{L1} is therefore equivalent to
\begin{equation}\label{L1-bis}
\int_1^{+\infty}\left(\frac{t}{f(t)}\right)^{\frac1{2s}}\;dt\ <\ +\infty.
\end{equation}
\end{rmk}

\begin{rmk}\rm
In \cite{keller} and \cite{osserman} condition \eqref{ko-intro}
is proven to be necessary and sufficient for the existence of a solution of
\[
\left\lbrace\begin{aligned}
-\lapl u = -f(u) \hbox{ in }\Omega, \\
\lim_{x\to\partial\Omega}u(x) =  +\infty. 
\end{aligned}\right.
\]
Note that if we set $s=1$ in \eqref{L1} then
\[
+\infty\ >\ \int_{u}^{+\infty}\phi(t)\;dt\lessgtr\int_{u}^{+\infty}\sqrt\frac{t}{f(t)}\;dt
\lessgtr\int_{u}^{+\infty}\frac{t}{\sqrt{F(t)}}\;dt 
\]
we get the condition to force the classical solution $u$ to be $L^1(\Omega)$.
Indeed in \cite[Theorem 1.6]{KO-dupaigne} it is proved that a solution $u$ satisfies
\begin{equation}\label{ddgr}
\lim_{x\rightarrow\partial\Omega}\frac{\phi(u(x))}{\delta(x)}=1
\end{equation}
which yields that $u\in L^1(\Omega)$ if and only if $\phi^{-1}$,
the inverse function of $\phi$ (recall it is monotone decreasing),
is integrable in a neighbourhood of $0$, 
i.e. with a change of integration variable
\[
+\infty>\int_0^\eta\phi^{-1}(r)\;dr
=\int_{\phi^{-1}(\eta)}^{+\infty}t|\phi'(t)|\;dt
=\int_{t_0}^{+\infty}\frac{t}{\sqrt{F(t)}}\;dt.
\]
\end{rmk}

Our results can be summarized as follows.

\begin{theo}\label{main1} Suppose that the nonlinear term $f$ satisfies hypotheses \eqref{f1},
\eqref{tech}, \eqref{L1} and
\begin{equation}\label{E}
\int_{t_0}^{+\infty}f(t)t^{-2/(1-s)}\;dt\ <\ +\infty.
\end{equation}
Then problem
\begin{equation}\label{problema}
\left\lbrace\begin{aligned}
\Ds u &= -f(u) & \hbox{ in }\Omega \\
u &= 0 & \hbox{ in }\R^N\setminus\Omega \\
\delta^{1-s}u &= +\infty &\hbox{ on }\partial\Omega
\end{aligned}\right.
\end{equation}
admits a solution $u\in L^1(\Omega)$. 
Moreover there exists $c>0$ for which
\begin{equation}\label{bbehav}
\phi(u(x))\ \geq\ c\,\delta(x)^s\qquad\hbox{near }\partial\Omega.
\end{equation}
\end{theo}

\begin{rmk}\rm The condition $u\in L^1(\Omega)$ is necessary to make sense
of the fractional Laplacian, as any different possible definition points out 
(see the Introduction). Also,
compare the boundary behaviour in this setting expressed by equation \eqref{bbehav},
with the classical one in equation \eqref{ddgr}.
\end{rmk}

\begin{theo}\label{main2} Suppose that the nonlinear term $f$ satisfies hypotheses \eqref{f1},
\eqref{tech}, \eqref{L1} and 
\begin{eqnarray}
& g:\R^N\setminus\Omega\longrightarrow[0,+\infty),\qquad g\in L^1(\R^N\setminus\Omega) & \nonumber \\
& \phi(g(x))\ \geq\ \delta(x)^s,\qquad\hbox{near }\partial\Omega. & \label{g2}
\end{eqnarray}
Then problem
\begin{equation}\label{gproblem}
\left\lbrace\begin{aligned}
\Ds u &= -f(u) & \hbox{ in }\Omega \\
u &= g & \hbox{ in }\R^N\setminus\Omega 
\end{aligned}\right.
\end{equation}
admits a solution $u\in L^1(\Omega)$. 
Moreover there exists $c>0$ for which
\[
\phi(u(x))\ \geq\ c\,\delta(x)^s\qquad\hbox{near }\partial\Omega.
\]
\end{theo}

Mind that in problem \eqref{gproblem} we do not prescribe the singular trace at $\partial\Omega$.

\subsection*{Notations}

In the following we will always use the following notations:
\begin{itemize}
\item $\C\Omega$, when $\Omega\subseteq\R^N$ is open, for $\R^N\setminus\overline{\Omega}$,
\item $\delta(x)$ for $\dist(x,\partial\Omega)$ once $\Omega\subseteq\R^N$ has been fixed,
\item $\mathcal{M}(\Omega)$, when $\Omega\subseteq\R^N$, for the space of measures on $\Omega$,
\item $\mathcal{H}$, for the $n-1$ dimensional Hausdorff measure, 
dropping the ``$n-1$'' subscript whenever there is no ambiguity,
\item $f\wedge g$, when $f,g$ are two functions, for the function $\min\{f,g\}$,
\item $C^{2s+\ep}(\Omega)=\{v:\R^N\rightarrow\R$ and for any $K$ compactly supported
in $\Omega$, there exists $\alpha=\alpha(K,v)$ such that $v\in C^{2s+\alpha}(K) \}$.
\end{itemize}

\chapter{Linear fractional Dirichlet problems}

\section{a mean value formula}\label{meanvalue-sec}

Definition \ref{sharm-def} of
$s$-harmonicity turns out to be equivalent to 
have a null fractional Laplacian.
Since we could not find a precise reference for this fact,
we provide a proof.
Indeed, on the one hand we have that
any function $u$ which is $s$-harmonic in an open set $\Omega$ solves
$$
\Ds u(x)=0\qquad\hbox{ in }\Omega,
$$
indeed condition \eqref{sharm-eq}
can be rewritten, using \eqref{cns2},
$$
\int_{\C B_r(x)}\frac{u(y)-u(x)}{|y-x|^N\,\left(|y-x|^2-r^2\right)^s}\;dy =0
\qquad \hbox{ for any }r\in(0,\dist(x,\partial\Omega))
$$
and therefore
\begin{eqnarray*}
0 & = & \lim_{r\downarrow 0}\int_{\C B_r(x)}\frac{u(y)-u(x)}{|y-x|^N\,\left(|y-x|^2-r^2\right)^s}\;dy \\
& = & PV \int_{\R^N}\frac{u(y)-u(x)}{|y-x|^{N+2s}}\;dy
\quad =\quad -\frac{\Ds u(x)}{C_{N,s}}.
\end{eqnarray*}
Indeed, by dominated convergence, far from $x$ it holds
$$
\int_{\C B_1(x)}\frac{u(y)-u(x)}{|y-x|^N\,\left(|y-x|^2-r^2\right)^s}\;dy
\xrightarrow{r\downarrow 0}\int_{\C B_1(x)}\frac{u(y)-u(x)}{|y-x|^{N+2s}}\;dy.
$$
Now, any function $u$ $s$-harmonic in $\Omega$ is smooth in $\Omega$:
this follows from the representation through the Poisson kernel on balls,
given in Theorem \ref{pointwise},
and the smoothness of the Poisson kernel,
see formula \eqref{pois-ball} below.
Since $u$ is a smooth function, a Taylor expansion when $|y-x|<1$
$$
u(y)-u(x)=\langle\,\grad u(x),y-x\,\rangle+\theta(y-x),
\quad\hbox{where}\quad|\theta(y-x)|\leq C\,|y-x|^2
$$
implies
\begin{eqnarray*}
 & & \left|\int_{B_1(x)\setminus B_r(x)}\frac{u(y)-u(x)}{|y-x|^N\,\left(|y-x|^2-r^2\right)^s}\;dy-
\int_{B_1(x)\setminus B_r(x)}\frac{u(y)-u(x)}{\left|y-x\right|^{N+2s}}\;dy\right|  \\
 & & \leq \int_{B_1(x)\setminus B_r(x)}\frac{|\theta(y-x)|}{{|y-x|}^N}\left(\frac{1}{\left(|y-x|^2-r^2\right)^s}-\frac{1}{\left|y-x\right|^{2s}}\right)\;dy 
 \\
 & & \leq C\int_{B_1(x)\setminus B_r(x)}\frac{1}{{|y-x|}^{N-2+2s}}\left[\left(1-\frac{r^2}{|y-x|^2}\right)^{-s}-1\right]\;dy  \\
 & & \leq 
 C\,r^{2-2s}\int_1^{1/r}t^{1-2s}\left[\left(1-\frac1{t^2}\right)^{-s}-1\right]dt
\ \xrightarrow[r\downarrow 0]{}\ 0
\end{eqnarray*}
as it can be checked for example with the L'H\^opital's rule.
So any $u$ which $s$-harmonic function in $\Omega$ 
satisfies $\Ds u=0$ in $\Omega$. The converse follows from the following theorem.

\begin{theo}\label{meanvalue} Let $u:\R^N\rightarrow\R$ a measurable function such that 
for some open set $\Omega\subseteq\R^N$ is $u\in C^{2s+\ep}(\Omega)$.
Also, suppose that
$$
\int_{\R^N}\frac{|u(y)|}{{1+|y|}^{N+2s}}\;dy<+\infty.
$$
Then for any $x\in\Omega$ and $r>0$ such that $\overline{B_r(x)}\subseteq\Omega$, one has
\begin{equation}\label{meanform}
u(x)\ =\ \int_{\C B_r}\eta_r(y)\:u(x-y)\;dy + \gamma(N,s,r)\Ds u(z), \qquad
\gamma(N,s,r)=\frac{\Gamma(N/2)\;r^{2s}}{2^{2s}\,\Gamma\left(\frac{N+2s}{2}\right)\,\Gamma(1+s)}
\end{equation}
for some $z=z(x,s,r)\in\overline{B_r(x)}$.
\end{theo}

\begin{proof}\rm 
Suppose, without loss of generality, that $x=0$.
Let $v=\Gamma_s-H$, where
$\Gamma_s$ is the fundamental solution\footnote{One possible construction 
and the explicit expression of the fundamental solution
can be found in \cite[paragraph 2.2]{extension}.} of the fractional Laplacian 
$$
\Gamma_s(x)\ =\ \frac{-C_{N,-s}}{{|x|}^{N-2s}}
$$
and $H$ solves in the pointwise sense
$$
\left\lbrace\begin{aligned}
\Ds H &=0 & & \hbox{in }B_r \\
H &=\Gamma_s & & \hbox{in }\C B_r.
\end{aligned}\right.
$$
We claim that $H$ satisfies equality
$$
H(x)=H_1(x):=\int_{\R^N}\Gamma_s(x-y)\,\eta_r(y)\;dy=(\Gamma_s*\eta_r)(x).
$$
Indeed
\begin{enumerate}
\item $\Ds H_1=0$ in $B_r$ because $H_1=\Gamma_s*\eta_r$,
$\Ds\Gamma_s=\delta_0$ in the sense of distributions
and $\eta_r=0$ in $B_r$,
\item since (cf. Appendix in \cite[p. 399ss]{landkof})
$$
\int_{\R^N}\Gamma_s(x-y)\,\eta_r(y)\;dy=
\int_{\C B_r}\Gamma_s(x-y)\,\eta_r(y)\;dy=
\Gamma_s(x),\qquad |x|>r,
$$
then $H_1=\Gamma_s$ in $\C B_r$, as desired.
\end{enumerate}
Finally, as in 1., note that 
$$
\Ds H(x)=\eta_r(x),\quad\hbox{when }|x|>r.
$$
Since $u\in C^{2s+\ep}(\Omega)$, $\Ds u\in C(\Omega)\subseteq C(\overline{B_r})$,
see \cite[Proposition 2.4]{silvestre} or Lemma \ref{lem-H} below.
Mollify $u$ in order to obtain a sequence $\{u_j\}_j\subseteq C^\infty(\R^N)$.
Hence,
\begin{multline}\label{73}
\int_{B_r}v\cdot\Ds u_j=\int_{B_r}\Gamma_s\cdot\Ds u_j-\int_{B_r}H\cdot\Ds u_j=\\=
\int_{\R^N}\Gamma_s\cdot\Ds u_j-\int_{\R^N}H\cdot\Ds u_j
=u_j(0)-\int_{\C B_r}u_j\,\eta_r,
\end{multline}
where we have used the integration by parts formula \eqref{intpartsprop}
and the definition of $\Gamma_s$.
On the one hand we have now that $\Ds u_j\xrightarrow[]{j\uparrow+\infty}\Ds u$ uniformly in $B_r$, since
\begin{align*}
& \sup_{x\in B_r}\left|\Ds u_j(x)-\Ds u(x)\right|= \\ 
& =\ \sup_{x\in B_r}\left|C_{N,s}\,PV\int_{\R^N}\frac{u_j(x)-u_j(y)-u(x)+u(y)}{{|x-y|}^{N+2s}}\;dy\right| \\
& \leq\ C_{N,s}\Arrowvert u_j-u\Arrowvert_{C^{2s+\ep}(B_{r+\delta})}
\sup_{x\in B_r}\int_{B_{r+\delta}}\frac{dy}{{|x-y|}^{N-\ep}}\ +\\
&\qquad\qquad+\ C_{N,s}
\int_{\R^N\setminus B_{r+\delta}}\frac{\Arrowvert u_j-u\Arrowvert_{L^\infty(B_r)}+|u_j(y)-u(y)|}{\left(|y|-r\right)^{N+2s}}\;dy,
\end{align*}
while on the other hand
$$
u_j(0)\xrightarrow[]{j\uparrow+\infty} u(0),\qquad
\int_{\C B_r}u_j\,\eta_r\xrightarrow[]{j\uparrow+\infty}\int_{\C B_r}u\,\eta_r
$$
so that we can let $j\rightarrow+\infty$ in equality \eqref{73}.
Collecting the information so far, we have
\begin{equation}\label{mv-first}
u(0)=
\int_{\C B_r} u\:\eta_r + \int_{B_r} v\;\Ds u=
\int_{\C B_r} u\:\eta_r + \Ds u(z)\cdot\int_{B_r}v
\end{equation}
for some $|z|\leq r$, by continuity of $\Ds u$ in $\overline{B_r}$
and since $v>0$ in $\overline{B_r}$.
The constant $\gamma(N,s,r)$ appearing in the statement 
equals to $\int_{B_r}v$.

Let us compute $\gamma(N,s,r)=\int_{B_r}v$.
If we consider the solution $\varphi_\delta$ to
$$
\left\lbrace\begin{aligned}
\Ds\varphi_\delta &=1 & & \hbox{in }B_{r+\delta} \\
\varphi_\delta &=0 & & \hbox{in }\C B_{r+\delta}
\end{aligned}\right.
$$
and we apply formula \eqref{mv-first} to $\varphi_\delta$ in place of $u$ to entail
$$
\varphi_\delta(0)=
\int_{\C B_r} \varphi_\delta\,\eta_r + \Ds \varphi_\delta(z)\cdot\int_{B_r}v
=\int_{B_{r+\delta}\setminus B_r} \varphi_\delta\,\eta_r +\int_{B_r}v
$$
The solution $\varphi_\delta$ is explicitly known (see \cite[equation (1.4)]{rosserra} and references therein) 
and given by
$$
\varphi_\delta(x)=\frac{\Gamma(N/2)}{2^{2s}\,\Gamma\left(\frac{N+2s}{2}\right)\,\Gamma(1+s)}
\left((r+\delta)^2-|x|^2\right)^s.
$$
Hence, by letting $\delta\downarrow 0$,
$$
\gamma(N,s,r)=\int_{B_r}v=\lim_{\delta\downarrow 0}\varphi_\delta(0)=
\frac{\Gamma(N/2)}{2^{2s}\,\Gamma\left(\frac{N+2s}{2}\right)\,\Gamma(1+s)}\;r^{2s}.
$$
\end{proof}

\subsection{The Liouville theorem for \texorpdfstring{$s$}{s}-harmonic functions}

Following the proof of the classic Liouville Theorem 
due to Nelson \cite{nelson}, it is possible to prove the analogous
result for the fractional Laplacian.
In the following 
we denote by $\omega_{N-1}=\mathcal{H}^{N-1}(\partial B)$,
the $(N-1)$-dimensional Hausdorff measure of the unit sphere. 

\begin{theo}\label{liouv} Let $u:\R^N\rightarrow\R$ be a function
which is $s$-harmonic throughout $\R^N$.
Then, if $u$ is bounded in $\R^N$, it is constant.
\end{theo}
\begin{proof} Take two arbitrary points $x_1,\ x_2\in\R^N$.
Both satisfy for all $r>0$
$$
u(x_1)=\int_{\C B_r(x_1)}\eta_r(y-x_1)\,u(y)\;dy,
\qquad
u(x_2)=\int_{\C B_r(x_2)}\eta_r(y-x_2)\,u(y)\;dy.
$$
Denote by $M:=\sup_{\R^N}|u|$, 
which is finite by hypothesis:
\begin{eqnarray*}
|u(x_1)-u(x_2)| & = & 
\left\arrowvert\int_{\C B_r(x_1)}\eta_r(y-x_1)\,u(y)\;dy-
\int_{\C B_r(x_2)}\eta_r(y-x_2)\,u(y)\;dy\right\arrowvert \\
& \leq & 
\int_{\C B_r(x_1)\cap B_r(x_2)}\eta_r(y-x_1)\,M\;dy+
\int_{\C B_r(x_2)\cap B_r(x_1)}\eta_r(y-x_2)\,M\;dy\ + \\
& & 
+\int_{\C B_r(x_1)\cap\C B_r(x_2)}\left\arrowvert
\eta_r(y-x_1)-\eta_r(y-x_1)\right\arrowvert\,M\;dy.
\end{eqnarray*}
Define $\delta:=|x_1-x_2|$. The first addend (and similarly the second)
vanish as $r\rightarrow+\infty$:
\begin{eqnarray*}
\int_{\C B_r(x_1)\cap B_r(x_2)}\eta_r(y-x_1)\;dy & \leq & 
\int_{B_{r+\delta}\setminus B_r}\frac{c(N,s)\,r^{2s}}{|y|^N\,(|y|^2-r^2)^s}\;dy \\
& = & 
\omega_{N-1}c(N,s)\,r^{2s}\int_r^{r+\delta}\frac{d\rho}{\rho\,(\rho^2-r^2)^s} \\
& \leq & 
\omega_{N-1}c(N,s)\,\frac{1}{r^{1-s}}
\int_r^{r+\delta}\frac{d\rho}{(\rho-r)^s}\xrightarrow[]{r\uparrow+\infty}0. 
\end{eqnarray*}
The third one is more delicate:
\[
\begin{aligned}
& \int_{\C B_r(x_1)\cap\C B_r(x_2)}\left\arrowvert
\frac{r^{2s}}{|y-x_1|^N\,(|y-x_1|^2-r^2)^s}-
\frac{r^{2s}}{|y-x_2|^N\,(|y-x_2|^2-r^2)^s}\right\arrowvert\;dy \\
&=\ \int_{\C B(x_1/r)\cap\C B(x_2/r)}\left\arrowvert
\frac{1}{|y-\frac{x_1}{r}|^N\,(|y-\frac{x_1}{r}|^2-1)^s}-
\frac{1}{|y-\frac{x_2}{r}|^N\,(|y-\frac{x_2}{r}|^2-1)^s}\right\arrowvert\;dy \\
&=\ \int_{\C B\cap\C B(\frac{x_2-x_1}{r})}\left\arrowvert
\frac{1}{|y|^N\,(|y|^2-1)^s}-
\frac{1}{|y-\frac{x_2-x_1}{r}|^N\,(|y-\frac{x_2-x_1}{r}|^2-1)^s}\right\arrowvert\;dy \\
&\left[\hbox{\small setting 
$x_r=(x_2-x_1)/r$, $|x_r|=\delta/r\rightarrow 0$ as $r\rightarrow+\infty$}\right] \\
&=\ \int_{\C B\cap\C B(x_r)}\left\arrowvert
\frac{1}{|y|^N\,(|y|^2-1)^s}-
\frac{1}{|y-x_r|^N\,(|y-x_r|^2-1)^s}\right\arrowvert\;dy \\
&\left[\hbox{\small taking wlog $x_r=\frac{\delta e_1}{r}$
and defining $H_r=\{x\in\R^N:x_1>\frac{\delta}{2r},\ |x-x_r|>1\}$}\right] \\
&=\ 2\int_{H_r}\left[
\frac{1}{|y-x_r|^N\,(|y-x_r|^2-1)^s}
-\frac{1}{|y|^N\,(|y|^2-1)^s}\right]\;dy \\
&\leq\ 2\,\omega_{N-1}
\int_1^{+\infty}\left[
\frac{1}{\rho\,(\rho^2-1)^s}
-\frac{\rho^{N-1}}{(\rho+\frac{\delta}{r})^N\,((\rho+\frac{\delta}{r})^2-1)^s}
\right]\;d\rho \\
&\leq\ 2\,\omega_{N-1}
\int_1^{+\infty}\frac{1}{\rho\,(1+\frac{\delta}{r\rho})^N}
\left[
\frac{1}{(\rho^2-1)^s}
-\frac{\rho^{N-1}}
{((\rho+\frac{\delta}{r})^2-1)^s}\right]\;d\rho\; 
\xrightarrow[]{r\uparrow+\infty}0
\end{aligned}
\]
thanks to the Monotone Convergence Theorem.
\end{proof}

\subsection{Asymptotics 
\texorpdfstring{as $s\uparrow 1$}{s up 1}}\label{meanform-app}

First of all, the proof of Theorem \ref{meanvalue} implies
$$
\lim_{s\uparrow 1}\gamma(N,s,r)=\gamma(N,1,r)=
\frac{\Gamma(N/2)}{4\,\Gamma\left(\frac{N+2}{2}\right)\,\Gamma(2)}\;r^2
=\frac{r^2}{4}\cdot\frac{\Gamma(N/2)}{\frac{N}{2}\,\Gamma(N/2)}=\frac{r^2}{2N}.
$$
Also
$$
1=\int_{\C B_r}\eta_r(y)\;dy=
\int_{\C B_r}\frac{c(N,s)\,r^{2s}}{|y|^N\left(|y|^2-r^2\right)^s}\;dy
\qquad\hbox{where }
c(N,s)=\frac{2\sin(\pi s)}{\pi\,\omega_{N-1}}.
$$
Then
\begin{multline*}
1\ =\ \frac{2\sin(\pi s)}{\pi\,\omega_{N-1}}
\int_{\C B_r}\frac{r^{2s}}{|y|^N\left(|y|^2-r^2\right)^s}\;dy 
\ =\ \frac{2\sin(\pi s)}{\pi\,\omega_{N-1}}
\int_{\C B}\frac{1}{|y|^N\left(|y|^2-1\right)^s}\;dy \ =\\
\ =\ \frac{2\sin(\pi s)}{\pi\,\omega_{N-1}}
\int_{\partial B}\left[
\int_1^{+\infty}\frac{d\rho}{\rho\,\left(\rho^2-1\right)^s}
\right]\;d\mathcal{H}^{N-1}(\theta) \ =\ 
\frac{2\sin(\pi s)}{\pi}\int_1^{+\infty}\frac{d\rho}{\rho\,\left(\rho^2-1\right)^s}.
\end{multline*}
With similar computation
$$
\int_{\C B_r}\eta_r(y)\:u(x-y)\;dy\ =\ 
\frac{2\sin(\pi s)}{\pi}\int_{\partial B}\left[
\int_1^{+\infty}\frac{u(x-r\rho\theta)}{\rho\,\left(\rho^2-1\right)^s}\;d\rho
\right]\;d\mathcal{H}^{N-1}(\theta).
$$
Fix a $\theta\in\partial B$ and consider the difference
$$
\left|\frac{2\sin(\pi s)}{\pi}\int_1^{+\infty}\frac{u(x-r\rho\theta)}{\rho\,\left(\rho^2-1\right)^s}\;d\rho
\ -\ u(x-r\theta)\right|\ \leq\ 
\frac{2\sin(\pi s)}{\pi}\int_1^{+\infty}\frac{|u(x-r\rho\theta)-u(x-r\theta)|}{\rho\,\left(\rho^2-1\right)^s}\;d\rho:
$$
since we are handling $C^2$ functions 
we can push a bit further the estimate to deduce
\begin{multline*}
\left|\frac{2\sin(\pi s)}{\pi}\int_1^{+\infty}\frac{u(x-r\rho\theta)}{\rho\,\left(\rho^2-1\right)^s}\;d\rho
\ -\ u(x-r\theta)\right|\ \leq\ \\
\leq\ \frac{2\sin(\pi s)}{\pi}\int_1^{1+\delta}\frac{C\,\left(\rho-1\right)^{1-s}}{\rho\,\left(\rho+1\right)^s}\;d\rho
+\frac{2\sin(\pi s)}{\pi}\int_{1+\delta}^{+\infty}\frac{|u(x-r\rho\theta)-u(x-r\theta)|}{\rho\,\left(\rho^2-1\right)^s}\;d\rho
\xrightarrow[s\uparrow 1]{}0
\end{multline*}
therefore
$$
\int_{\C B_r}\eta_r(y)\:u(x-y)\;dy\ 
\xrightarrow[s\uparrow 1]{}\ 
\frac{1}{\omega_{N-1}\,r^{N-1}}\int_{\partial B_r}u(x-y)\;d\mathcal{H}^{N-1}(y).
$$
The choice of the point $z\in \overline{B_r}$ in \eqref{meanform} depends 
on the value of $s$, but since these are all points belonging to a compact set,
we can build $\{s_k\}_{k\in\N}\subseteq(0,1)$, $s_k\rightarrow 1$ as $k\rightarrow+\infty$,
such that $z(s_k)\rightarrow z_0\in \overline{B_r}$.
Since it is known (see \cite[Proposition 4.4]{hitchhiker}) that
$$
\lim_{s\uparrow 1}\Ds u=-\lapl u,
$$
then
$$
|\Ds u\,(z(s_k))+\lapl u(z_0)|\leq
|\Ds u\,(z(s_k))-\Ds u\,(z_0)|+
|\Ds u\,(z_0)+\lapl u(z_0)|
\xrightarrow{k\rightarrow+\infty}0.
$$
Finally we have proven that
\begin{multline*}
u(x)\ =\ \int_{\C B_r}\eta_r(y)\:u(x-y)\;dy + \gamma(N,s,r)\Ds u(z(s_k))
\\ \xrightarrow{k\rightarrow+\infty}\ 
\frac{1}{\omega_{N-1}\,{r}^{N-1}}\int_{\partial B_r}u(x-y)\;d\mathcal{H}^{N-1}(y)
 - \frac{r^2}{2n}\,\lapl u(z_0)
\end{multline*}
which is a known formula for $C^2$ functions,
see e.g. \cite[Proposition A.1.2]{dupaigne-book}.

\section{existence and uniqueness}\label{liner-sec}

Assume $\Omega\subseteq\R^N$ is open and bounded, with $C^{1,1}$ boundary.
\medskip

\subsection{Proof of Proposition \ref{integrbypartsform}}\label{byparts-proof}
Assume first that $u\in\mathcal{S}$ and $v=0$ in $\R^N\setminus\Omega$, $v\in C^{2s+\ep}(\Omega)\cap C(\overline{\Omega})$
and $\Ds v\in L^1(\R^N)$; 
then we can regularize $v$, 
via the convolution with a mollifier ${\{\alpha_k(x)=k^N\alpha(kx):\alpha(x)=0\hbox{ for }|x|\geq1\}}_{k\in\N}$ 
in order to obtain a sequence
${\{v_k:=\alpha_k*v\}}_{k\in\N}\subseteq C^\infty_c(\R^N)\subseteq\mathcal{S}$
converging uniformly to $v$ in $\R^N$.
Also,
$$
\Ds v_k=v*\Ds\alpha_k
$$
indeed 
\begin{eqnarray*}
\Ds v_k(x) & = & C_{N,s}\,PV\int_{\R^N}\frac{v_k(x)-v_k(y)}{{|x-y|}^{N+2s}}\;dy \\
 & = & C_{N,s}\,\lim_{\ep\downarrow0}\int_{\C B_\ep(x)}\frac{\int_{\R^N}v(z)[\alpha_k(x-z)-\alpha_k(y-z)]\;dz}{{|x-y|}^{N+2s}}\;dy \\
 & = & C_{N,s}\lim_{\ep\downarrow 0}\int_{\R^N}v(z)\int_{\C B_\ep(x)}\frac{\alpha_k(x-z)-\alpha_k(y-z)}{{|x-y|}^{N+2s}}\;dy\;dz \\
 & = & C_{N,s}\lim_{\ep\downarrow 0}\int_{\R^N}v(z)\int_{\C B_\ep(x-z)}\frac{\alpha_k(x-z)-\alpha_k(y)}{{|x-z-y|}^{N+2s}}\;dy\;dz
\end{eqnarray*}
where we have
$$
C_{N,s}\int_{\C B_\ep(x)}\frac{\alpha_k(x)-\alpha_k(y)}{{|x-y|}^{N+2s}}\;dy
\xrightarrow[\ep\downarrow0]{}\Ds\alpha_k(x),\quad\hbox{ uniformly in }\R^N
$$
since, see \cite[Lemma 3.2]{hitchhiker},
\begin{eqnarray}
 & & \left|\Ds\alpha_k(x)-C_{N,s}\int_{\C B_\ep(x)}\frac{\alpha_k(x)-\alpha_k(y)}{{|x-y|}^{N+2s}}\;dy\right|\ = \nonumber \\
& & =\ \left|C_{N,s}\,PV\int_{B_\ep(x)}\frac{\alpha_k(x)-\alpha_k(y)}{{|x-y|}^{N+2s}}\;dy\right|  \\
& & =\ \left|\frac{C_{N,s}}{2}\int_{B_\ep}\frac{\alpha_k(x+y)+\alpha_k(x-y)-2\alpha_k(x)}{{|y|}^{N+2s}}\;dy\right| \nonumber \\
& & \leq\ \frac{C_{N,s}\,\Arrowvert\alpha_k\Arrowvert_{C^2(B_\ep(x))}}{2}\int_{B_\ep}\frac{dy}{{|y|}^{N+2s-2}} \label{111} \\
& & \leq\ \frac{C_{N,s}\,\Arrowvert\alpha_k\Arrowvert_{C^2(\R^N)}}{2}\int_{B_\ep}\frac{dy}{{|y|}^{N+2s-2}}
\xrightarrow[\ep\downarrow0]{}0 \nonumber.
\end{eqnarray}
Following the very same proof up to \eqref{111}, since $v\in C^{2s+\ep}(\Omega)$, it is also possible to prove
$$
\Ds v_k(x)=\left(\alpha_k*\Ds v\right)(x),\quad\hbox{ for }\delta(x)>\frac{1}{k}.
$$
Since $\Ds v\in C(\R^N\setminus\Omega)$, see \cite[Proposition 2.4]{silvestre}, we infer that
$$
\Ds v_k(x)\xrightarrow[k\uparrow+\infty]{}\Ds v,\quad\hbox{ for every } x\in\R^N\setminus\partial\Omega.
$$

We give now a pointwise estimate on $\Ds\alpha_k(x),\ x\in B_{1/k}$.
Since $\alpha_k\in C^\infty_c(\R^N)$, we can write (see \cite[Paragraph 2.1]{silvestre})
\begin{multline*}
\Ds\alpha_k(x)=\left[{(-\lapl)}^{s-1}\circ(-\lapl)\right]\alpha_k(x)={(-\lapl)}^{s-1}[-k^{N+2}\lapl\alpha(kx)]\ =\\
=C_{N,-s+1}\int_{\R^N}\frac{k^{N+2}\lapl\alpha(ky)}{{|x-y|}^{N+2s-2}}\;dy.
\end{multline*}
With a change of variable we entail
$$
\Ds\alpha_k(x)=C_{N,-s+1}\,k^{N+2s}\int_B\frac{\lapl\alpha(y)}{{|k\,x-y|}^{N+2s-2}}\;dy
$$
and therefore
\begin{equation}\label{dsalfa}
|\Ds\alpha_k(x)|\leq \frac{\omega_{N-1}}{2\,(1-s)}\, k^{N+2s}\Arrowvert\lapl\alpha\Arrowvert_{L^\infty(\R^N)}.
\end{equation}
Indeed,
$$
\int_B\frac{dy}{{|z-y|}^{N+2s-2}}
$$
is a bounded function of $z$, having in $z=0$ its maximum $\frac{\omega_{N-1}}{2\,(1-s)}$.
Indeed,
\begin{eqnarray*}
\int_B\frac{dy}{{|z-y|}^{N+2s-2}} & = & \int_{B\cap B(z)}\frac{dy}{{|z-y|}^{N+2s-2}}+\int_{B\setminus B(z)}\frac{dy}{{|z-y|}^{N+2s-2}} \\
& = & \int_{B\cap B(z)}\frac{dy}{{|y|}^{N+2s-2}}+\int_{B\setminus B(z)}\frac{dy}{{|z-y|}^{N+2s-2}} \\
& \leq & \int_{B\cap B(z)}\frac{dy}{{|y|}^{N+2s-2}}+\int_{B\setminus B(z)}\frac{dy}{{|y|}^{N+2s-2}} 
\  = \ \int_B\frac{dy}{{|y|}^{N+2s-2}}.
\end{eqnarray*}
Using \eqref{dsalfa}, the $L^1$-norm of $\Ds v_k$ can be estimated by
\begin{align*}
& \int_{\R^N}|\Ds v_k(x)|\;dx\ = \\
& = \int_{\{\delta(x)<1/k\}}|(v*\Ds\alpha_k)(x)|\;dx
+\int_{\{\delta(x)\geq1/k\}}|(\alpha_k*\Ds v)(x)|\;dx\\
& \leq\ \int_{\{\delta(x)<1/k\}}\int_{B_{1/k}(x)}|v(y)\Ds\alpha_k(x-y)|\;dy\;dx\ +\ 
\int_{\{\delta(x)\geq1/k\}}\int_{B_{1/k}}|\alpha_k(y)\Ds v(x-y)|\;dy\;dx \\
&\leq\ \int_{\{\delta(x)<1/k\}}\int_{B_{1/k}(x)}C{\delta(y)}^sk^{N+2s}\Arrowvert\lapl\alpha\Arrowvert_{L^\infty(\R^N)}\;dy\;dx\ 
+\ \int_{B_{1/k}}\alpha_k(y)\int_{\R^N}|\Ds v(x-y)|\;dx\;dy\\
& \leq\ \frac{C\,\Arrowvert\lapl\alpha\Arrowvert_{L^\infty(\R^N)}}{k^{1-s}}+\Arrowvert\Ds v\Arrowvert_{L^1(\R^N)},
\end{align*}
so that, by the Fatou's Lemma we have
$$
\int_{\R^N}|\Ds v|\leq \liminf_{k\uparrow+\infty}\int_{\R^N}|\Ds v_k|\leq
\limsup_{k\uparrow+\infty}\int_{\R^N}|\Ds v_k|\leq \int_{\R^N}|\Ds v|
$$
which means
$$
\Arrowvert\Ds v_k\Arrowvert_{L^1(\R^N)}\xrightarrow[k\uparrow+\infty]{}\Arrowvert\Ds v\Arrowvert_{L^1(\R^N)}.
$$
Apply the Fatou's lemma to $|\Ds v_k|+|\Ds v|-|\Ds v_k-\Ds v|\geq 0$ to deduce
\begin{multline*}
2\int_{\R^N}|\Ds v|\leq \liminf_{k\uparrow+\infty}\int_{\R^N}(|\Ds v_k|+|\Ds v|-|\Ds v_k-\Ds v|)=\\
=2\int_{\R^N}|\Ds v|-\limsup_{k\uparrow+\infty}\int_{\R^N}|\Ds v_k-\Ds v|
\end{multline*}
and conclude 
$$
\Arrowvert\Ds v_k-\Ds v\Arrowvert_{L^1(\R^N)}\xrightarrow[k\uparrow+\infty]{}0
$$
and, for any $u\in\mathcal{S}$,
$$
\int_{\R^N}u\Ds v_k\xrightarrow[k\uparrow+\infty]{}\int_{\R^N}u\Ds v.
$$
Note now that
$$
\int_{\Omega_k}v_k\Ds u=\int_{\R^N}v_k\Ds u\xrightarrow[]{k\uparrow 0}\int_{\R^N}v\Ds u=\int_{\Omega}v\Ds u
$$
since $\Arrowvert v_k-v\Arrowvert_{L^\infty(\R^N)}\xrightarrow[]{k\uparrow 0}0$, and this concludes the proof.
\medskip

\subsection{Preliminaries on fractional Green functions, Poisson kernels and Martin kernels}

Consider the function $G_\Omega:\Omega\times\R^N\rightarrow\R$
built as the family of solutions to the problems
$$
\hbox{for any }x\in\Omega\qquad
\left\lbrace\begin{array}{ll}
\Ds G_\Omega(x,\cdot)=\delta_x & \hbox{ in }\Omega, \\
G_\Omega(x,y)=0 & \hbox{ in }\C\Omega.
\end{array}\right.
$$
This function can be written as the sum
$$
G_\Omega(x,y)=\Gamma_s(y-x)-H(x,y),
$$
where $\Gamma_s$ is the fundamental solution to the fractional 
Laplacian,
and 
$$
\hbox{for any }x\in\Omega\qquad
\left\lbrace\begin{array}{ll}
\Ds H(x,\cdot)=0 & \hbox{ in }\Omega, \\
H(x,y)=\Gamma_s(y-x) & \hbox{ in }\C\Omega.
\end{array}\right.
$$

\begin{lem}\label{lem-H} Fix $x\in\Omega$. Then $H(x,\cdot)\in C^{2s+\ep}(\Omega)\cap C(\overline{\Omega})$.
\end{lem}
\begin{proof}\rm
Take $r=r(x)>0$ such that $\overline{B_r(x)}\subseteq\Omega$
and $\kappa$ a cutoff function
$$
0\leq \kappa\leq 1\hbox{ in }\R^N, \quad
\kappa\in C^\infty(\R^N), \quad
\kappa=1 \hbox{ in }\overline{\C\Omega}, \quad
\kappa=0 \hbox{ in }\overline{B_r(x)}
$$
and define $\gamma_s(y):=\kappa(y)\Gamma_s(y-x)$:
$\gamma_s\in C^\infty(\R^N),\gamma_s=\Gamma_s$ in $\C\Omega$, $\gamma_s=0$ in $B_r(x)$.
Then establish the equivalent problem
$$
\forall\;x\in\Omega\qquad
\left\lbrace\begin{array}{ll}
\Ds H(x,\cdot)=0 & \hbox{ in }\Omega \\
H(x,y)=\gamma_s(y-x) & \hbox{ in }\C\Omega
\end{array}\right.
$$
and by setting $h(x,y):=H(x,y)-\gamma_s(x,y)$ we obtain
$$
\forall\;x\in\Omega\qquad
\left\lbrace\begin{array}{ll}
\Ds h(x,\cdot)=-\Ds\gamma_s & \hbox{ in }\Omega, \\
h(x,y)=0 & \hbox{ in }\C\Omega.
\end{array}\right.
$$
Note that $\Ds\gamma_s\in L^\infty(\R^N)\cap C^\infty(\R^N)$ (see \cite[Proposition 2.7]{silvestre}):
from this we deduce that $h(x,\cdot)\in C^{2s+\ep}(\Omega)\cap C^s(\R^N)$ by \cite[Proposition 2.8]{silvestre} and \cite[Proposition 1.1]{rosserra} respectively
and then also $H(x,\cdot)\in C^{2s+\ep}(\Omega)\cap C^s(\R^N)$.
\end{proof}

\begin{lem} The function $G_\Omega(x,y)$ we have just obtained
satisfies the following properties:
\begin{description}
\item[\it i)] $G_\Omega$ is continuous in $\Omega\times\Omega$ except on the diagonal\footnote{where 
its singularity is inherited by the singularity in $0$ of $\Gamma_s$},
\item[\it ii)] $\Ds G_\Omega(x,\cdot)\in L^1(\C\Omega)$ for any $x\in\Omega$,
\item[\it iii)] for any $u\in C^{2s+\ep}(\Omega)\cap L^\infty(\R^N)$
and $x\in\Omega$
\begin{equation}\label{green-repr}
u(x)=\int_\Omega\Ds u(y)\:G_\Omega(x,y)\;dy-\int_{\C\Omega}u(y)\:\Ds G_\Omega(x,y)\;dy,
\end{equation}
and this formula is a Green's representation formula
($G_\Omega$ is the Green function
while its fractional Laplacian is the Poisson kernel),
\item[\it iv)] $\Ds G_\Omega(x,y)$, $x\in\Omega,\,y\in\C\Omega$
is given by the formula
\begin{equation}\label{green-pois}
\Ds G_\Omega(x,y)=
-C_{N,s}\,\int_{\Omega}\frac{G_\Omega(x,z)}{|z-y|^{N+2s}}\;dz,
\end{equation}
\item[\it v)] $G_\Omega(x,y)\geq 0$ for a.e. $(x,y)\in\Omega\times\Omega$, $x\neq y$,
and $\Ds G_\Omega(x,y)\leq 0$ for any $x\in\Omega,\,y\in\C\Omega$,
\item[\it vi)] it holds
\begin{align}
& \Gamma_s(y-x)=-\int_{\C\Omega}\Gamma_s(y-z)\cdot\Ds G_\Omega(x,z)\;dz
\qquad \hbox{for }x\in\Omega,\ y\in\C\Omega, \label{gammas-eq}\\
& G_\Omega(x,y)-\Gamma_s(y-x)=-\int_{\C\Omega}\Gamma_s(y-z)\cdot\Ds G_\Omega(x,z)\;dz
\qquad \hbox{for }x\in\Omega,\ y\in\Omega. 
\end{align}
\end{description}
\end{lem}
\begin{proof}\rm We prove all conclusions step by step.

\it Proof of ii). \rm
First of all, we use the estimate
$$
|h(x,y)|\leq C\Arrowvert\Ds\gamma_s(x,\cdot)\Arrowvert_\infty\;\delta(y)^s,
$$
where $h$ solves 
$$
\forall\;x\in\Omega\quad
\left\lbrace\begin{array}{ll}
\Ds h(x,\cdot)=-\Ds\gamma_s & \hbox{ in }\Omega \\
h(x,y)=0 & \hbox{ in }\C\Omega
\end{array}\right.
$$
and $\gamma_s$ is a regularization of $\Gamma_s$ as in Lemma \ref{lem-H};
for the inequality we refer to \cite[Proposition 1.1]{rosserra}. We deduce that,
for $y$ sufficiently close to $\partial\Omega$, it holds
\[
|G_\Omega(x,y)|=|\Gamma_s(y-x)-H(x,y)|=|\gamma_s(y-x)-H(x,y)|=
|h(x,y)|\leq C\Arrowvert\Ds\gamma_s(x,\cdot)\Arrowvert_\infty \delta(y)^s.
\]
Then, for $y\in\C\Omega$,
\begin{equation}\label{est-pois}
\begin{split}
|\Ds G_\Omega(x,y)|=\left|C_{N,s}\,PV\int_{\R^N}\frac{G_\Omega(x,z)-G_\Omega(x,y)}{|y-z|^{N+2s}}\;dz\right|
=\qquad\qquad\qquad\\
=C_{N,s}\,\int_{\Omega}\frac{G_\Omega(x,z)}{|y-z|^{N+2s}}\;dz\leq
C\int_{\Omega}\frac{\delta(z)^s}{|y-z|^{N+2s}}\;dz
\end{split}
\end{equation}
and
\begin{equation}
\begin{split}
\int_{\C\Omega}|\Ds G_\Omega(x,y)|\;dy
\leq C\int_{\C\Omega}\int_{\Omega}\frac{\delta(z)^s}{|y-z|^{N+2s}}\;dz\;dy\leq\qquad\qquad\qquad\\
\leq C\int_{\Omega}\delta(z)^s\int_{\delta(z)}^{+\infty}\rho^{-1-2s}\;d\rho\;dz\leq
\frac{C}{2s}\int_\Omega\frac{dz}{\delta(z)^s}<+\infty.
\end{split}
\end{equation}

\it Proof of iii). \rm 
The function $G_\Omega(x,y)=\Gamma_s(y-x)-H(x,y)$ can be 
integrated by parts against any $u\in C^{2s+\ep}(\R^N\setminus\Omega)\cap L^\infty(\R^N)$,
since, according to Lemma \ref{lem-H}, $H(x,\cdot)\in C^{2s+\ep}(\Omega)\cap C(\overline{\Omega})$ 
and so we can apply Proposition \ref{integrbypartsform}. Hence,
\begin{equation*}
u(x)=\int_\Omega\Ds u(y)\:G_\Omega(x,y)\;dy-\int_{\C\Omega}u(y)\:\Ds G_\Omega(x,y)\;dy.
\end{equation*}

\it Proof of iv). \rm We point how the computation 
of $\Ds G_\Omega(x,y)$, $x\in\Omega,\,y\in\C\Omega$
reduces to a more readable formula:
\[
-\Ds G_\Omega(x,y)=C_{N,s}\,PV\int_{\R^N}\frac{G_\Omega(x,z)-G_\Omega(x,y)}{|z-y|^{N+2s}}\;dz\ =\\
=C_{N,s}\,\int_{\Omega}\frac{G_\Omega(x,z)}{|z-y|^{N+2s}}\;dz,
\]
by simply recalling that $G_\Omega(x,z)=0$, when $z\in\C\Omega$.

\it Proof of v). \rm $G_\Omega(x,y)\geq 0$ for $(x,y)\in\Omega\times\Omega$, $x\neq y$,
in view of Lemma \ref{maxprinc} below applied to the function $H(x,\cdot)$.
Also, from this and \eqref{green-pois} we deduce that
\begin{equation}\label{pois-pos}
-\Ds G_\Omega(x,y)=C_{N,s}\,
\int_{\Omega}\frac{G_\Omega(x,z)}{|z-y|^{N+2s}}\;dz\geq 0.
\end{equation}

\it Proof of vi). \rm It suffices to apply \eqref{green-repr}
to the solution $H(x,y)$ of
$$
\left\lbrace\begin{array}{ll}
\Ds H(x,\cdot)=0 & \hbox{ in }\Omega \\
H(x,y)=\Gamma_s(y-x) & \hbox{ in }\C\Omega
\end{array}\right.
$$
to infer 
\begin{eqnarray*}
 & \displaystyle\Gamma_s(y-x)\ =\ -\int_{\C\Omega}\Gamma_s(y-z)\cdot\Ds G_\Omega(x,z)\;dz
\qquad \hbox{for }x\in\Omega\hbox{ and }y\in\C\Omega, & \\
 & \displaystyle H(x,y)\ =\ -\int_{\C\Omega}\Gamma_s(y-z)\cdot\Ds G_\Omega(x,z)\;dz
\qquad \hbox{for }x\in\Omega\hbox{ and }y\in\Omega. &
\end{eqnarray*}
\end{proof}

\begin{lem} For any $x\in\Omega,\,\theta\in\partial\Omega$ the function
$$
M_\Omega(x,\theta)=\lim_{\stackrel{\hbox{\scriptsize $y\in\Omega$}}{y\rightarrow\theta}}
\frac{G_\Omega(x,y)}{{\delta(y)}^s}
$$
is well-defined in $\Omega\times\partial\Omega$ and
for any $h\in C(\partial\Omega),\,\psi\in C^\infty_c(\Omega)$ one has
$$
\int_\Omega\left(\int_{\partial\Omega}M_\Omega(x,\theta)\,h(\theta)\;d\mathcal{H}(\theta)\right)\psi(x)\;dx
=\int_{\partial\Omega}h(\theta)\,D_s\phi(\theta)\;d\mathcal{H}(\theta),
$$
where
$$
D_s\phi(\theta)=\lim_{\stackrel{\hbox{\scriptsize $y\in\Omega$}}{y\rightarrow\theta}}
\frac{\phi(y)}{{\delta(y)}^s}
\qquad\hbox{ and }\qquad
\phi(y)=\int_\Omega G_\Omega(x,y)\,\psi(x)\;dx.
$$
\end{lem}
\begin{proof}\rm
For a small parameter $\ep>0$ define $\Omega_\ep=\{x\in\overline{\Omega}:0\leq\delta(x)<\ep\}$:
associate to any $x\in\Omega$ a couple $(\rho,\theta)$ where $\rho=\delta(x)$
and $\theta\in\partial\Omega$ satisfies $|x-\theta|=\rho$: such a $\theta$
is uniquely determined for small $\ep$ since $\partial\Omega\in C^{1,1}$.
Take also $\varphi\in C^\infty(\R)$, with $\varphi(0)=1$ and supported in $[-1,1]$.
With a slight abuse of notation define
\begin{equation}\label{h-approx}
f_\ep(x)=f_\ep(\rho,\theta)=h(\theta)\,\frac{\varphi(\rho/\ep)\,\rho^{-s}}{K_\ep},
\quad K_\ep=\frac{1}{1+s}\int_0^\ep\varphi(r/\ep)\;dr.
\end{equation}
Consider the functions
$$
u_\ep(x)=\int_{\Omega}f_\ep(y)\,G_\Omega(x,y)\;dy,
$$
and any function $\psi\in C^\infty_c(\Omega)$:
\begin{eqnarray*}
\int_\Omega u_\ep\psi & = & 
\int_{\Omega}f_\ep(y)\left[\int_\Omega\psi(x)\,G_\Omega(x,y)\;dx\right]dy
=\int_{\Omega_\ep}h(\theta(y))\,\frac{\varphi(\delta(y)/\ep)\,{\delta(y)}^{-s}}{K_\ep}\left[\int_\Omega\psi(x)\,G_\Omega(x,y)\;dx\right]dy \\
& = & \int_{\Omega_\ep}h(\theta(y))\,\frac{\varphi(\rho(y)/\ep)}{K_\ep}
\left[\int_\Omega\psi(x)\,\frac{G_\Omega(x,y)}{{\delta(y)}^s}\;dx\right]\;dy.
\end{eqnarray*}
Note that the function $\psi_1(y)=\int_\Omega\psi(x)\,G_\Omega(x,y)\,{\delta(y)}^{-s}\;dx$ is 
continuous in $\overline{\Omega_\ep}$, as a consequence of the boundary
behaviour of $G_\Omega$, see \cite[equation (2.13)]{chen}.
Then
\begin{multline*} 
\int_{\Omega_\ep}h(\theta(y))\,\frac{\varphi(\rho(y)/\ep)}{K_\ep}\,\psi_1(y)\;dy=
\frac{1}{(1+s)\,K_\ep}\int_0^\ep\varphi(r/\ep)\left(\int_{\partial\Omega}h(\theta)\,\psi_1(r,\theta)\;d\mathcal{H}(\theta)\right)\;dr \\
\xrightarrow{\ep\downarrow0}\ \int_{\partial\Omega}h(\theta)\,\psi_1(0,\theta)\;d\mathcal{H}(\theta)
=\int_{\partial\Omega}h(\theta)\int_\Omega \psi(x)\lim_{y\rightarrow\theta}\frac{G_\Omega(x,y)}{{\delta(y)}^{s}}
\;dx\;d\mathcal{H}(\theta).
\end{multline*}
Hence
\begin{equation}\label{000}
\int_\Omega u_\ep\psi\xrightarrow{\ep\downarrow0}\int_\Omega
\left(\int_{\partial\Omega}M_\Omega(x,\theta)\,h(\theta)\;d\mathcal{H}(\theta)\right)\psi(x)\;dx
\end{equation}
where
$$
M_\Omega(x,\theta)=\lim_{\stackrel{\hbox{\scriptsize $y\in\Omega$}}{y\rightarrow\theta}}\frac{G_\Omega(x,y)}{{\delta(y)}^{s}}.
$$
Note that, always in view of the boundary estimates on $G_\Omega$, the function $M_\Omega(x,\cdot)\in L^\infty(\partial\Omega)$ for any fixed $x\in\Omega$.
In addition,
\begin{multline}\label{001}
\int_\Omega u_\ep\psi =
\int_{\Omega_\ep}h(\theta(y))\,\frac{\varphi(\rho(y)/\ep)}{K_\ep}\left[\int_\Omega\psi(x)\,G_\Omega(x,y)\;dx\right]\;dy\ =\\
=\ \int_{\Omega_\ep}h(\theta(y))\,\frac{\varphi(\rho(y)/\ep)\,\delta(y)^s}{K_\ep\,\delta(y)^s}\,\phi(y)\;dy\  
\xrightarrow[]{\ep\downarrow 0}\  
\int_{\partial\Omega}h(\theta)\,D_s\phi(\theta)\;d\mathcal{H}(\theta).
\end{multline}
So the limits \eqref{000} and \eqref{001} must coincide.
\end{proof}

\begin{rmk}\rm The function $M_\Omega(x,\theta)$ we have just introduced
is closely related to the \it Martin kernel based at $x_0\in\Omega$\rm 
$$
\widetilde M_\Omega(x,\theta)=\lim_{\stackrel{\hbox{\scriptsize $y\in\Omega$}}{y\rightarrow\theta}}
\frac{G_\Omega(x,y)}{G_\Omega(x_0,y)}.
$$
For this reason we borrow
the usual notation of the Martin kernel.
\end{rmk}

\begin{lem}\label{Eu} For any $h\in C(\partial\Omega)$ define
$$
u(x)=\int_{\partial\Omega}M_\Omega(x,\theta)h(\theta)\;d\mathcal{H}(\theta),
\qquad x\in\Omega.
$$
Then for any $\theta^*\in\partial\Omega$
$$
Eu(\theta^*):=\lim_{\stackrel{\hbox{\scriptsize $x\in\Omega$}}{x\rightarrow\theta^*}}
\frac{u(x)}{\int_{\partial\Omega} M_\Omega(x,\theta)\;d\mathcal{H}(\theta)}=h(\theta^*).
$$
\end{lem}
\begin{proof}\rm
Denote by
$$
L(x)={\delta(x)}^{1-s}\int_{\partial\Omega} M_\Omega(x,\theta)\;d\mathcal{H}(\theta),
\qquad x\in\Omega
$$
for which we have
$$
\int_{\partial\Omega}\frac{\delta(x)}{c\,{|x-\theta|}^N}\;d\mathcal{H}(\theta)
\leq L(x)\leq
\int_{\partial\Omega}\frac{c\,\delta(x)}{{|x-\theta|}^N}\;d\mathcal{H}(\theta),
$$
and so $L$ is a bounded quantity.
Indeed, referring to estimates on the Green function in \cite[equation (2.13)]{chen}, we have inequalities
\begin{equation}\label{chen-martin}
\frac{{\delta(x)}^s}{c\,{|x-\theta|}^N}\leq M_\Omega(x,\theta)\leq\frac{c\,{\delta(x)}^s}{{|x-\theta|}^N},
\qquad x\in\Omega,\theta\in\partial\Omega.
\end{equation} 
Thus,
\begin{multline*}
\left|\frac{{\delta(x)}^{1-s}\,u(x)}{L(x)}-h(\theta^*)\right|=
\left|\frac{{\delta(x)}^{1-s}}{L(x)}\int_{\partial\Omega}M_\Omega(x,\theta)h(\theta)\;d\mathcal{H}(\theta)
-h(\theta^*)\frac{{\delta(x)}^{1-s}}{L(x)}\int_{\partial\Omega}M_\Omega(x,\theta)\;d\mathcal{H}(\theta)\right|\ \leq\\
\leq\ \frac{{\delta(x)}^{1-s}}{L(x)}\int_{\partial\Omega}
M_\Omega(x,\theta)\left|h(\theta)-h(\theta^*)\right|\;d\mathcal{H}(\theta)
\leq C\delta(x)\int_{\partial\Omega}\frac{\left|h(\theta)-h(\theta^*)\right|}{{|x-\theta|}^N}\;d\mathcal{H}(\theta).
\end{multline*}
Call $\theta_x\in\partial\Omega$ the point satisfying $\delta(x)=|x-\theta_x|$.
Describe $\partial\Omega$ as a graph in a neighbourhood of $0$,
i.e. $\Gamma\subseteq\partial\Omega$ open, and 
\begin{eqnarray*}
 & \Gamma\ni\theta=(\theta',\phi(\theta'))\hbox{ for some }\phi:B_r'\subseteq\R^{N-1}\rightarrow\R,\ (0,\phi(0))=\theta_x, & \\
& \phi\in C^{1,1}(B_r(0)),\ \nabla\phi(0)=0. &
\end{eqnarray*}
It holds
$\theta^*\in\Gamma$ for $x$ close enough to $\theta^*$.
Let us now write, 
$$
\int_{\Gamma}\frac{\left|h(\theta)-h(\theta^*)\right|}{{|x-\theta|}^N}\;d\mathcal{H}(\theta)
\leq \int_{\Gamma}\frac{\left|h(\theta)-h(\theta^*)\right|}
{\left[|x-\theta_x|^2+|\theta_x-\theta|^2-2\langle x-\theta_x,\theta_x-\theta\rangle\right]^{N/2}}
\;d\mathcal{H}(\theta)
$$
where
$$
2\langle x-\theta_x,\theta_x-\theta\rangle=2|x-\theta_x|\cdot|\theta_x-\theta|
\left\langle\frac{x-\theta_x}{|x-\theta_x|},\frac{\theta_x-\theta}{|\theta_x-\theta|}\right\rangle
\leq 2\mu|x-\theta_x|\cdot|\theta_x-\theta|,\ \theta\in\Gamma
$$
for some $\mu<1$, since $x-\theta_x$ is normal to $\partial\Omega$, 
while $\theta_x-\theta$ is ``almost tangent'' to $\partial\Omega$. Then
$$
\int_{\Gamma}\frac{\left|h(\theta)-h(\theta^*)\right|}{{|x-\theta|}^N}\;d\mathcal{H}(\theta)\leq
\frac{1}{(1-\mu)^{N/2}}\int_\Gamma\frac{\left|h(\theta)-h(\theta^*)\right|}
{\left[|x-\theta_x|^2+|\theta_x-\theta|^2\right]^{N/2}}\;d\mathcal{H}(\theta)
$$
Suppose without loss of generality that $\theta_x=0$ and denote by $\omega$ the modulus of continuity of $h$:
\begin{align*}
& \int_{\Gamma}\frac{\left|h(\theta)-h(\theta^*)\right|}{{|x-\theta|}^N}\;d\mathcal{H}(\theta)\leq
\frac{\sup_\Gamma\omega(|\theta-\theta^*|)}{(1-\mu)^{N/2}}\int_{B_r'}\frac{d\theta'}{\left[|x|^2+|\theta'|^2+|\phi(\theta')|^2\right]^{N/2}} \leq \\
& \leq  
\frac{\sup_\Gamma\omega(|\theta-\theta^*|)}{(1-\mu)^{N/2}}\int_0^r\frac{\rho^{N-2}\;d\rho}{\left[|x|^2+\rho^2\right]^{N/2}} 
\leq \\
& \leq\frac{\sup_\Gamma\omega(|\theta-\theta^*|)}{(1-\mu)^{N/2}}\int_0^r\frac{d\rho}{|x|^2+\rho^2}
\leq\frac{\sup_\Gamma\omega(|\theta-\theta^*|)}{|x|\,(1-\mu)^{N/2}}
\end{align*}
and therefore
\begin{multline*}
\delta(x)\int_{\partial\Omega}\frac{\left|h(\theta)-h(\theta^*)\right|}{{|x-\theta|}^N}\;d\mathcal{H}(\theta)=
\delta(x)\int_{\Gamma}\frac{\left|h(\theta)-h(\theta^*)\right|}{{|x-\theta|}^N}\;d\mathcal{H}(\theta)+
\delta(x)\int_{\partial\Omega\setminus\Gamma}\frac{\left|h(\theta)-h(\theta^*)\right|}{{|x-\theta|}^N}\;d\mathcal{H}(\theta)
\ \leq\\ \ \leq 
\frac{\sup_\Gamma\omega(|\theta-\theta^*|)}{(1-\mu)^{N/2}}+o(\delta(x)).
\end{multline*}
Now, since $\sup_\Gamma\omega(|\theta-\theta^*|)\leq\omega(\hbox{diam}\Gamma)$,
where diam$\Gamma=\sup_{\theta,\theta'\in\Gamma}|\theta-\theta'|$, and $\Gamma$ is arbitrary,
we deduce
$$
\delta(x)\int_{\partial\Omega}\frac{\left|h(\theta)-h(\theta^*)\right|}{{|x-\theta|}^N}\;d\mathcal{H}(\theta)
\xrightarrow[x\rightarrow\theta^*]{}0.
$$
\end{proof}
\medskip

\subsection{Linear theory for smooth data: proof of Theorem \ref{pointwise}}

We start by stating

\begin{lem}[maximum principle]\label{maxprinc} 
Let $\mathcal{O}\subset\R^N$ be an open
set such that $\overline{\Omega}\subset\mathcal{O}$.
Let $u:\R^N\rightarrow\R$ be a measurable function in
$C^{2s+\ep}(\Omega)\cap C(\mathcal{O})$,
and
$$
\int_{\R^N}\frac{|u(y)|}{1+|y|^{N+2s}}\;dy<+\infty,\qquad\hbox{and}\qquad
\left\lbrace\begin{array}{ll}
\Ds u\leq 0 & \hbox{ in }\Omega \\
u\leq 0 & \hbox{ in }\C\Omega.
\end{array}\right.
$$
Then $u\leq 0$ in $\Omega$.
\end{lem}
\begin{proof}\rm Call $\Omega^+=\{x\in\R^N:u>0\}$:
$\Omega^+$ is an open set contained in $\Omega$.
Assume by contradiction that $\Omega^+\neq\emptyset.$
By continuity of $u$ in $\overline{\Omega^+}\subset\mathcal{O}$, 
there exists $x_0\in\overline{\Omega^+}$ such that
$u(x_0)=\max\{u(x):x\in\overline{\Omega^+}\}$; moreover,
$x_0\in\Omega^+$, since $u\leq 0$ on $\partial\Omega^+$.
This point $x_0$
will be also a global maximum for $u$ since outside $\Omega^+$
the function $u$ is nonpositive.
Thus
$$
\Ds u(x_0)=C_{N,s}\,
PV\int_{\R^N}\frac{u(x_0)-u(y)}{|x_0-y|^{N+2s}}>0
$$
contradicting our hypotheses. Therefore $\Omega^+$ is empty.
\end{proof}

By splitting $g$ into its positive and negative part,
it suffices to prove Theorem \ref{pointwise} in the case where $g\geq0$.
So, from now on we will deal with nonnegative boundary data $g:\C\Omega\rightarrow[0,+\infty)$
which are measurable functions with
\begin{equation}\label{g}
0\leq-\int_{\C\Omega}g(y)\cdot\Ds G_\Omega(x,y)\;dy<+\infty
\qquad x\in\Omega.
\end{equation}
Note that, in view of equation \eqref{est-pois}, for $x\in\Omega,\ y\in\C\Omega$
\begin{equation}\label{est-pois2}
0\leq-\Ds G_\Omega(x,y)\leq C\int_{\C\Omega}\frac{\delta(y)^s}{\left|y-z\right|^{N+2s}}\;dz
\leq C\int_{\delta(y)}^{+\infty}\frac{d\rho}{\rho^{1+s}}=\frac{C}{s}\frac{1}{\delta(y)^s}.
\end{equation}
\medskip

\begin{proof}[Proof of Theorem \ref{pointwise}]
We split the proof by building the solution associated to each 
datum $f$, $g$ and $h$ separately.

\paragraph{First case: $f,\,h\equiv 0$.}

We present here a readaptation of \cite[Lemma 1.13]{landkof}.

The function $u$ defined by equation \eqref{linearsol} is 
continuous in $\Omega$ as an application of the Dominated Convergence Theorem
and inequality \eqref{est-pois2}.
The continuity up to the boundary is postponed to Paragraph \ref{cont-sharm-sec}.

For the sake of clarity we divide the proof in four steps:
for special forms of $g$, for $g$ regular enough, for $g$ bounded
and finally for any other $g$.

\it Step 1. \rm Suppose we have a measure $\nu$, such that $\nu(\Omega)=0$ and
$$
g(x)=\int_{\C\Omega}\Gamma_s(x-y)\;d\nu(y)\quad x\in\C\Omega.
$$
Then set
$$
\widetilde{u}(x)=\int_{\R^N}\Gamma_s(x-y)\;d\nu(y)\qquad\hbox{for any } x\in\R^N:
$$
\begin{description}
\item[\it a)] $\widetilde{u}=g$ in $\C\Omega$, since $\nu$ is supported in $\C\Omega$,
\item[\it b)] $\widetilde{u}=u$ in $\Omega$, where $u$ is given by
$$
u(x)=-\int_{\C\Omega}g(y)\,\Ds G_\Omega(x,y)\;dy,\quad x\in\Omega,
$$
indeed,
\begin{align*}
& -\int_{\C\Omega}g(y)\cdot\Ds G_\Omega(x,y)\;dy \ =
\ -\int_{\C\Omega}\Ds G_\Omega(x,y)\left(\int_{\C\Omega}\Gamma_s(z-y)\;d\nu(z)\right)\;dy \\ 
& =\ -\int_{\C\Omega}\left(\int_{\C\Omega}\Gamma_s(z-y)\cdot\Ds G_\Omega(x,y)\;dy\right)\;d\nu(z) \\ 
& \hbox{\small [in view of equation \eqref{gammas-eq}]} \\
& = \int_{\C\Omega}\Gamma_s(z-y)\;d\nu(z)
\ = \ \widetilde{u}(x),
\end{align*}
\item[\it c)] $u$ is $s$-harmonic in $\Omega$:
\begin{align*}
& \left(\eta_r*u\right)(x) \ = \
\int_{\C B_r(x)}\eta_r(y-x)\cdot u(y)\;dy \ =\ 
\int_{\C B_r(x)}\eta_r(y-x)\left(\int_{\R^N}\Gamma_s(z-y)\;d\nu(z)\right)\;dy \\
& =\ \int_{\R^N}\left(\int_{\C B_r(x)}\Gamma_s(z-y)\,\eta_r(y-x)\;dy\right)\;d\nu(z)\
 =\ \int_{\R^N}\Gamma_s(x-z)\;d\nu(z) \ = \ u(x)
\end{align*}
by choosing $0<r<\delta(x)\leq|z-x|$
and exploiting the $s$-harmonicity of $\Gamma_s$.
\end{description}

\it Step 2. \rm If $g\in C^\infty(\overline{\C\Omega})$ and supp$\,g$ is bounded,
then $g$ admits an extension $\tilde{g}\in C^\infty_c(\R^N)$ and (see \cite[Lemma 1.1]{landkof})
there exists an absolutely continuous measure $\nu$, with 
density $\Psi$, such that 
$$
\tilde{g}(x)=\int_{\C\Omega}\Gamma_s(x-y)\;d\nu(y)
\quad\hbox{for any $x\in\R^N$}.
$$
Denote by $\nu_\Omega$ the measure obtained by restricting $\nu$ to $\Omega$, i.e.
$\nu_\Omega(A)=\nu(A\cap\Omega)$ for any measurable $A$ and $\nu_\Omega$
has density $\Psi_\Omega=\Psi\chi_\Omega$, and 
$$
\nu_\Omega'(y)=\left\lbrace\begin{array}{ll}
\int_\Omega\Ds G_\Omega(x,y)\;d\nu_\Omega(x)
=-\int_\Omega\Ds G_\Omega(x,y)\,\Psi_\Omega(x)\;dx & y\in\C\Omega \\
0 & y\in\Omega
\end{array}\right.:
$$
the integral is well-defined because $\Psi_\Omega\in L^1(\Omega)$
while $\Ds G_\Omega(\cdot,y)\in C(\overline\Omega)$ for any fixed $y\in\C\Omega$
as a consequence of $G_\Omega(\cdot,y)\in C(\overline\Omega)$ and equation \eqref{green-pois}.
Define $\gamma=\nu-\nu_\Omega+\nu_\Omega'$ which is a measure supported in $\C\Omega$.
Then, when $x\in\C\Omega$,
\[
\int_{\R^N}\Gamma_s(x-y)\;d\nu_\Omega'(y)=
-\int_\Omega\left(\int_{\C\Omega}\Gamma_s(x-y)\,\Ds G_\Omega(z,y)\;dy\right)
\;d\nu_\Omega(z)=\int_\Omega\Gamma_s(x-z)\;d\nu_\Omega(z)
\]
where we have used \eqref{gammas-eq}. Therefore
$$
\int_{\C\Omega}\Gamma_s(x-y)\;d\gamma(y)
=\int_{\C\Omega}\Gamma_s(x-y)\;d\nu(y)=\tilde{g}(x)
$$
so that we can apply the previous step of the proof.

\it Step 3. \rm For $g\in L^\infty(\C\Omega)$,
consider a sequence 
${\{g_n\}}_{n\in\N}\subseteq C^\infty(\overline{\C\Omega})$ uniformly bounded 
and converging pointwisely to $g$.
The corresponding sequence of $s$-harmonic functions $u_n$ converges to $u$,
since
\[
u_n(x)=-\int_{\C\Omega}g_n(y)\,\Ds G_\Omega(x,y)\;dy
\xrightarrow[n\uparrow+\infty]{}
u(x)=-\int_{\C\Omega}g(y)\,\Ds G_\Omega(x,y)\;dy
\]
by Dominated Convergence.
Then, again by the Dominated Convergence theorem we have
$$
u_n\ =\ \eta_\delta*u_n\longrightarrow\eta_\delta*u,\quad \hbox{in }\Omega
$$
therefore
$$
u\ =\ \lim u_n\ =\ \eta_\delta*u,\quad \hbox{in }\Omega,
$$
i.e. $u$ is $s$-harmonic in $\Omega$.

\it Step 4. \rm For a general measurable nonnegative $g$ it suffices now to consider
an increasing sequence $g_n$ converging to $g$, e.g. $g_n=\min\{g,n\}$.
Then the corresponding sequence of $s$-harmonic functions $u_n$ converges to $u$.
Moreover, the sequence ${\{u_n\}}_{n\in\N}$ is increasing:
\[
u_{n+1}(x)=-\int_{\C\Omega}g_{n+1}(y)\,\Ds G_\Omega(x,y)\;dy
\geq -\int_{\C\Omega}g_n(y)\,\Ds G_\Omega(x,y)\;dy=u_n(x).
\]
Then, thanks to the Monotone Convergence theorem we have
$$
u_n\ =\ \eta_\delta*u_n\longrightarrow\eta_\delta*u,\quad \hbox{in }\Omega
$$
therefore
$$
u=\lim u_n=\eta_\delta*u,\quad \hbox{in }\Omega.
$$

\it Uniqueness. \rm Finally, if $g\in C(\overline{\C\Omega})$,
the solution we have built is the only solution in $C(\overline{\Omega})$
as an application of Lemma \ref{maxprinc}.


\begin{rmk}\rm Suppose to have $g:\C\Omega\rightarrow[0,+\infty)$
for which \eqref{g} fails and there is a set $\mathcal{O}\subseteq\Omega$,
$|\mathcal{O}|>0$, in which
$$
-\int_{\C\Omega}g(y)\,\Ds G_\Omega(x,y)\;dy=+\infty,\qquad x\in\mathcal{O}.
$$
It is not possible in this case to have a pointwise solution of
$$
\left\lbrace\begin{array}{ll}
\Ds u=0 & \hbox{ in }\Omega, \\
u=g & \hbox{ in }\C\Omega.
\end{array}\right.
$$
Indeed, if we set $g_n=\min\{g,n\}$, $n\in\N$, then
$g_n$ converges monotonically to $g$, and 
$$
u(x)\geq-\int_{\C\Omega}g_n(y)\,\Ds G_\Omega(x,y)\;dy
\xrightarrow{\ n\uparrow+\infty\ }+\infty
$$
for all $x\in\mathcal{O}$.
\end{rmk}

{\it Regularity. } As a last step we notice that the $C^\infty$ regularity 
is a consequence of the representation on balls and the regularity
of the associated Poisson kernel therein. Properties listed in \eqref{gg}
are a consequence of Paragraph \ref{traces} below and the estimates on $\Ds G_\Omega(x,y),\ x\in\Omega,\ y\in\C\Omega$.

\paragraph{Second case: $g,\,h\equiv 0$.}

Use the construction of $G_\Omega$ to write for $x\in\Omega$
$$
u(x)=\int_{\Omega}f(y)\,G_\Omega(y,x)\;dy=
\int_{\Omega}f(y)\,\Gamma_s(y-x)\;dy-\int_{\Omega}f(y)\,H(y,x)\;dy:
$$
the first addend is a function $u_1(x)$ which solves 
$\Ds u_1=f\chi_\Omega$ in $\R^N$, let us turn to the second one:
\begin{eqnarray*}
u_2(x)\ :=\ \int_{\Omega}f(y)H(y,x)\;dy & = & 
-\int_{\Omega}f(y)\left[
\int_{\C\Omega}\Gamma_s(y-z)\Ds H(x,z)\;dz\right]\;dy \\
& = & -\int_{\C\Omega}\Ds H(x,z)\left[
\int_{\Omega}f(y)\Gamma_s(y-z)\;dy\right]\;dz \\
& = & -\int_{\C\Omega}\Ds H(x,z)\,u_1(z)\;dz \ = \ -\int_{\C\Omega}u_1(z)\,\Ds G_\Omega(x,z)\;dz.
\end{eqnarray*}
According to the \it Step 4 \rm above, $u_2$ solves
$$
\left\lbrace\begin{array}{ll}
\Ds u_2=0 & \hbox{ in }\Omega \\
u_2=u_1 & \hbox{ in }\C\Omega
\end{array}\right.
$$
therefore $u=u_1-u_2$ and 
\begin{equation}\label{56}
\left\lbrace\begin{array}{ll}
\Ds u=\Ds u_1-\Ds u_2=f & \hbox{ in }\Omega \\
u=u_1-u_2=0 & \hbox{ in }\C\Omega.
\end{array}\right.
\end{equation}
Finally, $u\in C^{2s+\alpha}(\Omega)$ thanks to \cite[Proposition 2.8]{silvestre},
while for inequality
$$
|u(x)|\leq C \Arrowvert f\Arrowvert_\infty\, \delta(x)^s
$$
we refer to \cite[Proposition 1.1]{rosserra}.

\begin{rmk}\rm Note that these computations give an alternative integral representation
to the one provided in equation \eqref{green-repr}
for $u$, meaning that we have both
$$
\int_{\Omega} G_\Omega(x,y)\,\Ds u(y)\;dy
\ =\ u(x)\ =\ 
\int_{\Omega} G_\Omega(y,x)\,\Ds u(y)\;dy,
$$
so we must conclude that $G_\Omega(x,y)=G_\Omega(y,x)$ for $x,y\in\Omega,\ x\neq y$.
\end{rmk}

\paragraph{Third case: $f,\,g\equiv 0$.}

The function $u(x)=\int_{\partial\Omega}M_\Omega(x,\theta)\,h(\theta)\;d\mathcal{H}(\theta)$
is $s$-harmonic: to show this we use both the construction of $M_\Omega(x,\theta)$
and the mean value formula \eqref{meanform}. 
Using the notations of \eqref{h-approx}, for any $0<r<\delta(x)$
there exists $z\in \overline{B_r(x)}$
\[
u_\ep(x)=\int_{\C B_r(x)}\eta_r(y)\,u_\ep(x-y)\;dy + \gamma(N,s,r)\Ds u_\ep(z)=
\int_{\C B_r(x)}\eta_r(y)\,u_\ep(x-y)\;dy + \gamma(N,s,r)\,f_\ep(z)
\]
where the equality $\Ds u_\ep=f_\ep$ holds throughout $\Omega$.
Letting $\ep\downarrow 0$ we have both
$$
u_\ep(x)\xrightarrow{\ep\downarrow 0}u(x)=\int_{\partial\Omega}M_\Omega(x,\theta)\,h(\theta)\;d\mathcal{H}(\theta)
$$
and 
$$
f_\ep(z)\xrightarrow{\ep\downarrow 0} 0,\qquad\hbox{ for any }z\in\Omega.
$$
If now we note that, for any $p\in[1,N/(N-2s))$ it holds by the Jensen's Inequality
\[
\int_\Omega |u_\eps|^p\ \leq\ \|f_\eps\|^{p-1}_{L^1(\Omega)}\int_\Omega\int_\Omega G_\Omega(x,y)^p\;dx\;dy
\]
which is a quantity bounded independently of $\eps$ (see also Theorem \ref{reg} below),
we deduce that $\{u_\eps\}_{\eps}$ is an equibounded and equicontinuous family in $L^1(\Omega)$
and so the convergence $u_\eps\to u$ holds in $L^1(\Omega)$.
Moreover, when $x$ is far from $\partial\Omega$ and $\eps<\delta(x)$
\[
u_\eps(x)=\int_\Omega G_\Omega(x,y)\,f_\eps(y)\;dy\leq C\int_{\{\delta(y)<\eps\}}\delta(y)f_\eps(y)\;dy
\]
which again is bounded independently on $\eps$.
So we have boundedness in $L^\infty_{loc}(\Omega)$ and convergence in $L^1(\Omega)$.
This implies that we have equality
$$
u(x)=\int_{\C B_r}\eta_r(y)\,u(x-y)\;dy,
$$
i.e. $u$ is $s$-harmonic.

The fact that $\delta^{1-s}u\in C(\overline{\Omega})$ is done in Lemma \ref{Eu},
while the $C^\infty$ regularity can be done once again with 
the representation on balls.
\end{proof}

\begin{lem}[an explicit example on the ball]\label{expl-sol}
If $0<\sigma<1-s$, the functions
$$
u_\sigma(x)=\left\lbrace\begin{array}{ll}
\displaystyle\frac{c(N,s)}{\left(1-|x|^2\right)^\sigma} & |x|<1 \\
\displaystyle\frac{c(N,s+\sigma)}{\left(|x|^2-1\right)^\sigma} & |x|>1 \\
\end{array}\right.\qquad 
u_{1-s}(x)=\left\lbrace\begin{array}{ll}
\displaystyle\frac{c(N,s)}{\left(1-|x|^2\right)^{1-s}} & |x|<1 \\
0 & |x|>1
\end{array}\right.
$$
are $s$-harmonic in the ball $B=B_1(0)$, where $c(N,s)$ is given by \eqref{cns2}.
Moreover in this case
$$
M_B(x,\theta)=
\lim_{\stackrel{\hbox{\scriptsize $y\in\Omega$}}{y\rightarrow\theta\in\partial\Omega}}\frac{G_B^s(x,y)}{{\delta(y)}^s}
=\frac{C\left(1-|x|^2\right)^s}{|x-\theta|^N}
\qquad\hbox{for some }C=C(N,s).
$$
\end{lem}
\begin{proof}\rm
According to \cite[equation (1.6.11')]{landkof} and in view to
the computations due to Riesz \cite{riesz}, the Poisson kernel
for the ball $B$ of radius $1$ and centered at $0$ has the
explicit expression
\begin{equation}\label{pois-ball}
-\Ds G_B^s(x,y)=\frac{c(N,s)}{|x-y|^N}
\left(\frac{1-|x|^2}{|y|^2-1}\right)^s,\qquad
c(N,s)=\frac{\Gamma(N/2)\,\sin(\pi s)}{\pi^{1+N/2}}.
\end{equation}
We construct here the $s$-harmonic function
induced by data
$$
h(\theta)=0,\qquad
g_\sigma(y)=\frac{c(N,s+\sigma)}{\left(|y|^2-1\right)^\sigma},\quad 0<\sigma<1-s.
$$
Indeed, it can be explicitly computed
\[
\begin{aligned}
& \left(1-|x|^2\right)^\sigma u_\sigma(x) \ =\\
&=\ -\left(1-|x|^2\right)^\sigma\int_{\C B}g_\sigma(y)\cdot\Ds G_B^s(x,y)\;dy\ =\  
\int_{\C B}\frac{c(N,s)}{|x-y|^N}\cdot
\frac{{(1-|x|^2)}^{s+\sigma}}{{(|y|^2-1)}^s}\:g_\sigma(y)\;dy\ =\\
& =\ 
\int_{\C B}\frac{c(N,s)\,c(N,s+\sigma)}{|x-y|^N}\cdot\frac{{(1-|x|^2)}^{s+\sigma}}{{(|y|^2-1)}^{s+\sigma}}\;dy 
\ =\ 
-c(N,s)\int_{\C B}(-\lapl)^{s+\sigma}G_B^{s+\sigma}(x,y)\;dy
\ = \ c(N,s)
\end{aligned}
\]
therefore the function
$$
u_\sigma(x)=\left\lbrace\begin{array}{ll}
\displaystyle\frac{c(N,s)}{\left(1-|x|^2\right)^\sigma} & x\in B \\
\displaystyle\frac{c(N,s+\sigma)}{\left(|x|^2-1\right)^\sigma} & x\in\C B \\
\end{array}\right.
$$
solves the problem
$$
\left\lbrace\begin{array}{ll}
\Ds u_\sigma=0 & \hbox{ in }B \\
\displaystyle u_\sigma(x)=g_\sigma(x)=\frac{c(N,s+\sigma)}{\left(|x|^2-1\right)^\sigma} & \hbox{ in }\C B.
\end{array}\right.
$$
We are interested in letting $\sigma\rightarrow 1-s$.
Obviously,
$$
u_\sigma(x)\ \xrightarrow{\sigma\rightarrow 1-s}\ u(x)=
\left\lbrace\begin{array}{ll}
\displaystyle \frac{c(N,s)}{\left(1-|x|^2\right)^{1-s}} & \hbox{ in }B\\
\displaystyle 0 & \hbox{ in }\C B
\end{array}\right.
$$
everywhere in $\R^N\setminus\partial B$. The $s$-harmonicity is preserved,
since for $x\in B$ and any $r\in(0,1-|x|)$,
\begin{multline*}
u(x)=\lim_{\sigma\rightarrow 1-s}u_\sigma(x)=
\lim_{\sigma\rightarrow 1-s}\int_{\C B_r}u_\sigma(y)\,\eta_r(x-y)\;dy= \\
=\lim_{\sigma\rightarrow 1-s}\int_{\C B_r\cap B}u_\sigma(y)\,\eta_r(x-y)\;dy+
\lim_{\sigma\rightarrow 1-s}\int_{\C B}g_\sigma(y)\,\eta_r(x-y)\;dy= \\
=\int_{\C B_r\cap B}u(y)\,\eta_r(x-y)\;dy=
\int_{\C B_r}u(y)\,\eta_r(x-y)\;dy
\end{multline*}
since $g_\sigma\xrightarrow[]{\sigma\rightarrow 1-s}0$ in $L^1(\C B)$,
while $\eta_r(\,\cdot\,-x)$ is bounded in $\C B$ for $0<r<\delta(x)$.
For any $\psi\in C^\infty_c(B)$ we have
\begin{multline*}
\int_B u_\sigma\psi=-\int_{\C B}g_\sigma(x)\int_B \psi(z)\,\Ds G_B^s(z,x)\;dz\;dx\ =\\
=\ -\int_{\C B}\frac{c(N,s+\sigma)}{\left(|x|^2-1\right)^\sigma}\int_B \psi(z)\,\Ds G_B^s(z,x)\;dz\;dx.
\end{multline*}
Then on the one hand
$$
\int_B u_\sigma\psi\xrightarrow{\sigma\rightarrow 1-s}\int_B u\psi=c(N,s)\int_{B}\frac{\psi(x)}{\left(1-|x|^2\right)^{1-s}}\;dx.
$$
On the other hand
\begin{multline*}
-\int_{\C B}\frac{c(N,s+\sigma)}{\left(|x|^2-1\right)^\sigma}\int_B \psi(z)\,\Ds G_B^s(z,x)\;dz\;dx = \\
=
\int_{\C B}\frac{c(N,s+\sigma)}{\left(|x|^2-1\right)^\sigma}
\left[\int_B \psi(z)\frac{c(N,s)}{|z-x|^N}\left(\frac{1-|z|^2}{|x|^2-1}\right)^s\;dz\right]\;dx 
=\\=
\int_{\C B}\frac{c(N,s+\sigma)\,c(N,s)}{\left(|x|^2-1\right)^{\sigma+s}}\int_B\left[\psi(z)\frac{\left(1-|z|^2\right)^s}{|z-x|^N}\;dz\right]\;dx.
\end{multline*}
Note that the function $\int_B \psi(z)\cdot\frac{(1-|z|^2)^s}{|z-x|^N}\;dz$
is $C(\overline{\C B})$ in the $x$ variable, since $\psi\in C^\infty_c(B)$.
Splitting $x\in\C B$ in spherical coordinates, 
i.e. $x=\rho\theta$, $\rho=|x|\in(1,+\infty)$ and $|\theta|=1$, and denoting by $\phi\in C(\overline{\Omega})$
the function satisfying $\Ds\phi|_\Omega=\psi$, $\phi=0$ in $\C\Omega$,
\begin{equation}\label{678}
-c(N,s+\sigma)\int_{\C B}\frac{\Ds\phi(x)}{\left(|x|^2-1\right)^\sigma}\;dx =
 \int_1^{+\infty}\frac{c(N,s+\sigma)\,c(N,s)}{\left(\rho^2-1\right)^{\sigma+s}}\cdot\psi_1(\rho)\rho^{N-1}\;d\rho
\end{equation}
where $\psi_1(\rho)=\int_{\partial B}\int_B \psi(z)\,\frac{\left(1-|z|^2\right)^s}{|z-\rho\theta|^N}\;dz\;d\mathcal{H}(\theta)$
is continuous on $[1,+\infty)$ and 
has a decay at infinity which is comparable to that of $\rho^{-N}$.
Therefore, as $\sigma\rightarrow1-s$,
\begin{multline*}
-c(N,s+\sigma)\int_{\C B}\frac{\Ds\phi(x)}{\left(|x|^2-1\right)^\sigma}\;dx \longrightarrow
c(N,s)\,c(N,1/2)\,\frac{\psi_1(1)}{2}=\\
=\frac{c(N,s)\,c(N,1/2)}{2}\int_{\partial B}\int_B \psi(z)\cdot\frac{\left(1-|z|^2\right)^s}{|z-\theta|^N}\;dz\;d\mathcal{H}(\theta).
\end{multline*}
Indeed, by the definition of $c(N,s)$ in \eqref{cns2}, it holds 
$$
-c(N,s+\sigma)\int_1^{+\infty}\frac{\rho\;d\rho}{\left(\rho^2-1\right)^{\sigma+s}}=
-\frac{c(N,s+\sigma)}{2(1-s-\sigma)}
\xrightarrow[]{\sigma\uparrow 1-s}\frac{c(N,1/2)}{2\pi}
$$
and, since in \eqref{678} the product $\psi_1(\rho)\,\rho^{N-2}\in C([1,+\infty))$,
$$
-c(N,s+\sigma)\int_{\C B}\frac{\Ds\phi(x)}{\left(|x|^2-1\right)^\sigma}\;dx \longrightarrow
c(N,s)\,c(N,1/2)\,\frac{\psi_1(1)}{2}.
$$
So
$$
\int_B u\psi
=\frac{c(N,s)\,c(N,1/2)}{2}\int_{\partial B}\int_B \psi(z)\cdot\frac{\left(1-|z|^2\right)^s}{|z-\theta|^N}\;dz\;d\mathcal{H}(\theta),
$$
i.e. 
$$
u(x)=\frac{c(N,s)\,c(N,1/2)}{2}\int_{\partial B}\frac{\left(1-|x|^2\right)^s}{|x-\theta|^N}\;d\mathcal{H}(\theta)
$$
and then the kernel $M_B(x,\theta)$ for the ball is
$$
M_B(x,\theta)=
\lim_{\stackrel{\hbox{\scriptsize $y\in\Omega$}}{y\rightarrow\theta\in\partial\Omega}}\frac{G_B^s(x,y)}{{\delta(y)}^s}
=C\cdot\frac{\left(1-|x|^2\right)^s}{|x-\theta|^N}
$$
where $C=C(N,s)$.
\end{proof}
\medskip

\section{the linear dirichlet problem: an \texorpdfstring{$L^1$}{L1} theory}\label{stampacchia-frac}

We define $L^1$ solutions for the Dirichlet problem,
in the spirit of Stampacchia \cite{stampacchia}.
A proper functional space in which to consider 
test functions is the following.

\begin{lem}[test function space]\label{testspace}
For any $\psi\in C^\infty_c(\Omega)$ 
the solution $\phi$ of
$$
\left\lbrace\begin{aligned}
\Ds \phi &= \psi & & \hbox{in }\Omega \\
\phi &= 0 & & \hbox{in }\C\Omega \\
E\phi &= 0 & & \hbox{on }\partial\Omega
\end{aligned}\right.
$$
satisfies the following
\begin{enumerate}
\item $\Ds\phi\in L^1(\C\Omega)$ and for any $x\in\C\Omega$
\begin{equation}\label{repr-lapl}
\Ds\phi(x)=\int_\Omega\psi(z)\,\Ds G_\Omega(z,x)\;dz,
\end{equation}
\item for any $\theta\in\partial\Omega$
\begin{equation}\label{repr-deriv}
\int_\Omega M_\Omega(x,\theta)\psi(x)\;dx=D_s\phi(\theta).
\end{equation}
\end{enumerate}
\end{lem}
\begin{proof}\rm The solution $\phi$ is given by 
$\phi(x)=\int_{\R^N}\psi(y)\,G_\Omega(y,x)\;dy\in C^{2s+\ep}(\Omega)\cap C(\overline{\Omega})$.
Also, when $x\in\C\Omega$,
\begin{multline*}
\Ds\phi(x)\ =\ -C_{N,s}\,\int_\Omega\frac{\phi(y)}{|y-x|^{N+2s}}\;dy\ =
\ -\int_\Omega\frac{C_{N,s}}{|y-x|^{N+2s}}\left[\int_\Omega\psi(z)\,G_\Omega(z,y)\;dz\right]\;dy\ =\\
=\ -\int_\Omega\psi(z)\left[C_{N,s}\,\int_\Omega\frac{G_\Omega(z,y)}{|y-x|^{N+2s}}\;dy\right]\;dz\ =\
\int_\Omega\psi(z)\,\Ds G_\Omega(z,x)\;dz.
\end{multline*}
Thanks to the integrability of $\Ds G_\Omega(z,x)$ in $\C\Omega$,
also $\Ds \phi\in L^1(\C\Omega)$,
so that $\phi$ is an admissible function for the integration by parts formula
\eqref{intpartsprop}.
Moreover,
$$
\int_\Omega M_\Omega(x,\theta)\psi(x)\;dx=D_s\phi(\theta).
$$
Indeed,
\[
\int_\Omega M_\Omega(x,\theta)\psi(x)\;dx =
\int_\Omega \psi(x)\,
\lim_{\stackrel{\hbox{\scriptsize $z\!\in\!\Omega$}}{z\rightarrow\theta}}\frac{G_\Omega(x,z)}{{\delta(z)}^s}\;dx=
\lim_{\stackrel{\hbox{\scriptsize $z\!\in\!\Omega$}}{z\rightarrow\theta}}\int_\Omega 
\frac{\psi(x)\,G_\Omega(x,z)}{{\delta(z)}^s}\;dx =
\lim_{\stackrel{\hbox{\scriptsize $z\!\in\!\Omega$}}{z\rightarrow\theta}}\frac{\phi(z)}{{\delta(z)}^s} =
D_s\phi(\theta).
\]
\end{proof}

Then, our space of test functions will be 
$$
\T(\Omega)=\left\lbrace\phi\in C(\R^N):
\left\lbrace\begin{array}{ll}
\Ds \phi=\psi & \hbox{ in }\Omega \\
\phi=0 & \hbox{ in }\C\Omega \\
E\phi=0 & \hbox{ in }\partial\Omega
\end{array}\right.,\ 
\hbox{ when }\psi\in C^\infty_c(\Omega)\right\rbrace.
$$
Note that the map $D_s$ is well-defined from $\T(\Omega)$ to $C(\partial\Omega)$
as a consequence of the results in \cite[Theorem 1.2]{rosserra}.
We give the following

\begin{defi}[Weak $L^1$ solution] Given three Radon measures $\lambda\in\mathcal{M}(\Omega)$, 
$\mu\in\mathcal{M}(\C\Omega)$ and $\nu\in\mathcal{M}(\partial\Omega)$,
we say that a function $u\in L^1(\Omega)$
is a solution of 
$$
\left\lbrace\begin{aligned}
\Ds u &=\lambda & & \hbox{in }\Omega \\
u &=\mu & & \hbox{in }\C\Omega \\
Eu &=\nu & & \hbox{on }\partial\Omega
\end{aligned}\right.
$$
if for every $\phi\in\T(\Omega)$ it holds
\begin{equation}\label{weakdefi}
\int_\Omega u(x)\Ds \phi(x)\;dx\ =\ \int_\Omega\phi(x)\;d\lambda(x)
-\int_{\C\Omega}\Ds \phi(x)\;d\mu(x)+\int_{\partial\Omega}D_s\phi(\theta)\;d\nu(\theta)
\end{equation}
where $\displaystyle D_s\phi(\theta)=\lim_{\stackrel{\hbox{\scriptsize $x\in\Omega$}}{x\rightarrow\theta}}\frac{\phi(x)}{{\delta(x)}^s}$.
\end{defi}

\begin{prop}\label{strsol-weaksol}\rm Solutions provided by Theorem \ref{pointwise}
are solutions in the above $L^1$ sense.
\end{prop}
\begin{proof}\rm Indeed, if 
$$
u(x)=\int_\Omega f(y)\,G_\Omega(y,x)\;dy-\int_{\C\Omega}g(y)\,\Ds G_\Omega(x,y)\;dy
+\int_{\partial\Omega}M_\Omega(x,\theta)\,h(\theta)\;d\mathcal{H}(\theta),
$$
then for any $\phi\in\T(\Omega)$,
\begin{align*}
 & \int_\Omega u(x)\,\Ds\phi(x)\;dx\ = \\ 
 & =\ \int_\Omega \left[\int_\Omega f(y)\,G_\Omega(y,x)\;dy-\int_{\C\Omega}g(y)\,\Ds G_\Omega(x,y)\;dy
 +\int_{\partial\Omega}M_\Omega(x,\theta)\,h(\theta)\;d\mathcal{H}(\theta)\right]\,\Ds\phi(x)\;dx \\ 
 & =\ \int_\Omega f(y)\int_\Omega \Ds\phi(x)\,G_\Omega(y,x)\;dx\;dy
 -\int_{\C\Omega}g(y)\int_\Omega\Ds\phi(x)\,\Ds G_\Omega(x,y)\;dx\;dy\ +\\
 & \qquad+\ \int_{\partial\Omega}h(\theta)\int_\Omega M_\Omega(x,\theta)\Ds\phi(x)
 	d\mathcal{H}(\theta).
\end{align*}
Apply now \eqref{repr-lapl} and \eqref{repr-deriv} to deduce the thesis.
\end{proof}

\begin{proof}[Proof of Theorem \ref{existence-weak2}]
In case $\nu=0$, we claim that the solution is given by formula
$$
u(x)=\int_\Omega G_\Omega(y,x)\;d\lambda(y)-\int_{\C\Omega}\Ds G_\Omega(x,y)\;d\mu(y).
$$
Take $\phi\in\T(\Omega)$:
\begin{align*}
& \int_\Omega u(x)\,\Ds\phi(x)\;dx\ =
\ \int_\Omega \left[\int_\Omega G_\Omega(y,x)\;d\lambda(y)-\int_{\C\Omega}\Ds G_\Omega(x,y)\;d\mu(y)\right]\,\Ds\phi(x)\;dx\ =\\
& =\ \int_\Omega\left[\int_\Omega \Ds\phi(x)\,G_\Omega(y,x)\;dx\right]\;d\lambda(y)\ 
-\ \int_{\C\Omega}\left[\int_\Omega\Ds\phi(x)\,\Ds G_\Omega(x,y)\;dx\right]\;d\mu(y)\ =\\
& =\ \int_\Omega\phi(y)\;d\lambda(y)-\int_{\C\Omega}\Ds\phi(y)\;d\mu(y)
\end{align*}
again using \eqref{repr-lapl}.
Then, we claim that the function
$$
u(x)=\int_{\partial\Omega}M_\Omega(x,\theta)\;d\nu(\theta)
$$
solves problem
$$
\left\lbrace\begin{aligned}
\Ds u &=0 & & \hbox{ in }\Omega, \\
u &=0 & & \hbox{ in }\C\Omega, \\
Eu &=\nu & & \hbox{ on }\partial\Omega.
\end{aligned}\right.
$$
This, along with the first part of the proof, proves our thesis.
Take $\phi\in\T(\Omega)$ and call $\psi=\Ds\phi|_\Omega\in C^\infty_c(\Omega)$.
Then
\begin{eqnarray*}
\int_\Omega\left(\int_{\partial\Omega}M_\Omega(x,\theta)\;d\nu(\theta)\right)\psi(x)\;dx
& = & 
\int_{\partial\Omega}\left(\int_\Omega M_\Omega(x,\theta)\psi(x)\;dx\right)\;d\nu(\theta) \\
& = & 
\int_{\partial\Omega}\left(\int_\Omega \psi(x)\,
\lim_{\stackrel{\hbox{\scriptsize $z\in\Omega$}}{z\rightarrow\theta}}\frac{G_\Omega(x,z)}{{\delta(z)}^s}\;dx\right)\;d\nu(\theta) \\
& = & 
\int_{\partial\Omega}\lim_{\stackrel{\hbox{\scriptsize $z\in\Omega$}}{z\rightarrow\theta}}\left(\int_\Omega 
\frac{\psi(x)\,G_\Omega(x,z)}{{\delta(z)}^s}\;dx\right)\;d\nu(\theta) \\
& = & 
\int_{\partial\Omega}\lim_{\stackrel{\hbox{\scriptsize $z\in\Omega$}}{z\rightarrow\theta}}\frac{\phi(z)}{{\delta(z)}^s}\;d\nu(\theta) 
\ = \ 
\int_{\partial\Omega}D_s\phi(\theta)\;d\nu(\theta).
\end{eqnarray*}

The uniqueness is due to Lemma \ref{max-princweak2} below. 
Theorem \ref{reg-l1sol} below proves the estimate on the $L^1$ norm of the solution.
\end{proof}

\begin{cor}\label{subdomain} The solution provided by Theorem \ref{existence-weak2} satisfies,
for any open $A\subseteq\Omega$ such that $\overline{A}\subseteq\Omega$
and any $\phi_A\in\T(A)$,
\[
\int_A u(x)\Ds\phi_A(x)\;dx=
\int_A\phi_A(x)\;d\lambda(x)-\int_{\Omega\setminus A}u(x)\Ds\phi_A(x)\;dx-\int_{\C\Omega}\Ds\phi(x)\;d\mu(x).
\]
\end{cor}
\begin{proof}
For the sake of clarity we give the proof in the particular case $\mu=\nu=0$;
the other cases follow with the same computations with the necessary modifications.
Establish problem
\[
\left\lbrace\begin{aligned}
& \Ds u_A=\lambda & \hbox{ in }A \\
& u_A=u & \hbox{ in }\C A \\
& Eu_A=0 & \hbox{ on }\partial A
\end{aligned}\right.\quad:
\]
we want to prove that $u_A=u$ in $A$.
The function $u_A$ is given by
\begin{align*}
& u_A(x)=\int_A G_A(x,z)\;d\lambda(z)-\int_{\Omega\setminus A}\Ds G_A(x,z)u(z)\;dz \\
& = \int_A G_A(x,z)\;d\lambda(z)-\int_{\Omega\setminus A}\Ds G_A(x,z)\left[\int_\Omega G_\Omega(z,y)\;d\lambda(y)\right]\;dz \\
& = \int_A G_A(x,z)\;d\lambda(z)-\int_A\left[\int_{\Omega\setminus A}\Ds G_A(x,y)\,G_\Omega(y,z)\;dy\right]d\lambda(z) \\
& \qquad-\ \int_{\Omega\setminus A}\left[\int_{\Omega\setminus A}\Ds G_A(x,y)\,G_\Omega(y,z)\;dy\right]d\lambda(z).
\end{align*}
Since now, by construction of Green functions,
\begin{align*}
& G_A(x,z)-\int_{\Omega\setminus A}\Ds G_A(x,y)\,G_\Omega(y,z)\;dy=G_\Omega(x,z)\, \qquad x,\,z\in A, \\
& \int_{\Omega\setminus A}\Ds G_A(x,y)\,G_\Omega(y,z)\;dy=\Ds G_\Omega(x,z) \qquad x\in A,\,z\in \C A,
\end{align*}
we deduce $u_A(x)=u(x),\ x\in A$.
\end{proof}

\begin{lem}[maximum principle]\label{max-princweak2} Let $u\in L^1(\Omega)$ be a solution of
$$
\left\lbrace\begin{aligned}
\Ds u &\leq 0 & & \hbox{ in }\Omega, \\
u &\leq 0 & & \hbox{ in }\C\Omega, \\
Eu &\leq 0 & & \hbox{ on }\partial\Omega.
\end{aligned}\right.
$$
Then $u\leq 0$ a.e. in $\R^N$.
\end{lem}
\begin{proof}\rm Take $\psi\in C^\infty_c(\Omega)$, $\psi\geq 0$ and
the associated $\phi\in\T(\Omega)$ for which $\Ds\phi|_\Omega=\psi$:
it is $\phi\geq 0$ in $\Omega$, in view of Lemma \ref{maxprinc}.
This implies that for $y\in\C\Omega$ it holds
$$
\Ds\phi(y)=-C_{N,s}\,\int_\Omega\frac{\phi(z)}{|y-z|^N}\;dz\leq0,
$$
and also $D_s\phi\geq 0$ throughout $\partial\Omega$.
In particular 
$$
\int_\Omega u\psi\leq 0
$$
as a consequence of \eqref{weakdefi}.
\end{proof}
\medskip

\subsection{Regularity theory}

\begin{theo}\label{reg-l1sol} Given three Radon measures $\lambda\in\mathcal{M}(\Omega)$, 
$\mu\in\mathcal{M}(\C\Omega)$ and $\nu\in\mathcal{M}(\partial\Omega)$,
consider the solution $u\in L^1(\Omega)$ of problem
$$
\left\lbrace\begin{aligned}
\Ds u &=\lambda & & \hbox{ in }\Omega, \\
u &=\mu & & \hbox{ in }\C\Omega, \\
Eu &=\nu & & \hbox{ on }\partial\Omega.
\end{aligned}\right.
$$
Then 
$$
\Arrowvert u\Arrowvert_{L^1(\Omega)}\leq C\left(
\Arrowvert\lambda\Arrowvert_{\mathcal{M}(\Omega,{\delta(x)}^s\,dx)}+
\Arrowvert\mu\Arrowvert_{\mathcal{M}(\C\Omega,{\delta(x)}^{-s}\wedge{\delta(x)}^{-n-2s}\,dx)}+
\Arrowvert\nu\Arrowvert_{\mathcal{M}(\partial\Omega)}\right).
$$
\end{theo}
\begin{proof}\rm 
Consider $\zeta$ to be the solution of 
$$
\left\lbrace\begin{aligned}
\Ds\zeta &=1 & & \hbox{ in }\Omega \\
\zeta &=0 & & \hbox{ in }\C\Omega \\
E\zeta &=0 & & \hbox{ on }\partial\Omega
\end{aligned}\right.
$$
which we know to satisfy $0\leq\zeta(x)\leq C\,\delta(x)^s$ in $\Omega$,
see \cite{rosserra}.
Note also that, by approximating $\zeta$ with functions in $\T(\Omega)$ and
by \eqref{repr-lapl}, for $x\in\C\Omega$
$$
\Ds\zeta(x)=\int_\Omega\Ds G_\Omega(z,x)\;dz,\qquad x\in\C\Omega,
$$
and therefore, when $x\in\C\Omega$ and $\delta(x)<1$,
$$
0\leq-\Ds\zeta(x)\leq \int_\Omega\frac{C\,\delta(y)^s}{|y-x|^{N+2s}}\;dy\leq
C\int_{\delta(x)}^{+\infty}\frac{dt}{t^{1+s}}=\frac{C}{\delta(x)^s},
$$
while for $x\in\C\Omega$ and $\delta(x)\geq1$
$$
0\leq-\Ds\zeta(x)\leq \int_\Omega\frac{C\,\delta(y)^s}{|y-x|^{N+2s}}\;dy\leq
\frac{C}{\delta(x)^{N+2s}}\int_\Omega\delta(y)^s\;dy.
$$
Furthermore, $D_s\zeta$ is a well-defined function on $\partial\Omega$,
again thanks to the results in \cite[Proposition 1.1]{rosserra}. Indeed, consider
an increasing sequence ${\{\psi_k\}}_{k\in\N}\subseteq C^\infty_c(\Omega)$,
such that $0\leq\psi_k\leq 1$, $\psi_k\uparrow 1$ in $\Omega$.
Call $\phi_k$ the function in $\T(\Omega)$ associated with $\psi_k$,
i.e. $\Ds\phi_k|_\Omega=\psi_k$.
In this setting
\[
0\leq\lim_{\stackrel{\hbox{\scriptsize $z\!\in\!\Omega$}}{z\rightarrow\theta}}
\frac{\zeta(z)-\phi_k(z)}{{\delta(z)}^s}=
\lim_{\stackrel{\hbox{\scriptsize $z\!\in\!\Omega$}}{y\rightarrow\theta}}
\int_\Omega\frac{G_\Omega(y,z)}{{\delta(z)}^s}\left(1-\psi_k(y)\right)\;dy
\leq\int_\Omega M_\Omega(y,\theta)\left(1-\psi_k(y)\right)\;dy
\xrightarrow{k\rightarrow+\infty}0,
\]
hence $D_s\zeta(\theta)$ is well-defined and it holds
$$
D_s\zeta(\theta)=\int_\Omega M_\Omega(x,\theta)\;dx.
$$
Finally, we underline how $D_s\zeta\in L^\infty(\Omega)$,
thanks to \eqref{chen-martin}.

We split the rest of the proof by using the integral representation of $u$:
\begin{description}
\item[\it i] the mass induced by the right-hand side is
\[
\int_\Omega\left|\int_\Omega G_\Omega(y,x)\;d\lambda(y)\right|\;dx
\leq\int_\Omega\int_\Omega G_\Omega(y,x)\;dx\;d|\lambda|(y)
\leq\int_\Omega\zeta(y)\;d|\lambda|(y)
\leq C\int_\Omega\delta(y)^s\;d|\lambda|(y),
\]
\item[\it ii] the one induced by the external datum
\begin{multline*}
\int_\Omega\left|\int_{\C\Omega} \Ds G_\Omega(x,y)\;d\mu(y)\right|\;dx\leq
\int_{\C\Omega}-\int_\Omega \Ds G_\Omega(x,y)\;dx\;d|\mu|(y)=\\
=\int_{\C\Omega}-\Ds\zeta(y)\;d|\mu|(y)\leq
\int_{\C\Omega}\left(\frac{C}{\delta(y)^s}\wedge\frac{C}{\delta(y)^{N+2s}}\right)\;d|\mu|(y),
\end{multline*}
\item[\it iii] finally the mass due to the boundary behavior
\begin{multline*}
\int_\Omega\left|\int_{\partial\Omega} M_\Omega(x,\theta)\;d\nu(\theta)\right|\;dx\leq
\int_{\partial\Omega}\int_\Omega M_\Omega(x,\theta)\;dx\;d|\nu|(\theta)\ =\\
=\ \int_{\partial\Omega}D_s\zeta(\theta)\;d|\nu|(\theta)
\leq \Arrowvert D_s\zeta\Arrowvert_\infty\,|\nu|(\partial\Omega).
\end{multline*}
\end{description}
Note that the smoothness of the domain is needed 
only to make the last point go through,
and we can repeat the proof in case $\nu=0$
without requiring $\partial\Omega\in C^{1,1}$.
\end{proof}

To gain higher integrability on a solution,
the first step we take is the following

\begin{lem}\label{reg-t} For any $q>N/s$ there exists 
$C=C(N,s,\Omega,q)$ such that for all $\phi\in\T(\Omega)$
$$
\left\Arrowvert\frac{\phi}{\delta^s}\right\Arrowvert_{L^\infty(\Omega)}
\leq C\, \Arrowvert\Ds\phi\Arrowvert_{L^q(\Omega)}.
$$
\end{lem}
\begin{proof}\rm Call as usual $\psi=\Ds\phi|_\Omega\in C^\infty_c(\Omega)$.
Let us work formally
\[
\left\Arrowvert\frac{\phi}{\delta^s}\right\Arrowvert_{L^\infty(\Omega)}
\leq \sup_{x\in\Omega}\,\frac{1}{\delta(x)^s}\int_\Omega G_\Omega(x,y)\,|\psi(y)|\;dy
\leq\sup_{x\in\Omega}\,\frac{\left\Arrowvert G_\Omega(x,\cdot)\right\Arrowvert_{L^p(\Omega)}}{\delta(x)^s}
\cdot \Arrowvert\Ds\phi\Arrowvert_{L^q(\Omega)},
\]
where $\frac{1}{p}+\frac{1}{q}=1$. Thus we need only to understand 
for what values of $p$ we are not writing a trivial inequality.
The main tool here is inequality
\begin{equation}\label{alter-Gbound}
G_\Omega(x,y)\leq \frac{c\,\delta(x)^s}{|x-y|^N}\left(|x-y|^s\wedge\delta(y)^s\right),
\end{equation}
which holds in $C^{1,1}$ domains, see \cite[equation 2.14]{chen}.
We have then
$$
\int_\Omega\left|\frac{G_\Omega(x,y)}{\delta(x)^s}\right|^p\;dy\leq
c\int_\Omega \frac{1}{|x-y|^{Np}}\left(|x-y|^{sp}\wedge\delta(y)^{sp}\right)\;dy
\leq c\int_\Omega \frac{1}{|x-y|^{(N-s)p}}
$$
which is uniformly bounded in $x$ for $p<N/(N-s)$.
This condition on $p$ becomes $q>N/s$.
\end{proof}

In view of this last lemma,
we are able to provide the following theorem,
which is the fractional counterpart of a classical result
(see e.g. \cite[Proposition A.9.1]{dupaigne-book}).

\begin{theo}\label{reg} For any $p<N/(N-s)$, there exists a constant $C$ such that
any solution $u\in L^1(\Omega)$ of problem
$$
\left\lbrace\begin{aligned}
\Ds u &=\lambda & & \hbox{ in }\Omega \\
u &=0 & & \hbox{ in }\C\Omega \\
Eu &=0 & & \hbox{ on }\partial\Omega
\end{aligned}\right.
$$
has a finite $L^p$-norm controlled by
$$
\Arrowvert u\Arrowvert_{L^p(\Omega)}\leq C\,
\Arrowvert\lambda\Arrowvert_{L^1(\Omega,{\delta(x)}^s dx)}
$$
\end{theo}
\begin{proof}\rm For any $\psi\in C^\infty_c(\Omega)$, let $\phi\in\T(\Omega)$ be
chosen in such a way that $\Ds\phi|_\Omega=\psi$. We have
\[
\left|\int_\Omega u(x)\,\psi(x)\;dx\right|=
\left|\int_\Omega \phi(x)\;d\lambda(x)\right|
\leq\int_\Omega\frac{|\phi(x)|}{\delta(x)^s}\cdot\delta(x)^s\;d|\lambda|(x)\leq
C\Arrowvert\psi\Arrowvert_{L^q(\Omega)}\int_\Omega\delta(x)^s\;d|\lambda|(x).
\]
where $q$ is the conjugate exponent of $p$, according to Lemma \ref{reg-t}.
By density of $C^\infty_c(\Omega)$ in $L^q(\Omega)$ and 
the isometry between $L^p(\Omega)$ and the dual space of $L^q(\Omega)$,
we obtain our thesis.
\end{proof}
\medskip

\section{asymptotic behaviour at the boundary}\label{bblin-sec}

\subsection{Right-hand side blowing up at the boundary: proof of \texorpdfstring{\eqref{udelta}}{\ref{udelta}}}
\label{rhs-blowing}

In this paragraph we study the boundary behaviour
of the solution $u$ to the problem
$$
\left\lbrace\begin{aligned}
\Ds u(x) &={\delta(x)}^{-\beta} & & \hbox{ in }\Omega,\ 0<\beta<1+s \\
u &=0 & & \hbox{ in }\C\Omega \\
Eu &=0 & & \hbox{ on }\partial\Omega
\end{aligned}\right.
$$
starting from the radial case $\Omega=B$.
Then
$$
u(x)=\int_B\frac{G_B(x,y)}{{\delta(y)}^\beta}\;dy
$$
and by using \eqref{est-green}, up to multiplicative constants,
$$
u(x)\leq\int_B \frac{\left[{|x-y|}^2\wedge\delta(x)\delta(y)\right]^s}{{|x-y|}^N}
\cdot\frac{dy}{{\delta(y)}^\beta}.
$$
Set
$$
\ep:=\delta(x)\hbox{ and }x=(1-\ep)e_1,\qquad
\theta:=\frac{y}{|y|},\,r:=\delta(y)\hbox{ and }y=(1-r)\theta,
$$
and rewrite
$$
u(x)\leq\int_{S^{N-1}}\int_0^1 \frac{\left[{|(1-\ep)e_1-(1-r)\theta|}^2\wedge\ep r\right]^s}{{|(1-\ep)e_1-(1-r)\theta|}^N}
\cdot\frac{{(1-r)}^{N-1}}{{r}^\beta}\;dr\;d\mathcal{H}^{N-1}(\theta).
$$
Split the angular variable in $\theta=(\theta_1,\theta')$ where
$\theta_1=\langle\theta,e_1\rangle$ and $\theta_1^2+|\theta'|^2=1$: then, for a general $F$,
$$
\int_{S^{N-1}} F(\theta)d\mathcal{H}^{N-1}(\theta)=
\int_{S^{N-2}}\left(\int_{-1}^1\left(1-\theta_1^2\right)^{(N-3)/2} F(\theta_1,\theta')\;d\theta_1\right)\;d\mathcal{H}^{N-2}(\theta'),
$$
and we write
\begin{align*}
& u(x)\leq\int_{S^{N-2}}\int_0^1\left(1-\theta_1^2\right)^{(N-3)/2}\int_0^1 \\
& \qquad \frac{\left\{[(1-\ep)^2+(1-r)^2-2(1-\ep)(1-r)\theta_1]\wedge\ep r\right\}^s(1-r)^{N-1}}
{{[(1-\ep)^2+(1-r)^2-2(1-\ep)(1-r)\theta_1]}^{N/2}\;{r}^\beta}
\,dr\;d\theta_1\,d\mathcal{H}^{N-2}(\theta') \\
& =\mathcal{H}^{N-2}(S^{N-2})\int_0^1\left(1-\theta_1^2\right)^{(N-3)/2}\int_0^1 \\
&\qquad \frac{\left\{[(1-\ep)^2+(1-r)^2-2(1-\ep)(1-r)\theta_1]\wedge\ep r\right\}^s{(1-r)}^{N-1}}
{{[(1-\ep)^2+(1-r)^2-2(1-\ep)(1-r)\theta_1]}^{N/2}\;{r}^\beta}\,dr\,d\theta_1.
\end{align*}
From now on we will drop all multiplicative constants
and all inequalities will have to be interpreted to hold up to constants.
Let us apply a first change of variables
$$
t=\frac{1-r}{1-\ep}\quad\longleftrightarrow\quad r=1-(1-\ep)t
$$
to obtain
\begin{align}
& u(x)\leq\int_0^1\left(1-\theta_1^2\right)^{(N-3)/2}\int_0^{1/(1-\ep)} \\
& \qquad \frac{\left\{[(1-\ep)^2+(1-\ep)^2t^2-2(1-\ep)^2t\theta_1]\wedge\ep(1-(1-\ep)t)\right\}^s}
{{[(1-\ep)^2+(1-\ep)^2t^2-2(1-\ep)^2t\theta_1]}^{N/2}}
\frac{(1-\ep)^Nt^{N-1}}{{(1-(1-\ep)t)}^\beta}\;dt\;d\theta_1 \nonumber \\
& =\int_0^1\left(1-\theta_1^2\right)^{(N-3)/2}\int_0^{1/(1-\ep)}
\frac{\left[(1-\ep)^2(1+t^2-2t\theta_1)\wedge\ep(1-(1-\ep)t)\right]^s}
{{[1+t^2-2t\theta_1]}^{N/2}}\cdot
\frac{t^{N-1}}{{(1-(1-\ep)t)}^\beta}\;dt\;d\theta_1 \nonumber \\
& \leq\int_0^{1/(1-\ep)}\int_0^1
\frac{\left\{[(1-t)^2+2t\sigma]\wedge\ep[1-(1-\ep)t]\right\}^s}
{{[(1-t)^2+2t\sigma]}^{N/2}}\,{\sigma}^{(N-3)/2}
\;d\sigma\;\frac{t^{N-1}}{{(1-(1-\ep)t)}^\beta}\;dt 	\label{67},
\end{align}
where $\sigma=1-\theta_1$.
We compute now the integral in the variable $\sigma$.
Let $\sigma'$ be defined by equality
$$
(1-t)^2+2t\sigma'=\ep[1-(1-\ep)t]
\qquad\Longleftrightarrow\qquad
\sigma'=\frac{\ep[1-(1-\ep)t]-(1-t)^2}{2t}
$$
and let
$$
\sigma^*=\max\{\sigma',0\}.
$$
The quantity $\sigma^*$ equals $0$ if and only if
$$
\ep[1-(1-\ep)t]\leq(1-t)^2
$$
and it is easy to verify that this happens whenever
$$
t\leq t_1(\ep):=1-\frac{\ep-\ep^2+\ep\sqrt{\ep^2-2\ep+5}}{2},\  
t\geq t_2(\ep):=1+\frac{-\ep+\ep^2+\ep\sqrt{\ep^2-2\ep+5}}{2}.
$$
\begin{rmk}\rm For small $\ep>0$ we have $t_2(\ep)<\frac{1}{1-\ep}$, $0<t_1(\ep)<1$
\end{rmk}

Hence
$$
\sigma^*=\max\{\sigma',0\}=\left\lbrace
\begin{array}{ll}
\displaystyle
\frac{\ep[1-(1-\ep)t]-(1-t)^2}{2t} &
\hbox{if }
t\in\left(t_1(\ep),t_2(\ep)\right)\subseteq\left(0,\frac{1}{1-\ep}\right) \\
0 & \hbox{if }
t\in\left(0,t_1(\ep)\right)\cup\left(t_2(\ep),\frac{1}{1-\ep}\right).
\end{array}\right.
$$
We split now integral \eqref{67} into four pieces, as following
\begin{multline}\label{splitted}
\int_0^{\frac{1}{1-\ep}}dt\int_0^1d\sigma\ =\\
=\ \int_{t_1(\ep)}^{t_2(\ep)}dt\int_0^{\sigma^*}d\sigma\ 
+\ \int_{t_1(\ep)}^{t_2(\ep)}dt\int_{\sigma^*}^1d\sigma\ 
+\ \int_0^{t_1(\ep)}dt\int_0^1d\sigma\ 
+\ \int_{t_2(\ep)}^\frac{1}{1-\ep}dt\int_0^1d\sigma,
\end{multline}
and we treat each of them separately:
\begin{itemize}
\item for the first one we have
\begin{multline*}
\int_0^{\sigma^*}\frac{{\sigma}^{(N-3)/2}}{{[(1-t)^2+2t\sigma]}^{N/2-s}}\;d\sigma
\leq\frac{1}{{(2t)}^{(N-3)/2}}\int_0^{\sigma^*}
{[(1-t)^2+2t\sigma]}^{s-3/2}\;d\sigma\leq\\
\leq \frac{1}{{t}^{(N-1)/2}}\frac{1}{s-\frac{1}{2}}
\left[\ep^{s-1/2}\left(1-(1-\ep)t\right)^{s-1/2}-|1-t|^{2s-1}\right]
\end{multline*}
and therefore the first integral is less than
\begin{equation}\label{1}
\int_{t_1(\ep)}^{t_2(\ep)}
\frac{t^{(N-1)/2}}{\left(1-(1-\ep)t\right)^\beta}\cdot
\frac{1}{s-\frac{1}{2}}
\left[\ep^{s-1/2}\left(1-(1-\ep)t\right)^{s-1/2}-|1-t|^{2s-1}\right]\;dt;
\end{equation}
note now that $t_2(\ep)-t_1(\ep)\sim\ep$, up to a multiplicative constant, and $\frac{1}{1-\ep}-1\sim\ep$
and thus \eqref{1} is of magnitude
$$
\ep^{-\beta+2s};
$$
\item for the second integral we have
\begin{multline*}
\int_{\sigma^*}^1\frac{{\sigma}^{(N-3)/2}}{{[(1-t)^2+2t\sigma]}^{N/2}}\;d\sigma
\leq\frac{1}{{(2t)}^{(N-3)/2}}\int_0^{\sigma^*}
{[(1-t)^2+2t\sigma]}^{-3/2}\;d\sigma\leq\\
\leq \frac{1}{{t}^{(N-1)/2}}
\left[\ep^{-1/2}\left(1-(1-\ep)t\right)^{-1/2}-(1+t^2)^{-1/2}\right]
\end{multline*}
and therefore the second integral is less than
\begin{equation}\label{2}
\int_{t_1(\ep)}^{t_2(\ep)}
\frac{t^{(N-1)/2}}{\left(1-(1-\ep)t\right)^\beta}\cdot
\left[\ep^{-1/2}\left(1-(1-\ep)t\right)^{-1/2}-(1+t^2)^{-1/2}\right]\;dt;
\end{equation}
and \eqref{2} is of magnitude
$$
\ep^{-\beta+2s};
$$
\item for the third integral we have
$$
\int_0^1\frac{{\sigma}^{(N-3)/2}}{{[(1-t)^2+2t\sigma]}^{N/2}}\;d\sigma
\leq \frac{1}{t^{(N-3)/2}}\int_0^1\frac{d\sigma}{\left[(1-t)^2+2t\sigma\right]^{3/2}}
\leq \frac{1}{t^{(N-1)/2}}\cdot\frac{1}{|1-t|}
$$
so that the third integral is less than
\begin{equation}\label{3}
\begin{split}
\int_0^{t_1(\ep)}
\frac{t^{(N-1)/2}\,\ep^s}{\left(1-(1-\ep)t\right)^{\beta-s}}
\cdot\frac{dt}{1-t}\leq
\ep^s\int_0^{t_1(\ep)}
\frac{dt}{(1-t)\,\left(1-(1-\ep)t\right)^{\beta-s}}\leq\\
\leq\ep^s\int_0^{t_1(\ep)}
\frac{dt}{\left(1-t\right)^{\beta-s+1}}=
\frac{\ep^s\left(1-t_1(\ep)\right)^{\beta-s}}{\beta-s}
-\frac{\ep^s}{\beta-s}
\end{split}
\end{equation}
if $\beta\neq s$
and \eqref{3} is of magnitude
\begin{eqnarray*}
& \ep^s & \hbox{ for }0\leq\beta< s,\\
& \ep^s\log\frac{1}{\ep} & \hbox{ for }\beta=s,\\
& \ep^{-\beta+2s} & \hbox{ for }s<\beta<1+s;
\end{eqnarray*}
\item for the fourth integral we have
$$
\int_0^1\frac{{\sigma}^{(N-3)/2}}{{[(1-t)^2+2t\sigma]}^{N/2}}\;d\sigma
\leq \frac{1}{t^{(N-1)/2}}\cdot\frac{1}{|1-t|}
$$
so that the fourth integral is less than
\begin{equation}\label{4}
\int_{t_2(\ep)}^{\frac{1}{1-\ep}}
\frac{t^{(N-1)/2}\,\ep^s}{\left(1-(1-\ep)t\right)^{\beta-s}}
\cdot\frac{dt}{t-1}\leq
\ep^s\int_{t_2(\ep)}^{\frac{1}{1-\ep}}
\frac{dt}{(t-1)\,\left(1-(1-\ep)t\right)^{\beta-s}}
\end{equation}
and \eqref{4} is of magnitude
\begin{eqnarray*}
\ep^s\int_{t_2(\ep)}^{\frac{1}{1-\ep}}\frac{dt}{\left(t-1\right)^{\beta-s+1}} & \sim\ep^s &
\hbox{ for }0\leq\beta< s, \\
\ep^s\int_{t_2(\ep)}^{\frac{1}{1-\ep}}\frac{dt}{t-1}
& \sim\ep^s\log\frac{1}{\ep} & \hbox{ for }\beta=s, \\
\frac{\ep^s}{t_2(\ep)-1}\int_{t_2(\ep)}^{\frac{1}{1-\ep}}
\frac{dt}{\left(1-(1-\ep)t\right)^{\beta-s}}
& \sim\ep^{-\beta+2s} &
\hbox{ for }s<\beta<1+s.
\end{eqnarray*}
\end{itemize}
Resuming the information collected so far,
what we have gained is that, up to constants,
\begin{equation}\label{rhs-estimate}
u(x)\leq \left\lbrace
\begin{array}{ll}
\displaystyle
{\delta(x)}^s &
\hbox{ for }0\leq\beta< s, \\ 
\displaystyle
{\delta(x)}^s\log\frac{1}{\delta(x)} & \hbox{ for }\beta=s,\\
\displaystyle
{\delta(x)}^{-\beta+2s} &
\hbox{ for }s<\beta<1+s.
\end{array}\right.
\end{equation}

This establishes an upper bound for the solutions.
Note now that the integral \eqref{67} works also as
a lower bound, of course up to constants.
Using the split expression \eqref{splitted} we entail
\begin{align*}
& u(x)\geq \\
& \int_0^{\frac{1}{1-\ep}}
\int_0^1\frac{\left\{[(1-t)^2+1+2t\sigma]\wedge\ep[1-(1-\ep)t]\right\}^s}
{{[(1-t)^2+2t\sigma]}^{N/2}}\,{\sigma}^{(N-3)/2}
\;d\sigma\;\frac{t^{N-1}}{{(1-(1-\ep)t)}^\beta}
\;dt \\ 
& \geq\int_0^{t_1(\ep)}\frac{t^{N-1}\,\ep^s}{\left(1-(1-\ep)t\right)^{\beta-s}}
\int_0^1\frac{{\sigma}^{(N-3)/2}}{\left[(1-t)^2+2t\sigma\right]^{N/2}}\;d\sigma\;dt,
\end{align*}
where we have used only the expression with the third 
integral in \eqref{splitted}.
We claim that
\begin{equation}\label{45}
\int_0^1\frac{{\sigma}^{(N-3)/2}}{\left[(1-t)^2+2t\sigma\right]^{N/2}}
\;d\sigma\geq\frac{1}{1-t}\geq\frac{1}{1-(1-\ep)t},
\end{equation}
where the inequality is intended to hold up to constants.
In case \eqref{45} holds and $\beta\neq s$ we have
\begin{multline*}
\int_0^{t_1(\ep)}\frac{t^{N-1}\,\ep^s}{\left(1-(1-\ep)t\right)^{\beta-s}}
\int_0^1\frac{{\sigma}^{(N-3)/2}}{\left[(1-t)^2+2t\sigma\right]^{N/2}}\;d\sigma\;dt
\geq\int_0^{t_1(\ep)}\frac{t^{N-1}\,\ep^s}{\left(1-(1-\ep)t\right)^{\beta-s+1}}\;dt=\\
=\left.\frac{1}{(1-\ep)\,(\beta-s)}\cdot
\frac{t^{N-1}\,\ep^s}{\left(1-(1-\ep)t\right)^{\beta-s}}\right\arrowvert_0^{t_1(\ep)}
-\frac{N-1}{(1-\ep)(\beta-s)}\int_0^{t_1(\ep)}\frac{t^{N-2}\,\ep^s}{\left(1-(1-\ep)t\right)^{\beta-s}}\;dt.
\end{multline*}
The integral on the second line is a bounded quantity as $\ep\downarrow 0$
since $\beta<1+s$. Now, we are left with
$$
u(x)\geq\frac{\ep^{-\beta+2s}}{\beta-s}-\frac{\ep^s}{\beta-s}
$$
so that
$$
u(x)\geq\left\lbrace\begin{array}{ll}
\displaystyle\ep^s & \hbox{ when }0\leq\beta< s \\
\displaystyle\ep^s\log\frac{1}{\ep} & \hbox{ when }\beta =s \\
\displaystyle\ep^{-\beta+2s} & \hbox{ when }s<\beta<1+s.
\end{array}\right.
$$
We still have to prove \eqref{45}: note that an integration by parts yields, for $N\geq4$,
\begin{multline*}
\int_0^1\frac{{\sigma}^{(N-3)/2}}{\left[(1-t)^2+2t\sigma\right]^{N/2}}
\:d\sigma=
\frac{1}{2t\,(1-N/2)}\cdot\frac{1}{\left[(1-t)^2+2t\right]^{N/2-1}}\ +\\
+\ \frac{1}{2t\,(N/2-1)}\int_0^1\frac{{\sigma}^{(N-5)/2}}{\left[(1-t)^2+2t\sigma\right]^{N/2-1}}\:d\sigma
\end{multline*}
so that we can show \eqref{45} only in dimensions $N=2,\,3$
and deduce the same conclusions for any other value of $N$
by integrating by parts a suitable number of times.
For $N=2$
\[
\int_0^1\frac{d\sigma}{\sqrt{\sigma}\left[(1-t)^2+2t\sigma\right]}
=\frac{\pi}{\sqrt{2t}(1-t)}.
\]
For $N=3$
$$
\int_0^1\frac{d\sigma}{\left[(1-t)^2+2t\sigma\right]^{3/2}}=
\frac{1}{4t}\left(\frac{1}{1-t}-\sqrt{\frac{1}{1+t^2}}\right)
$$
which completely proves our claim \eqref{45}.

So far we have worked only on spherical domains.
In a general domain $\Omega$ with $C^{1,1}$ boundary,
split the problem in two by setting $u=u_1+u_2$,
for $0<\beta<1+s$ and a small $\delta_0>0$
$$
\left\lbrace\begin{aligned}
\Ds u_1 &=\frac{\chi_{\{\delta\geq\delta_0\}}}{{\delta}^\beta} & & \hbox{in }\Omega \\
u_1 &=0 & & \hbox{in }\C\Omega \\
Eu_1 &=0 & & \hbox{on }\partial\Omega,
\end{aligned}\right.
\qquad
\left\lbrace\begin{aligned}
\Ds u_2 &=\frac{\chi_{\{\delta<\delta_0\}}}{{\delta}^\beta} & & \hbox{in }\Omega \\
u_2 &=0 & & \hbox{in }\C\Omega \\
Eu_2 &=0 & & \hbox{on }\partial\Omega.
\end{aligned}\right.
$$
Note that $u_1\in C^s(\R^N)$, see \cite[Proposition 1.1]{rosserra}.
Since $\partial\Omega\in C^{1,1}$, we choose now $\delta_0$ sufficiently small in order to have
that for any $y\in\Omega$ with $\delta(y)<\delta_0$ it is uniquely determined $\theta=\theta(y)\in\partial\Omega$
such that $|y-\theta|=\delta(y)$. Then in inequalities (see \eqref{est-green})
\begin{multline*}
\int_{\{\delta(y)<\delta_0\}}\frac{\left[{|x-y|}^2\,\wedge\,\delta(x)\delta(y)\right]^s}{c_2\,|x-y|^N\,{\delta(y)}^\beta}\;dy
\leq u_2(x)=\int_{\{\delta(y)<\delta_0\}}\frac{G_\Omega(x,y)}{{\delta(y)}^\beta}\;dy\leq \\
\leq\int_{\{\delta(y)<\delta_0\}}\frac{c_2\,\left[{|x-y|}^2\,\wedge\,\delta(x)\delta(y)\right]^s}{|x-y|^N\,{\delta(y)}^\beta}\;dy
\end{multline*}
set
$$
\begin{array}{l}
\nu(\theta),\ \theta\in\partial\Omega,\ \hbox{the outward unit normal to }\partial\Omega\hbox{ at }\theta, \\
\ep:=\delta(x),\ \theta^*=\theta(x),\ \hbox{ and }x=\theta^*-\ep\,\nu(\theta^*), \\
\theta=\theta(y),\ r:=\delta(y),\ \hbox{ and }y=\theta-r\,\nu(\theta).
\end{array}
$$
Then, using the Fubini's Theorem, we write
\begin{equation}\label{09}
\int_{\{\delta(y)<\delta_0\}}\frac{\left[{|x-y|}^2\,\wedge\,\delta(x)\delta(y)\right]^s}{|x-y|^N\,{\delta(y)}^\beta}\;dy=
\int_0^{\delta_0}\int_{\partial\Omega}\frac{\left[{|\theta^*-\ep\,\nu(\theta^*)-\theta+r\,\nu(\theta)|}^2\,\wedge\,\ep r\right]^s}
{{|\theta^*-\ep\,\nu(\theta^*)-\theta+r\,\nu(\theta)|}^N\,{r}^\beta}\;d\mathcal{H}(\theta)\;dr.
\end{equation}
Split the integration on $\partial\Omega$ into the integration on 
$\Gamma:=\{\theta\in\partial\Omega:|\theta-\theta^*|<\delta_1\}$ and $\partial\Omega\setminus\Gamma$,
and choose $\delta_1>0$ small enough to have a $C^{1,1}$ diffeomorphism 
$$
\begin{array}{rcl}
\varphi:\widetilde\Gamma\subseteq S^{N-1} & \longrightarrow & \Gamma \\
\omega & \longmapsto & \theta=\varphi(\omega).
\end{array}
$$
We build now
$$
\begin{array}{rcl}
\overline\varphi:\widetilde\Gamma\times(0,\delta_0)\subseteq B_1 & \longrightarrow & \{y\in\Omega:\delta(y)<\delta_0,\ |\theta(y)-\theta^*|<\delta_1\} \\
(1-\delta)\omega=\omega-\delta\omega & \longmapsto & y=\theta-\delta\,\nu(\theta)=\varphi(\omega)-\delta\,\nu(\varphi(\omega)),
\end{array}
$$
where we suppose $e_1\in\widetilde\Gamma$ and $\phi(e_1)=\theta^*$.
With this change of variables \eqref{09} becomes
\[
\int_{\{\delta(y)<\delta_0\}}\frac{\left[{|x-y|}^2\,\wedge\,\delta(x)\delta(y)\right]^s}{|x-y|^N\,{\delta(y)}^\beta}\;dy
=\int_{0}^{\delta_0}\int_{\widetilde\Gamma}\frac{\left[{|e_1-\ep\,e_1-\omega+r\,\omega|}^2\,\wedge\,\ep r\right]^s}
{{|e_1-\ep\,e_1-\omega+r\,\omega|}^N\,{r}^\beta}\;|D\overline\varphi(\omega)|\;d\mathcal{H}(\omega)\;dr
\]
which, since $|D\overline\varphi(\omega)|$ is a bounded continuous quantity far from $0$,
is bounded below and above in terms of
$$
\int_{0}^{\delta_0}\int_{\widetilde\Gamma}\frac{\left[{|(1-\ep)e_1-(1-r)\omega|}^2\,\wedge\,\ep r\right]^s}
{{|(1-\ep)e_1-(1-r)\omega|}^N\,{r}^\beta}\;d\mathcal{H}(\omega)\;dr,
$$ 
i.e. we are brought back to the spherical case.

\subsection{Boundary continuity of \texorpdfstring{$s$}{s}-harmonic functions}\label{cont-sharm-sec}

Consider $g:\C\Omega\rightarrow\R$ and
$$
u(x)=-\int_{\C\Omega}g(y)\Ds G_\Omega(x,y)\;dy
$$
and think of letting $x\rightarrow\theta\in\partial\Omega$.
Suppose that for any small $\ep>0$ there exists $\delta>0$
such that $|g(y)-g(\theta)|<\ep$ for any $y\in\C\Omega\cap B_\delta(\theta)$.
Then
\begin{align*}
& |u(x)-g(\theta)|=\left|\int_{\C\Omega}(g(\theta)-g(y))\Ds G_\Omega(x,y)\;dy\right|\leq\\
& \leq \int_{\C\Omega\cap B_\delta(\theta)}\left|(g(\theta)-g(y))\Ds G_\Omega(x,y)\right|dy
+\int_{\C\Omega\setminus B_\delta(\theta)}\left|(g(\theta)-g(y))\Ds G_\Omega(x,y)\right|dy.
\end{align*}
The first addend satisfies
$$
\int_{\C\Omega\cap B_\delta(\theta)}\left|(g(\theta)-g(y))\Ds G_\Omega(x,y)\right|\;dy
\leq -\ep \int_{\C\Omega\cap B_\delta(\theta)}\Ds G_\Omega(x,y)\;dy\leq\ep.
$$
For the second one we exploit \eqref{est-poisson}:
\[
\int_{\C\Omega\setminus B_\delta(\theta)}\left|(g(\theta)-g(y))\Ds G_\Omega(x,y)\right|dy
\leq\ {\delta(x)}^s\int_{\C\Omega\setminus B_\delta(\theta)}\frac{c_1|(g(\theta)-g(y))|}{{\delta(y)}^s\left(1+\delta(y)\right)^s{|x-y|}^N}
\]
which converges to $0$ as $x\rightarrow\theta\in\partial\Omega$.
So we have that
$$
\lim_{x\rightarrow\theta}|u(x)-g(\theta)|\leq\ep,
$$
and by arbitrarily choosing $\ep$, we conclude 
$$
\lim_{x\rightarrow\theta}u(x)=g(\theta).
$$

\subsection{Explosion rate of large \texorpdfstring{$s$}{s}-harmonic functions: proof of \texorpdfstring{\eqref{uleqg}}{\ref{uleqg}}}
\label{expl-sharm-sec}

We study here the rate of divergence of 
$$
u(x)=-\int_{\C\Omega}g(y)\:\Ds G_\Omega(x,y)\;dy
\quad\hbox{ as }x\rightarrow\partial\Omega,\ x\in\Omega
$$
which is the solution to 
$$
\left\lbrace\begin{aligned}
\Ds u & =0 & & \hbox{ in }\Omega \\
u &=g & & \hbox{ in }\C\Omega \\
Eu &=0 & & \hbox{ on }\partial\Omega
\end{aligned}\right.
$$
in case $g$ explodes at $\partial\Omega$.

\begin{rmk}\rm The asymptotic behaviour of $u$ depends only on
the values of $g$ near the boundary, since we can split
$$
g=g\chi_{\{d<\eta\}}+g\chi_{\{d\geq\eta\}}
$$
and the second addend has a null contribution on the boundary,
in view of Paragraph \ref{cont-sharm-sec}.
Therefore in our computations we will suppose
that $g(y)=0$ for $\delta(y)>\eta$.
\end{rmk}

In the further assumption that $g$ explodes like a power, i.e.
there exist $\eta, k, K>0$ for which
$$
\frac{k}{\delta(y)^\tau}\leq g(y)\leq\frac{K}{\delta(y)^\sigma},
\qquad \hbox{ for } 0\leq\tau\leq\sigma<1-s,\ 0<\delta(y)<\eta
$$
(the choice $\sigma<1-s$ is in order to have \eqref{g},
see \eqref{est-poisson} above)
our proof doesn't require heavy computations
and it is as follows.

Dropping multiplicative constants in inequalities and for $\Omega_\eta=\{y\in\C\Omega:\delta(y)<\eta\}$:
\begin{multline*}
\delta(x)^\sigma u(x)=
-\delta(x)^\sigma\int_{\C\Omega}g(y)\cdot\Ds G_\Omega(x,y)\;dy\leq \\
\leq\int_{\Omega_\eta}\frac{\delta(x)^{s+\sigma}}{\delta(y)^{s+\sigma}\,\left(1+\delta(y)\right)^s\,|x-y|^N}\;dy 
\leq-\int_{\C\Omega}\chi_{\Omega_\eta}(y)\cdot{(-\lapl)}^{s+\sigma}G_\Omega^{s+\sigma}(x,y)\;dy 
\ \leq\ 1
\end{multline*}
Similarly one can treat also the lower bound:
\begin{align*}
& \delta(x)^\tau u(x)=
-\delta(x)^\tau\int_{\C\Omega}g(y)\cdot\Ds G_\Omega(x,y)\;dy\ \geq \\
& \geq\int_{\Omega_\eta}\frac{\delta(x)^{s+\tau}}{\delta(y)^{s+\tau}\,\left(1+\delta(y)\right)^s\,|x-y|^N}\;dy
\geq-\int_{\C\Omega}\chi_{\Omega_\eta}(y)\cdot{(-\lapl)}^{s+\tau}G_\Omega^{s+\tau}(x,y)\;dy 
\ \xrightarrow[x\rightarrow\partial\Omega]{}1.
\end{align*}
The limit we have computed above is the continuity up to the boundary
of $\widehat{u}$ solution of
$$
\left\lbrace\begin{aligned}
{(-\lapl)}^{s+\tau}\widehat{u} &=0 & & \hbox{ in }\Omega, \\
\widehat{u} &=\chi_{\Omega_\eta} & & \hbox{ in }\C\Omega, \\
E\widehat{u} &=0 & & \hbox{ on }\partial\Omega.
\end{aligned}\right.
$$

\begin{rmk}\rm Both the upper and the lower estimate are optimal,
thanks to what have been shown in Example \ref{expl-sol}.
\end{rmk}

In the case of a general boundary datum $g$
we start from the case $\Omega=B$,
recalling that in this setting, according to \cite[equation (1.6.11')]{landkof},
$$
-\Ds G_B(x,y)=\frac{c(N,s)}{{|x-y|}^N}\left(\frac{1-|x|^2}{|y|^2-1}\right)^s
$$
and therefore
$$
u(x)=\int_{\C B}\frac{c(N,s)}{{|x-y|}^N}\left(\frac{1-|x|^2}{|y|^2-1}\right)^s\:g(y)\;dy.
$$
Suppose without loss of generality
$$
\begin{array}{l}
x=(1-\ep)\,e_1\quad \ep=\delta(x) \\
y=(1+r)\,\theta\quad r=\delta(y),\quad\theta=\frac{y}{|y|},\quad\theta_1=e_1\cdot\theta
\end{array}
$$
so that
\[
u(x)=\int_0^{\eta}\left[
\int_{\partial B}
\frac{c(N,s)}{{|(1-\ep)^2+(1+r)^2-2(1-\ep)(1+r)\theta_1|}^{N/2}}
\left(\frac{\ep(2-\ep)}{r(2+r)}\right)^s\:g(r\theta)
\;d\mathcal{H}(\theta)
\right](1+r)^{N-1}\;dr.
\]
Denote now by $\overline{g}(r)=\sup_{r\leq t\leq\eta}\sup_{\delta(x)=t}g(x)$,
which is decreasing in $0$.
Splitting the integral in the $\theta$ variable into two integrals
in the variables $(\theta_1,\theta')$ where $\theta_1^2+|\theta'|^2=|\theta|^2=1$,
up to constants we obtain
\begin{multline}\label{34}
u(x)\leq \int_0^{\eta}\left[
\int_{-1}^1
\frac{(1-\theta_1^2)^{(N-3)/2}}{{|(1-\ep)^2+(1+r)^2-2(1-\ep)(1+r)\theta_1|}^{N/2}}
\cdot\frac{\ep^s}{r^s}
\,d\theta_1\right]\overline{g}(r)\,(1+r)^{N-1}\,dr\leq \\
\leq\ep^s\int_0^{\eta}\left[
\int_{-1}^1
\frac{(1-\theta_1^2)^{(N-3)/2}}{{|(1-\ep)^2+(1+r)^2-2(1-\ep)(1+r)\theta_1|}^{N/2}}
\;d\theta_1\right]\frac{\overline{g}(r)}{r^s}\;dr. 
\end{multline}
Define $M:=\frac{1+r}{1-\ep}>1$ and look at the inner integral:
\begin{multline*}
\int_{-1}^1
\frac{(1-\theta_1^2)^{(N-3)/2}}{{|1+M^2-2M\theta_1|}^{N/2}}
\;d\theta_1
=
\int_{-1}^1
\frac{(1-\theta_1^2)^{(N-3)/2}}{{|1-\theta_1^2+(M-\theta_1)^2|}^{N/2}}
\;d\theta_1
\leq \\ \leq 
\int_{-1}^1
\left|1-\theta_1^2+(M-\theta_1)^2\right|^{-3/2}
\;d\theta_1.
\end{multline*}
The integral from $-1$ to $0$ contributes by a bounded quantity
so that we are left with
\begin{eqnarray*}
u(x) & \leq & 
\int_0^1
\left|1-\theta_1+(M-\theta_1)^2\right|^{-3/2}
\;d\theta_1
\ =\ \int_0^1
\left|\tau+(M-1+\tau)^2\right|^{-3/2}
\;d\tau \\
& \leq & 
\int_0^1
\left|\tau+(M-1)^2\right|^{-3/2}
\;d\tau
\ =\ 
-2\left.\left(\tau+(M-1)^2\right)^{-1/2}\right|_{\tau=0}^1  
\ \leq\ \frac{1}{M-1}\quad=\quad\frac{1-\ep}{r+\ep}.
\end{eqnarray*}
Thus
$$
u(x)\ \leq\ \ep^s\int_0^{\eta}
\frac{\overline{g}(r)}{r^s}\cdot\frac{1-\ep}{r+\ep}\;dr
\ \leq\ \ep^s\int_0^{\ep}
\frac{\overline{g}(r)}{r^s}\cdot\frac{1}{r+\ep}\;dr+
\int_1^{\eta/\ep}\frac{\overline{g}(\ep\tau)}{\tau^s}\cdot\frac{1}{1+\tau}\;d\tau.
$$
Our claim  now is that this last expression
is controlled by $\overline{g}(\ep)$ as $\ep\downarrow0$.
Since $\overline{g}$ is exploding in $0$, for small $\ep$ it holds
$\overline{g}(\tau\ep)\leq\overline{g}(\ep)$ for $\tau>1$ and
$$
\frac{1}{\overline{g}(\ep)}\int_1^{\eta/\ep}\frac{\overline{g}(\ep\tau)}{\tau^s}\cdot\frac{1}{1+\tau}\;d\tau
\leq \int_1^{+\infty}\frac{1}{\tau^s}\cdot\frac{1}{1+\tau}\;d\tau.
$$
For the other integral
$$
\frac{\ep^s}{\overline{g}(\ep)}\int_0^{\ep}
\frac{\overline{g}(r)}{r^s}\cdot\frac{1}{r+\ep}\;dr\leq
\frac{\ep^s}{\ep\,\overline{g}(\ep)}\int_0^{\ep}
\frac{\overline{g}(r)}{r^s}\;dr.
$$
To compute the limit as $\ep\downarrow 0$
we use a Taylor expansion:
$$
\frac{1}{\ep}\int_0^\ep G(r)\;dr
=G(\ep)+G'(\ep)\,\frac{\ep^2}{2}+o(\ep^2),
$$
where we have denoted by $G(\ep)=\ep^{-s}\,\overline{g}(s)$.
Thus
$$
\frac{\ep^s}{\overline{g}(\ep)}\int_0^{\ep}
\frac{\overline{g}(r)}{r^s}\cdot\frac{1}{r+\ep}\;dr\leq
1+\frac{G'(\ep)}{2\,G(\ep)}\,\ep+o\left(\frac{\ep}{G(\ep)}\right).
$$
We are going to show now that
$$
\frac{|G'(\ep)|}{G(\ep)}\,\ep<1.
$$
Indeed
$$
-\int_\ep^\eta\frac{G'(\xi)}{G(\xi)}\;d\xi<\int_\ep^\eta\frac{d\xi}{\xi}
\qquad\Longleftrightarrow\qquad G(\ep)<G(\eta)\,\frac{\eta}{\ep}
$$
which is guaranteed by the fact that $G$ is integrable in a neighbourhood of $0$.

These computations show that, in the case of the ball,
the explosion rate of the $s$-harmonic function induced 
by a large boundary datum is the almost the same as
the rate of the datum itself.

Note now that up to \eqref{34} the same computations provide a lower estimate for $u$
if we substitute $\overline{g}$ with $\underline{g}(r)=\inf_{\delta(x)=r}g(x)$.
Then
\begin{multline*}
\int_{-1}^1
\frac{(1-\theta_1^2)^{(N-3)/2}}{{|1+M^2-2M\theta_1|}^{N/2}}
\;d\theta_1=
\int_{-1}^1
\frac{(1-\theta_1^2)^{(N-3)/2}}{{|1-\theta_1^2+(M-\theta_1)^2|}^{N/2}}
\;d\theta_1\geq\\
\geq\int_0^1\frac{\sigma^{(N-3)/2}}{\left|\sigma+(M-1+\sigma)^2\right|^{N/2}}\;d\sigma
\geq \frac{1}{M-1}
\end{multline*}
where the last inequality is \eqref{45}.
Finally we need only to repeat the above computations replacing
$\overline{g}$ with $\underline{g}$ and other minor modifications.

In the case of a general smooth domain,
we can reduce to the spherical case 
as we did to conclude Paragraph \ref{rhs-blowing}.

\chapter{Semilinear fractional Dirichlet problems}\label{nonlin-sec}

\section[the method of sub- and supersolution: proof of theorem \ref{NL-CS}]
{the method of sub- and supersolution:\\ proof of theorem \ref{NL-CS}}

The proof is an adaptation of the result
by Cl\'ement and Sweers \cite{clem-sweers}.

\it Existence. \rm
We can reduce the problem to homogeneous boundary condition,
indeed by considering the solution of 
$$
\left\lbrace\begin{aligned}
\Ds v &=0 & & \hbox{ in }\Omega \\
v &=g & & \hbox{ in }\C\Omega \\
Ev &=0 & & \hbox{ on }\partial\Omega
\end{aligned}\right.
$$
we can think of solving the problem
(note that the boundedness of $g$ implies \eqref{gintro})
$$
\left\lbrace\begin{aligned}
\Ds u &=-f(x,v+u) & & \hbox{ in }\Omega \\
u &=0 & & \hbox{ in }\C\Omega \\
Eu &=0 & & \hbox{ on }\partial\Omega
\end{aligned}\right.
$$
therefore from now on we will suppose $g\equiv 0$.
Note also that since $v$ is continuous in $\Omega$ and bounded then
$(x,t)\mapsto f(x,v(x)+t)$ satisfies \it f.1) \rm too.

Modify $f$ by defining
$$
F(x,u)=\left\lbrace\begin{array}{ll}
f(x,\overline{u}(x)) & \hbox{if }u>\overline{u}(x) \\
f(x,u) & \hbox{if }\underline{u}(x)\leq u\leq \overline{u}(x) \\
f(x,\underline{u}(x)) & \hbox{if }u<\underline{u}(x)
\end{array}\right.\qquad\hbox{ for every }x\in\overline{\Omega},\, u\in\R:
$$
the function $F(x,u)$ is continuous and bounded on $\Omega\times\R$,
by hypothesis \it f.1) \rm and the boundedness of $\underline{u},\overline{u}$. 
We can write a solution of 
$$
(\ast)\qquad\left\lbrace\begin{aligned}
\Ds u &= -F(x,u) & & \hbox{ in }\Omega \\
u &=0 & & \hbox{ in }\C\Omega \\
Eu &=0 & & \hbox{ on }\partial\Omega
\end{aligned}\right.
$$
as a fixed-point of the map obtained as the composition
$$
\begin{array}{ccccl}
L^\infty(\Omega) & 
\longrightarrow & L^\infty(\Omega) & 
\longrightarrow & L^\infty(\Omega) \\
u & \longmapsto & -F(x,u(x)) & \longmapsto & w \:\hbox{ s.t. }
\left\lbrace\begin{aligned}
& \Ds w=-F(x,u(x))\hbox{ in }\Omega,\\ 
& w=0\hbox{ in }\C\Omega,\ Ew=0\hbox{ on }\partial\Omega.
\end{aligned}\right.
\end{array}
$$
The first map sends $L^\infty(\Omega)$ 
in a bounded subset of $L^\infty(\Omega)$,
by continuity of $f$ and boundedness of $\underline{u},\,\overline{u}$.
The second map is compact since $w\in C^s(\R^N)$, thanks to
the results in\footnote{The notion of solution used in \cite{rosserra} is different
from the one we use here, but look at Section \ref{weakvsl1}.} \cite[Proposition 1.1]{rosserra}.
Then the composition admits a fixed point in view of the Schauder Fixed Point Theorem.

Note that a solution to the original problem lying between $\underline{u}$ and $\overline{u}$ in $\Omega$, 
is also a solution of $(\ast)$.
Moreover, \it any \rm solution of $(\ast)$ is between $\underline{u}$ and $\overline{u}$.
Indeed consider $A:=\{x\in\Omega:u(x)>\overline{u}(x)\}$, which is open by the continuity of $u$ and $\overline{u}$.
For any $\psi\in C^\infty_c(A)$, $\psi\geq 0$,
with the corresponding $\phi_A\in\T(A)$, by Corollary \ref{subdomain},
\begin{eqnarray*}
\int_A u(x)\,\psi(x)\;dx & = & -\int_A F(x,u(x))\,\phi_A(x)\;dx-\int_{\Omega\setminus A}u(x)\,\Ds\phi_A(x)\;dx \\
& = & -\int_A F(x,\overline{u}(x))\,\phi_A(x)\;dx-\int_{\Omega\setminus A}u(x)\,\Ds\phi_A(x)\;dx \\
& \leq & \int_A\overline{u}(x)\,\psi(x)\;dx
\end{eqnarray*}
which implies $u\leq\overline{u}$ in $A$,
by positivity of $\psi$, proving $A=\emptyset$.

\it Uniqueness. \rm If we have two continuous solutions $u$ and $w$
$$
\left\lbrace\begin{aligned}
\Ds u &= -f(x,u) & & \hbox{ in }\Omega \\
u &=g & & \hbox{ in }\C\Omega \\
Eu &=0 & & \hbox{ in }\partial\Omega
\end{aligned}\right.
\qquad\qquad
\left\lbrace\begin{aligned}
\Ds w &=-f(x,w) & & \hbox{ in }\Omega \\
w &=g & & \hbox{ in }\C\Omega \\
Ew &=0 & & \hbox{ in }\partial\Omega
\end{aligned}\right.
$$
then for the difference $u-w$ it holds 
$$
\left\lbrace\begin{aligned}
\Ds(u-w) &= -f(x,u)+f(x,w) & & \hbox{ in }\Omega \\
u-w &= 0 & & \hbox{ in }\C\Omega \\
Eu-Ew &= 0 & & \hbox{ in }\partial\Omega.
\end{aligned}\right.
$$
Defining $\Omega_1=\{x\in\Omega:w(x)<u(x)\}$, thanks to the monotony of $f$,
$$
\left\lbrace\begin{aligned}
\Ds(u-w) &\leq 0 & & \hbox{ in }\Omega_1 \\
u-w &\leq 0 & & \hbox{ in }\C\Omega_1 \\
Eu-Ew &=0 & & \hbox{ in }\partial\Omega_1
\end{aligned}\right.
$$
but then, according to Lemma \ref{max-princweak2}, $u\leq w$ in $\Omega_1$.
This means $\Omega_1$ is empty. By reversing the roles of $u$ and $w$,
we deduce $u=w$ in $\Omega$.

{\it Minimal solution.}
We refer the reader to the proof in \cite[Corollary 2.2]{KO-dupaigne}.

\section{proof of theorem \ref{LARGE-BUILDING}}

In the case of negative right-hand side,
Theorem \ref{LARGE-BUILDING} follows from Theorem \ref{-SIGN}.
So, assume the right-hand side is positive and consider 
$$
\left\lbrace\begin{array}{l}
\Ds v= f(x,v) \quad \hbox{ in }\Omega, \\ \displaystyle
\lim_{\stackrel{\hbox{\scriptsize $x\!\rightarrow\!\partial\Omega$}}{x\in\Omega}}v(x)=+\infty.
\end{array}\right.
$$
We look for a suitable shape $g$ of $v$ outside $\Omega$
and exploding at $\partial\Omega$:
the large $s$-harmonic function $v_0$ induced by $g$
in $\Omega$ will be a subsolution of our equation,
and in particular will imply that the blow-up condition 
at $\partial\Omega$ is fulfilled.
Then, in order to prove the existence part,
we need a supersolution.

Consider $F:\R\rightarrow\R$ continuous, increasing
and such that $F(t)\geq f(x,t)$ for any $t\geq 0$:
for example,
$$
F(t)={t}^{\frac{4s}{1-s}}+\max_{0\leq r\leq t}\left[\max_{x\in\overline\Omega}f(x,r)\right].
$$
Choose
$$
g(x):=\frac{F^{-1}(I(x))-1}{\overline{c}},
$$
where $\overline{c}=\overline{c}(N,s,\Omega)$ is the constant of equation
\eqref{uleqg} giving the upper control
of large $s$-harmonic functions in terms of the boundary datum
(see Paragraph \ref{expl-sharm-sec}), and define for $x\in\R^N,\ \delta(x)\leq\max_{x'\in\overline\Omega}\delta(x')$
$$
I(x)=\min\left\lbrace C_{N,s}\int_{\C\Omega}\frac{dy}{{|z-y|}^{N+2s}}:z\in\Omega,\ 
\delta(z)=\delta(x)\right\rbrace,
$$
while $I(x)=0$ when $\displaystyle x\in\C\Omega,\ \delta(x)>\max_{x'\in\overline\Omega}\delta(x')$.
Note that when $\delta(x)$ is small
$$
I(x)\leq C_{N,s}\,\omega_{N-1}\int_{\delta(x)}^{+\infty}\frac{d\rho}{{\rho}^{1+2s}}=\frac{C_{N,s}\,\omega_{N-1}}{2s}\cdot\frac{1}{{\delta(x)}^{2s}}.
$$
Such $g$ satisfies hypothesis \eqref{g}, since when $\delta(x)$ is small
$$
F(t)\geq {t}^{\frac{4s}{1-s}}\ \Rightarrow\ 
F^{-1}(t)\leq {t}^{\frac{1-s}{4s}} \ \Rightarrow\ 
F^{-1}\left(I(x)\right)\leq\left(\frac{C_{N,s}\,\omega_{N-1}}{2s}\right)^{\frac{1-s}{4s}}\frac{1}{{\delta(x)}^{(1-s)/2}}.
$$

Call $v_0$ the solution to
$$
\left\lbrace\begin{aligned}
\Ds v_0 &= 0 & & \hbox{ in }\Omega, \\
v_0 &=g & & \hbox{ in }\C\Omega, \\
Ev_0 &=0 & & \hbox{ on }\partial\Omega.
\end{aligned}\right.
$$
Denote by $w:=v-v_0$: our claim
is that problem
$$
\left\lbrace\begin{aligned}
\Ds w &=f(x,v_0+w) & & \hbox{ in }\Omega \\
w &=0 & & \hbox{ in }\C\Omega \\
Ew &=0 & & \hbox{ on }\partial\Omega
\end{aligned}\right.
$$
admits a solution $w$.
Indeed, we have a subsolution which is
the function constant to $0$ in $\Omega$
and $\chi_\Omega$ turns out to be a supersolution.
To show this we consider the problem
$$
\left\lbrace\begin{aligned}
\Ds w &\geq F(v_0+w) & & \hbox{ in }\Omega \\
w &=0 & & \hbox{ in }\C\Omega \\
Ew &=0 & & \hbox{ on }\partial\Omega
\end{aligned}\right.
$$
and observe that, when $x\in\Omega$
$$
\Ds\chi_\Omega(x)=C_{N,s}\int_{\C\Omega}\frac{dy}{{|x-y|}^{N+2s}}\geq I(x)
$$
and
$$
F(v_0(x)+1)\leq F(\overline{c}g(x)+1)=F(F^{-1}(I(x)))=I(x).
$$
Finally, the property $F(v_0+w)\geq f(x,v_0+w)$
concludes the construction of the supersolution.
Then Lemma \ref{sub+super} below provides 
the existence of a solution.

\section{proof of theorem \ref{-SIGN}}\label{proof-SIGN}

For any $n\in\N$, denote by $g_n=\min\{g,n\}$. Also, with the notation of equation \eqref{h-approx},
for a small parameter $r>0$ denote by
$$
f_r(y)=f_r(\rho,\theta)=h(\theta)\,\frac{\varphi(\rho/r)}{K_r},
\qquad K_r=\frac{1}{1+s}\int_{\Omega_r}\varphi(\delta(y)/r)\,{\delta(y)}^s\;dy,
$$
and recall that this is an approximation of the $h$ boundary datum.
Finally call $u_{n,r}$ the minimal solution of
$$
\left\lbrace\begin{aligned}
\Ds u_{n,r}(x) &=-f(x,u_{n,r}(x))+f_r(x) & & \hbox{ in }\Omega \\
u_{n,r} &=g_n & & \hbox{ in }\C\Omega \\
Eu_{n,r} &=0 & & \hbox{ on }\partial\Omega
\end{aligned}\right.
$$
provided by Theorem \ref{NL-CS}, with $\underline{u}\equiv 0$ and $\overline{u}=u_r^0$ defined below.
Note that for any $r>0$, 
the sequence $\{u_{n,r}\}_{n\in\N}$ we obtain is increasing in $n$:
indeed, $u_{n+1,r}$ is a supersolution for the problem defining $u_{n,r}$,
since it has larger boundary values and the minimality property
on $u_{n,r}$ gives $u_{n,r}\leq u_{n+1,r}$.
Moreover, ${\{u_{n,r}\}}_{n\in\N}$ is bounded by the 
function $u_r^0$ associated with 
the linear problem with data $g$ and $f_r$, i.e.
$$
\left\lbrace\begin{aligned}
\Ds u_r^0 &=f_r & & \hbox{ in }\Omega, \\
u_r^0 &=g & & \hbox{ in }\C\Omega, \\
Eu_r^0 &=0 & & \hbox{ on }\partial\Omega.
\end{aligned}\right.
$$
Therefore $u_{n,r}$ admits a pointwise limit in $\R^N$ as $n\uparrow+\infty$. 
Call $u_r$ this limit: 
obviously $u_r=g$ in $\C\Omega$. Take any nonnegative $\phi\in\T(\Omega)$
with $0\leq\psi=\Ds\phi|_\Omega\in C^\infty_c(\Omega)$: then
\begin{multline*}
\int_\Omega[f(x,u_r)-f_r]\,\phi\leq\liminf_{n\uparrow+\infty}\int_\Omega[f(x,u_{n,r})-f_r]\,\phi=\\
=-\limsup_{n\uparrow+\infty}\int_\Omega u_{n,r}\psi-\int_{\C\Omega}g\,\Ds\phi
=-\int_\Omega u_r\psi-\int_{\C\Omega}g\,\Ds\phi,
\end{multline*}
where we have used the Fatou lemma and the continuity 
of the map $t\mapsto f(x,t)$.
This means that $u_r$ is a subsolution.

We are left to prove that $u_r$ is also a supersolution.
Call $\Omega'=$ supp$\psi\subset\subset\Omega$ 
and build a sequence ${\{\Omega_k\}}_{k\in\N}$
such that $\Omega'\subseteq\Omega_k\subseteq\Omega$ and $\Omega_k\nearrow\Omega$.
Since $\psi\in C^\infty_c(\Omega_k)$ for any $k$, then the we can build 
the sequence of functions $\phi_k\in\T(\Omega_k)$ induced by $\psi$:
this sequence is increasing and converges pointwise to $\phi$.
Moreover, for any $k$, since $\Ds\phi_k\leq0$ in $\C\Omega_k$
\begin{align*}
& \int_\Omega[f(x,u_r)-f_r]\,\phi_k=\lim_{n\uparrow+\infty}\int_{\Omega_k}[f(x,u_{n,r})-f_r]\,\phi_k \\
& =\lim_{n\uparrow+\infty}\left(-\int_{\Omega_k}u_{n,r}\psi-\int_{\C\Omega_k}u_{n,r}\,\Ds\phi_k\right)\\
& \geq\lim_{n\uparrow+\infty}\left(-\int_{\Omega}u_{n,r}\psi-\int_{\C\Omega}g_n\,\Ds\phi_k\right)
=-\int_\Omega u_r\psi-\int_{\C\Omega}g\,\Ds\phi_k:
\end{align*}
letting both sides of the inequality pass to the limit as $k\uparrow+\infty$
we obtain
$$
\int_\Omega[-f(x,u_r)-f_r]\,\phi\geq-\int_\Omega u_r\psi-\int_{\C\Omega}g\,\Ds\phi,
$$
because recall that for $x\in\C\Omega\subseteq\C\Omega_k$
$$
-\Ds\phi_k(x)=\int_{\Omega_k}\frac{\phi_k(y)}{{|x-y|}^{N+2s}}\;dy
$$
increases to $-\Ds\phi(x)$.

This means that $u_r$ is both a sub- and a supersolution
and it solves
\begin{equation}\label{ur}
\left\lbrace\begin{aligned}
\Ds u_r(x) &= -f(x,u_r(x))+f_r(x) & & \hbox{ in }\Omega, \\
u &=g & & \hbox{ in }\C\Omega, \\
Eu &=0 & & \hbox{ on }\partial\Omega.
\end{aligned}\right.
\end{equation}

\begin{rmk}\rm Note that we have just solved all problems 
with null $h$ boundary datum, i.e.
$$
\left\lbrace\begin{aligned}
\Ds u(x) &=-f(x,u(x)) & & \hbox{ in }\Omega \\
u &=g & & \hbox{ in }\C\Omega \\
Eu &=0 & & \hbox{ on }\partial\Omega
\end{aligned}\right.
$$
since in this case $f_r\equiv 0$.
\end{rmk}

We want now to push $r\downarrow0$.
We claim that the family ${\{u_r\}}_r$ is uniformly bounded and
equicontinuous on every compact subset of $\Omega$:
there would exist then a $u\in L^1(\Omega)$ 
and a sequence ${\{r_k\}}_{k\in\N}$
such that $r_k\rightarrow 0$ as $k\rightarrow+\infty$ and
$u_{r_k}\rightarrow u$ a.e. in $\Omega$ and uniformly
on compact subsets. Then for any $\phi\in\T(\Omega)$
$$
\begin{array}{ccccc}
\displaystyle
\int_\Omega u_{r_k}\Ds\phi & = & \displaystyle
-\int_\Omega f(x,u_{r_k})\,\phi &
\displaystyle +\int_\Omega f_{r_k}\,\phi &
\displaystyle -\int_{\C\Omega}g\,\Ds\phi \\
\downarrow & & \downarrow & \downarrow & \\
\displaystyle
\int_\Omega u\Ds\phi & = & \displaystyle
-\int_\Omega f(x,u)\,\phi &
\displaystyle +\int_{\partial\Omega} h\,D_s\phi &
\displaystyle -\int_{\C\Omega}g\,\Ds\phi
\end{array}
$$
where the convergence $\int_\Omega f(x,u_{r_k})\phi\rightarrow\int_\Omega f(x,u)\phi$
holds since
$$
f(x,u_{r_k})|\phi|\leq\Lambda(au_{r_k})\,C{\delta}^s
\leq\Lambda(au_0)\,C{\delta}^s\leq
\Lambda(C\delta^{s-1})\,{\delta}^s\in L^1(\Omega).
$$
We still have to prove that our claim is true.
The uniform boundedness on compact subsets of ${\{u_r\}}_r$ is a
consequence of inequalities
$$
0\leq u_r \leq u_r^0 \xrightarrow{r\downarrow 0} u_0,\quad\hbox{ for any }r,
$$
where
$$
\left\lbrace\begin{aligned}
\Ds u_r^0 &=f_r & & \hbox{ in }\Omega \\
u_r^0 &=g & & \hbox{ in }\C\Omega \\
Eu_r^0 &=0 & & \hbox{ on }\partial\Omega,
\end{aligned}\right.
\qquad\qquad
\left\lbrace\begin{aligned}
\Ds u_0 &=0 & & \hbox{ in }\Omega \\
u_0 &=g & & \hbox{ in }\C\Omega \\
Eu_0 &=h & & \hbox{ on }\partial\Omega,
\end{aligned}\right.
$$
and the convergence of $v_r$ to $u_0$ is uniform
in compact subsets of $\Omega$.
Then
\begin{align*}
& |u_r(x)-u_r(z)|  = 
\left|\int_\Omega\left(G_\Omega(x,y)-G_\Omega(z,y)\right)\,[-f(y,u_r(y))+f_r(y)]\;dy\right|\\
& \leq \int_\Omega\left|G_\Omega(x,y)-G_\Omega(z,y)\right|\left(a_1+C\,a_2{\delta(y)}^{p(s-1)}\right)\;dy 
+\int_\Omega\left|G_\Omega(x,y)-G_\Omega(z,y)\right|\,C\,{\delta(y)}^{-1}\;dy \\
\end{align*}
implies the equicontinuity.
\bigskip
If there exist $b_1,T>0$
such that 
$$
f(x,t)\geq b_1{t}^{\frac{1+s}{1-s}},\quad \hbox{ for }t>T.
$$
and $h\not\equiv 0$, then a solution would satisfy
\[
f(\cdot,u)\,\delta^s\ \not\in L^1(\Omega);
\]
in this case the integration by parts formula defining a weak solution
would not make sense.

Note that this proof exploits the negativity of the right-hand side 
only in considering the $s$-harmonic function induced by
$g$ and $h$ as a supersolution of problem 
$$
\left\lbrace\begin{aligned}
-\Ds u &=f(x,u) & & \hbox{ in }\Omega \\
u &=g & & \hbox{ in }\C\Omega \\
Eu &=h & & \hbox{ in }\partial\Omega.
\end{aligned}\right.
$$
With minor modifications to the proof we can state

\begin{lem}\label{sub+super}
Let $f:\overline{\Omega}\times\R\rightarrow[0,+\infty)$ be a function 
satisfying f.1) and f.2).
Let $g:\C\Omega\rightarrow\R^+$ be a measurable function
satisfying \eqref{gintro} and $h\in C(\partial\Omega),\ h\geq0$.
Assume the nonlinear problem
$$
\left\lbrace\begin{aligned}
\Ds u &=f(x,u) & & \hbox{ in }\Omega \\
u &=g & & \hbox{ in }\C\Omega \\
Eu &=h & & \hbox{ on }\partial\Omega
\end{aligned}\right.
$$
admits a subsolution $\underline{u}\in L^1(\Omega)$ and a supersolution $\overline{u}\in L^1(\Omega)$.
Assume also $\underline{u}\leq\overline{u}$ in $\Omega$.
Then the above nonlinear problem has a weak solution $u\in L^1(\Omega)$ 
satisfying
$$
\underline{u}\leq u\leq\overline{u}.
$$
\end{lem}
\begin{proof} Replace, in the above proof, the function
$u^0$ with the supersolution $\overline{u}$.
\end{proof}

\section[sublinear nonnegative nonlinearity: proof of theorem \ref{SUBLINEAR}]
{sublinear nonnegative nonlinearity:\\ proof of theorem \ref{SUBLINEAR}}

We first prove a Lemma which will make the proof easily go through.

\begin{lem} There exists $m=m(\Lambda)>0$ sufficiently large for which 
any problem of the form
$$
\left\lbrace\begin{aligned}
\Ds u(x) &=\Lambda (u(x)) & & \hbox{ in }\Omega, &  \\
u &=g & & \hbox{ in }\C\Omega, & g\geq m>0, \\
Eu &=h & & \hbox{ on }\partial\Omega, &
\end{aligned}\right.
$$
is solvable.
\end{lem}

\begin{proof}\rm We can equivalently solve the integral equation
$$
u(x)\ =\ u_0(x)+\int_\Omega G_\Omega(x,y)\,\Lambda (u(y))\;dy,
$$
where $u_0$ is the s-harmonic function induced by $g$ and $h$ in $\Omega$.

Define the map
$$
\begin{array}{rcl}
\mathcal{K}\ :\ L^1(\Omega) & \longrightarrow & L^1(\Omega) \\
u(x) & \longmapsto & \displaystyle
u_0(x)+\int_\Omega G_\Omega(x,y)\,\Lambda (u(y))\;dy
\end{array}
$$
The condition $g\geq m$ in $\C\Omega$ entails $u_0\geq m$ in $\Omega$;
also, for any $w\in L^1(\Omega),\ w\geq 0$ implies $\mathcal{K} w\geq u_0\geq m$,
therefore $\mathcal{K}$ sends the subset $D_m:=\{w\in L^1(\Omega):w\geq m\}$ of $L^1(\Omega)$
into itself. Moreover, for $u,v\in D_m$
\begin{multline*}
\int_\Omega\left|\mathcal{K}u(x)-\mathcal{K}v(x)\right|\;dx
\leq
\int_\Omega\left|\Lambda (u(y))-\Lambda (v(y))\right|\int_\Omega G_\Omega(x,y)\;dx\;dy
\leq\\
\leq\Arrowvert\zeta\Arrowvert_\infty
\sup_{x\in\Omega}\Lambda' (u(x))\int_\Omega\left| u(y)-v(y)\right|\;dy
\leq\Arrowvert\zeta\Arrowvert_\infty
\Lambda'(m)\int_\Omega\left| u(y)-v(y)\right|\;dy
\end{multline*}
where $\zeta(y)=\int_\Omega G_\Omega(x,y)\;dx$.
Now, if $m$ is very large, we have 
$$
\Arrowvert\zeta\Arrowvert_\infty\Lambda'(m)<1,
$$
i.e. $\mathcal{K}$ is a contraction on $D_m$,
and $\mathcal{K}$ has a fixed point in $D_m$.
\end{proof}

In general, for the problem
$$
\left\lbrace\begin{aligned}
\Ds u &=f(x,u) & & \hbox{ in }\Omega, \\
u &=g & & \hbox{ in }\C\Omega, \\
Eu &=h & & \hbox{ on }\partial\Omega,
\end{aligned}\right.
$$
we have a subsolution which is the $s$-harmonic function satisfying 
the boundary conditions.

But we are now able to provide a supersolution:
this can be done by setting $g_m=\max\{g,m\}$
and by solving, for some large value of $m$,
$$
\left\lbrace\begin{aligned}
\Ds\overline{u} &= \Lambda(\overline{u})\geq f(x,\overline{u}) & & \hbox{ in }\Omega,  \\
\overline{u} &=g_m\geq g & & \hbox{ in }\C\Omega,\\
E\overline{u} &=h & & \hbox{ on }\partial\Omega.
\end{aligned}\right.
$$

It is sufficient to apply the classical iteration scheme starting from
the $s$-harmonic function $u_0$ and with iteration step
$$
\hbox{for any }k\in\N\qquad
\left\lbrace\begin{aligned}
\Ds u_k &=f(x,u_{k-1}(x)) & & \hbox{ in }\Omega, \\
u_k &=g & & \hbox{ in }\C\Omega,\\
Eu_k &=h & & \hbox{ on }\partial\Omega.
\end{aligned}\right.
$$
In such a way we build an increasing sequence ${\{u_k\}_{k\in\N}}\subseteq L^1(\Omega)$
which is uniformly bounded from above by $\overline{u}$.
Indeed, on the one hand we have that 
$$
\left\lbrace\begin{aligned}
\Ds (u_1-u_0) &=f(x,u_0(x))\geq 0 & & \hbox{ in }\Omega \\
u_1-u_0 &=0 & & \hbox{ in }\C\Omega \\
Eu_1-Eu_0 &=0 & & \hbox{ on }\partial\Omega
\end{aligned}\right.
$$
entails that $u_1-u_0\geq0$, while on the other hand
an induction argument relying on the monotony of $t\mapsto f(x,t)$
finally shows that $u_k$ increases.
Call $u(x):=\lim_{k}u_k(x)$, which is finite in view of 
the upper bound furnished by $\overline{u}$.
Then
\[
u(x)=\lim_{k\uparrow+\infty}u_k(x)=u_0(x)+
\lim_{k\uparrow+\infty}\int_\Omega f(y,u_{k-1}(y))\,G_\Omega(x,y)\; dy=\\
=u_0(x)+\int_\Omega f(y,u(y))\,G_\Omega(x,y)\; dy.
\]

\section[superlinear nonnegative nonlinearity: proof of theorem \ref{SUPERLINEAR}]
{superlinear nonnegative nonlinearity:\\ proof of theorem \ref{SUPERLINEAR}}

We give the proof for problem 
$$
\left\lbrace\begin{aligned}
\Ds u &= \lambda f(x,u) & & \hbox{ in }\Omega \\
u(x) &= \delta(x)^{-\beta} & & \hbox{ in }\C\Omega,\ 0<\beta<1-s \\
Eu &=0 & & \hbox{ on }\partial\Omega
\end{aligned}\right.
$$
while for the other one it is sufficient to replace $\beta$ with $1-s$
and repeat the same computations.

To treat the case of a general nonlinearity
we use again the equivalent integral equation
$$
u(x)\ =\ u_0(x)+\lambda\int_\Omega G_\Omega(x,y)\,f(y,u(y))\;dy,
$$
where $u_0$ is the s-harmonic function induced in $\Omega$ by the boundary data.
In this case the computations in Section \ref{bblin-sec} on the 
rate of explosion at the boundary turn out to be very useful.
Indeed on the one hand we have that $u_0$ inherits its explosion
from the boundary data $g$ and $h$: briefly, in our case
\begin{equation*}
g(x)= \frac{1}{{\delta(x)}^\beta},\quad 0<\beta<1-s
\qquad  \longrightarrow \qquad
\frac{\underline{c}}{{\delta(x)}^\beta}\leq u_0(x)\leq \frac{\overline{c}}{{\delta(x)}^\beta}.
\end{equation*}
Since $u_0$ is a subsolution,
our first goal is to build a supersolution and we build it of the form
$$
\overline{u}=u_0+\zeta,
$$
where
$$
\left\lbrace\begin{aligned}
\Ds\zeta &={\delta(x)}^{-\gamma} & & \hbox{ in }\Omega,\quad\gamma>0\\
\zeta &=0 & & \hbox{ in }\C\Omega, \\
E\zeta &=0 & & \hbox{ on }\partial\Omega.
\end{aligned}\right.
$$
Recall that, \eqref{udelta} says
$$
\zeta(x)\geq \left\lbrace\begin{array}{ll}\displaystyle
c_1{\delta(x)}^s & \hbox{ if }0\leq\gamma<s, \\ \displaystyle 
c_3{\delta(x)}^s\,\log\frac{1}{\delta(x)} & \hbox{ if }\gamma=s, \\ \displaystyle 
c_5{\delta(x)}^{-\gamma+2s} & \hbox{ if }s<\gamma<1+s.
\end{array}\right.
$$
The function $\overline{u}$ is a supersolution if
$$
\overline{u}(x)\geq u_0(x)+\lambda\int_\Omega G_\Omega(x,y)\,f(y,\overline{u}(y))\;dy,
$$
or, equivalently formulated
\begin{equation}\label{12}
\zeta(x)\geq\lambda\int_\Omega G_\Omega(x,y)\,f(y,u_0(y)+\zeta(y))\;dy.
\end{equation}

If $f(x,t)$ has an algebraic behavior\footnote{For $p<1$ we are actually in the case of the previous paragraph.}
$$
f(x,t)\leq a_1+a_2\,{t}^p,\quad a_1,a_2>0,\ p\geq 1
$$
then
$$
f(y,u_0(y)+\zeta(y))\leq a_1+a_2\,\left(u_0(y)+\zeta(y)\right)^p\leq
\left\lbrace\begin{array}{ll}
\displaystyle\frac{C}{{\delta(x)}^{p\beta}} & \hbox{ if }\gamma-2s\leq\beta, \\
\displaystyle\frac{C}{{\delta(x)}^{p(\gamma-2s)}} & \hbox{ if }\gamma-2s>\beta.
\end{array}\right.
$$
In case $\gamma-2s\leq\beta$ we have
$$
\int_\Omega G_\Omega(x,y)\,f(y,u_0(y)+\zeta(y))\;dy\leq 
\left\lbrace\begin{array}{ll}
\displaystyle c_2\,C{\delta(x)}^s & \hbox{ if }p\beta<s \\
\displaystyle c_4\,C{\delta(x)}^s\log\frac{1}{\delta(x)} & \hbox{ if }p\beta= s \\
\displaystyle\frac{c_6\,C}{{\delta(x)}^{p\beta-2s}} & \hbox{ if }s< p\beta<1+s 
\end{array}\right.
$$
again in view of Paragraph \ref{rhs-blowing}, so that
we can choose $\gamma=p\beta$ provided $p\beta-2s\leq\beta$.

If this is not the case then it means we need powers $\gamma$
satisfying $\gamma-2s>\beta$. If $\gamma-2s>\beta$ we have
$$
\int_\Omega G_\Omega(x,y)\,f(y,u_0(y)+\zeta(y))\;dy\leq 
\left\lbrace\begin{aligned}
& c_2\,C{\delta(x)}^s\ \hbox{ if }p(\gamma-2s)<s \\
& c_4\,C{\delta(x)}^s\log\frac{1}{\delta(x)}\ \hbox{ if }p(\gamma-2s)= s \\
& \frac{c_6\,C}{{\delta(x)}^{p(\gamma-2s)-2s}}\ \hbox{ if }s< p(\gamma-2s)<1+s 
\end{aligned}\right.
$$
and a suitable choice for $\gamma$ would be 
$$
\gamma=\max\left\lbrace \frac{2s\,p}{p-1},\beta+2s+\ep\right\rbrace
$$
which fulfils both inequalities
$$
\gamma-2s>\beta,\qquad\gamma\geq p(\gamma-2s).
$$
This is an admissible choice for $\gamma$ provided $\gamma<1+s$, i.e. only if $p>(1+s)/(1-s)$;
in case $p$ doesn't satisfy this lower bound then
$$
p\leq\frac{1+s}{1-s} \quad \Longrightarrow \quad p\beta\leq \frac{1+s}{1-s}\,\beta\leq \beta+2s
$$
and we are in the previous case.

Finally, if $q\beta>1+s$ then a solution $u$ should satisfy, whenever $\delta(x)<1$,
$$
f(x,u(x))\geq b {u(x)}^q \geq b{u_0(x)}^q \geq c{\delta(x)}^{-q\beta}
$$
which would imply
$$
\int_\Omega G_\Omega(x,y)\,f(y,u(y))\;dy=+\infty,\quad x\in\Omega,
$$
which means that the problem is not solvable.

\section{complete blow-up: proof of theorem \ref{COMPLETE-BLOWUP}}

Consider a nondecreasing sequence $\{f_k\}_{k\in\N}$ of bounded functions 
such that $f_k\uparrow f$ pointwisely as $k\uparrow+\infty$, as in Definition \ref{complete-blowup-defi}.
Let us first prove the theorem in the case of null boundary data.
The first claim is that
$$
\int_\Omega f_k(x,u_k(x))\,{\delta(x)}^s\;dx\ \xrightarrow{k\uparrow+\infty}\ +\infty.
$$
Suppose by contradiction that the sequence 
of integrals is bounded by a constant $C$.
Consider an increasing sequence of nonnegative
${\{\psi_n\}}_{n\in\N}\subseteq C^\infty_c(\Omega)$
such that $\psi_n\uparrow1$ in $\Omega$
and pick $\phi_n\in\T(\Omega)$ in such a way that
$\Ds\phi_n=\psi_n$ holds in $\Omega$. Then
$$
\int_\Omega u_k\,\psi_n=\int_\Omega f_k(x,u_k)\,\phi_n\leq c\int_\Omega f_k(x,u_k(x))\,{\delta(x)}^s\;dx
$$
for some constant $c>0$ not depending on $n$, see \cite[Proposition 1.2]{rosserra}.
By letting $n\uparrow+\infty$, we deduce that $u_k$ is a bounded sequence in $L^1(\Omega)$.
Take now $\underline{u}_k$ as the minimal solution to the $k$-th nonlinear problem:
since $\Ds\underline{u}_{k+1}=f_{k+1}(x,\underline{u}_{k+1})\geq f_k(x,\underline{u}_{k+1})$
then $\underline{u}_k$ is an increasing sequence and it admits 
a pointwise limit $u$. This $u$ is limit also in the $L^1$-norm,
since $\underline{u}_k$ is bounded in this norm.
But then, for any $\phi\in\T(\Omega)$, we have 
by the Monotone Convergence Theorem
$$
\int_\Omega u\,\Ds\phi=\lim_{k\uparrow+\infty}\int_\Omega\underline{u}_k\,\Ds\phi=
\lim_{k\uparrow+\infty}\int_\Omega f_k(x,\underline{u}_k)\,\phi=
\int_\Omega f(x,u)\,\phi,
$$
that is $u\in L^1(\Omega)$ would be a weak solution, a contradiction.

Our second claim is that
$$
\underline{u}_k(x)\geq c\left[\int_\Omega f_k(y,\underline{u}_k(y))\,{\delta(y)}^s\;dy\right]\,{\delta(x)}^s,
$$
for some constant $c>0$ independent of $k$.
To do this, we exploit \eqref{est-green}. Call
$\Omega_1(x)=\{y\in\Omega:|y-x|\leq\delta(x)\delta(y)\}$,
$\Omega_2(x)=\{y\in\Omega:|y-x|>\delta(x)\delta(y)\}$
and $d(\Omega)=\sup\{|y-x|:x,\,y\in\Omega\}$:
\begin{eqnarray*}
\underline{u}_k(x) & = & 
\int_\Omega G_\Omega(x,y)\,f_k(y,\underline{u}_k(y))\;dy  \\
& \geq & 
c_2 \int_\Omega\left[{|x-y|}^{2s}\wedge{\delta(x)}^s{\delta(y)}^s\right]f_k(y,\underline{u}_k(y))\;\frac{dy}{{|x-y|}^N}  \\
& = & c_2 
\int_{\Omega_1(x)}f_k(y,\underline{u}_k(y))\;\frac{dy}{{|x-y|}^{N-2s}}
+c_2\,{\delta(x)}^s \int_{\Omega_2(x)}{\delta(y)}^s f_k(y,\underline{u}_k(y))\;\frac{dy}{{|x-y|}^N}  \\
& \geq & 
\frac{c_2}{{d(\Omega)}^{2s}} \int_{\Omega_1(x)}{\delta(x)}^s{\delta(y)}^s f_k(y,\underline{u}_k(y))\;\frac{dy}{{|x-y|}^{N-2s}}+
\frac{c_2}{{d(\Omega)}^N}\,{\delta(x)}^s \int_{\Omega_2(x)}{\delta(y)}^s f_k(y,\underline{u}_k(y))\;dy  \\
& \geq & 
\frac{c_2}{{d(\Omega)}^N}\,{\delta(x)}^s \int_\Omega{\delta(y)}^s f_k(y,\underline{u}_k(y))\;dy.
\end{eqnarray*} 
Let now $k\uparrow+\infty$: by the monotone pointwise convergence of $f_k$ to $f$
and by the monotonicity of $\{\underline{u}_k\}_{k\in\N}$ (as noticed above),
the last term in the inequalities above converges to $+\infty$,
and we see then how we have complete blow-up.

In the case of nonhomogeneous boundary conditions,
we consider the $s$-harmonic function $u_0$ induced by data $g$ and $h$,
and we denote by $F(x,t)=f(x,u_0(x)+t)$ for $x\in\Omega$, $t\geq 0$.
By hypothesis we have then that there is no weak solution to
$$
\left\lbrace\begin{aligned}
\Ds v &= F(x,v) & & \hbox{ in }\Omega, \\
v &=0 & & \hbox{ in }\C\Omega, \\
Ev &=0 & & \hbox{ on }\partial\Omega.
\end{aligned}\right.
$$
Since any monotone approximation on $f$ is also a monotone approximation of $F$,
then there is complete blow-up in the problem for $v$
and this bears the complete blow-up for the problem on $u$.

\chapter{A Keller-Osserman condition for \texorpdfstring{$\Ds$}{Ds}}\label{fracKO}

In this Chapter we are going to provide the proofs
of the results listed in Section \ref{fracKO-intro}.
Recall the list of assumptions we have set ourselves in.

\section{preliminaries}

Hypothesis \eqref{tech} implies that $f(t)t^{-1-M}$ is monotone decreasing
and $f(t)t^{-1-m}$ is monotone increasing, since
\[
\frac{d}{dt}\frac{f(t)}{t^{1+M}}=\frac{1}{t^{1+M}}\left(f'(t)-(1+M)\frac{f(t)}{t}\right)\ \leq\ 0, 
\qquad
\frac{d}{dt}\frac{f(t)}{t^{1+m}}=\frac{1}{t^{1+m}}\left(f'(t)-(1+m)\frac{f(t)}{t}\right)\ \geq\ 0:
\]
we write this monotonicity conditions as
\begin{equation}\label{mono-f}
c^{1+m}f(t)\ \leq\ f(ct)\ \leq\ c^{1+M}f(t),\qquad c>1,\ t>0.
\end{equation}
The function $F$ satisfies two inequalities similar to \eqref{tech}:
\begin{equation}\label{Ftech}
2+m\ \leq\ \frac{t\,f(t)}{F(t)}\ \leq\ 2+M,
\end{equation}
indeed by integrating \eqref{tech} we deduce 
\[
(1+m)F(t)\leq\int_0^t\tau\,f'(\tau)\;d\tau=tf(t)-F(t).
\]
Let $\psi=\phi^{-1}$ be the inverse of $\phi$, so that
\begin{equation}\label{psi}
v\ =\ \int_{\psi(v)}^{+\infty}\frac{dt}{\sqrt{F(t)}},\qquad v\geq 0.
\end{equation}
The function $\psi$ is decreasing and $\psi(v)\uparrow+\infty$ as $v\downarrow 0$.
Moreover, by Remark \ref{phi-rmk} and \eqref{Ftech}, for $u>0$ and some $y\in(u,+\infty)$
\begin{equation*}
\frac{\phi(u)}{u|\phi'(u)|}=\frac{\sqrt{F(u)}}{u}\int_u^{+\infty}\frac{dt}{\sqrt{F(t)}}=
\frac{-\frac{1}{\sqrt{F(y)}}}{\frac{1}{\sqrt{F(y)}}-\frac{yf(y)}{2F(y)^{3/2}}}=
\frac{1}{\frac{yf(y)}{2F(y)}-1}\quad
\left\lbrace\begin{aligned}
& \geq \frac2M \\
& \leq \frac2m
\end{aligned}\right.
\end{equation*}
which in turn says that it holds, by setting $v=\phi(u)$,
\begin{equation}\label{psitech}
\frac2M\ \leq\ \frac{v|\psi'(v)|}{\psi(v)}\ \leq\ \frac2m,\qquad
\end{equation}
and one can also prove
\begin{equation}\label{mono-psi}
\psi(cv)\ \leq\ c^{-2/M}\psi(v),\qquad c\in(0,1),\ v>0.
\end{equation}
as we have done for \eqref{mono-f} above. Also, by \eqref{Ftech} and \eqref{psitech},
\begin{equation}\label{mono-psi2}
\frac{v^2\,\psi''(v)}{\psi(v)}=\frac{v^2\,f(\psi(v))}{2\,\psi(v)}\lessgtr
\frac{v^2\,F(\psi(v))}{\psi(v)^2}=\frac{v^2\,\psi'(v)^2}{\psi(v)^2}\lessgtr 1.
\end{equation}

\section{construction of a supersolution}

In this paragraph we prove the key point for the proof of Theorems 
\ref{main1} and \ref{main2}, that is we build a supersolution
to both problems by handling the function $U$ defined in \eqref{U} below.

Since by assumption $\partial\Omega\in C^2$, the function $\dist(x,\partial\Omega)$
is $C^2$ in an open strip around the boundary, except on $\partial\Omega$ itself.
Consider a positive function $\delta(x)$ which is obtained by
extending $\dist(x,\partial\Omega)$ smoothly to $\R^N\setminus\partial\Omega$.
Define 
\begin{equation}\label{U}
U(x)\ =\ \psi(\delta(x)^s),\qquad x\in\R^N.
\end{equation}

\begin{lem}\label{U-L1} 
The function $U$ defined in \eqref{U} is in $L^1(\Omega)$.
\end{lem}
\begin{proof}
Since both $\psi$ and $\delta^s$ are continuous in $\Omega$, then $U\in L^1_{loc}(\Omega)$.
Fix $\delta_0>0$ small and consider $\Omega_0=\{x\in\Omega:\delta(x)<\delta_0\}$. We have
(using once the coarea formula)
\[
\int_{\Omega_0}\psi(\delta(x)^s)\;dx\ 
\leq\ C\int_0^{\delta_0}\psi(t^s)\;dt:
\]
apply now the transformation $\psi(t^s)=\eta$ to get
\[
\int_{\Omega_0}U(x)\;dx \ \leq \ C\int_{\eta_0}^{+\infty}\eta\,\phi(\eta)^{(1-s)/s}|\phi'(\eta)|\;d\eta 
\]
where, by Remark \ref{phi-rmk},
\[
|\phi'(\eta)|\lessgtr\frac1{\sqrt{\eta\,f(\eta)}}\qquad\hbox{and}\qquad\phi(\eta)\lessgtr\sqrt{\frac{\eta}{f(\eta)}},
\]
therefore
\[
\int_{\Omega_0}U(x)\;dx \ \leq \ C\int_{\eta_0}^{+\infty}\left(\frac{\eta}{f(\eta)}\right)^\frac1{2s}d\eta 
\]
which is finite by \eqref{L1-bis}.
\end{proof}

\begin{prop}\label{impo}
The function $U$ defined in \eqref{U} satisfies for some $C,\,\delta_0>0$
\begin{equation}\label{impo-eq}
\Ds U \geq-Cf(U),\qquad \hbox{in }\Omega_{\delta_0}=\{x\in\Omega:\delta(x)<\delta_0\}.
\end{equation}
\end{prop}
\begin{proof}
We start by writing, for $x\in\Omega$
\begin{equation}\label{split}
\frac{\Ds U(x)}{C_{N,s}}=PV\int_\Omega\frac{\psi(\delta(x)^s)-\psi(\delta(y)^s)}{|x-y|^{N+2s}}\:dy
+\int_{\C\Omega}\frac{\psi(\delta(x)^s)-\psi(\delta(y)^s)}{|x-y|^{N+2s}}\:dy.
\end{equation}

Let us begin with an estimate for 
\[
PV\int_\Omega\frac{\psi(\delta(x)^s)-\psi(\delta(y)^s)}{|x-y|^{N+2s}}\:dy.
\]
Split the integral into
\[
\int_{\Omega_1}\frac{\psi(\delta(x)^s)-\psi(\delta(y)^s)}{|x-y|^{N+2s}}\:dy
+PV\int_{\Omega_2}\frac{\psi(\delta(x)^s)-\psi(\delta(y)^s)}{|x-y|^{N+2s}}\:dy+
\int_{\Omega_3}\frac{\psi(\delta(x)^s)-\psi(\delta(y)^s)}{|x-y|^{N+2s}}\:dy
\]
where we have set
\begin{align*}
\Omega=\Omega_1\cup\Omega_2\cup\Omega_3,\text{ with: } 
 & \Omega_1=\left\{y\in\Omega:\delta(y)>\frac32\delta(x)\right\} \\
 & \Omega_2=\left\{y\in\Omega:\frac12\delta(x)\leq\delta(y)\leq\frac32\delta(x)\right\} \\
 & \Omega_3=\left\{y\in\Omega:\delta(y)<\frac12\delta(x)\right\}.
\end{align*}
In $\Omega_1$ we have in particular $\delta(y)>\delta(x)$ so that,
since $\psi$ decreasing function, the first integral contributes by a positive quantity.

Now, let us turn to integrals on $\Omega_2$ and $\Omega_3$.
Set $x=\theta+\delta(x)\grad\delta(x),\ \theta\in\partial\Omega$:
up to a rotation and a translation,
we can suppose that $\theta=0$ and $\grad\delta(x)=e_N$.
By compactness, $\partial\Omega$ can be covered by a finite number of open portions $\Gamma_j\subset\partial\Omega$, $j=1,\ldots,n$.
For any $j=1,\ldots,n$, the function $\eta\mapsto\dist(\eta,\partial\Omega\setminus\Gamma_j)$ is continuous in $\partial\Omega$
and so is $\eta\mapsto\max_j\dist(\eta,\partial\Omega\setminus\Gamma_j)$: there is a point $\eta^*\in\partial\Omega$
where $\eta\mapsto\max_j\dist(\eta,\partial\Omega\setminus\Gamma_j)$ attains its minimum. 
Such a minimum cannot be $0$ because $\eta^*$ belongs at least to one of the $\Gamma_j$.
This implies that for any $\eta\in\partial\Omega$ there exists $i\in\{1,\ldots,n\}$ for which
\begin{equation}\label{00}
\dist(\eta,\partial\Omega\setminus\Gamma_i)\geq \max_j\dist(\eta^*,\partial\Omega\setminus\Gamma_j)
\end{equation}
and this in particular implies $\eta\in\Gamma_i$. 
Let $\Gamma$ be a neighbourhood of $0$ on $\partial\Omega$ chosen from $\{\Gamma_j\}_{j=1}^n$ and
for which \eqref{00} is fulfilled. Let also
\[
\omega=\{y\in\R^N:y=\eta+\delta(y)\grad\delta(y),\ \eta\in\Gamma\}. 
\]
The set $\Gamma\subset\partial\Omega$ can be described via as the graph of a $C^2$ function
\begin{eqnarray*}
\gamma:B'_r(0)\subseteq\R^{N-1} & \longrightarrow & \R \\
\eta' & \longmapsto & \gamma(\eta')\quad\text{ s.t }\eta=(\eta',\gamma(\eta'))\in\Gamma
\end{eqnarray*}
satisfying $\gamma(0)=|\grad\gamma(0)|=0$.

The integration on $(\Omega_2\cup\Omega_3)\setminus\omega$
is lower order with respect to the one on $(\Omega_2\cup\Omega_3)\cap\omega$
since in the latter we have the singular in $x$ to deal with,
while in the former $|x-y|$ is a quantity bounded below independently on $x$.
Indeed when $y\in(\Omega_2\cup\Omega_3)\setminus\omega$
\[
|x-y|\geq|\eta+\delta(y)\grad\delta(y)|-\delta(x)\geq|\eta|-\delta(y)-\delta(x)\geq\dist(0,\partial\Omega\setminus\Gamma)-\frac{5}{2}\delta(x)
\]
where $\delta(x)$ is small and the first addend is bounded uniformly in $x$ by \eqref{00}.

We are left with:
\[
C\cdot PV\int_{\Omega_2\cap\omega}\frac{\psi(\delta(x)^s)-\psi(\delta(y)^s)}{|x-y|^{N+2s}}\:dy+
C\int_{\Omega_3\cap\omega}\frac{\psi(\delta(x)^s)-\psi(\delta(y)^s)}{|x-y|^{N+2s}}\:dy.
\]
Let us split the remainder of the estimate in steps.

{\it First step: the distance between $x$ and $y$.} We claim that there exists $c>0$ such that
\begin{equation}\label{distxy}
\begin{aligned}
& |x-y|^2\ \geq\ c\left(|\delta(x)-\delta(y)|^2+|\eta'|^2\right),\quad y\in(\Omega_2\cup\Omega_3)\cap\omega,\\ 
& y=\eta+\delta(y)\grad\delta(y),\ \eta=(\eta',\gamma(\eta')).
\end{aligned}
\end{equation}
Since in our set of coordinates $x=\delta(x)e_N$, we can write
\begin{multline*}
|x-y|^2=|\delta(x)e_N-\delta(y)e_N+\delta(y)e_N-y_Ne_N-y'|^2\ \geq \\
\geq\ |\delta(x)-\delta(y)|^2-2|\delta(x)-\delta(y)|\cdot|\delta(y)-y_N|+|\delta(y)-y_N|^2+|y'|^2.
\end{multline*}
We concentrate our attention on $|\delta(y)-y_N|$: the idea is to show
that this is a small quantity; indeed, in the particular case when $\Gamma$ 
lies on the hyperplane $y_N=0$, this quantity is actually zero. As in the definition
of $\omega$, we let $y=\eta+\delta(y)\grad\delta(y)$ and $\eta=(\eta',\gamma(\eta'))\in\Gamma$:
thus $y_N=\gamma(\eta')+\delta(y)\langle\grad\delta(y),e_N\rangle$ where 
$\grad\delta(y)$ is the inward unit normal to $\partial\Omega$ at the point $\eta$, so that
\[
\grad\delta(y)=\frac{(-\grad\gamma(\eta'),1)}{\sqrt{|\grad\gamma(\eta')|^2+1}}
\]
and
\begin{equation}\label{yeta}
y'=\eta'-\frac{\delta(y)\,\grad\gamma(\eta')}{\sqrt{|\grad\gamma(\eta')|^2+1}},
\qquad y_N=\gamma(\eta')+\frac{\delta(y)}{\sqrt{|\grad\gamma(\eta')|^2+1}}.
\end{equation}
Now, since $y\in\Omega_2\cup\Omega_3$, it holds $|\delta(x)-\delta(y)|\leq\delta(x)$ and
\[
|\delta(y)-y_N|\leq |\gamma(\eta')|+\delta(y)\left(1-\frac{1}{\sqrt{|\grad\gamma(\eta')|^2+1}}\right)
\leq C|\eta'|^2+2C\delta(x)|\eta'|^2
\]
where, in this case, $C=\|\gamma\|_{C^2(B_r)}$ depends only on the geometry of $\partial\Omega$ and not on $x$.
By \eqref{yeta}, we have 
\[
|\eta'|^2\leq 2|y'|^2+\frac{2\delta(y)^2\,|\grad\gamma(\eta')|^2}{|\grad\gamma(\eta')|^2+1}
\leq 2|y'|^2+2C\delta(y)^2\,|\eta'|^2\leq 2|y'|^2+C\delta(x)^2\,|\eta'|^2,
\]
so that $|\eta'|^2\leq C|y'|^2$ when $\delta(x)$ is small enough.
Finally
\begin{multline*}
|x-y|^2\geq |\delta(x)-\delta(y)|^2+|y'|^2-2|\delta(x)-\delta(y)|\cdot|\delta(y)-y_N|\ \geq \\
\geq\ |\delta(x)-\delta(y)|^2+c|\eta'|^2-2C\delta(x)|\eta'|^2,
\end{multline*}
where, again, $C=\|\gamma\|_{C^2(B_r)}$ and \eqref{distxy} is proved provided $x$ is close enough to $\partial\Omega$.

{\it Second step: integration on $\Omega_2\cap\omega$.} Using the regularity of $\psi$ and $\delta$ we write
\[
\psi(\delta(x)^s)-\psi(\delta(y)^s)\geq \grad(\psi\circ\delta^s)(x)\cdot(x-y)
-\|D^2(\psi\circ\delta^s)\|_{L^\infty(\Omega_2\cap\omega)}|x-y|^2
\]
where
\[
D^2(\psi\circ\delta^s)=\frac{s\psi'(\delta^s)}{\delta^{1-s}}\,D^2\delta
+\frac{s^2\,\psi''(\delta^s)}{\delta^{2-2s}}\,\grad\delta\otimes\grad\delta
+\frac{s(s-1)\,\psi'(\delta^s)}{\delta^{2-s}}\,\grad\delta\otimes\grad\delta
\]
so that 
\[
\|D^2(\psi\circ\delta^s)\|_{L^\infty(\Omega_2\cap\omega)}\ \leq\ 
C\left\|\frac{\psi'(\delta^s)}{\delta^{1-s}}\right\|_{L^\infty(\Omega_2\cap\omega)}
+C\left\|\frac{\psi''(\delta^s)}{\delta^{2-2s}}\right\|_{L^\infty(\Omega_2\cap\omega)}
+C\left\|\frac{\psi'(\delta^s)}{\delta^{2-s}}\right\|_{L^\infty(\Omega_2\cap\omega)}.
\]
By definition of $\Omega_2$ and by \eqref{mono-psi} we can control the sup-norm by 
the value at $x$:
\begin{multline*}
\|D^2(\psi\circ\delta^s)\|_{L^\infty(\Omega_2\cap\omega)}\leq
C\,\frac{|\psi'(\delta(x)^s)|}{\delta(x)^{1-s}}
+C\,\frac{\psi''(\delta(x)^s)}{\delta(x)^{2-2s}}
+C\,\frac{|\psi'(\delta(x)^s)|}{\delta(x)^{2-s}}\ \leq \\
\leq\ C\,\frac{\psi''(\delta(x)^s)}{\delta(x)^{2-2s}}
+C\,\frac{|\psi'(\delta(x)^s)|}{\delta(x)^{2-s}}
\end{multline*}
and using equation \eqref{mono-psi2} we finally get 
\[
\|D^2(\psi\circ\delta^s)\|_{L^\infty(\Omega_2\cap\omega)}\leq 
C\,\frac{\psi''(\delta(x)^s)}{\delta(x)^{2-2s}}.
\]
If we now retrieve the whole integral and exploit \eqref{distxy}
\begin{multline*}
PV\int_{\Omega_2\cap\omega}\frac{\psi(\delta(x)^s)-\psi(\delta(y)^s)}{|x-y|^{N+2s}}\:dy\ \geq\ 
-C \frac{\psi''(\delta(x)^s)}{\delta(x)^{2-2s}}
\int_{\Omega_2\cap\omega}\frac{dy}{|x-y|^{N+2s-2}}\ \geq\\
\geq\ -C \frac{\psi''(\delta(x)^s)}{\delta(x)^{2-2s}}
\int_{\Omega_2\cap\omega}\frac{dy}{\left(|\delta(x)-\delta(y)|^2+|\eta|^2\right)^{(N+2s-2)/2}}.
\end{multline*}
We focus our attention on the integral on the right-hand side: by the coarea formula
\begin{align*}
& \int_{\Omega_2\cap\omega}\frac{dy}{\left(|\delta(x)-\delta(y)|^2+|\eta|^2\right)^{(N+2s-2)/2}}\ = \\
& =\ \int_{\delta(x)/2}^{3\delta(x)/2}dt \int_{\{\delta(y)=t\}\cap\omega}
\frac{d\mathcal{H}^{N-1}(\eta)}{\left(|\delta(x)-t|^2+|\eta|^2\right)^{(N+2s-2)/2}}\\
& \leq\ C\int_{\delta(x)/2}^{3\delta(x)/2}dt \int_{B_r}\frac{d\eta'}{\left(|\delta(x)-t|^2+|\eta'|^2\right)^{(N+2s-2)/2}} \\
& \leq\ C\int_{\delta(x)/2}^{3\delta(x)/2}dt \int_0^r\frac{\rho^{N-2}}{\left(|\delta(x)-t|^2+\rho^2\right)^{(N+2s-2)/2}}\;d\rho \\
& \leq\ C\int_{\delta(x)/2}^{3\delta(x)/2}dt \int_0^r\frac{\rho}{\left(|\delta(x)-t|^2+\rho^2\right)^{(2s+1)/2}}d\rho\ \leq  
\ C\int_{\delta(x)/2}^{3\delta(x)/2}\frac{dt}{|t-\delta(x)|^{2s-1}}.
\end{align*} 
We can retrieve now the chain of inequalities we stopped above:
\[
\int_{\Omega_3\cap\omega}\frac{\psi(\delta(x)^s)-\psi(\delta(y)^s)}{|x-y|^{N+2s}}\:dy\ \geq\ 
-C\frac{\psi''(\delta(x)^s)}{\delta(x)^{2-2s}}
\int_{\delta(x)/2}^{3\delta(x)/2}\frac{dt}{|\delta(x)-t|^{-1+2s}}
\geq\ -C\,\psi''(\delta(x)^s).
\]

{\it Third step: integration on $\Omega_3\cap\omega$.}
We use \eqref{distxy} once again:
\begin{align*}
& \int_{\Omega_3\cap\omega}\frac{\psi(\delta(x)^s)-\psi(\delta(y)^s)}{|x-y|^{N+2s}}\:dy\ \geq\\
& \geq\ -\int_{\Omega_3\cap\omega}\frac{\psi(\delta(y)^s)}{|x-y|^{N+2s}}\:dy\ 
\geq\ -C\int_{\Omega_3\cap\omega}\frac{\psi(\delta(y)^s)}{\left(|\delta(x)-\delta(y)|^2+|\eta'|^2\right)^\frac{N+2s}2}\:dy\\
& \geq\ -C\int_0^{\delta(x)/2}\frac{\psi(t^s)}{(\delta(x)-t)^{1+2s}}\;dt\ \geq
\ -\frac{C}{\delta(x)^{1+2s}}\int_0^{\delta(x)/2}\psi(t^s)\;dt.
\end{align*}
The term we have obtained is of the same order of $\delta(x)^{-2s}\psi(\delta(x)^s)$, by \eqref{psitech}:
\[
\int_0^{\delta(x)/2}\psi(t^s)\;dt\lessgtr\int_0^{\delta(x)/2}t^s\psi'(t^s)\;dt=
\frac{\delta(x)}{2s}\psi\left(\frac{\delta(x)^s}{2^s}\right)-\frac1s\int_0^{\delta(x)/2}\psi(t^s)\;dt
\]
so that
\begin{equation}\label{03}
\int_0^{\delta(x)/2}\psi(t^s)\;dt\lessgtr\delta(x)\psi(\delta(x)^s)=\delta(x)^{1+2s}\cdot\frac{\psi(\delta(x)^s)}{\delta(x)^{2s}}.
\end{equation}
Recall now that $\psi(\delta(x)^s)\delta(x)^{-2s}$ is in turn of the same size of $\psi''(\delta(x)^s)$ by \eqref{mono-psi2}.

{\it Fourth step: the outside integral in \eqref{split}.} 
We focus now our attention on
\[
\int_{\C\Omega}\frac{\psi(\delta(y)^s)-\psi(\delta(x)^s)}{|x-y|^{N+2s}}\;dy.
\]
First, by using the monotonicity of $\psi$, we write
\begin{multline*}
\int_{\C\Omega}\frac{\psi(\delta(y)^s)-\psi(\delta(x)^s)}{|x-y|^{N+2s}}\;dy\ \leq\ 
\int_{\{y\in\C\Omega:\delta(y)<\delta(x)\}\cap\omega}\frac{\psi(\delta(y)^s)-\psi(\delta(x)^s)}{|x-y|^{N+2s}}\;dy\ +\\
+\ \int_{\{y\in\C\Omega:\delta(y)<\delta(x)\}\setminus\omega}\frac{\psi(\delta(y)^s)-\psi(\delta(x)^s)}{|x-y|^{N+2s}}\;dy.
\end{multline*}
The second integral gives
\[
\int_{\{y\in\C\Omega:\delta(y)<\delta(x)\}\setminus\omega}\frac{\psi(\delta(y)^s)-\psi(\delta(x)^s)}{|x-y|^{N+2s}}\;dy
\ \leq\ C\|\psi(\delta^s)\|_{L^1(\R^N)}
\]
because the distance between $x$ and $y$ is bounded there.
Again we point out that
\begin{align*}
& \int_{\{\delta(y)<\delta(x)\}\cap\omega}\frac{\psi(\delta(y)^s)-\psi(\delta(x)^s)}{|x-y|^{N+2s}}\;dy
\leq C\int_0^{\delta(x)}\frac{\psi(t^s)-\psi(\delta(x)^s)}{|\delta(x)+t|^{1+2s}}\;dt\ \leq \\
& \leq\ C\int_0^{\delta(x)/2}\frac{\psi(t^s)}{|\delta(x)+t|^{1+2s}}\;dt
+C\int_{\delta(x)/2}^{\delta(x)}\frac{\psi(t^s)}{|\delta(x)+t|^{1+2s}}\;dt \\
& \leq\ C\delta(x)^{-1-2s}\int_0^{\delta(x)/2}\psi(t^s)\;dt
+C\psi\left(\frac{\delta(x)^s}{2^s}\right)\int_{\delta(x)/2}^{\delta(x)}(\delta(x)+t)^{-1-2s}\;dt \\
& \leq\ C\delta(x)^{-1-2s}\int_0^{\delta(x)}\psi(t^s)\;dt
+C\psi(\delta(x)^s)\delta(x)^{-2s}
\end{align*}
which is of the order of $\psi''(\delta(x)^s)$, by \eqref{03} and \eqref{mono-psi2}.
\smallskip

{\it Conclusion.} We have proved that for $\delta(x)$ sufficiently small
\[
\Ds U(x)\geq -C\psi''(\delta(x)^s).
\]
Recall now that $\psi''(\delta^s)=f(\psi\circ\delta^s)/2$ and $U=\psi\circ\delta^s$ in $\Omega$,
so that
\[
\Ds U\ \geq\ -Cf(U)
\]
holds in a neighbourhood of $\partial\Omega$.
\end{proof}

\section{existence}\label{exist}

\begin{lem}\label{EU} If the nonlinear term $f$ satisfies the growth condition \eqref{E}
then the function $U$ defined in \eqref{U} satisfies
\[
\lim_{x\to\partial\Omega}\delta(x)^{1-s}U(x) =\ +\infty.
\]
\end{lem}
\begin{proof} Write
\[
\liminf_{x\to\partial\Omega}\delta(x)^{1-s}\psi(\delta(x)^{s})\ =
\ \liminf_{u\uparrow+\infty}u\,\phi(u)^{\frac{1-s}{s}}.
\]
Such a limit is $+\infty$ if and only if 
\[
\liminf_{u\uparrow+\infty}u^{\frac{s}{1-s}}\int_u^{+\infty}\frac{dt}{\sqrt{2F(t)}}=+\infty.
\]
If we use the L'H\^opital's rule to
\[
\frac{\displaystyle\int_u^{+\infty}\dfrac{dt}{\sqrt{2F(t)}}}{u^{-s/(1-s)}}
\] 
we get the ratio $u^{\frac{1}{1-s}}/\sqrt{2F(u)}$
and applying once again the L'H\^opital's rule, this time to
$u^{\frac{2}{1-s}}/F(u)$,
we get $u^{\frac{1+s}{1-s}}/f(u)$
which diverges by hypothesis \eqref{E}. Indeed, since $f$ is increasing,
\[
u^{-\frac{1+s}{1-s}}f(u)= f(u)\cdot\frac{1-s}{1+s}\int_u^{+\infty}t^{-2/(1-s)}dt\leq \int_u^{+\infty}f(t)t^{-2/(1-s)}dt
\xrightarrow[u\uparrow+\infty]{}0.
\] 
\end{proof}

\begin{lem}\label{mod-supersol} Let $v:\R^N\to\R$ be a function which satisfies
$\Ds v\in C(\Omega)$.
If there exist $C,\,\delta_0>0$ such that
\[
\Ds v \ \geq\ -Cf(v) \qquad \hbox{ in }\Omega_{\delta_0}:=\{x\in\Omega:\delta(x)<\delta_0\}
\]
then there exists $\overline{u}\geq v$ such that $\Ds\overline{u}\geq-f(\overline{u})$ throughout $\Omega$. 
\end{lem}
\begin{proof}
Define $\xi:\R^N\to\R$ as the solution to
\begin{equation}\label{xi}
\left\lbrace\begin{aligned}
\Ds\xi &=1 & \hbox{ in }\Omega \\
\xi &=0 & \hbox{ in }\C\Omega \\
E\xi &=0 & \hbox{ on }\partial\Omega
\end{aligned}\right.
\end{equation}
and consider $\overline u=\mu v+\lambda\xi$, where $\mu,\lambda\geq1$. If $C\in(0,1]$ then 
$\Ds v\geq -f(v)$ in $\Omega_{\delta_0}$, so choose $\mu=1$. If $C>1$, then choose $\mu=C^{1/M}>1$
in order to have in $\Omega_{\delta_0}$
\begin{multline*}
\Ds\overline u+f(\overline u)=\Ds (\mu v+\lambda\xi)+f(\mu v+\lambda\xi)\geq -\mu\,Cf(v)+f(\mu v)\ \geq \\
\geq\ (-\mu\,C+\mu^{1+M})f(v)=0
\end{multline*}
where we have heavily used the positivity of $\xi$ and \eqref{mono-f}.
Now, since $\Ds v\in C(\overline{\Omega\setminus\Omega_{\delta_0}})$ we can choose $\lambda=\mu\|\Ds v\|_{L^\infty(\Omega\setminus\Omega_{\delta_0})}$
so that also in $\Omega\setminus\Omega_{\delta_0}$
\[
\Ds \overline u=\Ds(\mu v+\lambda\xi)=\mu\Ds v+\lambda\geq 0\geq -f(\overline u).
\]
\end{proof}

Collecting the information so far, 
we have that Lemmas \ref{U-L1}, \ref{mod-supersol} and \ref{EU} 
fully prove the following theorems.
\begin{theo}\label{U-super}
If the nonlinear term $f$ satisfies \eqref{f1},
\eqref{tech}, \eqref{L1} and the growth condition \eqref{E},
then there exists a function $\overline{u}$ supersolution to \eqref{problema}.
Moreover
\[
\overline{u}\ =\ \mu \psi(\delta^s)+\lambda\xi,\qquad\hbox{in }\Omega
\]
where $\xi$ is the solution of \eqref{xi}, $\lambda>0$, $\mu=\max\{1,C^{1/M}\}$ where $C>0$ is the constant in
\eqref{impo-eq} and $M>0$ the one in \eqref{tech}.
\end{theo}

\begin{theo}\label{U-gsuper}
If the nonlinear term $f$ satisfies \eqref{f1},
\eqref{tech}, \eqref{L1}, then there exists a function $\overline{u}$ supersolution to \eqref{gproblem}.
Moreover
\[
\overline{u}\ =\ \mu \psi(\delta^s)+\lambda\xi,\qquad\hbox{in }\Omega
\]
where $\xi$ is the solution of \eqref{xi}, $\lambda>0$, $\mu=\max\{1,C^{1/M}\}$ where $C>0$ is the constant in
\eqref{impo-eq} and $M>0$ the one in \eqref{tech}.
\end{theo}

\subsection{Proof of Theorem \ref{main1}}

Theorems \ref{U-super} bears as a consequence the following.
Build the sequence of weak solutions to problems
\begin{equation}\label{kproblem}
\left\lbrace\begin{aligned}
\Ds u_k\  &=\ -f(u_k) & & \hbox{ in }\Omega \\
u_k\ &=\ 0 & & \hbox{ in }\C\Omega \\
E\,u_k\ &=\ k & & \hbox{ on }\partial\Omega, & k\in\N.
\end{aligned}\right.
\end{equation}
The existence of any $u_k$ can be proved as in Theorem \ref{-SIGN}, in view of hypothesis \eqref{E},
since it implies
\[
\int_0^{\delta_0}f(\delta^{s-1})\delta^s\;d\delta\ <\ +\infty.
\]
We need here auxiliary regularity results that will be proved in  the next 
chapter. Using Lemma \ref{regularity} and Corollary \ref{weak-point}, we deduce
that $u_k\in C^{2s+\eps}(\Omega)$ and 
\[
\Ds u_k(x)\ =\ -f(u_k(x)),\qquad\hbox{for any }x\in\Omega
\]
in a pointwise sense.

\paragraph*{\it Step 0: $\{u_k\}_{k\in\N}$ is increasing with $k$.}
Consider
\[
\left\lbrace\begin{aligned}
\Ds(u_{k+1}-u_k) &=\ f(u_k)-f(u_{k+1}) & \hbox{ in }\Omega \\
u_{k+1}-u_k &=\ 0 & \hbox{ in }\C\Omega \\
E(u_{k+1}-u_k) &=\ 1 & \hbox{ on }\partial\Omega.
\end{aligned}\right.
\]
Call $\Omega^+_k:=\{u_k>u_{k+1}\}$ which satisfies $\overline{\Omega^+_k}\subset\Omega$ as 
it is implied by the singular boundary trace. In $\overline{\Omega^+_k}\subset\Omega$
the difference $u_k-u_{k+1}$ attains its maximum at some point $x_k$:
at this point we must have $\Ds(u_k-u_{k+1})(x_k)>0$ because it 
is also a global maximum. But at the same time $f(u_{k+1}(x_k))-f(u_k(x_k))\leq 0$,
because $f$ is increasing, proving that $\Omega^+_k$ must be empty.

\paragraph*{\it Step 1: $\{u_k\}_{k\in\N}$ has a pointwise limit.}

Any $u_k$ lies below $\overline{u}$: since $E(\overline{u}-u_k)>0$,
there exists a compact set $U_k\subset\Omega$ for which $u_k\leq\overline{u}$ in $\Omega\setminus U_k$.
Inside $U_k$ we have
\[
\left\lbrace\begin{aligned}
\Ds(\overline{u}-u_k) &\geq\ f(u_k)-f(\overline{u}) & \hbox{ in }U_k \\
\overline{u}-u_k &\geq\ 0 & \hbox{ in }\C U_k \\
E_{\partial U_k}(\overline{u}-u_k) &=\ 0 & \hbox{ on }\partial U_k
\end{aligned}\right.
\]
in a pointwise sense, where $u_k,\overline{u}\in C^{2s+\eps}(\overline{U_k})$.
Then an argument similar to the one in {\it Step 0} yields $u_k\leq\overline{u}$ also in $U_k$.

Finally, $\{u_k\}_{k\in\N}$ is increasing and pointwisely bounded by $\overline{u}$ throughout $\Omega$.
This entails that
\[
u(x)\ :=\ \lim_{k\uparrow+\infty}u_k(x)
\]
is well-defined and finite for any $x\in\Omega$. Also, $0\leq u\leq\overline{u}$ in $\Omega$
and since $\overline{u}\in L^1(\Omega)$ by Lemma \ref{U-L1}, then $u\in L^1(\Omega)$.

\paragraph*{\it Step 2: $u\in C(\Omega)$.}
Fix any compact $D\subset\Omega$ and choose a $c>0$ for which $\delta(x)>2c$ for any $x\in D$.
Let $\widetilde{D}:=\{y\in\Omega:\delta(y)>c\}$. For any $k,\,j\in\N$ it holds
\[
\Ds(u_{k+j}-u_k)=f(u_k)-f(u_{k+j})\leq 0, \qquad\hbox{ in }\widetilde{D}
\]
and therefore
\[
0\leq u_{k+j}(x)-u_k(x)\leq\int_{\C\widetilde{D}}P_{\widetilde{D}}(x,y)\left[u_{k+j}(y)-u_k(y)\right]dy
\]
where $P_{\widetilde{D}}(x,y)$ is the Poisson kernel associated to $\widetilde{D}$, which satisfies (see \cite[Theorem 2.10]{chen})
\[
P_{\widetilde{D}}(x,y)\ \leq\ \frac{C\,{\delta_{\widetilde{D}}(x)}^s}{{\delta_{\widetilde{D}}(y)}^s\,|x-y|^N},
\qquad x\in\widetilde{D},\ y\in\C\widetilde{D}.
\]
When $x\in D\subset\widetilde{D}$ one has $|x-y|>c$ for any $y\in\C\widetilde{D}$, and therefore
\[
0\leq u_{k+j}(x)-u_k(x)\leq C\int_{\C\widetilde{D}}\frac{u_{k+j}(y)-u_k(y)}{{\delta_{\widetilde{D}}(y)}^s}\;dy
\leq C\int_{\C\widetilde{D}}\frac{u(y)-u_k(y)}{{\delta_{\widetilde{D}}(y)}^s}\;dy
\]
where the last integral converges by Monotone Convergence to $0$ independently on $x$.
This means the convergence $u_k\to u$ is uniform on compact subsets 
and since $\{u_k\}_{k\in\N}\subset C(\Omega)$ (cf. the proof of Theorem \ref{-SIGN}, Section \ref{proof-SIGN}),
then also $u\in C(\Omega)$.
 
\paragraph*{\it Step 3: $u\in C^2(\Omega)$.}
This is a standard bootstrap argument using the elliptic regularity in \cite[Propositions 2.8 and 2.9]{silvestre}.

\paragraph*{\it Step 4: $u$ solves \eqref{problema} in a pointwise sense.} The function $\Ds u(x)$ is well-defined 
for any $x\in\Omega$ because $u\in C^2(\Omega)\cap L^1(\R^N)$. Using the regularity results in \cite[Propositions 2.8 and 2.9]{silvestre},
we have
\[
\Ds u=\lim_{k\uparrow+\infty}\Ds u_k=-\lim_{k\uparrow+\infty}f(u_k)=-f(u).
\]
Also, $\delta^{1-s}u\geq\delta^{1-s}u_k$ holds in $\Omega$ for any $k\in\N$. Therefore, for any $k\in\N$,
\[
\liminf_{x\to\partial\Omega}\delta(x)^{1-s}u(x)\geq\lim_{x\to\partial\Omega}\delta(x)^{1-s}u_k(x)\geq \lambda Eu_k=\lambda k
\]
for some constant $\lambda>0$ depending on $\Omega$ and not on $k$. This entails
\[
\lim_{x\to\partial\Omega}\delta(x)^{1-s}u(x)\ =\ +\infty
\]
and completes the proof of Theorem \ref{main1}.

\begin{rmk}\rm The proof of Theorem \ref{main2} is alike. Indeed,
in the same way, the sequence of solutions to
\begin{equation}\label{gkproblem}
\left\lbrace\begin{aligned}
\Ds u_k\  &=\ -f(u_k) & & \hbox{ in }\Omega \\
u_k\ &=\ g_k:=\min\{k,g\} &  & \hbox{ in }\C\Omega,\ k\in\N \\
E\,u_k\ &=\ 0 & & \hbox{ on }\partial\Omega, & 
\end{aligned}\right.
\end{equation}
approaches a solution of problem \eqref{gproblem}
which lies below the supersolution provided by Theorem \ref{U-gsuper}.
\end{rmk}

\section{examples}

\begin{ex}[\bf Power nonlinearity]\label{pow}\rm
Let us consider $f(t)=t^p$, for $p>1$.
In this case
\[
\frac{tf'(t)}{f(t)}=p.
\]
The function $\phi$ reads as (cf. \eqref{phi})
\[
\phi(u)=\int_u^{+\infty}\sqrt{\frac{p+1}{2}}\:t^{-\frac{p+1}{2}}\;dt=\sqrt{\frac{2(p+1)}{p-1}}\:u^{\frac{1-p}{2}}
\]
and hypothesis \eqref{L1} can then be written
\[
\int_u^{+\infty}\eta^{\frac{1-p}{2s}}\;d\eta\ <\ +\infty
\]
that holds if and only if $p>1+2s$.
On the other hand hypothesis \eqref{E} becomes
\[
p-\frac2{1-s}<-1,\qquad\hbox{i.e.}\ \ p<\frac{1+s}{1-s}
\]
\hfill$\blacklozenge$
\end{ex}

In the next two examples we look at the two critical cases in the power-like nonlinearity,
adding a logarithmic weight.

\begin{ex}[\bf Lower critical case for powers]\rm
We consider here $f(t)=t^{1+2s}\ln^\alpha(1+t)$, $\alpha>0$.
In this case 
\[
\frac{tf'(t)}{f(t)}=\frac{(1+2s)f(t)+\frac{\alpha t f(t)}{(1+t)\ln(1+t)}}{f(t)}=1+2s+\frac{\alpha t}{(1+t)\ln(1+t)}.
\]
Condition \eqref{L1-bis} turns into
\[
\int_u^{+\infty}\left(\frac{t}{t^{1+2s}\ln^\alpha(1+t)}\right)^{1/(2s)}dt=
\int_u^{+\infty}\frac{dt}{t\ln^{\alpha/(2s)}(1+t)}\ <\ +\infty
\]
which is fulfilled only for $\alpha>2s$. Also, hypothesis \eqref{E} becomes
\[
\int_{t_0}^{+\infty}t^{1+2s-2/(1-s)}\ln^\alpha(1+t)\;dt\ <\ +\infty
\]
which is satisfied by any $\alpha>0$ since $(1+2s)(1-s)-2<s-1$.
\hfill$\blacklozenge$
\end{ex}

\begin{ex}[\bf Upper critical case for powers]\rm
We consider here $f(t)=t^{\frac{1+s}{1-s}}\ln^{-\beta}(1+t)$, $\beta>0$.
In this case 
\[
\frac{tf'(t)}{f(t)}=\frac{\frac{1+s}{1-s}f(t)-\frac{\beta t f(t)}{(1+t)\ln t}}{f(t)}=\frac{1+s}{1-s}-\frac{\beta t}{(1+t)\ln(1+t)}
\]
Hypothesis \eqref{L1-bis} turns into
\[
\int_u^{+\infty}\left(\frac{t\ln^\beta(1+t)}{t^{(1+s)/(1-s)}}\right)^{1/(2s)}\;dt=
\int_u^{+\infty}\frac{\ln^{\beta/(2s)}(1+t)}{t^{1/(1-s)}}\;dt
\ <\ +\infty
\]
which is fulfilled for any $\beta>0$. Also, hypothesis \eqref{E} becomes
\[
\int_{t_0}^{+\infty}t^{-1}\ln^{-\beta}(1+t)\;dt\ <\ +\infty
\]
which is satisfied by any $\beta>1$.
\hfill$\blacklozenge$
\end{ex}

\section{comments on hypotheses ({\rm \ref{L1}}) and ({\rm \ref{E}})}

The present section will be devoted to
the explanation of the difficulties of problem \eqref{problema} when 
one of hypotheses \eqref{E} or \eqref{L1} fails.

We recall that hypothesis \eqref{L1} is the one needed to guarantee that
that the function $U$ defined in \eqref{U} belongs to $L^1(\Omega)$ (cf. Lemma \ref{U-L1}), 
while \eqref{E} has been used to show that $EU=+\infty$ (cf. Lemma \ref{EU}).
These features are essential in proving that $U$ is a supersolution to \eqref{problema}.
Roughly speaking, condition \eqref{L1} gives a lower growth condition at infinity of the nonlinear term $f$:
in the power case $f(t)=t^p$ it corresponds to $p>1+2s$ (cf. Example \ref{pow}).
On the other hand hypothesis \eqref{E} gives an upper growth condition.
Note that in case \eqref{E} fails, we have two issues: not only the candidate supersolution
$U$ does not satisfy $EU=+\infty$, but also the approximate problem \eqref{kproblem} does not have any solution.

\begin{lem}\label{nonexist}
In case \eqref{E} fails, problem 
\begin{equation}
\left\lbrace\begin{aligned}
\Ds u_1\  &=\ -f(u_1) & & \hbox{ in }\Omega \\
u_1\ &=\ 0 & & \hbox{ in }\C\Omega \\
E\,u_1\ &=\ 1 & & \hbox{ on }\partial\Omega
\end{aligned}\right.
\end{equation} 
does not admit any weak or pointwise solution.
\end{lem}
\begin{proof}
In both cases the solution would satisfy $u_1\geq c\delta^{s-1}$ in $\Omega$, for some $c>0$.
If $u_1$ was a weak solution then for any $\phi\in\T(\Omega)$
\[
\int_\Omega u_1\,\Ds\phi+\int_\Omega f(u_1)\phi\ =\ \int_{\partial\Omega}D_s\phi
\]
where 
\[
\int_\Omega f(u_1)\phi \geq C\int_\Omega f(c\,\delta^{s-1})\delta^s = +\infty
\]
because \eqref{E} does not hold, a contradiction.

If $u_1$ was a pointwise solution, then by Lemma \ref{point-weak} it 
would be a weak solution on any subdomain $D\subset\overline{D}\subset\Omega$.
Therefore
\[
u_1(x)\ =\ -\int_D G_D(x,y)\,f(u_1(y))\;dy + \int_{\C D} P_D(x,y)\,u_1(y)\;dy.
\]
If $u_0$ denotes the $s$-harmonic function induced by $Eu=1$, then $u_1\leq u_0$ in $\Omega$
and
\[
u_1(x)\ \leq\ -\int_D G_D(x,y)\,f(u_1(y))\;dy + \int_{\C D} P_D(x,y)\,u_0(y)\;dy
=-\int_D G_D(x,y)\,f(u_1(y))\;dy + u_0(x).
\]
Fix $x\in\Omega$.
Letting now $D\nearrow\Omega$ we have that $G_D(x,y)\uparrow G_\Omega(x,y)$
and 
\[
\int_\Omega G_\Omega(x,y)\,f(u_1(y))\;dy\ \geq\ 
c\delta(x)^s\int_{\{2\delta(y)<\delta(x)\}} \delta(y)^s\,f(u_1(y))\;dy =+\infty
\]
because \eqref{E} does not hold, a contradiction.
\end{proof}

\chapter{Final remarks}

In this Chapter we would like to point out some elements
that may risk to be unclear if left implicit.
In the first paragraph we discuss the relation between pointwise solutions
and weak $L^1$ solutions. The second paragraph states 
in what sense the singular boundary trace is attained 
by weak solutions, when the data are not smooth.
The third and paragraph deals with the definition
of weak $L^1$ solution given by Chen and V\'eron \cite{chen-veron},
which amounts to be equivalent to Definition \ref{weakdefiintro}.

\section{pointwise solutions vs. weak \texorpdfstring{$L^1$}{L1} solutions}

\begin{lem}\label{weak-distr} Solutions provided by Theorem \ref{existence-weak2} are
distributional solutions, meaning that for any $\psi \in C^\infty_c(\Omega)$
it holds
\[
\int_\Omega u\Ds\psi\ =\ \int_\Omega\psi\;d\lambda-\int_{\C\Omega}\Ds\psi\;d\mu.
\]
\end{lem}
\begin{proof}
Consider a sequence ${\{\eta_k\}}_{k\in\N}$ of bump functions, such that
$0\leq\ldots\leq\eta_k\leq\eta_{k+1}\leq\ldots\leq 1$ and $\eta_k\uparrow\chi_\Omega$.
Define $f_k:=\eta_k\Ds\psi\in C^\infty_c(\Omega)$ and the corresponding $\phi_k\in \T(\Omega)$.
Then we know that
\[
\int_\Omega u\,f_k\ =\ \int_\Omega\phi_k\;d\lambda-\int_{\C\Omega}\Ds\phi_k\;d\mu+\int_{\partial\Omega}D_s\phi_k\:d\nu.
\]
Since $|f_k|\leq|\Ds\psi|<C$ in $\Omega$, the left-hand side converges to $\int_\Omega u\Ds\psi$ by dominated convergence.
In the same way $|\phi_k|\leq \int_\Omega G_\Omega(\cdot,y)|f_k(y)|dy\leq C\delta^s$ and
\[
|\Ds\phi_k(x)|\leq C_{N,s}\int_\Omega\frac{|\phi_k(y)|}{{|x-y|}^{N+2s}}\leq
C\int_\Omega\frac{\delta(y)^s}{{|x-y|}^{N+2s}}\leq C\delta(x)^{-s}\min\{1,\delta(x)^{-N-s}\},
\]
thus we have the convergence on the first and the second addend on the left-hand side.
Finally, we have that $\phi_k\,\delta^{-s}$ converges uniformly towards $\psi\,\delta^{-s}$ in $\Omega$,
so that $D_s\psi_k\to D_s\psi=0$ as $k\uparrow+\infty$; this, along with
\[
\left|D_s\phi_k(z)\right|=\left|\int_\Omega M_\Omega(x,z)\,f_k(x)\:dx\right|\leq
\int_\Omega M_\Omega(x,z)\,|\Ds\psi(x)|\:dx< C(N,\Omega,s)
\]
would prove the convergence
\[
\int_{\partial\Omega}D_s\phi_k\:d\nu\xrightarrow[\ k\uparrow+\infty\ ]{}\int_{\partial\Omega}D_s\psi\:d\nu=0
\]
completing the proof of the lemma. Indeed, by \eqref{alter-Gbound}, it holds
\[
\left|\phi_k(x)-\psi(x)\right|\leq\int_\Omega G_\Omega(x,y)(1-\eta_k(y))\left|\Ds\psi(x)\right|\:dy
\leq \delta(x)^s\left\|\Ds\psi\right\|_{L^\infty(\Omega)}\int_{\Omega_k}\frac{dy}{\left|x-y\right|^{N-s}}
\]
where $\Omega_k:=\{\eta_k<1\}\subset\Omega$ has measure converging to $0$,
if the sequence ${\{\eta_k\}}_{k\in\N}$ is properly chosen.
\end{proof}

\begin{lem}\label{regularity} Suppose $f\in L^\infty_{loc}(\Omega)$, and consider 
a $u\in L^1(\R^N,(1+|x|)^{-N-2s}dx)$ satisfying
\[
\Ds u\ =\ f\qquad\hbox{in a distributional sense }
\]
i.e. for any $\psi\in C^\infty_c(\Omega)$
\[
\int_\Omega u\Ds\psi\ =\ \int_{\Omega} f\psi-\int_{\C\Omega}u\Ds\psi.
\]
Then $u\in C^\beta(\Omega)$ for any $\beta\in [0,2s)$.
If $f\in C^\alpha_{loc}(\Omega)$ for some $\alpha>0$
such that $\alpha+2s\not\in\N$, then $u\in C^{2s+\alpha}_{loc}(\Omega)$.
\end{lem}
\begin{proof}
This can be proved as in \cite[Proposition 2.8]{silvestre}.
\end{proof}

\begin{cor}\label{weak-point} Take $f\in C^\alpha_{loc}(\Omega)$ for some $\alpha>0$
such that $\alpha+2s\not\in\N$ with
\[
\int_\Omega |f|\delta^s\ <\ +\infty,
\]
a Radon measure $\mu\in \mathcal{M}(\C\Omega)$ such that
\[
\int_{\C\Omega}\delta^{-s}\min\{1,\delta^{-N-s}\}\;d|\mu|\ <\ +\infty,
\]
a finite Radon measure $\nu\in\mathcal{M}(\partial\Omega)$ and $u:\R^N\to\R$ a weak $L^1$ solution to
\[
\left\lbrace\begin{aligned}
\Ds u &=\ f & \hbox{ in }\Omega \\
u &=\ \mu & \hbox{ in }\C\Omega \\
Eu &=\ \nu & \hbox{ on }\partial\Omega.
\end{aligned}\right.
\]
Then $\Ds u(x)=f(x)$ holds pointwisely for any $x\in\Omega$.
\end{cor}
\begin{proof} This is a combination of the previous Lemma
with \cite[Proposition 2.4]{silvestre}.
\end{proof}

\begin{lem}\label{point-weak} Take $f\in C^\alpha_{loc}(\Omega)$ for some $\alpha\in(0,1)$, 
$h\in C(\partial\Omega)$ and $u:\R^N\to\R$
a pointwise solution to
\begin{equation}\label{98}
\left\lbrace\begin{aligned}
\Ds u &=\ f &\hbox{ in }\Omega \\
u &=\ g &\hbox{ in }\C\Omega \\
Eu &=\ h &\hbox{ on }\partial\Omega.
\end{aligned}\right.
\end{equation}
If
\[
\int_\Omega |f|\delta^s\ <\ +\infty,\qquad \int_{\C\Omega}|g|\delta^{-s}\min\{1,\delta^{-N-s}\}\ <\ +\infty,
\qquad h\in C(\partial\Omega)
\]
then $u$ is also a weak $L^1$ solution to the same problem.
\end{lem}
\begin{proof}
We refer to Theorem \ref{existence-weak2} for the existence and uniqueness of a weak $L^1$ solution $v$ to problem \eqref{98}.
By Corollary \ref{weak-point}, $v\in C^{2s+\alpha}_{loc}(\Omega)$ 
and $\Ds v = f$ holds pointwisely in $\Omega$. So
\begin{equation*}
\left\lbrace\begin{aligned}
\Ds (u-v) &=\ 0 &\hbox{ in }\Omega \\
u-v &=\ 0 &\hbox{ in }\C\Omega \\
E(u-v) &=\ 0 &\hbox{ on }\partial\Omega.
\end{aligned}\right.
\end{equation*}
in a pointwise sense. In particular, $u-v\in C(\Omega)$ since harmonic functions are smooth.
Define $\Omega^+:=\{x\in\Omega:u(x)>v(x)\}$: $u-v$ is a nonnegative $s$-harmonic function
and, by \cite[Lemma 5 and Theorem 1]{bogdan-repr-sharm}, it decomposes into the sum of the $s$-harmonic function
induced by the $E_{\Omega^+}(u-v)$ trace and the one by its values on $\C\Omega^+$.
But $E_{\Omega^+}(u-v)=0$ on $\partial\Omega^+$ as it is implied 
by the singular trace datum in \eqref{98} and the continuity on $\partial\Omega^+\cap\Omega$,
while $u-v\leq 0$ in $\C\Omega^+$. This yields $\Omega^+=\emptyset$ and $v\geq u$ in $\Omega$.
Repeating the argument we deduce also $u\leq v$ and this completes the proof of the lemma.
\end{proof}



\section{variational weak solutions vs. weak \texorpdfstring{$L^1$}{L1} solutions}\label{weakvsl1}

\begin{defi} Given $f\in L^\infty(\Omega)$, a variational weak solution of
\begin{equation}\label{345435345}
\left\lbrace\begin{aligned}
\Ds u &= f & & \hbox{in }\Omega \\
u &= 0 & & \hbox{in }\C\Omega \\
\end{aligned}\right.
\end{equation}
is a function $u\in H^s(\R^N)$ such that $u\equiv 0$ in $\C\Omega$ and
for any other $v\in H^s(\R^N)$ such that $v\equiv 0$ in $\C\Omega$ it holds
\[
\int_{\R^N}{(-\lapl)}^{s/2}u\,{(-\lapl)}^{s/2}v\ =\ \int_\Omega f\,v.
\]
\end{defi}

\begin{lem}\label{644351} $\T(\Omega)\subset H^s(\R^N)$.
\end{lem}
\begin{proof}
Consider $\phi\in\T(\Omega)$. The fractional Laplacian ${(-\lapl)}^{s/2}\phi$
is a continuous function decaying like $|x|^{-N-s}$ at infinity.
So $\|{(-\lapl)}^{s/2}\phi\|_{L^2(\R^N)}<\infty$ and we can apply \cite[Proposition 3.6]{hitchhiker}.
\end{proof}

\begin{prop} Let $f\in L^\infty(\Omega)$. Let $u$ be a variational
weak solution of \eqref{345435345}. Then it is also a weak $L^1$ solution
to the problem
\[
\left\lbrace\begin{aligned}
\Ds u &= f & & \hbox{in }\Omega \\
u &= 0 & & \hbox{in }\C\Omega \\
Eu &= 0 & & \hbox{on }\partial\Omega.
\end{aligned}\right.
\]
\end{prop}
\begin{proof}
Consider $\phi\in\T(\Omega)$:
\[
\int_\Omega u\Ds\phi=\int_{\R^N}{(-\lapl)}^{s/2}u\,{(-\lapl)}^{s/2}\phi=\int_\Omega f\phi.
\]
where we have used Lemma \ref{644351} on $\phi$.
\end{proof}

\begin{prop} Let $f\in L^\infty(\Omega)$. Let $u$ be a weak $L^1$ solution
to the problem
\[
\left\lbrace\begin{aligned}
\Ds u &= f & & \hbox{in }\Omega \\
u &= 0 & & \hbox{in }\C\Omega \\
Eu &= 0 & & \hbox{on }\partial\Omega.
\end{aligned}\right.
\]
Then it is also a variational
weak solution of \eqref{345435345}. 
\end{prop}
\begin{proof}
Call $w$ the variational weak solution of \eqref{345435345}. 
By the previous Lemma, $w$ is also a weak $L^1$ solution.
We must conclude $u=w$ by the uniqueness of a weak $L^1$ solution.
\end{proof}

\section{traces}\label{traces}

With Section \ref{bblin-sec} and Lemma \ref{Eu} we provided some facts on the boundary 
behaviour of solutions to linear problems with some assumptions on the data.
In this section we recover the boundary behaviour in presence 
of rough data. We use the notion of weak trace introduced by Ponce \cite[Proposition 3.5]{poncebook}.

\begin{prop} Let $\lambda\in\mathcal{M}(\Omega)$ be a Radon measure such that
\begin{equation}\label{564}
\int_\Omega \delta^s\:d|\lambda|\ <\ +\infty.
\end{equation}
Then the weak solution to
\[
\left\lbrace\begin{aligned}
\Ds u &= \lambda & & \hbox{in }\Omega \\
u &= 0 & & \hbox{in }\C\Omega \\
Eu &= 0 & & \hbox{on }\partial\Omega
\end{aligned}\right.
\]
satisfies
\begin{equation}
\lim_{\eta\downarrow 0}\frac 1\eta\int_{\{\delta(x)<\eta\}}\delta(x)^{1-s}\, u(x)\:dx
\ =\ 0.
\end{equation}
\end{prop}
\begin{proof}
By using the Jordan decomposition on $\lambda=\lambda^+-\lambda^-$,
we can suppose that $\lambda$ is a nonnegative measure
without loss of generality.
Fix $\sigma\in(0,s)$ and exchange the order of integration in
\[
\int_{\{\delta(x)<\eta\}}\delta(x)^{1-s}\, u(x)\:dx=
\int_{\delta(x)<\eta}\delta(x)^{1-s}\int_\Omega G_\Omega(x,y)\:d\lambda(y)\:dx.
\]
Our claim is that
\begin{equation}\label{768}
\int_\Omega G_\Omega(x,y)\,\delta(x)^{1-s}\chi_{(0,\eta)}(\delta(x))\;dx\ \leq\ 
\left\lbrace\begin{aligned}
& C\eta^{1+\sigma}\delta(y)^{s-\sigma} & \hbox{ if }\delta(y)\geq\eta \\
& C\eta\delta(y)^s & \hbox{ if }\delta(y)<\eta.
\end{aligned}\right.
\end{equation}
This would prove
\begin{multline*}
\frac1\eta\int_\Omega\left(\int_\Omega G_\Omega(x,y)\,\delta(x)^{1-s}\chi_{(0,\eta)}(\delta(x))\;dx\right)d\lambda(y)\ \leq \\
\leq\ C\eta^\sigma\int_{\{\delta(y)\geq\eta\}\cap\Omega}\delta(y)^{s-\sigma}\;d\lambda(y)
+C\int_{\{\delta(y)<\eta\}\cap\Omega}\delta(y)^s\;d\lambda(y)
\end{multline*}
where the second addend converges to $0$ as $\eta\downarrow 0$ by \eqref{564}.
For the first addend, we have that $\eta^\sigma\delta(y)^{s-\sigma}$ converges pointwisely to zero
for any $y\in\Omega$ and $\eta^\sigma\delta(y)^{s-\sigma}\leq \delta(y)^s$
if $y\in\Omega\cap\{\delta(y)>\eta\}$, therefore we have the convergence to $0$ by dominated convergence.

We turn now to the proof of \eqref{768}. For any $y\in\Omega$ one has 
\begin{equation}\label{654}
\int_\Omega G_\Omega(x,y)\,\delta(x)^{1-s}\chi_{(0,\eta)}(\delta(x))\;dx\leq
\eta^{1+\sigma}\int_\Omega G_\Omega(x,y)\,\delta(x)^{-s-\sigma}\;dx\leq
C\eta^{1+\sigma}\delta(y)^{s-\sigma}
\end{equation}
where we have used the regularity at the boundary in \eqref{udelta}.
In particular \eqref{654} holds when $\delta(y)>\eta$. 

To prove the other part of \eqref{768} we write 
(dropping from now on multiplicative constants depending on $N,\Omega$ and $s$)
\[
\int_\Omega G_\Omega(x,y)\,\delta(x)^{1-s}\chi_{(0,\eta)}(\delta(x))\;dx
\leq\ \eta^{1-s}\int_{\{\delta(x)<\eta\}\cap\Omega}\frac{\left(\delta(x)\delta(y)\wedge|x-y|^2\right)^s}{{|x-y|}^N}\delta(x)^{1-s}dx
\]
and, by exploiting Lemma \eqref{lemma39} below, we are allowed to perform 
the computations only in the case where $\partial\Omega$ is locally flat
where the above reads as
\[
\int_0^\eta\int_B\frac{\left[x_Ny_N\wedge(|x'-y'|^2+|x_N-y_N|^2)\right]^s}{{(|x'-y'|^2+|x_N-y_N|^2)}^{N/2}}\cdot x_N^{1-s}\;dx'\:dx_N.
\]
where $x=(x',x_N)\in\R^{N-1}\times\R$ and $y=(y',y_N)\in\R^{N-1}\times\R$.
First note that we can suppose $y'=0$ without loss of generality 
and $a\wedge b\leq 2ab/(a+b)$ when $a,b>0$. With the change of variable 
$x_N=y_N\,t$ and $x'=y_N\,\xi$, we reduce to
\[
y_N^{1+s}\int_0^{\eta/y_N}\int_{B_{1/y_N}}\frac{t}{\left(|\xi|^2+|t-1|^2\right)^{N/2-s}}
\cdot\frac{d\xi}{\left(|\xi|^2+|t-1|^2+t\right)^s}\;dt
\]
and, passing to polar coordinates in the $\xi$ variable
\begin{multline*}
y_N^{1+s}\int_0^{\eta/y_N}\int_0^{1/y_N}\frac{t\,\rho^{N-2}}{\left(\rho^2+|t-1|^2\right)^{N/2-s}}
\cdot\frac{d\rho}{\left(\rho^2+|t-1|^2+t\right)^s}\;dt\ \leq\\
\leq\ y_N^{1+s}\int_0^{\eta/y_N}\int_0^{1/y_N}\frac{t\,\rho}{\left(\rho^2+|t-1|^2\right)^{3/2-s}}
\cdot\frac{d\rho}{\left(\rho^2+|t-1|^2+t\right)^s}\;dt.
\end{multline*}
We deal first with the integral in the $\rho$ variable\footnote{The
computation which follows is not valid in the particular case $s=1/2$,
but with some minor natural modifications the same idea will work.}
\begin{align*}
& t\int_0^{1/y_N}\frac{\rho}{\left(\rho^2+|t-1|^2\right)^{3/2-s}}
\cdot\frac{d\rho}{\left(\rho^2+|t-1|^2+t\right)^s}\ \leq \\
& \leq\ \frac{t}{\left(|t-1|^2+t\right)^s}\int_0^1\frac{\rho}{\left(\rho^2+|t-1|^2\right)^{3/2-s}}\;d\rho+
t\int_1^{1/y_N}\frac{\rho}{\left(\rho^2+|t-1|^2\right)^{3/2}}\;d\rho \\
& \leq\ \frac{t}{\left(|t-1|^2+t\right)^s}\cdot\left.\frac{\left(\rho^2+|t-1|^2\right)^{-1/2+s}}{2s-1}\right|_{\rho=0}^1+
\frac{t}{\left(1+|t-1|^2\right)^{1/2}}.
\end{align*}
We then have
\begin{multline*}
t\int_0^{1/y_N}\frac{\rho}{\left(\rho^2+|t-1|^2\right)^{3/2-s}}
\cdot\frac{d\rho}{\left(\rho^2+|t-1|^2+t\right)^s}\ \leq\\
\leq\ \left\lbrace\begin{aligned}
& \frac{t}{\left(|t-1|^2+t\right)^s\,\left|t-1\right|^{1-2s}}+\frac{t}{\left(1+|t-1|^2\right)^{1/2}} & s\in(0,1/2), \\
& \frac{t\,\left(1+|t-1|^2\right)^{s-1/2}}{\left(|t-1|^2+t\right)^s}+\frac{t}{\left(1+|t-1|^2\right)^{1/2}} & s\in(1/2,1).
\end{aligned}\right.
\end{multline*}
The two quantities are both integrable in $t=1$
and converge to a positive constant as $t\uparrow+\infty$, therefore
\[
y_N^{1+s}\int_0^{\eta/y_N}\int_0^{1/y_N}\frac{t\,\rho^{N-2}}{\left(\rho^2+|t-1|^2\right)^{N/2-s}}
\cdot\frac{d\rho}{\left(\rho^2+|t-1|^2+t\right)^s}\;dt\ \leq\ 
\eta\,y_N^s=\eta\,\delta(y)^s
\]
completing the proof of \eqref{768}.
\end{proof}

\begin{prop} Let $\mu\in\mathcal{M}(\C\Omega)$ be a Radon measure such that
\begin{equation}\label{nvlwkn}
\int_{\C\Omega}\delta^{-s}\,\min\{1,\delta^{-N-s}\}\;d|\mu|\ <\ +\infty.
\end{equation}
Then the $s$-harmonic function $u$ induced 
in $\Omega$ by $\mu$ satisfies
\begin{equation}\label{ciciic}
\lim_{\eta\downarrow 0}\frac 1\eta\int_{\{\delta(x)<\eta\}}\delta(x)^{1-s}\, u(x)\:dx
\ =\ 0.
\end{equation}
\end{prop}
\begin{proof} Again, we can assume without loss of generality
that $\mu$ is a nonnegative measure. Recall that $u$ is represented by
\[
u(x)\ =\ -\int_{\C\Omega}\Ds G_\Omega(x,y)\;d\mu(y),\qquad x\in\Omega
\]
and by \eqref{est-poisson}
\[
u(x)\ \leq\ \int_{\C\Omega}\frac{\delta(x)^s}{\left|x-y\right|^N}\cdot\frac{d\mu(y)}{\delta(y)^s\left(1+\delta(y)\right)^s}.
\]
Choose $R>0$ large enough to have $\Omega\subset\subset B_R$ and split
\[
\mu=\mu_1+\mu_2:=\chi_{B_R}\mu+\chi_{\C B_R}\mu\ :
\]
note that 
\[
\delta(x)^{1-s}\int_{\C\Omega}\frac{\delta(x)^s}{\left|x-y\right|^N}\cdot\frac{d\mu_2(y)}{\delta(y)^s\left(1+\delta(y)\right)^s}
\leq \delta(x)\int_{\C B_R}\frac{d\mu_2(y)}{\delta(y)^s\left(1+\delta(y)\right)^{N+s}}
\]
which converges uniformly to $0$ as $\delta(x)\downarrow 0$ in view of \eqref{nvlwkn}.
Thus we can suppose that $\mu$ is supported in $B_R$ and 
\[
\frac 1\eta\int_{\{\delta(x)<\eta\}}\delta(x)^{1-s}\, u(x)\:dx\leq
\int_{\C\Omega}\left(\frac 1\eta\int_{\{\delta(x)<\eta\}}\frac{\delta(x)}{\left|x-y\right|^N}\;dx\right)
\frac{d\mu(y)}{\delta(y)^s}.
\]
Once again, we reduce to the flat case as we did in (and with the same notation of) the last proposition.
Let $y=(0,-y_N)$ with $y_N>0$. We have to handle
\[
\frac 1\eta\int_0^\eta\int_{B} x_N\left[(x_N+y_N)^2+|x'|^2\right]^{-N/2}dx'\:dx_N.
\]
With the change of variable $x'=(x_N+y_N)\xi$ we have
\[
\frac 1\eta\int_0^\eta\int_{B} x_N\left[(x_N+y_N)^2+|x'|^2\right]^{-N/2}dx'\:dx_N=
\frac 1\eta\int_0^\eta\frac{x_N}{x_N+y_N}\int_{B_{1/(x_N+y_N)}}\left[1+|\xi|^2\right]^{-N/2}d\xi\:dx_N
\]
which is uniformly bounded in $\eta$ for $y_N>0$. Moreover,
\begin{multline*}
\frac 1\eta\int_0^\eta\int_{B} x_N\left[(x_N+y_N)^2+|x'|^2\right]^{-N/2}dx'\:dx_N=
\int_0^1\int_B \eta\,t\left[(\eta\,t+y_N)^2+|x'|^2\right]^{-N/2}dx'\:dt\ \leq\\
\leq\ \eta\int_0^1\int_B\left[y_N^2+|x'|^2\right]^{-N/2}dx'\:dt
\end{multline*}
converges to $0$ as $\eta\downarrow 0$ for any $y_N>0$. Therefore, since $\delta^{-s}\mu$
is a finite measure on $\C\Omega$, by dominated convergence the limit
\[
\lim_{\eta\downarrow 0}\int_{\C\Omega}\left(\frac 1\eta\int_{\{\delta(x)<\eta\}}\frac{\delta(x)}{\left|x-y\right|^N}\;dx\right)
\frac{d\mu(y)}{\delta(y)^s}\ =\ 0
\]
holds.
\end{proof}

\section{the test function space}\label{chenv}

In \cite{chen-veron} the following definition of weak solution is given.

\begin{defi}\label{weakdefi2} Given a Radon measure $\nu$ such that $\delta^s\in L^1(\Omega,d\nu)$
a function $u\in L^1(\Omega)$ is a weak solution of
\[
\left\lbrace\begin{aligned}
\Ds u+ f(u) &=\ \nu & \hbox{ in }\Omega \\
u &=\ 0 & \hbox{ in }\C\Omega
\end{aligned}\right.
\]
if $f(u)\in L^1(\Omega,\delta^s\,dx)$
\[
\int_\Omega u\Ds\xi+\int_\Omega f(u)\xi\ =\ \int_\Omega \xi\;d\nu
\] 
for any $\xi\in\mathbb{X}_s\subset C(\R^N)$, i.e. 
\begin{enumerate}
\item supp$\xi\subseteq\overline{\Omega}$
\item $\Ds\xi(x)$ is pointwisely defined for any $x\in\Omega$ and $\|\Ds\xi\|_{L^\infty(\Omega)}<+\infty$ 
\item there exist a positive $\varphi\in L^1(\Omega,\delta^s dx)$ and $\eps_0>0$ such that
\[
|\Ds_\eps\xi(x)|=\left|\int_{\C B_\eps(x)}\frac{\xi(x)-\xi(y)}{|x-y|^{N+2s}}\;dx\right|\leq \varphi(x)
\qquad \hbox{ for all }\eps\in(0,\eps_0].
\]
\end{enumerate}
\end{defi}

The test space $\mathbb{X}_s$ in Definition \ref{weakdefi2} is quite different 
from the space $\T(\Omega)$ which is used in Definition \ref{weakdefiintro}.
Still, testing a Dirichlet problem against one or the other does not yield two different solutions,
i.e. the two notions of weak $L^1$ solutions are equivalent.
We split the proof of this fact into two lemmas.

\begin{lem} $\T(\Omega)\subset\mathbb{X}_s$.
\end{lem}
\begin{proof} Pick $\phi\in\T(\Omega)$.
Properties {\it 1.} and {\it 2.} of Definition \ref{weakdefi2} are satisfied by construction.
In order to prove {\it 3.} write for $\delta(x)<2\eps$
\begin{multline}\label{1111}
\Ds_\eps\phi(x)=\psi(x)-PV\int_{B_\eps(x)}\frac{\phi(x)-\phi(y)}{|x-y|^{N+2s}}\;dy\ =\\
=\ \psi(x)-PV\int_{B_{\delta(x)/2}(x)}\frac{\phi(x)-\phi(y)}{|x-y|^{N+2s}}\;dy
-\int_{B_\eps(x)\setminus B_{\delta(x)/2}(x)}\frac{\phi(x)-\phi(y)}{|x-y|^{N+2s}}\;dy.
\end{multline}
with $\psi:=\Ds\phi|_\Omega\in C^\infty_c(\Omega)$.
Consider $\alpha\in(0,s)$. For the first integral
\begin{align*}
& \left|PV\int_{B_{\delta(x)/2}(x)}\frac{\phi(x)-\phi(y)}{|x-y|^{N+2s}}\;dy\right|\ 
 \leq\ \|\phi\|_{C^{2s+\alpha}(B_{\delta(x)/2}(x))}\int_{B_{\delta(x)/2}(x)}\frac{dy}{|x-y|^{N-\alpha}}\ \\
& =\ \|\phi\|_{C^{2s+\alpha}(B_{\delta(x)/2}(x))}\frac{\omega_{N-1}}{\alpha}\left(\frac{\delta(x)}{2}\right)^\alpha
\end{align*}
where, by \cite[Proposition 2.8]{silvestre}
\begin{align*}
& \|\phi\|_{C^{2s+\alpha}(B_{\delta(x)/2}(x))}=2^{2s+\alpha}\delta(x)^{-2s-\alpha}\left\|\phi\left(x+\frac{\delta(x)}{2}\ \cdot\right)\right\|_{C^{2s+\alpha}(B)}\ \leq \\
& \leq\ C\delta(x)^{-2s-\alpha}\left(\left\|\phi\left(x+\frac{\delta(x)}{2}\ \cdot\right)\right\|_{L^\infty(B)}+
\left\|\psi\left(x+\frac{\delta(x)}{2}\ \cdot\right)\right\|_{C^{\alpha}(B)}\right) \\
& \leq\ C\delta(x)^{-2s-\alpha}\left(\|\phi\|_{L^\infty(B_{\delta(x)/2}(x))}
+\delta(x)^\alpha\|\psi\|_{C^{\alpha}(B_{\delta(x)/2}(x))}\right) \\
& \leq\ C\delta(x)^{-2s-\alpha}\left(\|\psi\|_{L^\infty(\R^N)}\delta(x)^s
+\delta(x)^\alpha\|\psi\|_{C^{\alpha}(\R^N)}\right)\ 
\leq\ C\|\psi\|_{C^\alpha(\R^N)}\delta(x)^{-2s}.
\end{align*}
The integration far from $x$ gives
\begin{align*}
& \left|\int_{B_\eps(x)\setminus B_{\delta(x)/2}(x)}\frac{\phi(x)-\phi(y)}{|x-y|^{N+2s}}\;dy\right|\ 
 \leq\ \|\phi\|_{C^s(\R^N)} \int_{B_\eps(x)\setminus B_{\delta(x)/2}(x)}\frac{dy}{|x-y|^{N+s}}\ \leq \\ 
& \leq\ \|\phi\|_{C^s(\R^N)} \int_{\R^N\setminus B_{\delta(x)/2}}\frac{dz}{|z|^{N+s}}\ 
 \leq\ \|\phi\|_{C^s(\R^N)} \frac{\omega_{N-1}}{s}\left(\frac2{\delta(x)}\right)^s.
\end{align*}
All this entails
\[
\delta(x)^s|\Ds_\eps\phi(x)|\ \leq\ \delta(x)^s|\psi(x)|+C\|\psi\|_{C^\alpha(\R^N)}\delta(x)^{\alpha-s}+C\|\phi\|_{C^s(\R^N)},
\qquad \hbox{when }\delta(x)<2\eps.
\]
For $\delta(x)\geq2\eps$ one does not have the second integral on the right-hand side of \eqref{1111}
whereas the first one is computed on the ball of radius $\eps$,
but the same computations can be run.
This proves the statement of the Lemma.
\end{proof}

\begin{lem} Given a Radon measure $\nu\in\mathcal{M}(\Omega)$ such that $\delta^s\in L^1(\Omega,d\nu)$,
if a function $u\in L^1(\Omega)$ satisfies 
\begin{equation}\label{7890}
\int_\Omega u\Ds\xi\ =\ \int_\Omega\xi\;d\nu,\qquad\hbox{ for any }\xi\in\T(\Omega),
\end{equation}
then the same holds true for any $\xi\in\mathbb{X}_s$. 
\end{lem}
\begin{proof}
Pick $\xi\in\mathbb{X}_s$: by definition, $\zeta:=\Ds\xi\in L^\infty(\Omega)$. 
Consider the standard mollifier $\eta\in C^\infty_c(\R^N)$ and $\eta_\eps(x):=\eps^{-N}\eta(x/\eps)$.
Then
\begin{equation}\label{zetaeps}
\zeta_\eps:=\zeta\chi_\Omega*\eta_\eps \in C^\infty(\R^N)\hbox{ and }\|\zeta_\eps\|_{L^\infty(\Omega)}\leq\|\zeta\|_{L^\infty(\Omega)}.
\end{equation}
Define $\xi_\eps$ as the solution to
\[
\left\lbrace\begin{aligned}
\Ds\xi_\eps &=\ \zeta_\eps & \hbox{ in }\Omega \\
\xi_\eps &=\ 0 & \hbox{ in }\C\Omega \\
E\xi_\eps &=\ 0 & \hbox{ on }\partial\Omega.
\end{aligned}\right.
\]
Also, for $\rho>0$ small consider
\[
\Omega_\rho:=\{x\in\Omega:\delta(x)>\rho\}
\]
and a bump function $b_\rho\in C^\infty_c(\R^N)$ such that
\[
b_\rho\equiv 1\hbox{ in }\Omega_{2\rho},\ b_\rho\equiv0\hbox{ in }\R^N\setminus\Omega_\rho,\ 0\leq b_\rho\leq 1\hbox{ in }\R^N.
\]
Then $\zeta_{\eps,\rho}:=b_\rho\zeta_\eps\in C^\infty_0(\Omega)$. 
Let $\xi_{\eps,\rho}\in\T(\Omega)$ be the function induced by $\zeta_{\eps,\rho}$.
By \eqref{7890},
\begin{equation}\label{231456}
\int_\Omega u\zeta_{\eps,\rho}\ =\ \int_\Omega \xi_{\eps,\rho}\;d\nu.
\end{equation}
It holds $\zeta_{\eps,\rho}\rightarrow\zeta_\eps$ as $\rho\downarrow0$,
with $\|\zeta_{\eps,\rho}\|_{L^\infty(\Omega)}\leq\|\zeta_\eps\|_{L^\infty(\Omega)}$ and
\[
|\xi_{\eps,\rho}(x)|\ \leq\ C\|\zeta_{\eps,\rho}\|_{L^\infty(\Omega)}\delta(x)^s
\leq \|\zeta_{\eps}\|_{L^\infty(\Omega)}\delta(x)^s
\]
so that we can push equality \eqref{231456} to the limit to deduce,
by dominated convergence,
\begin{equation}\label{2356}
\int_\Omega u\zeta_\eps\ =\ \int_\Omega \xi_\eps\;d\nu.
\end{equation}
Similarly, since $\|\zeta_{\eps}\|_{L^\infty(\Omega)}\leq \|\zeta\|_{L^\infty(\Omega)}$,
letting $\eps\downarrow 0$ yields
\[
\int_\Omega u\zeta\ =\ \int_\Omega \xi\;d\nu.
\]
\end{proof}

\part[Boundary~blow-up~solutions for~{\texorpdfstring{$\As$}{As}}]
{Boundary~blow-up~solutions for~{\huge\texorpdfstring{$\As$}{As}}}

\chapter{Introduction and main results}\label{introparttwo}

Given a bounded domain $\Omega$ of the Euclidean space $\R^N$, 
the {\it spectral fractional Laplacian} operator $\As,\ s\in(0,1),$ is classically defined as 
a fractional power of the Laplacian with homogeneous Dirichlet 
boundary conditions, 
seen as a self-adjoint operator in the Lebesgue space $L^2(\Omega)$, see \eqref{As} below.
This provides a nonlocal operator of elliptic type with {\it homogeneous} boundary conditions.
Recent bibliography on this operator can be found e.g. in 
Bonforte, Sire and V\'azquez  \cite{bsv}, Grubb \cite{grubb3}, 
Caffarelli and Stinga \cite{caffarelli-stinga}, Servadei and Valdinoci \cite{servadei-valdinoci}.

One aspect of the theory is however left unanswered:
the formulation of natural {\it nonhomogeneous} boundary conditions. 
A first attempt can be found in the work of Dhifli, M\^{a}agli and Zribi \cite{dhifli}. 
The investigations that have resulted in the present Part turn out, we hope, to shed some further light on this question. 
We provide a weak formulation, which is well-posed in the sense of Hadamard, for linear problems of the form  
\begin{equation}\label{main}
\left\{\begin{aligned}
\As u &=\ \mu & \hbox{ in }\Omega, \\
\frac{u}{h_1} &=\ \zeta & \hbox{ on }\partial\Omega,
\end{aligned}\right.
\end{equation}
where $h_1$ is a reference function, see \eqref{h1} below, 
with prescribed singular behaviour at the boundary. 
Namely, $h_1$ is bounded above and below by constant multiples of $\delta^{-(2-2s)}$
and the left-hand side of the boundary condition must be understood as 
a limit as $\delta$ converges to zero. 
In other words, unlike the classical Dirichlet problem for the Laplace operator, 
nonhomogeneous boundary conditions must be singular. 
In addition, if the data $\mu, \zeta$ are smooth, the solution blows up\footnote{This 
is very similar indeed to the theory of nonhomogeneous boundary conditions for the 
fractional Laplacian - although in that case the blow-up rate is of order $\delta^{-(1-s)}$ -
as analysed in Part \ref{partone} of this manuscript and from a different perspective by Grubb \cite{grubb}.}
at the fixed rate $\delta^{-(2-2s)}$.
In fact, for the special case of positive $s$-harmonic functions, that is when $\mu=0$, 
the singular boundary condition was already identified in previous works 
emphasizing the probabilistic and potential theoretic aspects of the problem: 
see e.g. Song and Vondra\v{c}ek \cite{song-vondra}, 
Glover, Pop-Stojanovic, Rao, \v{S}iki\'c, Song and Vondra\v{c}ek \cite{gprssv} and Song \cite{song}
for the spectral fractional Laplacian. 

Turning to nonlinear problems, even more singular boundary conditions arise: 
in the above system, if $\mu=-u^p$ for suitable values of $p$, one may choose $\zeta=+\infty$, 
in the sense that the solution $u$ will blow up at a higher rate
with respect to $\delta^{-(2-2s)}$ and 
controlled by the (scale-invariant) one $\delta^{-2s/(p-1)}$. 
Note that the value $\zeta=+\infty$ is not admissible for linear problems.
This has been already observed in the context of the fractional Laplacian in the previous Part, see \cite{a2},
and this is what we prove here for the spectral fractional Laplacian. 
Interestingly, the range of admissible exponents $p$ is different according to which operator one works with.

\section{main results} 
For clarity, we list here the definitions and statements that we use. 
First recall the definition of the spectral fractional Laplacian:
\begin{defi}
Let $\Omega\subset\R^N$ a bounded 
domain and let ${\{\varphi_j\}}_{j\in\N}$ 
be a Hilbert basis of $L^2(\Omega)$ consisting of eigenfunctions 
of the Dirichlet Laplacian $\left.-\lapl\right\vert_\Omega$, 
associated to the eigenvalues $\lambda_j$, $j\in\N$, i.e.\footnote{See Brezis \cite[Theorem 9.31]{brezis}.} 
$\varphi_j\in H^1_0(\Omega)\cap C^\infty(\Omega)$ and $-\lapl\varphi_j=\lambda_j\varphi_j$ in $\Omega$.
Given $s\in(0,1)$, consider the Hilbert space\footnote{When 
$\Omega$ is smooth, $H(2s)$ coincides with the Sobolev space $H^{2s}(\Omega)$ if $0<s<1/4$, 
$H_{00}^{s}(\Omega)$ if $s\in\{1/4, 3/4\}$, $H_0^{2s}(\Omega)$ otherwise; 
see Lions and Magenes \cite[Theorems 11.6 and 11.7 pp. 70--72]{lions-magenes}.} 
\[
H(2s):=\left\{v=\sum_{j=1}^\infty \widehat v_j\varphi_j\in L^2(\Omega):\|v\|_{H(2s)}^2=\sum_{j=0}^\infty\lambda_j^{2s}\vert\widehat v_j\vert^2<\infty\right\}.
\]
The spectral fractional Laplacian of  $u\in H(2s)$ is the function belonging to $L^2(\Omega)$ given by the formula
\begin{equation}\label{As}
\As u\ =\ \sum_{j=1}^\infty\lambda_j^s\widehat u_j\,\varphi_j.
\end{equation}
\end{defi}
Note that $C^\infty_c(\Omega)\subset H(2s) \hookrightarrow L^2(\Omega)$. 
So, the operator $\As$ is unbounded, densely defined and with bounded inverse $\Dms$ in $L^2(\Omega)$.
Alternatively, for almost every $x\in\Omega$,
\begin{equation}
\As u(x)\ =\  PV\int_\Omega[u(x)-u(y)]J(x,y)\;dy+\kappa(x)u(x),\label{As2}
\end{equation}
where, letting $p_{\Omega}(t,x,y)$ denote the heat kernel of $\left.-\lapl\right\vert_\Omega$, 
\begin{equation}\label{jkappa}
J(x,y)=\frac{s}{\Gamma(1-s)}\int_0^\infty\frac{p_{\Omega}(t,x,y)}{t^{1+s}}\;dt\quad\hbox{and}\quad
\kappa(x)=\frac{s}{\Gamma(1-s)}\int_\Omega\left(1-\int_\Omega p_{\Omega}(t,x,y)\;dy\right)\frac{dt}{t^{1+s}}
\end{equation}
are respectively\footnote{in the language of potential theory of killed stochastic processes.
Note that the integral in \eqref{As2} must be understood in the sense of principal values. 
To see this, look at \eqref{Jbound}.} 
the {\it jumping kernel} and the {\it killing measure,}
see Song and Vondra\v{c}ek \cite[formulas (3.3) and (3.4)]{song-vondra}. 
For the reader's convenience, we provide a proof of \eqref{As2} in Paragraph \ref{app1}. 
We assume from now on that  
\[
\text{$\Omega$ is of class $C^{1,1}$.}
\] 
In particular,  
sharp bounds are known for the heat kernel $P_\Omega$, see \eqref{hkb} below, 
and provide in turn sharp estimates for $J(x,y)$, see \eqref{Jbound} below, 
so that the right-hand side of \eqref{As2} remains well-defined
for every $x\in\Omega$ under the assumption that 
$u\in C^{2s+\eps}(\Omega)\cap L^1(\Omega,\delta(x)dx)$ 
for some $\eps>0$. 
This allows us to {\it define} the spectral fractional Laplacian
of functions which {\it do not} vanish on the boundary of $\Omega$. 
As a simple example, observe that the function $u(x)\equiv 1$, $x\in\Omega$,
does not belong to $H(2s)$ if $s\ge1/4$, yet it solves
\eqref{main} for $\mu=\kappa$ and $\zeta=0$.

\begin{defi} The Green function and the Poisson kernel of the spectral fractional Laplacian are defined respectively by
\begin{equation}\label{green}
G_{|\Omega}^s(x,y)\ =\ \frac1{\Gamma(s)}\int_0^\infty p_{\Omega}(t,x,y)\,t^{s-1}\;dt,
\qquad x,\,y\in\Omega,x\neq y,\, s\in(0,1],
\end{equation}
where $p_{\Omega}$ denotes the heat kernel of $\left.-\lapl\right|_\Omega$, and by
\begin{equation}\label{poisson}
P_{|\Omega}^s(x,y)\ :=\ -\left.\frac{\partial}{\partial\nu_y}\right. G^s_{|\Omega}(x,y), \qquad x\in\Omega, y\in\partial\Omega.
\end{equation}
where $\nu$ is the outward unit normal to $\partial\Omega$.  
\end{defi}

In Section \ref{green-sect}, we shall prove that $P_{|\Omega}^s$ is well-defined
(see Lemma \ref{lemma8}) and review some useful identities
involving the Green function $G_{|\Omega}^s$ 
and the Poisson kernel $P_{|\Omega}^s$. 
Now, let us define weak solutions of \eqref{main}. 
\begin{defi} Consider the test function space
\begin{equation}\label{test}
\T(\Omega)\ :=\ \rest^{-s}C^\infty_c(\Omega)
\end{equation}
and the weight
\begin{equation}\label{h1}
h_1(x) = \int_{\partial\Omega} P^s_{|\Omega}(x,y)\,d\HH(y),\qquad x\in\Omega.
\end{equation}
Given two Radon measures $\mu\in\mathcal{M}(\Omega)$ 
and $\zeta\in\mathcal{M}(\partial\Omega)$ with
\begin{equation}\label{hypo}
\int_\Omega\delta(x)\;d|\mu|(x)\ <\ \infty,\qquad |\zeta|(\partial\Omega)\ <\ \infty,
\end{equation}
a function $u\in L^1_{loc}(\Omega)$ is a weak solution to
\begin{equation}\label{prob}
\left\{\begin{aligned}
\As u &=\ \mu & \hbox{ in }\Omega \\
\frac{u}{h_1} &=\ \zeta & \hbox{ on }\partial\Omega
\end{aligned}\right.
\end{equation}
if, for any $\psi\in\T(\Omega)$,
\begin{equation}\label{byparts-spectral}
\int_\Omega u\,\As\psi\ =\ \int_\Omega \psi\,d\mu - \int_{\partial\Omega} \frac{\partial\psi}{\partial\nu}\;d\zeta.
\end{equation}
\end{defi}

We shall prove that $\T(\Omega)\subseteq C^1_0(\overline{\Omega})$, 
see Lemma \ref{lemma-test-spectral}, 
so that all integrals above are well-defined. 
Equation \eqref{byparts-spectral} is indeed a weak formulation of
\eqref{prob}, as the following lemma shows.

\begin{lem}\label{clas-weak} 
\begin{enumerate}
\item (weak solutions are distributional solutions) 
Assume that $u\in L^1_{loc}(\Omega)$ is a weak solution of \eqref{prob}. 
Then in fact, $u\in L^1(\Omega,\delta(x)dx)$ and $\As u=\mu$ in the sense
of distributions i.e. for any $\psi\in C^\infty_c(\Omega)$,
$\frac{\As\psi}{\delta}$ is bounded in $\Omega$ and
\[
\int_\Omega u \As\psi = \int_\Omega \psi \;d\mu,
\]
moreover the boundary condition holds in the sense that for every $\phi\in C(\super\Omega)$
\[
\lim_{t\downarrow 0}\frac1t\int_{\{\delta(x)<t\}}\frac{u(x)}{h_1(x)}\,\phi(x)\:d\HH(x)= \int_{\partial\Omega}\phi(x)\:d\zeta(x)
\]
whenever $\mu\in\mathcal{M}(\Omega)$ satisfies \eqref{hypo}
and $\zeta\in L^1(\partial\Omega)$.
\item (for smooth data, weak solutions are classical) 
Assume that $u\in L^1_{loc}(\Omega)$ is a weak solution of \eqref{prob}, 
where $\mu\in C^\alpha(\overline\Omega)$ for some $\alpha$ such that $\alpha+2s\not\in\N$ and $\zeta\in C(\partial\Omega)$.
Then, $\As u$ is well-defined by \eqref{As2} for every $x\in\Omega$, 
$\As u(x)=\mu(x)$ for all $x\in\Omega$ and for all $x_0\in\partial\Omega$,
\[
\lim_{\stackrel{\hbox{\footnotesize $x\to x_0$}}{x\in\Omega}}\frac{u(x)}{h_1(x)} = \zeta(x_0).
\]
\item (classical solutions are weak solutions) 
Assume that $u\in C^{2s+\eps}(\Omega)$ is such that $u/h_1\in C(\super\Omega)$. 
Let $\mu=\As u$ be given by \eqref{As2} and $\zeta=\left. u/h_1\right\vert_{\partial\Omega}$. 
Then, $u$ is a weak solution of \eqref{prob}.
\end{enumerate}
\end{lem}

We present some facts about harmonic functions in Section \ref{harm-sect} with an eye kept on their singular boundary trace in Section \ref{bb-sect}. 

We prove the well-posedness of \eqref{prob} in Section \ref{dir-sect}, namely
\begin{theo}\label{point} Given two Radon measures $\mu\in\mathcal{M}(\Omega)$ and $\zeta\in\mathcal{M}(\partial\Omega)$
such that \eqref{hypo} holds, there exists a unique function $u\in L^1_{loc}(\Omega)$ 
which is a weak solution to \eqref{prob}. 
Moreover, for a.e. $x\in\Omega$,
\begin{equation}\label{repr}
u(x)\ =\ \int_\Omega G^s_{|\Omega}(x,y)\;d\mu(y)+
\int_{\partial\Omega}P_{|\Omega}^s(x,y)\;d\zeta(y)
\end{equation}
and
\begin{equation}\label{cont}
\|u\|_{L^1(\Omega,\delta(x) dx)}\ \leq\ 
C(\Omega,N,s)\left(\|\delta\mu\|_{\mathcal{M}(\Omega)}+\|\zeta\|_{\mathcal{M}(\partial\Omega)}\right).
\end{equation}
In addition, the following  estimates hold.
\begin{align}
\|u\|_{L^p(\Omega,\delta(x) dx)}\ &\leq\ C_1\|\delta\mu\|_{\mathcal{M}(\Omega)} &
\qquad\text{if $\zeta=0$ and $p\in\left[1,\frac{N+1}{N+1-2s}\right)$}\label{lp}\\
\|u\|_{C^\alpha(\overline\omega)}\ &\leq\ 
C_2\left(\|\mu\|_{L^\infty(\omega)}+\|\zeta\|_{\mathcal M(\partial\Omega)}\right) &
\qquad\text{if $\omega\subset\subset\Omega$ and $\alpha\in(0,2s)$}\label{calpha}\\
\|u\|_{C^{2s+\alpha}(\overline\omega)}\ &\leq\ 
C_3\left(\|\mu\|_{C^\alpha(\overline\omega)}+\|\zeta\|_{\mathcal M(\partial\Omega)}\right) &
\qquad\text{if $\omega\subset\subset\Omega$ and $2s+\alpha\not\in\N$.}\label{ctwosplusalpha}
\end{align}
In the above $C_1=C_1(\Omega,N,s,p)$, $C_2=C_2(\Omega,\omega,N,s,\alpha)$, $C_3=C_3(\Omega,\omega,N,s,\alpha)$.
\end{theo}

In Section \ref{nonlin-sect} we solve nonlinear Dirichlet problems, by proving
\begin{theo}\label{nonhom-cor} Let $g(x,t):\Omega\times\R^+\longrightarrow\R^+$ 
be a Carath\'eodory function and $h:\R^+\to\R^+$ a nondecreasing function 
such that $g(x,0)=0$ and for a.e. $x\in\Omega$ and all $t>0$,
\[
0\le g(x,t)\le h(t) \qquad\text{where}\qquad  h(\delta^{-(2-2s)})\delta\in L^1(\Omega).
\]
Then, problem
\begin{equation}\label{nonhom-prob}
\left\lbrace\begin{aligned}
\As u &=\; -g(x,u) & \hbox{ in }\Omega \\
\frac{u}{h_1} &=\; \zeta & \hbox{ on }\partial\Omega
\end{aligned}\right.
\end{equation}
has a solution $u\in L^1(\Omega,\delta(x)dx)$
for any $\zeta\in C(\partial\Omega), \zeta\ge0$. In addition, if $t\mapsto g(x,t)$ is nondecreasing then the solution is unique.
\end{theo}

Finally, with Section \ref{extra-sect} we prove

\begin{theo}\label{xlarge-theo} Let 
\[
p\in\left(1+s,\frac1{1-s}\right).
\]
Then, there exists a nonnegative function $u\in L^1(\Omega,\delta(x)dx)\cap C^\infty(\Omega)$ solving
\begin{equation}\label{large-prob}
\left\lbrace\begin{aligned}
\As u &=\; -u^p & &\hbox{ in }\Omega, \\
\frac{u}{h_1} &=\; +\infty & & \hbox{ on }\partial\Omega
\end{aligned}\right.
\end{equation}
in the following sense:
the first equality holds pointwise and in the sense of distributions, 
the boundary condition is understood as a pointwise limit. 
In addition, there exists a constant $C=C(\Omega, N,s,p)$ such that
\[
0\le u\le C\delta^{-\frac{2s}{p-1}}.
\]
\end{theo}

\chapter{Nonhomogeneous boundary conditions}\label{spectral}

\section{green function and poisson kernel}\label{green-sect}
In the following three lemmas\footnote{which hold even if the domain $\Omega$ is not $C^{1,1}$}, we establish some useful identities for the Green function defined by \eqref{green}.
\begin{lem}\label{inverse}{\rm (\cite[formula (17)]{grss})}
Let $f\in L^2(\Omega)$. For almost every $x\in\Omega$, $G_{|\Omega}^s(x,\cdot)f\in L^1(\Omega)$ and
\[
\rest^{-s}f(x) = \int_\Omega G_{|\Omega}^s(x,y)f(y)\;dy\qquad \text{for a.e. $x\in\Omega$}.
\]
\end{lem}

\begin{proof} If $\varphi_j$ is an eigenfunction of $\left.-\lapl\right\vert_\Omega$, then
\begin{align*}
& \int_\Omega G_{|\Omega}^s(x,y)\varphi_j(y)\; dy \ =\ \int_0^\infty\frac{t^{s-1}}{\Gamma(s)}\int_\Omega p_\Omega(t,x,y)\varphi_j(y)\;dy\;dt \\
& =\ \int_0^\infty\frac{t^{s-1}}{\Gamma(s)}\,e^{-\lambda_jt}\varphi_j(x)\;dt\ =\ \frac{\lambda_j^{-s}}{\Gamma(s)}\varphi_j(x)\int_0^\infty t^{s-1}e^{-t}\;dt
\ =\ \lambda_j^{-s}\varphi_j(x)=\Dms\varphi_j(x).
\end{align*}
By linearity, if $f\in L^2(\Omega)$ is a linear combination of eigenvectors  $f=\sum_{j=1}^M \widehat f_j\varphi_j$, then 
\[
\int_\Omega G_{|\Omega}^s(x,y)\sum_{j=1}^M \widehat f_j\varphi_j(y)\; dy\ =\ \sum_{j=1}^M \widehat f_j\lambda_j^{-s}\varphi_j(x)
\]
and so, letting 
\begin{equation}\label{identity}
\mathbb{G}^s_{|\Omega} f\ :=\ \int_\Omega G_{|\Omega}^s(\cdot,y)\,f(y)\;dy,
\end{equation}
we have
\begin{equation}\label{riesz-id}
\|\mathbb{G}^s_{|\Omega} f\|_{H(2s)}^2=\sum_{j=1}^M\lambda_j^{2s}\cdot \vert\widehat f_j\vert^2\lambda_j^{-2s} = \|f\|_{L^2(\Omega)}^2.
\end{equation}
Thus the map $\mathbb{G}_{|\Omega}^s: f\longmapsto \mathbb G^s_{|\Omega} f$
uniquely extends to a linear isometry from $L^2(\Omega)$ to $H(2s)$, 
which coincides with $\Dms$. It remains to prove that the identity \eqref{identity} remains valid a.e. for $f\in L^2(\Omega)$.
By standard parabolic theory, the function $(t,x)\mapsto \int_\Omega p_\Omega(t,x,y)dy$ is bounded (by 1) and smooth in $[0,T]\times\omega$ for every $T>0$, $\omega\subset\subset \Omega$. Hence, for every $x\in\Omega$, $G_{|\Omega}^s(x,\cdot)\in L^1(\Omega)$.
Assume first that $f=\psi\in C^\infty_c(\Omega)$ and take a sequence ${\{\psi_k\}}_{k\in\N}$
in the linear span of the eigenvectors
${\{\varphi_j\}}_{j\in\N}$ such that ${\{\psi_k\}}_{k\in\N}$ 
converges to $\psi$ in $L^2(\Omega)$. The convergence is in fact uniform and so \eqref{identity} holds for $f=\psi$. Indeed,
by standard elliptic regularity, there exist constants $C=C(N,\Omega), k=k(N)$ such that any eigenfunction satisfies 
\[
\|\grad\varphi_j\|_{L^\infty(\Omega)}\leq (C\lambda_j)^k \|\varphi_j\|_{L^2(\Omega)}=(C\lambda_j)^k.
\]
In particular, taking $C$ larger if needed,
\begin{equation}\label{esti-ef}
\left\|\frac{\varphi_j}{\delta}\right\|_{L^\infty(\Omega)}\leq(C\lambda_j)^k.
\end{equation}
Now write $\psi=\sum_{j=1}^{\infty}\widehat\psi_j\varphi_j$ and fix $m\in\N$. Integrating by parts $m$ times yields
\[
\widehat\psi_j=\int_\Omega\psi\varphi_j=-\frac1{\lambda_j}\int_\Omega\psi\,\lapl\varphi_j
=-\frac1{\lambda_j}\int_\Omega\lapl\psi\,\varphi_j=\ldots=\frac{(-1)^m}{\lambda_j^m}\int_\Omega\lapl^m\psi\,\varphi_j
\]
which implies that
\begin{equation}\label{spectral-coef}
|\widehat\psi_j|\leq \frac{\|\lapl^m\psi\|_{L^2(\Omega)}}{\lambda_j^m},
\end{equation}
i.e. the spectral coefficients of $\psi$ converge to $0$ faster than any polynomial. This and \eqref{esti-ef} imply that ${\{\psi_k\}}_{k\in\N}$ 
converges to $\psi$ uniformly, as claimed.

Take at last $f\in L^2$ and a sequence ${\{f_k\}}_{k\in\N}$ 
in $C^\infty_c(\Omega)$ of nonnegative functions such that ${\{f_k\}}_{k\in\N}$ 
converges to $\vert f\vert$ a.e. and in $L^2(\Omega)$. 
By \eqref{riesz-id}, 
$\|\mathbb{G}^s_{|\Omega} f_k\|_{L^2}\le \|f_k\|_{L^2}$ 
and by  Fatou's lemma, we deduce that
$G_{|\Omega}^s(x,\cdot)f\in L^1(\Omega)$ 
for a.e. $x\in\Omega$ and the desired identity follows.
\end{proof}

\begin{lem}{\rm (\cite[formula (8)]{grss})}
For a.e. $x,y\in\Omega$, 
\begin{equation}\label{compo}
\int_{\Omega} G^{1-s}_{|\Omega}(x,\xi)G^s_{|\Omega}(\xi,y)\:d\xi= G^1_{|\Omega}(x,y).
\end{equation}
\end{lem}
\begin{proof}
Clearly, given an eigenfunction $\varphi_j$,
\[
\rest^{-s}\rest^{s-1}\varphi_j = \lambda_j^{-s}\lambda_j^{s-1}\varphi_j = \rest^{-1}\varphi_j
\]
so $\rest^{-s}\circ\rest^{s-1}=\rest^{-1}$ in $L^2(\Omega)$. By the previous lemma and Fubini's theorem, we deduce that for $\varphi\in L^2(\Omega)$ and a.e. $x\in\Omega$,
\[
\int_{\Omega^2} G^{1-s}_{|\Omega}(x,\xi)G^s_{|\Omega}(\xi,y)\varphi(y)\;d\xi\:dy= \int_\Omega G^1_{|\Omega}(x,y)\varphi(y)\; dy
\]
and so \eqref{compo} holds almost everywhere.
\end{proof}

\begin{lem} For any $\psi\in C^\infty_c(\Omega)$, 
\begin{equation}\label{adg}
\As\psi\ =\ (-\lapl)\circ \rest^{s-1} \psi
\ =\ \rest^{s-1}\circ(-\lapl)  \psi.
\end{equation}
\end{lem}
\begin{proof} 
The identity clearly holds if $\psi$ is an eigenfunction. If $\psi\in C^\infty_c(\Omega)$, its spectral coefficients have fast decay and the result follows by writing the spectral decomposition of $\psi$. Indeed, 
thanks to \eqref{spectral-coef} and \eqref{esti-ef},  we may easily work by density to establish \eqref{adg}.
\end{proof}

Let us turn to the definition and properties of the Poisson kernel. 
Recall that, 
for $x\in\Omega, y\in\partial\Omega$, the Poisson kernel of the Dirichlet Laplacian is given by
\[
P_{|\Omega}^1(x,y)\ =\ -\left.\frac{\partial}{\partial\nu_y}\right. G^1_{|\Omega}(x,y).
\]

\begin{lem}\label{lemma8} 
The function
\[
P_{|\Omega}^s(x,y)\ :=\ -\frac{\partial}{\partial\nu_y} G^s_{|\Omega}(x,y)
\]
is well-defined for $x\in\Omega, y\in\partial\Omega$ and 
$P_{|\Omega}^s(x,\cdot)\in C(\partial\Omega)$ for any $x\in\Omega$.
Furthermore, there exists a constant $C>0$ depending on $N,s,\Omega$ only such that
\begin{equation}\label{poissbound}
\frac1C\frac{\delta(x)}{|x-y|^{N+2-2s}}\le P^s_{|\Omega}(x,y)\le C\frac{\delta(x)}{|x-y|^{N+2-2s}}
\end{equation}
and 
\begin{equation}\label{pois-id}
\int_{\Omega} G^{1-s}_{|\Omega}(x,\xi)P^s_{|\Omega}(\xi,y)\;d\xi\ =\ P_{|\Omega}^1(x,y).
\end{equation}
\end{lem}
\begin{rmk}\rm
When $\Omega$ is merely Lipschitz, one must work with the Martin kernel in place of the Poisson kernel, see \cite{gprssv}.
\end{rmk}

\begin{proof} {\it of Lemma \ref{lemma8}. }
Take $x,z\in\Omega$, $y\in\partial\Omega$. Then,
\[\frac{G_{|\Omega}^s(x,z)}{\delta(z)}=
\frac{1}{\Gamma(s)\,\delta(z)}\int_0^\infty p_\Omega(t,x,z)t^{s-1}\;dt=
\frac{\vert z-x\vert^{2s}}{\Gamma(s)\,\delta(z)} \int_0^\infty p_\Omega(\vert z-x\vert^{2}\tau,x,z)\tau^{s-1}\;d\tau.
\]
Since $\Omega$ has $C^{1,1}$ boundary, 
given $x\in\Omega$, $p_\Omega(\cdot,x,\cdot)\in C^1((0,+\infty)\times\overline\Omega)$ 
and the following heat kernel bound holds (cf. Davies, Simon and Zhang \cite{davies1,davies2, zhang}):
\begin{equation}\label{hkb}
\left[\frac{\delta(x)\delta(y)}{t}\wedge 1\right]\frac{1}{c_1 t^{N/2}}e^{-|x-y|^2/(c_2t)}\le p_\Omega(t,x,y)\leq\left[\frac{\delta(x)\delta(y)}{t}\wedge 1\right]\frac{c_1}{t^{N/2}}e^{-c_2|x-y|^2/t},
\end{equation}
where $c_1,c_2$ are constants depending on $\Omega, N$ only.
So,
\begin{equation}\label{prepois}
\frac{\vert z-x\vert^{2s}}{\,\delta(z)} p_\Omega(\vert z-x\vert^{2}\tau,x,z)\tau^{s-1}\le C
\vert z-x\vert^{2s-N-2}\delta(x)\tau^{s-2-N/2}e^{-c_2/\tau}
\end{equation}
and the reverse inequality 
\[
\frac 1C
\vert z-x\vert^{2s-N-2}\delta(x)\tau^{s-2-N/2}e^{-1/(c_2\tau)}\le
\frac{\vert z-x\vert^{2s}}{\,\delta(z)} p_\Omega(\vert z-x\vert^{2}\tau,x,z)\tau^{s-1}
\]
also holds for $\tau\ge \delta(x)\delta(z)\vert z-x\vert^{-2}$.
As $z\to y\in\partial\Omega$, 
the right-hand-side of \eqref{prepois} obviously converges 
in $L^1(0,+\infty,d\tau)$ so we may apply 
the generalized dominated convergence theorem to deduce that
$P_{|\Omega}^s(x,y)$ is well-defined, satisfies \eqref{poissbound} and 
\[
P_{|\Omega}^s(x,y)=-\frac{\partial}{\partial\nu_y} G^s_{|\Omega}(x,y)=\lim_{z\to y}\frac{G^s_{|\Omega}(x,z)}{\delta(z)}
=-\frac1{\Gamma(s)}\int_0^\infty\frac{\partial}{\partial\nu_y}p_\Omega(t,x,y)\,t^{s-1}\,dt.
\]
From this last formula we deduce also that, for any fixed $x\in\Omega$,
the function $P_{|\Omega}^s(x,\cdot)\in C(\partial\Omega)$: indeed,
having chosen a sequence ${\{y_k\}}_{k\in\N}\subset\partial\Omega$
converging to some $y\in\partial\Omega$, we have
\[
\left|P_{|\Omega}^s(x,y_k)-P_{|\Omega}^s(x,y)\right|
\leq\frac1{\Gamma(s)}\int_0^\infty\left|\frac{\partial}{\partial\nu_y}p_\Omega(t,x,y_k)
-\frac{\partial}{\partial\nu_y}p_\Omega(t,x,y)\right|t^{s-1}\,dt
\]
where, by \eqref{hkb}
\[
\left|\frac{\partial}{\partial\nu_y}p_\Omega(t,x,y)\right|\leq
\frac{c_1\,\delta(x)}{t^{N/2+1}}\,e^{-c_2|x-y|^2/t}\leq
\frac{c_1\,\delta(x)}{t^{N/2+1}}\,e^{-c_2\delta(x)^2/t}
\qquad\hbox{for any }y\in\partial\Omega,
\]
so that $\left|P_{|\Omega}^s(x,y_k)-P_{|\Omega}^s(x,y)\right|\to 0$ as $k\uparrow\infty$
by dominated convergence.

By similar arguments, $G_{|\Omega}^s$ is a continuous function on $\super\Omega^2\setminus\{(x,y):x= y\}$.
And so, by \eqref{compo}, we have
\[
-\frac{\partial}{\partial\nu_y}\int_{\Omega} G^{1-s}_{|\Omega}(x,\xi)G^s_{|\Omega}(\xi,y)\;d\xi =\ P_{|\Omega}^1(x,y).
\]
Let us compute the derivative of the left-hand side alternatively. We have
\[
\int_\Omega G_{|\Omega}^{1-s}(x,\xi)\frac{G_{|\Omega}^s(\xi,z)}{\delta(z)}\;d\xi= 
 \int_{\R^+\times\Omega} f(t,\xi,z)\;dt d\xi,
\]
where, having fixed $x\in\Omega$,
\[
f(t,\xi,z) = \frac{G_{|\Omega}^{1-s}(x,\xi)}{\Gamma(s)\,\delta(z)}  p_\Omega(t,\xi,z)t^{s-1} \le C \vert x-\xi\vert^{2s-N}\left[\frac{\delta(x)\delta(\xi)}{\vert x-\xi\vert^2}\wedge 1\right]t^{s-2-N/2}\delta(\xi)e^{-c_2\vert z-\xi\vert^2/t}.
\]
For a fixed $\eps>0$, and $\xi\in\Omega\setminus B(y,\eps)$, $z\in B(y,\eps/2)$, we deduce that
\[
f(t,\xi,z) \le C \vert x-\xi\vert^{2s-N}t^{s-2-N/2}e^{-c_\eps/t}\in L^1((0,+\infty)\times\Omega)
\]
Similarly, if $t>\eps$,
\[
f(t,\xi,z)\le C \vert x-\xi\vert^{2s-N}t^{s-2-N/2}\in L^1((\eps,+\infty)\times\Omega).
\]
Now,
\[
\int_0^\eps t^{s-1-N-2}e^{-c_2\frac{\vert\xi-z\vert^2}{t}}dt \le \vert\xi-z\vert^{2s-N} \int_0^{+\infty} \tau^{s-1-N/2}e^{-c_2/\tau}d\tau.
\]
Hence, there exists a constant $C>0$ independent of $\eps$ such that 
\[
\int_{(0,\eps)\times B(y,\eps)}f(t,\xi,z) dt\;d\xi \le C \eps^{2+2s}.
\]
It follows from the above estimates and dominated convergence that
\begin{equation}
P_{|\Omega}^1(x,y)=\lim_{z\to y}\int_{(0,+\infty)\times\Omega} f(t,\xi,z)\;dt d\xi = \int_{\Omega} G^{1-s}_{|\Omega}(x,\xi)P^s_{|\Omega}(\xi,y)d\xi
\end{equation}
i.e. \eqref{pois-id} holds.
\end{proof}

\begin{rmk} \rm
Thanks to the heat kernel bound \eqref{hkb}, the following estimate also holds. 
\begin{equation}\label{green-behav}
\frac1C \frac1{{|x-y|}^{N-2s}}\left(1\wedge\frac{\delta(x)\,\delta(y)}{|x-y|^2}\right) \le G_{|\Omega}^s(x,y)\le  \frac C{{|x-y|}^{N-2s}}\left(1\wedge\frac{\delta(x)\,\delta(y)}{|x-y|^2}\right)
\end{equation}
for some constant $C=C(\Omega,N,s)$.
Also observe for computational convenience that 
\[
\frac12\left(1\wedge\frac{\delta(x)\,\delta(y)}{|x-y|^2}\right)\le \frac{\delta(x)\,\delta(y)}{\delta(x)\,\delta(y)+|x-y|^2}\le \left(1\wedge\frac{\delta(x)\,\delta(y)}{|x-y|^2}\right).
\]
\end{rmk}

\section{harmonic functions and interior regularity}
\label{harm-sect}

By the heat kernel bound \eqref{hkb}, there exists $C=C(\Omega,N,s)>0$ such that
\begin{equation}\label{Jbound}
\frac{1}{C{|x-y|}^{N+2s}}
\left[\frac{\delta(x)\delta(y)}{{|x-y|}^2}\wedge 1\right]\leq J(x,y)\leq \frac{C}{{|x-y|}^{N+2s}}
\left[\frac{\delta(x)\delta(y)}{{|x-y|}^2}\wedge 1\right].
\end{equation}

\begin{defi}\label{harm-def} 
A function $h\in L^1(\Omega,\delta(x)dx)$ is $s$-harmonic in $\Omega$ 
if for any $\psi\in C^\infty_c(\Omega)$ there holds
\[
\int_\Omega h\:\As\psi\ =\ 0.
\]
\end{defi}
\noindent The above definition makes sense thanks to the following lemma.
\begin{lem}\label{As-smooth} 
For any $\psi\in C^\infty_c(\Omega)$, $\As\psi\in C^1_0(\super\Omega)$ and there exists a constant $C=C(s,N,\Omega,\psi)>0$ such that
\begin{equation}\label{ineq-lemma1}
|\As \psi|\ \leq\ C\delta\qquad \text{in }\Omega.
\end{equation}
In addition, if $\psi\ge0$, $\psi\not\equiv0$, then
\begin{equation}
\As\psi\ \leq\ -C\delta \qquad \text{in } \Omega\setminus\hbox{supp}\,\psi.
\end{equation}
\end{lem}
\begin{proof}
Thanks to \eqref{esti-ef} and \eqref{spectral-coef}, 
$\As\psi\in C^1_0(\overline\Omega)$ and 
\[
\left|\frac{\As \psi}{\delta}\right|\leq
\sum_{j=1}^{\infty}\lambda_j^s|\widehat\psi_j|\left\|\frac{\varphi_j}{\delta}\right\|_{L^\infty(\Omega)}
<\infty
\]
and \eqref{ineq-lemma1} follows. Let us turn to the case where $\psi\ge0$, $\psi\not\equiv0$.
We apply formula \eqref{As2} and assume 
that $x\in\Omega\setminus\hbox{supp}\psi$. Denote by $x^*$ a point of maximum of $\psi$ and let $2r=\dist(x^*,\hbox{supp}\,\psi)$. Then for $y\in B_r(x^*)$, it holds $\psi(y),\delta(y)\ge c_1>0$, $r\leq \vert x-y\vert\le c_2$ and so, applying \eqref{Jbound},
\begin{eqnarray*}
\As \psi(x) & = & -\int_\Omega\psi(y)\,J(x,y)\;dy \\
& \leq & -C\int_{B_r(x^*)}\frac{\psi(y)}{{|x-y|}^{N+2s}}
\left[\frac{\delta(x)\delta(y)}{{|x-y|}^2}\wedge 1\right]dy \\
& \leq & -C\cdot\int_{B_r(x^*)}\frac{c_1}{c_2^{N+2s}}\left[\frac{\delta(x)\,c_1}{c_2^2}\wedge 1\right]dy \\
& \leq & -\tilde C\,\delta(x).
\end{eqnarray*}
\end{proof}

\begin{lem}\label{pois-harm} 
The function $P_{|\Omega}^s(\cdot,z)\in L^1(\Omega,\delta(x)dx)$ 
is $s$-harmonic in $\Omega$ for any fixed $z\in\partial\Omega$.
\end{lem}
\begin{proof}
Thanks to \eqref{poissbound}, 
$P_{|\Omega}^s(\cdot,z)\in L^1(\Omega,\delta(x)dx)$.
Pick $\psi\in C^\infty_c(\Omega)$ and exploit \eqref{adg}:
\[
\int_\Omega P_{|\Omega}^s(\cdot,z)\,\As\psi=
\int_\Omega P_{|\Omega}^s(\cdot,z)\,\rest^{s-1}[-\lapl\psi].
\]
Applying Lemma \ref{inverse}, the Fubini's Theorem and \eqref{pois-id}, the above quantity is equal to
\[
\int_\Omega P^1_{|\Omega}(\cdot,z)\,(-\lapl)\psi=0.
\]
\end{proof}

\begin{lem}\label{harm}  
For any finite Radon measure $\zeta\in\mathcal{M}(\partial\Omega)$, let
\begin{equation}\label{repres}
h(x)\ =\ \int_{\partial\Omega} P^s_{|\Omega}(x,z)\;d\zeta(z),\qquad x\in\Omega.
\end{equation}
Then, 
$h$ is $s$-harmonic in $\Omega$. In addition, there exists a constant $C=C(N,s,\Omega)>0$ such that
\begin{equation}\label{blah}
\Vert h\Vert_{L^1(\Omega,\delta(x)dx)} \le C \Vert \zeta\Vert_{\mathcal M(\partial\Omega)}.
\end{equation}
Conversely, for any $s$-harmonic function $h\in L^1(\Omega,\delta(x)dx)$, $h\geq 0$,
there exists a finite Radon measure $\zeta\in\mathcal{M}(\partial\Omega)$, $\zeta\geq 0$,
such that \eqref{repres} holds.
\end{lem}
\begin{proof} Since $P_{|\Omega}^s(x,\cdot)$ is continuous, $h$ is well-defined. By \eqref{poissbound},
\[
\delta(x)|h(x)|\leq
C\int_{\partial\Omega}\frac{d|\zeta|(z)}{|x-z|^{N-2s}}
\]
so that $h\in L^1(\Omega,\delta(x)dx)$ and \eqref{blah} holds.
Pick now $\psi\in C^\infty_c(\Omega)$:
\[
\int_\Omega h(x)\,\As\psi(x)\;dx=
\int_{\partial\Omega}\left(\int_\Omega P_{|\Omega}^s(x,z)\,\As\psi(x)\;dx\right)d\zeta(z)=0
\]
in view of Lemma \ref{pois-harm}. 
Conversely, let $h$ denote a nonnegative $s$-harmonic function.
By Definition \ref{harm-def} and by equation \eqref{adg}, 
we have for any $\psi\in C^\infty_c(\Omega)$
\begin{multline*}
0=\int_\Omega h(x)\As\psi(x)\:dx=
\int_\Omega h(x)\,\rest^{s-1}\circ(-\lapl)\,\psi(x)\:dx\ =\\
=\ \int_\Omega \left(\int_\Omega G_{|\Omega}^{1-s}(x,\xi)h(\xi)\,d\xi\right)(-\lapl)\psi(x)\:dx,
\end{multline*}
so that $\int_\Omega G_{|\Omega}^{1-s}(x,\xi)h(\xi)d\xi$ is a (standard) nonnegative harmonic function.
In particular (cf. e.g. \cite[Corollary 6.15]{axler}),
there exists a finite Radon measure $\zeta\in\mathcal{M}(\partial\Omega)$ 
such that $\int_\Omega G_{|\Omega}^{1-s}(x,\xi)h(\xi)d\xi=\int_{\partial\Omega} P_{|\Omega}^1(x,y)d\zeta(y)\in C^\infty(\Omega)$. 
We now exploit equation \eqref{pois-id} to deduce that 
\[
\int_\Omega G^{1-s}_{|\Omega}(x,\xi)\left[h(\xi)-\int_{\partial\Omega}P^s_{|\Omega}(\xi,y)d\zeta(y)\right]d\xi =0.
\]
Since 
\[
\int_\Omega \varphi_1(x)\ \left(\int_\Omega G_{|\Omega}^{1-s}(x,\xi) h(\xi)\:d\xi\right)dx=\int_\Omega h\, \rest^{s-1}\varphi_1=\frac1{\lambda_1^{1-s}}\int_\Omega h\,\varphi_1 <\infty,
\]
it holds $\int_\Omega G_{|\Omega}^{1-s}(x,\xi) h(\xi)d\xi\,\in L^1(\Omega,\delta(x)dx)$. Thanks to \eqref{green-behav}, we are allowed to let $G_{|\Omega}^{s}$ act on it. 
By \eqref{compo}, this leads to 
\[
\int_\Omega G^1_{|\Omega}(x,\xi)\left[h(\xi)-\int_{\partial\Omega}P^s_{|\Omega}(\xi,y)\;d\zeta(y)\right]d\xi =0.
\]
Take at last $\psi\in C^\infty_c(\Omega)$ and $\varphi= (\left.-\lapl\right\vert_\Omega)^{-1}\psi$. Then,
\[
0=
\int_\Omega \varphi(x)\left[\int_\Omega G^1_{|\Omega}(x,\xi)\left[h(\xi)-\int_{\partial\Omega}P^s_{|\Omega}(\xi,y)d\zeta(y)\right]\:d\xi\right]dx=
\int_\Omega \psi(\xi)\left[h(\xi)-\int_{\partial\Omega}P^s_{|\Omega}(\xi,y)\;d\zeta(y)\right]d\xi
\]
and so \eqref{repres} holds a.e. and in fact everywhere thanks to Lemma \ref{clas-harm} below.
\end{proof}

\begin{lem}\label{clas-harm}
Take $\alpha>0$ such that $2s+\alpha\not\in\N$ and $f\in C^{\alpha}_{loc}(\Omega)$. If $u\in L^1(\Omega,\delta(x)dx)$ solves 
\[
\As u = f \qquad\text{in $\mathcal D'(\Omega)$,}
\]
then $u\in C^{2s+\alpha}_{loc}(\Omega)$, the above equation holds pointwise, and given any compact sets $K\subset\subset K'\subset\subset\Omega$, there exists a constant $C=C(s,N,\alpha,K,K',\Omega)$ such that
\[
\Vert u \Vert_{C^{2s+\alpha}(K)} \le C \left(\Vert f \Vert_{C^{\alpha}(K')}+\Vert u \Vert_{L^1(\Omega,\delta(x)dx)}\right).
\]
Similarly, if $f\in L^\infty_{loc}(\Omega)$ and $\alpha\in(0,2s)$,
\[
\Vert u \Vert_{C^\alpha(K)} \le C \left(\Vert f \Vert_{L^\infty(K')}+\Vert u \Vert_{L^1(\Omega,\delta(x)dx)}\right).
\]
In particular, if $h$ is $s$-harmonic, then $h\in C^\infty(\Omega)$ and the equality
$\As h(x)=0$ holds at every point $x\in\Omega$.
\end{lem}
\begin{proof} We only prove the former inequality, the proof of the latter follows mutatis mutandis.
Given $x\in\Omega$, let
\[
v(x)=\int_\Omega G_{|\Omega}^{1-s}(x,y)u(y)\;dy.
\]
Observe that $v$ is well-defined and 
\begin{equation}\label{esti-l1d}
\Vert v\Vert_{L^1(\Omega,\delta(x)dx)} \le C(\Omega,N,s)\Vert u\Vert_{L^1(\Omega,\delta(x)dx)}
\end{equation}
Indeed, letting $\varphi_1>0$ denote an eigenvector associated to the principal eigenvalue of the Laplace operator, it follows from the Fubini's theorem and Lemma \ref{inverse} that
\[
\int_\Omega \varphi_1(x)\int_\Omega G_{|\Omega}^{1-s}(x,y)\left\vert u(y)\right\vert dy\, dx = \lambda_1^{s-1}\int_\Omega \left\vert u(y)\right\vert\varphi_1(y)\, dy. 
\]
In addition, $-\lapl v= f$ in $\mathcal D'(\Omega)$, since for $\varphi\in C^\infty_c(\Omega)$,
\[
\int_\Omega v\:(-\lapl)\varphi = 
\int_{\Omega}u\:\As\varphi,
\]
thanks to the Fubini's theorem, equation \eqref{adg}, Lemma \ref{inverse} and Definition \ref{harm-def}. 
Observe now that if $\varphi\in C^\infty_c(\Omega)$, then 
\[
\Dums\varphi=\frac{s}{\Gamma(1-s)}\int_0^{+\infty}t^{s-1}
\left(\frac{\varphi-e^{\trest}\varphi}{t}\right)dt.
\]
The above identity is straightforward if $\varphi$ is an eigenfunction and remains true for $\varphi\in C^\infty_c(\Omega)$ by density, using the fast decay of spectral coefficients, see \eqref{spectral-coef}. 
So,
\begin{multline*}
\int_\Omega u\varphi\:dx= 
\int_\Omega v\Dums\varphi\:dx\ =\\
=\ \frac{s}{\Gamma(1-s)}\int_{\Omega}\int_0^\infty
v\,t^{s-1}\left(\frac{\varphi-e^{\trest}\varphi}{t}\right)dt\:dx\ =\\
=\ \frac{s}{\Gamma(1-s)}\int_{\Omega}\int_0^\infty 
\varphi t^{s-1}\left(\frac{v-e^{\trest}v}{t}\right)dt\:dx
\end{multline*}
and 
\[
u = \frac{s}{\Gamma(1-s)}\int_0^{+\infty}t^{s-1}
\left(\frac{v-e^{\trest}v}{t}\right)dt.
\]
Choose $\super f\in C^\alpha_c(\R^N)$ such that $\super f=f$ in $K'$, $\Vert \super f\Vert_{C^{\alpha}(\R^N)}\le C \Vert f\Vert_{C^{\alpha}(K')}$ and let
\[
\super u(x) = c_{N,s}\int_{\R^N}\vert x-y\vert^{-(N-2s)}\super f(y)\;dy
\]
solve $(-\lapl)^s \super u = \super f$ in $\R^N$. It is well-known (see e.g. \cite{silvestre}) that $\Vert \super u\Vert_{C^{2s+\alpha}(\R^N)}\le C \Vert \super f\Vert_{C^{\alpha}(\R^N)}$, for a constant $C$ depending only on $s,\alpha,N$ and the measure of the support of $\super f$. It remains to estimate $u-\super u$. Letting 
\[
\super v(x) = c_{N,1-s}\int_{\R^N} \vert x-y\vert^{-(N-2(1-s))}\super u(y)\;dy,
\]
we have as previously that $-\lapl\super v=\super f$ and
\[
\super u = \frac{s}{\Gamma(1-s)}\int_0^{+\infty}t^{s-1}
\left(\frac{\super v-e^{-t\lapl}\super v}{t}\right)dt,
\]
where this time $e^{-t\lapl}\super v(x)= \frac1{(4\pi t)^{N/2}}\int_{\R^N}e^{-\frac{\vert x-y\vert^2}{4t}}\super v(y)\;dy$.
Hence,
\[
\frac{\Gamma(1-s)}s(u-\super u) = \int_0^{+\infty}t^{s-1}
\left(\frac{(v-\super v)-e^{-t\lapl}(v-\super v)}{t}\right)dt + \int_0^{+\infty}t^{s-1}
\left(\frac{e^{\trest}v-e^{-t\lapl}\super v}{t}\right)dt.
\]
Fix a compact set $K''$ such that $K\subset\subset K''\subset\subset K'$.
Since $v-\super v$ is harmonic in $K'$,
\[\Vert v-\super v\Vert_{C^{2s+\alpha+2}(K'')}\le C\Vert v-\super v\Vert_{L^1(K')}\le C\Vert u\Vert_{L^1(\Omega,\delta(x)dx)}
\]
By parabolic regularity,
\[
\left\Vert\int_0^{+\infty}t^{s-1}
\left(\frac{(v-\super v)-e^{-t\lapl}(v-\super v)}{t}\right)dt
\right\Vert_{C^{2s+\alpha}(K)}\le C\Vert u\Vert_{L^1(\Omega,\delta(x)dx)}.
\]
In addition, the function $w=e^{\trest}v-e^{-t\lapl}\super v$ solves the heat equation inside $\Omega$ with initial condition $w(0,\cdot)=v-\super v$. Since $\lapl (v-\super v)=0$ in $K'$, it follows from parabolic regularity again that $w(t,x)/t$ remains bounded in $C^{2s+\alpha}(K)$ as $t\to0^+$, so that again
\[
\left\Vert\int_0^{+\infty}t^{s-1}
\left(\frac{e^{\trest}v-e^{-t\lapl}\super v}{t}\right)dt\right\Vert_{C^{2s+\alpha}(K)}\le C\Vert u\Vert_{L^1(\Omega,\delta(x)dx)}.
\]

\end{proof}

\begin{lem}\label{u-to-lapl} Take $\alpha>0$, $\alpha\not\in\N$,
and 
$u\in C^{2s+\alpha}_{loc}(\Omega)\cap L^1(\Omega,\delta(x)dx)$.
Given any compact set $K\subset\subset K'\subset\subset\Omega$, there exists a constant $C=C(s,N,\alpha,K, K',\Omega)$ such that
\[
\|\As u\|_{C^\alpha(K)}\leq
C\left(\|u\|_{C^{2s+\alpha}(K')} +
\|u\|_{L^1(\Omega,\delta(x)dx)}\right).
\]
\end{lem}
\begin{proof} With a slight abuse of notation, we write
\[
\Asmu u(x) = \int_\Omega G_{|\Omega}^{1-s}(x,y)\,u(y)\:dy.
\]
By Lemma \ref{clas-harm} we have
\begin{multline*}
\|\Asmu u\|_{C^{2+\alpha}(K)}\leq C\left(
\|u\|_{C^{2s+\alpha}(K')}+\|\Asmu u\|_{L^1(\Omega,\delta(x)dx)}\right)\ \leq \\
\leq\ C\left(
\|u\|_{C^{2s+\alpha}(K')}+\|u\|_{L^1(\Omega,\delta(x)dx)}\right).
\end{multline*}
Obviously it holds also
\[
\|(-\lapl)\circ\Asmu u\|_{C^\alpha(K)}\leq\|\rest^{s-1}u\|_{C^{2+\alpha}(K)}.
\]
By \eqref{adg},
\[
(-\lapl)\circ\Asmu u
=\As u \qquad \hbox{in }\mathcal{D}'(\Omega),
\]
which concludes the proof.
\end{proof}

\begin{prop}\label{Lp-reg} Let $f\in L^1(\Omega,\delta(x)dx)$ and $u\in L^1_{loc}(\Omega)$. 
The function
\[
u(x)=\int_\Omega G_{|\Omega}^s(x,y)\,f(y)\;dy
\]
belongs to $L^p(\Omega,\delta(x)dx)$ for any $p\in\left[1,\dfrac{N+1}{N+1-2s}\right)$.
\end{prop}
\begin{proof}
We start by applying the Jensen's Inequality, 
\[
\left|\int_\Omega G_{|\Omega}^s(x,y)\,f(y)\;dy\right|^p
\leq \|f\|^{p-1}_{L^1(\Omega,\delta(x)dx)}
\int_\Omega\left|\frac{G_{|\Omega}^s(x,y)}{\delta(y)}\right|^p\delta(y)\,f(y)\;dy,
\]
so that
\[
\int_\Omega|u(x)|^p\,\delta(x)\;dx\leq \|f\|^p_{L^1(\Omega,\delta(x)dx)}
\sup_{y\in\Omega}\int_\Omega\left|\frac{G_{|\Omega}^s(x,y)}{\delta(y)}\right|^p\delta(x)\;dx
\]
and by \eqref{green-behav} we have to estimate
\[
\sup_{y\in\Omega}\int_\Omega\frac1{\left|x-y\right|^{(N-2s)p}}\cdot
\frac{\delta(x)^{p+1}}{\left[|x-y|^2+\delta(x)\delta(y)\right]^p}\;dx
\]
Pick $\eps>0$. Clearly,
\[
\sup_{\{y\;:\;\delta(y)\ge\eps\}}\int_\Omega\frac1{\left|x-y\right|^{(N-2s)p}}\cdot
\frac{\delta(x)^{p+1}}{\left[|x-y|^2+\delta(x)\delta(y)\right]^p}\;dx\le C_\eps.
\]
Thanks to Lemma \ref{lemma39}, we may now reduce to the case where the boundary is flat, i.e. when in a neighbourhood $A$ of a given point
$y\in\Omega$ such that $\delta(y)<\eps$, there holds $A\cap\partial\Omega\subseteq\{y_N=0\}$ and $A\cap\Omega\subseteq\{y_N>0\}$. Without loss of generality, we assume that $y=(0,y_N)$ and 
$x=(x',x_N)\in B\times(0,1)\subseteq\R^{N-1}\times\R$.
We are left with proving that
\[
\int_B dx'\int_0^1 dx_N \frac1{\left[|x'|^2+|x_N-y_N|^2\right]^{(N-2s)p/2}}\cdot
\frac{x_N^{p+1}}{\left[|x'|^2+|x_N-y_N|^2+x_Ny_N\right]^p}
\]
is a bounded quantity. Make the change of variables $x_N=y_N t$ and pass to
polar coordinates in $x'$, with $|x'|=y_N\rho$. Then, the above integral becomes
\begin{equation}\label{4576}
y_N^{-(N+1-2s)p+N+1}\int_0^{1/y_N}d\rho\int_0^{1/y_N}dt\;\frac{\rho^{N-2}}{\left[\rho^2+|t-1|^2\right]^{(N-2s)p/2}}
\cdot\frac{t^{p+1}}{\left[\rho^2+|t-1|^2+t\right]^p}.
\end{equation}
Now, we split the integral in the $t$ variable into $\int_0^{1/2}+\int_{1/2}^{3/2}+\int_{3/2}^{1/y_N}$.
Note that the exponent $-(N+1-2s)p+N+1$ is positive for $p<(N+1)/(N+1-2s)$.
We drop multiplicative constants in the computations that follow.
The first integral is bounded above by a constant multiple of
\[
\int_0^{1/y_N}d\rho\int_0^{1/2}dt\;\frac{\rho^{N-2}}{\left[\rho^2+1\right]^{(N-2s)p/2}}
\cdot\frac{t^{p+1}}{\left[\rho^2+1+t\right]^p}\lesssim
\int_0^{1/y_N}d\rho\;\frac{\rho^{N-2}}{\left[\rho^2+1\right]^{(N+2-2s)p/2}}
\]
which remains bounded as $y_N\downarrow 0$ since
\[
p\geq 1>\frac{N-1}{N+2-2s}\qquad\hbox{implies}\qquad (N+2-2s)p-N+2>1.
\]
The second integral is of the order of
\begin{align*}
& \int_0^{1/y_N}d\rho\int_{1/2}^{3/2}dt\;\frac{\rho^{N-2}}{\left[\rho^2+|t-1|^2\right]^{(N-2s)p/2}}\cdot\frac1{\left[\rho^2+1\right]^p} \ =\\
& =\ \int_0^{1/y_N}d\rho\int_0^{1/2}dt\;\frac{\rho^{N-2}}{\left[\rho^2+t^2\right]^{(N-2s)p/2}}\cdot\frac1{\left[\rho^2+1\right]^p} \\
& =\ \int_0^{1/y_N}d\rho\;\frac{\rho^{N-1-(N-2s)p}}{\left[\rho^2+1\right]^p}\int_0^{1/(2\rho)}dt\;\frac1{\left[1+t^2\right]^{(N-2s)p/2}} \\
& \lesssim\ \int_0^\infty d\rho\;\frac{\rho^{N-1-(N-2s)p}}{\left[\rho^2+1\right]^p}
\end{align*}
which is finite since $p<(N+1)/(N+1-2s)<N/(N-2s)$ implies $N-1-(N-2s)p>-1$ 
and $p\geq 1>N/(N+2-2s)$ implies $2p-N+1+(N-2s)p>1$.

We are left with the third integral which is controlled by
\begin{align*}
& \int_0^{1/y_N}d\rho\;\rho^{N-2}\int_{3/2}^{1/y_N}dt\;\frac{t^{p+1}}{\left[\rho^2+t^2\right]^{(N+2-2s)p/2}}\ \leq\\
& \leq\ \int_0^{1/y_N}d\rho\;\rho^{N-(N+1-2s)p}\int_{3/(2\rho)}^{1/(y_N\rho)}dt\;\frac{t^{p+1}}{\left[1+t^2\right]^{(N+2-2s)p/2}}\\
& \lesssim\ \int_0^{1/y_N}d\rho\;\rho^{N-(N+1-2s)p}.
\end{align*}
The exponent $N-(N+1-2s)p>-1$ since $p<(N+1)/(N+1-2s)$, so this third integral is bounded above by a constant multiple of
$y_N^{-1-N+(N+1-2s)p}$ which simplifies with the factor in front of \eqref{4576}.
\end{proof}

\section{boundary behaviour}\label{bb-sect}

We first provide the boundary behaviour of the reference function $h_1$. Afterwards, in Proposition \ref{bound-cont} below, we will deal with the weighted trace left on the boundary by harmonic functions induced by continuous boundary data.

\begin{lem} Let $h_1$ be given by \eqref{h1}. There exists a constant $C=C(N,\Omega,s)>0$ such that
\begin{equation}\label{h1-behav}
\frac1C \delta^{-(2-2s)}\le h_1\le C \delta^{-(2-2s)}.
\end{equation}
\end{lem}
\begin{proof}
Restrict without loss of generality to the case where $x$ lies in a neighbourhood of $\partial\Omega$.
Take $x^*\in\partial\Omega$ such that $|x-x^*|=\delta(x)$, 
which exists by compactness of $\partial\Omega$.
Take $\Gamma\subset\partial\Omega$ a neighbourhood of $x^*$ in the topology of $\partial\Omega$. By Lemma \ref{lemma39} in the Appendix,
we can think of $\Gamma\subset\{x_N=0\}$, $x^*=0$ and $x=(0,\delta(x))\in\R^{N-1}\times\R$ without loss of generality.
in such a way that it is possible to compute
\[
\int_\Gamma \frac{\delta(x)}{|x-z|^{N+2-2s}}\;d\HH(z)
\asymp\int_{B_r}\frac{\delta(x)}{[|z'|^2+\delta(x)^2]^{N/2+1-s}}\;dz'.
\]
Recalling \eqref{poissbound}, we have reduced the estimate to 
\begin{align*}
& \int_{\partial\Omega}P^s_{|\Omega}(x,z)\;d\HH(z)\ \asymp \\
& \asymp\ \int_{B_r}\frac{\delta(x)}{[|z'|^2+\delta(x)^2]^{N/2+1-s}}\;dz'
\ \asymp\ \int_0^r\frac{\delta(x)\,t^{N-2}}{[t^2+\delta(x)^2]^{N/2+1-s}}\;dt \\
& =\ \int_0^{r/\delta(x)}\frac{\delta(x)^{N}t^{N-2}}{[\delta(x)^2t^2+\delta(x)^2]^{N/2+1-s}}\;dt\ \asymp\ \delta(x)^{2s-2}\int_0^{r/\delta(x)}\frac{t^{N-2}\;dt}{[t^2+1]^{N/2+1-s}}
\end{align*}
and this concludes the proof, since
\[
\int_0^{r/\delta(x)}\frac{t^{N-2}\;dt}{[t^2+1]^{N/2+1-s}}\asymp 1.
\]
\end{proof}

In the following we will use the notation
\[
\mathbb{P}_{|\Omega}^sg\ :=\ \int_{\partial\Omega}P_{|\Omega}^s(\cdot,\theta)\,g(\theta)\;d\HH(\theta)
\]
where $\sigma$ denotes the Hausdorff measure on $\partial\Omega$,
whenever $g\in L^1(\Omega)$.

\begin{prop}\label{bound-cont} Let $\zeta\in C(\partial\Omega)$. Then, for any $z\in\partial\Omega$,
\begin{equation}
\frac{\mathbb{P}_{|\Omega}^s \zeta(x)}{h_1(x)}\ \xrightarrow[x\to z]{x\in\Omega} \ \zeta(z)\qquad\hbox{uniformly on }\partial\Omega.
\end{equation}
\end{prop}
\begin{proof}
Let us write
\begin{multline*}
\left|\frac{\mathbb{P}_{|\Omega}^s\zeta(x)}{h_1(x)}-\zeta(z)\right|=
\left|\frac1{h_1(x)}\int_{\partial\Omega}P^s_{|\Omega}(x,\theta)\,\zeta(\theta)\;d\HH(\theta)-\frac{h_1(x)\,\zeta(z)}{h_1(x)}\right|\ \leq\\
\leq\ \frac1{h_1(x)}\int_{\partial\Omega}P^s_{|\Omega}(x,\theta)|\zeta(\theta)-\zeta(z)|\;d\HH(\theta)\leq
C\delta(x)^{3-2s}\int_{\partial\Omega}\frac{|\zeta(\theta)-\zeta(z)|}{|x-\theta|^{N+2-2s}}\;d\HH(\theta)\ \leq \\
\leq\ C\delta(x)\int_{\partial\Omega}\frac{|\zeta(\theta)-\zeta(z)|}{|x-\theta|^N}\;d\HH(\theta).
\end{multline*}
It suffices now to repeat the computations in 
Lemma \ref{Eu}
to show that the obtained quantity 
converges to $0$ as $x\to z$.
\end{proof}

With an approximation argument started from the last Proposition,
we can deal with a $\zeta\in L^1(\partial\Omega)$ datum.

\begin{theo}\label{bound-l1} For any $\zeta\in L^1(\partial\Omega)$
and any $\phi\in C^0(\super\Omega)$ it holds
\[
\frac1t\int_{\{\delta(x)\leq t\}}
\frac{\mathbb{P}_{|\Omega}^s\zeta(x)}{h_1(x)}\,\phi(x)\;dx
\xrightarrow[t\downarrow 0]{} \int_{\partial\Omega}\phi(y)\,\zeta(y)\;d\HH(y).
\]
\end{theo}
\begin{proof}
For a general $\zeta\in L^1(\partial\Omega)$, consider a sequence
${\{\zeta_k\}}_{k\in\N}\subset C(\partial\Omega)$ such that 
\begin{equation}\label{topos}
\int_{\partial\Omega}\left|\zeta_k(y)-\zeta(y)\right|d\HH(y)\xrightarrow[k\uparrow\infty]{} 0.
\end{equation}
For any fixed $k\in\N$, we have
\begin{align}
& \left| \frac1t\int_{\{\delta(x)<t\}}\frac{\mathbb{P}_{|\Omega}^s\zeta(x)}{h_1(x)}\,\phi(x)\:dx 
- \int_{\partial\Omega}\phi(x)\,\zeta(x)\:d\HH(x) \right| \leq \nonumber \\
& \left|\frac1t\int_{\{\delta(x)<t\}}\frac{\mathbb{P}_{|\Omega}^s\zeta(x)-\mathbb{P}_{|\Omega}^s\zeta_k(x)}{h_1(x)}\,\phi(x)\:dx \right| \label{1first}\\
& +\ \left| \frac1t\int_{\{\delta(x)<t\}}\frac{\mathbb{P}_{|\Omega}^s\zeta_k(x)}{h_1(x)}\,\phi(x)\:dx 
- \int_{\partial\Omega}\phi(x)\,\zeta_k(x)\:d\HH(x) \right| \label{2second} \\
& +\ \left|\int_{\partial\Omega}\phi(x)\,\zeta_k(x)\:d\HH(x)-\int_{\partial\Omega}\phi(x)\,\zeta(x)\:d\HH(x) \right|. \label{3third}
\end{align}
Call $\lambda_k:=\zeta_k-\zeta$: 
the term \eqref{1first} equals
\[
\frac1t\int_{\{\delta(x)<t\}}\frac{\mathbb{P}_{|\Omega}^s\lambda_k(x)}{h_1(x)}\,\phi(x)\:d\HH(x) =
\int_{\partial\Omega}\left(\frac1t\int_{\{\delta(x)<t\}}\frac{P_{|\Omega}^s(x,y)}{h_1(x)}\,\phi(x)\:dx\right)\lambda_k(y)\:d\HH(y).
\]
Call
\[
\Phi(t,y):=\frac1t\int_{\{\delta(x)<t\}}\frac{P_{|\Omega}^s(x,y)}{h_1(x)}\,\phi(x)\:dx.
\]
Combining equations \eqref{poissbound}, \eqref{h1-behav} and the boundedness of $\phi$,
we can prove that $\Phi$ is uniformly bounded in $t$ and $y$. Indeed,
\[
|\Phi(t,y)|\leq\frac{\|\phi\|_{L^\infty(\Omega)}}{t}
\int_{\{\delta(x)<t\}}\frac{\delta(x)^{3-2s}}{|x-y|^{N+2-2s}}\:dx
\leq
\frac{\|\phi\|_{L^\infty(\Omega)}}{t}
\int_{\{\delta(x)<t\}}\frac{\delta(x)}{|x-y|^N}\:dx
\]
and reducing our attention to the flat case (see Lemma \ref{lemma39} in the Appendix for the complete justification)
we estimate (the ' superscript denotes an object living in $\R^{N-1}$)
\[
\frac1t\int_0^t\int_{B'} \frac{x_N}{\left[|x'|^2+x_N^2\right]^{N/2}}\:dx'\:dx_N
\leq \frac1t\int_0^t\int_{B'_{1/x_N}}\frac{d\xi}{\left[|\xi|^2+1\right]^{N/2}}\:dx_N
\leq \int_{\R^{N-1}}\frac{d\xi}{\left[|\xi|^2+1\right]^{N/2}}.
\]
Thus $\int_{\partial\Omega}\Phi(t,y)\lambda_k(y)d\HH(y)$ 
is arbitrarily small in $k$ in view of \eqref{topos}.

The term \eqref{2second} converges to $0$ as $t\downarrow0$ because the convergence
\[
\frac{\mathbb{P}_{|\Omega}^s\zeta_k(x)}{h_1(x)}\,\phi(x)\xrightarrow[x\to z]{x\in\Omega} \zeta_k(z)\,\phi(z)
\]
is uniform in $z\in\partial\Omega$ in view of Proposition \ref{bound-cont}.

Finally, the term \eqref{3third} is arbitrarily small with $k\uparrow+\infty$, because of \eqref{topos}.
This concludes the proof of the theorem, because 
\begin{multline*}
\lim_{t\downarrow 0}\left| \frac1t\int_{\{\delta(x)<t\}}\frac{\mathbb{P}_{|\Omega}^s\zeta(x)}{h_1(x)}\,\phi(x)\:dx 
- \int_{\partial\Omega}\phi(x)\,\zeta(x)\:d\HH(x) \right|
\ \leq \\
\leq\ \|\Phi\|_{L^\infty((0,t_0)\times\partial\Omega)}\int_{\partial\Omega}|\zeta_k(y)-\zeta(y)|\;d\HH(y)
+\|\phi\|_{L^\infty(\partial\Omega)}\int_{\partial\Omega}|\zeta_k(y)-\zeta(y)|\;d\HH(y)
\end{multline*}
and letting $k\uparrow+\infty$ we deduce the thesis as a consequence of \eqref{topos}.
\end{proof}

Moreover we have also
\begin{theo}\label{bound-g} 
For any $\mu\in\mathcal{M}(\Omega)$, such that 
\begin{equation}\label{mu1}
\int_\Omega\delta\,d|\mu|<\infty,
\end{equation}
and any $\phi\in C^0(\super\Omega)$ it holds
\begin{equation}\label{333}
\frac1t\int_{\{\delta(x)\leq t\}}\frac{\mathbb{G}_{|\Omega}^s\mu(x)}{h_1(x)}\,\phi(x)
\:dx
\ \xrightarrow[t\downarrow 0]{}\ 0.
\end{equation}
\end{theo}
\begin{proof}
By using the Jordan decomposition of $\mu=\mu^+-\mu^-$ into its positive and negative part, we can suppose without loss of generality that $\mu\geq 0$. 
Fix some $s'\in(0,s\wedge 1/2)$. Exchanging the order of integration we claim that
\begin{equation}\label{cla}
\int_{\{\delta(x)\leq t\}}G_{|\Omega}^s(x,y)\,\delta(x)^{2-2s}dx\leq
\left\lbrace\begin{aligned}
& C\,t^{2-2s'}\,\delta(y)^{2s'} & \delta(y)\geq t \\
& C\,t\,\delta(y) & \delta(y)<t
\end{aligned}\right.
\end{equation}
where $C=C(N,\Omega,s)$ and does not depend on $t$, which yields
\[
\frac1t\int_\Omega\left(\int_{\{\delta(x)<t\}}G_{|\Omega}^s(x,y)\,\frac{dx}{h_1(x)}\right)d\mu(y)
\ \leq\ C\,t^{1-2s'}\int_{\{\delta(y)\geq t\}}\delta(y)^{2s'}d\mu(y)
+C\int_{\{\delta(x)<t\}}\delta(y)\:d\mu(y).
\]
The second addend converges to $0$ as $t\downarrow0$ by \eqref{mu1}.
Since $t^{1-2s'}\delta(y)^{2s'}$ converges pointwisely to $0$ in $\Omega$ as $t\downarrow0$ and
$t^{1-2s'}\delta(y)^{2s'}\leq\delta(y)$ in $\{\delta(y)\geq t\}$, 
then the first addend converges to $0$ by dominated convergence.
This suffices to deduce our thesis \eqref{333}.

Let us turn now to the proof of the claimed estimate \eqref{cla}. 
For the first part we refer to \cite[Proposition 7]{dhifli} to say
\[
\int_{\{\delta(x)\leq t\}}G_{|\Omega}^s(x,y)\,\delta(x)^{2-2s}dx\leq
t^{2-2s'}\int_{\{\delta(x)\leq t\}}G_{|\Omega}^s(x,y)\,\delta(x)^{-2s+2s'}dx\leq C\,t^{2-2s'}\,\delta^{2s'}.
\]
We focus our attention on the case where $\partial\Omega$ is locally flat, 
i.e. we suppose that in a neighbourhood $A$ of $y$ it holds $A\cap\partial\Omega\subseteq\{x_N=0\}$
(see Lemma \ref{lemma39} in the Appendix to reduce the general case to this one).
So, since $\delta(x)=x_N$ and retrieving estimate \eqref{green-behav} on the Green function, 
we are dealing with
(the $'$ superscript denotes objects that live in $\R^{N-1}$)
\[
\int_0^t\int_{B'(y')}\frac{y_N\,x_N^{3-2s}}{|x'-y'|^2+(x_N-y_N)^2+x_N\,y_N}
\cdot\frac{dx'}{\left[|x'-y'|^2+(x_N-y_N)^2\right]^{(N-2s)/2}}\;dx_N.
\]
From now on we drop multiplicative constants depending only on $N$ and $s$.
Suppose without loss of generality $y'=0$. Set $x_N=y_N\eta$ and switch to polar coordinates in the $x'$ variable:
\[
\int_0^{t/y_N}\int_0^1\frac{y_N^{5-2s}\,\eta^{3-2s}}{r^2+y_N^2(\eta-1)^2+\eta\,y_N^2}\cdot\frac{r^{N-2}dr}{\left[r^2+y_N^2(\eta-1)^2\right]^{(N-2s)/2}}\;d\eta
\]
and then set $r=y_N\rho$ to get
\begin{multline*}
y_N^2\int_0^{t/y_N}\eta^{3-2s}\int_0^{1/y_N}\frac{\rho^{N-2}}{\left[\rho^2+(\eta-1)^2\right]^{(N-2s)/2}}
\cdot\frac{d\rho}{\rho^2+(\eta-1)^2+\eta}\;d\eta\ \leq\\
\leq\ y_N^2\int_0^{t/y_N}\eta^{3-2s}\int_0^{1/y_N}\frac{\rho}{\left[\rho^2+(\eta-1)^2\right]^{(3-2s)/2}}
\cdot\frac{d\rho}{\rho^2+(\eta-1)^2+\eta}\;d\eta.
\end{multline*}
Consider now $s\in(1/2,1)$. 
The integral in the $\rho$ variable is less than
\[
\int_0^{1/y_N}\frac{\rho}{\left[\rho^2+(\eta-1)^2\right]^{(5-2s)/2}}\;d\rho\leq|\eta-1|^{-3+2s}
\]
so that, integrating in the $\eta$ variable,
\[
y_N^2\int_0^{t/y_N}\eta^{3-2s}\,|\eta-1|^{-3+2s}\:d\eta
\leq t\,y_N
\]
and we prove \eqref{cla} in the case $s\in(1/2,1)$.
Now we study the case $s\in(0,1/2]$. Split the integration in the $\eta$ variable into $\int_0^2$ and $\int_2^{1/y_N}$: the latter can be treated in the same way as above. For the other one we exploit the inequality $\rho^2+(\eta-1)^2+\eta\geq \frac34$ to deduce\footnote{In the computation that follows, in the particular case $s=\frac12$ the term $|1-\eta|^{2s-1}$ must be replaced by $-\ln|1-\eta|$, but this is harmless.}:
\begin{multline*}
y_N^2\int_0^2\eta^{3-2s}\left(\int_0^{1/y_N}\frac{\rho}{\left[\rho^2+(\eta-1)^2+\eta\right]^{(3-2s)/2}}\cdot\frac{d\rho}{\rho^2+(\eta-1)^2}\right)d\eta\ \leq\\
\leq\ y_N^2\int_0^2\eta^{3-2s}\left(\int_0^{1/y_N}\frac{\rho}{\left[\rho^2+(\eta-1)^2\right]^{(3-2s)/2}}\:d\rho\right)d\eta\leq y_N^2\int_0^2\frac{\eta^{3-2s}}{|\eta-1|^{1-2s}}\:d\eta
\leq y_N^2.
\end{multline*}
Note now that, in our set of assumptions, $y_N=\delta(y)<t$. So $y_N^2\leq ty_N$ and we get to the desired conclusion \eqref{cla} also in the case $s\in(0,1/2]$.
\end{proof}

\section{the dirichlet problem}\label{dir-sect}
Recall the definition of test functions \eqref{test}.

\begin{lem}\label{lemma-test-spectral} $\T(\Omega)\subseteq C^1_0(\overline{\Omega})\cap C^\infty(\Omega)$. 
Moreover, for any $\psi\in\T(\Omega)$ and $z\in\partial\Omega$,
\begin{equation}\label{dertest}
-\frac{\partial\psi}{\partial\nu}(z)\ =\ \int_\Omega P_{|\Omega}^s(y,z)\,\As\psi(y)\;dy.
\end{equation}
\end{lem}
\begin{proof}
Take $\psi\in\T(\Omega)$ and let $f=\As\psi$. Since $f\in C^\infty_c(\Omega)$, the spectral coefficients of $f$ have fast decay (see \eqref{spectral-coef}) and so the same holds true for $\psi$. It follows that $\psi\in C^1_0(\overline\Omega)$ and $\T(\Omega)\subseteq C^1_0(\overline{\Omega})$.
By Lemma \ref{inverse}, for all $x\in\overline\Omega$,
\[
\psi(x) = \int_\Omega G_{|\Omega}^s(x,y)\,f(y)\:dy
\]
Using Lemma \ref{lemma8}, \eqref{prepois} and the dominated convergence theorem, \eqref{dertest} follows.

Since $\As$ is self-adjoint in $H(2s)$, we know that the equality $\As\psi=f$ holds in $\mathcal{D}'(\Omega)$ and the interior regularity follows from Lemma \ref{clas-harm}.
\end{proof}

\begin{lem}[maximum principle for classical solutions]\label{maxprinc-clas}
Consider $u\in C^{2s+\eps}(\Omega)\cap L^1(\Omega,\delta(x)\,dx)$ such that
\[
\As u \geq 0\ \hbox{ in }\Omega,\qquad \liminf_{x\to\partial\Omega}u(x)\geq 0.
\]
Then $u\geq 0$ in $\Omega$. In particular this holds when $u\in\T(\Omega)$.
\end{lem}
\begin{proof} Suppose $x^*\in\Omega$ such that 
$\displaystyle u(x^*)=\min_\Omega u<0$. Then
\[
\As u(x^*)=\int_\Omega [u(x^*)-u(y)]\,J(x,y)\;dy + \kappa(x^*) u(x^*) < 0,
\]
a contradiction.
\end{proof}

\begin{lem}[maximum principle for weak solutions] \label{maxprinc-weak}
Let $\mu\in\mathcal{M}(\Omega),\zeta\in\mathcal{M}(\partial\Omega)$ be two Radon measures satisfying \eqref{hypo}
with $\mu\geq 0$ and $\zeta\geq 0$.
Consider $u\in L^1_{loc}(\Omega)$ 
a weak solution to the Dirichlet problem \eqref{prob}. Then $u\geq 0$ a.e. in $\Omega$. 
\end{lem}
\begin{proof}
Take $f\in C^\infty_c(\Omega)$, $f\ge0$ and $\psi=\Dms f\in\T(\Omega)$. By Lemma \ref{maxprinc-clas}, $\psi\geq 0$ in $\Omega$ and by Lemma \ref{lemma-test-spectral} $-\frac{\partial\psi}{\partial\nu}\geq 0$ on $\partial\Omega$. Thus, by \eqref{byparts-spectral}, $\int_\Omega uf\geq 0$. Since this is true for every $f\in C^\infty_c(\Omega)$, the result follows.
\end{proof}

\subsection{Proof of Theorem \ref{point}}
Uniqueness is a direct consequence of the Maximum Principle, Lemma \ref{maxprinc-weak}.
Let us prove that formula \eqref{repr} defines the desired weak solution. Observe that if $u$ is given by \eqref{repr}, then $u\in L^1(\Omega,\delta(x)dx)$. Indeed, 

\begin{multline}\label{L1delta}
\int_\Omega\left|\varphi_1(x) \int_\Omega G^s_{|\Omega}(x,y)d\mu(y)\right|dx
\leq\int_\Omega \int_\Omega G_{|\Omega}^s(x,y)\varphi_1(x)dx\;d|\mu|(y)
\\=\frac1{\lambda_1^s}\int_\Omega\varphi_1(y)\;d|\mu|(y)\le C \Vert \delta\mu\Vert_{\mathcal M(\Omega)}.
\end{multline}
This, along with Lemma \ref{harm}, proves that $u\in L^1(\Omega,\delta(x)dx)$ and \eqref{cont}.
Now, pick $\psi\in\T(\Omega)$ and compute, via the Fubini's Theorem, Lemma \ref{inverse} and Lemma \ref{lemma-test-spectral},
\begin{multline*}
\int_\Omega u(x)\As\psi(x)\;dx = \\
=\ \int_\Omega\left(\int_\Omega G_{|\Omega}^s(x,y)d\mu(y)\right)
\As\psi(x)\;dx 
+\int_\Omega\left(\int_{\partial\Omega}P_{|\Omega}^s(x,z)\:d\zeta(z)\right)\As\psi(x)\;dx\ = \\
=\ \int_\Omega\left(\int_\Omega G_{|\Omega}^s(x,y)\As\psi(x)\:dx\right)d\mu(y) + \int_{\partial\Omega}\left(\int_\Omega P^s_{|\Omega}(x,z)\As\psi(x)\;dx\right)d\zeta(z)\ =\\
=\ \int_\Omega\psi(y)\;d\mu(y)-\int_{\partial\Omega}\frac{\partial\psi}{\partial\nu}(z)\;d\zeta(z).
\end{multline*}
\hfill$\square$

\subsection{Proof of Lemma \ref{clas-weak}}

\noindent {\it Proof of 1.}
Consider a sequence $\{\eta_k\}_{k\in\N}\subset C^\infty_c(\Omega)$ of bump functions such that
$0\leq\eta_1\leq\ldots\leq\eta_k\leq\eta_{k+1}\leq\ldots\leq 1$ and 
$\eta_k(x)\uparrow\chi_\Omega(x)$ as $k\uparrow\infty$. 
Consider $\psi\in C^\infty_c(\Omega)$ and define
$f_k:=\eta_k\As\psi\in C^\infty_c(\Omega)$, $\psi_k:=\rest^{-s}f_k\in\T(\Omega)$.

Let us first note that the integral 
\[
\int_\Omega u\,\As\psi
\]
makes sense in view of \eqref{ineq-lemma1} and \eqref{cont}. The sequence ${\{f_k\}}_{k\in\N}$ trivially converges a.e. to $\As\psi$, while
\[
|\psi_k(x)-\psi(x)|\leq\int_\Omega G_{|\Omega}^s(x,y)|\As\psi(y)|
\cdot(1-\eta_k(y))\;dy
\]
converges to 0 for any $x\in\Omega$ by dominated convergence.
Since $u$ is a weak solution, it holds
\[
\int_\Omega u\,f_k\ =\ \int_\Omega\psi_k\;d\mu-\int_{\partial\Omega}\frac{\partial\psi_k}{\partial\nu}\;d\zeta.
\]
By dominated convergence $\int_\Omega u\,f_k\to \int_\Omega u\As\psi$ and $\int_\Omega\psi_k\:d\mu\to\int_\Omega\psi\:d\mu$.
Indeed for any $k\in\N$,
\[
|f_k|\leq |\As\psi|\qquad\hbox{and}\qquad
|\psi_k|\leq\frac1{\lambda_1^s}\left\|\frac{\As\psi}{\varphi_1}\right\|_{L^\infty(\Omega)}\varphi_1,
\]
where the latter inequality follows from the maximum principle or the representation formula $\psi_k(x)=\int_\Omega G_{|\Omega}^s(x,y)f_k(y)dy$.
Finally, the convergence 
\[
\int_{\partial\Omega}\frac{\partial\psi_k}{\partial\nu}\;d\zeta\xrightarrow[k\uparrow\infty]{}0
\]
holds by dominated convergence, since
\[
\int_{\partial\Omega}\frac{\partial\psi_k}{\partial\nu}\;d\zeta
=\int_\Omega\mathbb{P}_{|\Omega}^s\zeta\,f_k
\]
by the Fubini's Theorem, and $\mathbb{P}_{|\Omega}^s\zeta\in L^1(\Omega,\delta(x)dx)$ while $|f_k|\leq |\As\psi|\leq C\delta$
by Lemma \ref{As-smooth}, so that
\[
\int_\Omega\mathbb{P}_{|\Omega}^s\zeta\,f_k
\xrightarrow[k\uparrow\infty]{}
\int_\Omega\mathbb{P}_{|\Omega}^s\zeta\,f\ =\ 0
\]
because $\mathbb{P}_{|\Omega}^s\zeta$ is $s$-harmonic and $f\in\ C^\infty_c(\Omega)$.

The proof of the boundary trace can be found in Theorems \ref{bound-l1} and \ref{bound-g}, by recalling the representation formula provided by Theorem \ref{point} for the solution to \eqref{prob}.
\hfill$\square$

\noindent {\it Proof of 2.}
Recall that $u$ is represented by
\[
u(x)=\int_\Omega G_{|\Omega}^s(x,y)\,\mu(y)\:dy
+\int_{\partial\Omega}P_{|\Omega}^s(x,y)\,\zeta(y)\:d\HH(y).
\]
By Point 1. and Lemma \ref{clas-harm}, $u\in C^{2s+\alpha}_{loc}(\Omega)$.
Moreover, $u\in L^1(\Omega,\delta(x)dx)$ thanks to \eqref{cont}. So,
we can compute pointwise $\As u$ by using \eqref{As2} and \eqref{def2}:
this entails by the self-adjointness of the operator in \eqref{def2} that
\[
\int_\Omega \As u\,\psi = \int_\Omega u\As\psi = \int_\Omega \mu\psi,
\qquad\text{for any }\psi\in C^\infty_c(\Omega)
\]
and we must conclude that $\As u=\mu$ a.e.
By continuity the equality holds everywhere.

We turn now to the boundary trace. The contribution given by $\mathbb G_{|\Omega}^s\mu$
is irrelevant, because it is a bounded function as it follows from 
\[
|\mathbb G_{|\Omega}^s\mu(x)|\leq C\|\mu\|_{L^\infty(\Omega)}\int_\Omega\frac{dy}{\left|x-y\right|^{N-2s}}
\]
where we have used \eqref{green-behav}.
Therefore, by Proposition \ref{bound-cont}, there also holds for all $x_0\in\partial\Omega$,
\[
\lim_{x\to x_0, x\in\Omega} \frac{u(x)}{h_1(x)}
= \lim_{x\to x_0, x\in\Omega} \frac{\mathbb{G}_{|\Omega}^s\mu(x)+\mathbb{P}_{|\Omega}^s\zeta(x)}{h_1(x)}
= \zeta(x_0).
\]
\hfill$\square$

\noindent {\it Proof of 3.} By Lemma \ref{u-to-lapl} $\mu\in C^\eps_{loc}(\Omega)$. In addition, we have assumed that $\zeta\in C(\partial\Omega)$. Consider 
\[
v(x)=\int_\Omega G_{|\Omega}^s(x,y)\,\mu(y)\;dy+\int_{\partial\Omega}P_{|\Omega}^s(x,z)\,\zeta(z)\;d\HH(z)
\]
the weak solution associated to data $\mu$ and $\zeta$. 
By the previous point of the Lemma, $v$ is a classical solution to the equation, 
so that in a pointwise sense it holds
\[
\As(u-v)=0\hbox{ in }\Omega,\qquad \frac{u-v}{h_1}=0\hbox{ on }\partial\Omega.
\]
By applying Lemma \ref{maxprinc-clas} we conclude that $|u-v|\leq \eps h_1$ for any $\eps>0$ and thus $u-v\equiv 0$. 
\hfill$\square$

\chapter{Nonlinear problems and large solutions for the power nonlinearity}\label{specKO}

\section{the nonlinear problem}\label{nonlin-sect}

\begin{lem}[kato's inequality]\label{kato} For $f\in L^1(\Omega,\delta(x)dx)$ let $w\in L^1(\Omega,\delta(x)dx)$
weakly solve 
\[
\left\lbrace\begin{aligned}
\As w &= f & & \hbox{in }\Omega \\
\frac{w}{h_1} &= 0 & & \hbox{on }\partial\Omega.
\end{aligned}\right.
\]
For any convex $\Phi:\R\to\R$, $\Phi\in C^2(\R)$ such that $\Phi(0)=0$
and $\Phi(w)\in L^1_{loc}(\Omega)$, it holds
\[
\As\Phi(w)\leq\Phi'(w)\As w.
\]
Moreover, the same holds for $\Phi(t)=t^+=t\wedge 0$.
\end{lem}
\begin{proof}
Let us first assume that $f\in C^\alpha_{loc}(\Omega)$.
In this case, by Lemma \ref{clas-harm}, $w\in C^{2s+\alpha}_{loc}(\Omega)$
and the equality $\As w=f$ holds in a pointwise sense. Then
\begin{align*}
\As\Phi\circ w(x) & =\ \int_\Omega[\Phi(w(x))-\Phi(w(y))]\,J(x,y)\:dy+\kappa(x)\,\Phi(w(x)) \\
& =\ \Phi'(w(x))\int_\Omega[w(x)-w(y)]\,J(x,y)\:dy+\kappa(x)\,\Phi(w(x)) \\
& \ \ -\int_\Omega[w(x)-w(y)]^2\,J(x,y)
\int_0^1\Phi''(w(x)+t[w(y)-w(x)])(1-t)\:dt\:dy \\
&\leq\ \Phi'(w(x))\,\As w(x)
\end{align*}
where we have used that $\Phi''\geq 0$ in $\R$ and 
that $\Phi'(t)\leq t\Phi(t)$, which follows from $\Phi(0)=0$.

We deal now with $f\in L^\infty(\Omega)$.
Pick ${\{f_j\}}_{j\in\N}\subseteq C^\infty_c(\Omega)$ converging to $f$
in $L^1(\Omega,\delta(x),dx)$ and bounded in $L^\infty(\Omega)$. 
The corresponding ${\{w_j=\mathbb{G}_\Omega^s f_j\}}_{j\in\N}$
converges to $w$ in $L^1(\Omega,\delta(x)dx)$,
is bounded in $L^\infty(\Omega)$ and without loss of generality
we assume that $f_j\to f$ and $w_j\to w$ a.e. in $\Omega$.
We know that for any $\psi\in \T(\Omega),\ \psi\geq 0$
\[
\int_\Omega\Phi(w_j)\,\As\psi\leq\int_\Omega f_j\,\Phi'(w_j)\psi.
\]
By the continuity of $\Phi$ and $\Phi'$ we have $\Phi(w_j)\to\Phi(w)$, $\Phi'(w_j)\to\Phi'(w)$ a.e. in $\Omega$ and that ${\{\Phi(w_j)\}}_{j\in\N}$, ${\{\Phi'(w_j)\}}_{j\in\N}$ are bounded in $L^\infty(\Omega)$. Since ${\{f_j\}}_{j\in\N}$ is converging to $f$ in $L^1(\Omega,\delta(x)dx)$, then 
\[
\int_\Omega\Phi(w_j)\,\As\psi\longrightarrow
\int_\Omega\Phi(w)\,\As\psi
\qquad\hbox{and}\qquad
\int_\Omega f_j\,\Phi'(w_j)\psi\longrightarrow
\int_\Omega f\,\Phi'(w)\psi
\]
by dominated convergence.

For a general $f\in L^1(\Omega,\delta(x)dx)$ define $f_{j,k}:=(f\wedge j)\vee(-k)$, $j,k\in\N$.
Also, split the expression of $\Phi=\Phi_1-\Phi_2$ into the difference of two increasing function: this can be done in
the following way. The function $\Phi'$ is continuous and increasing in $\R$,
so that it can either have constant sign or there exists $t_0\in\R$ such that $\Phi'(t_0)=0$.
If it has constant sign than $\Phi$ can be increasing or decreasing and 
we can choose respectively $\Phi_1=\Phi,\Phi_2=0$ or $\Phi_1=0,\Phi_2=-\Phi$.
Otherwise we can take
\[
\Phi_1(t)=\left\lbrace\begin{aligned}
& \Phi(t) & t>t_0 \\
& \Phi(t_0) & t\leq t_0
\end{aligned}\right.
\qquad\hbox{and}\qquad
\Phi_2(t)=\left\lbrace\begin{aligned}
& 0 & t>t_0 \\
& \Phi(t_0)-\Phi(t) & t\leq t_0
\end{aligned}\right. .
\]
We already know that
for any $\psi\in \T(\Omega),\ \psi\geq 0$
\[
\int_\Omega\Phi(w_{j,k})\,\As\psi\leq\int_\Omega f_{j,k}\,\Phi'(w_{j,k})\psi.
\]
On the right-hand side we can use twice the monotone convergence, 
letting $j\uparrow\infty$ first and then $k\uparrow\infty$.
On the left hand side, by writing $\Phi=\Phi_1-\Phi_2$
again we can exploit several times the monotone convergence by splitting
\begin{multline*}
\int_\Omega\Phi(w_{j,k})\,\As\psi = \int_\Omega\Phi_1(w_{j,k})\,[\As\psi]^+-\int_\Omega\Phi_1(w_{j,k})\,[\As\psi]^-\ +\\
-\ \int_\Omega\Phi_2(w_{j,k})\,[\As\psi]^++\int_\Omega\Phi_2(w_{j,k})\,[\As\psi]^-
\end{multline*}
to deduce the thesis.

Finally, note that $\Phi(t)=t^+$ can be monotonically
approximated by 
\[
\Phi_j(t)=\frac12\sqrt{t^2+\frac1{j^2}}+\frac t2 -\frac1{2j}
\]
which is convex, $C^2$ and $\Phi_j(0)=0$. So 
\[
\int_\Omega\Phi_j(w)\,\As\psi\leq\int_\Omega f\,\Phi_j'(w)\psi.
\]
Since $\Phi_j(t)\uparrow t^+$ and $2\Phi_j'(t)\uparrow 1+\hbox{sgn}(t)=2\chi_{(0,+\infty)}(t)$,
we prove the last statement of the Lemma.
\end{proof}

\begin{theo}\label{exist-nonlinspec} 
Let $f(x,t):\Omega\times\R\longrightarrow\R$ be a Carath\'eodory function. 
Assume that there exists a subsolution and a supersolution $\underline u,\overline u\in L^1(\Omega,\delta(x)dx)\cap L^\infty_{loc}(\Omega)$ 
to
\begin{equation}\label{098}
\left\lbrace\begin{aligned}
\As u &= f(x,u) & \hbox{ in }\Omega, \\
\frac{u}{h_1} &= 0 & \hbox{ on }\partial\Omega.
\end{aligned}\right.
\end{equation}
Assume in addition that $f(\cdot,v)\in L^1(\Omega,\delta(x)dx)$ for every $v\in L^1(\Omega,\delta(x)dx)$ such that $\underline u\le v\le \overline u$ a.e.
Then, there exist weak solutions $u_1,u_2\in L^1(\Omega,\delta(x)dx)$ in $[\underline u,\overline u]$ 
such that any solution $u$ in the interval $[\underline u,\overline u]$ satisfies
\[
\underline u\le u_1\le u\le u_2\le \overline u \qquad\text{a.e.}
\]
Moreover, if the nonlinearity $f$ is decreasing in the second variable, then the solution is unique.
\end{theo}
\begin{proof}
According to Montenegro and Ponce \cite{ponce}, the mapping $v\mapsto F(\cdot,v)$, where
\[
F(x,t):=f(x,[t\wedge\overline u(x)]\vee\underline u(x)),\qquad x\in\Omega,t\in\R,
\]
acts continuously from $L^1(\Omega,\delta(x)dx)$ into itself.
In addition, the operator 
\begin{eqnarray*}
\mathcal K:\ L^1(\Omega,\delta(x)dx) & \longrightarrow &  L^1(\Omega,\delta(x)dx) \\
 v(x)\qquad & \longmapsto & \mathcal K(v)(x) = \int_\Omega G_{|\Omega}^s(x,y) F(y,v(y))dy
\end{eqnarray*}
is compact. Indeed, take a bounded sequence ${\{v_n\}}_{n\in\N}$ in $L^1(\Omega,\delta(x)dx)$. On a compact set  $K\subset\subset\Omega$, $\underline u, \overline u$ are essentially bounded and so must be the sequence ${\{F(\cdot,v_n)\}}_{n\in\N}$. By Theorem \ref{point} and Lemma \ref{clas-harm}, ${\{\mathcal K(v_n)\}}_{n\in\N}$ is bounded in $C^\alpha_{loc}(K)\cap L^p(\Omega,\delta(x)dx)$, $p\in[1,(N+1)/(N+1-2s))$. In particular, a subsequence ${\{v_{n_k}\}}_{k\in\N}$ converges locally uniformly to some $v$. By H\"older's inequality, we also have 
\[
\Vert v_{n_k}-v \Vert_{L^1(\Omega\setminus K,\delta(x)dx)} \le
\Vert v_{n_k}-v \Vert_{L^p(\Omega\setminus K,\delta(x)dx)} \Vert \chi_{\Omega\setminus K}\Vert_{L^{p'}(\Omega\setminus K,\delta(x)dx)}.
\]
Hence,
\[
\Vert v_{n_k}-v \Vert_{L^1(\Omega,\delta(x)dx)}\le \Vert v_{n_k}-v \Vert_{L^\infty(K,\delta(x)dx)}\Vert\delta\Vert_{L^1(\Omega)}+C\Vert \chi_{\Omega\setminus K}\Vert_{L^{p'}(\Omega\setminus K,\delta(x)dx}.
\]
Letting $k\to+\infty$ and then $K\to\Omega$, we deduce that $\mathcal K$ is compact and by the Schauder's Fixed Point Theorem, $\mathcal K$ has a fixed point $u\in L^1(\Omega,\delta(x)dx)$. We then may prove that $\underline u\leq u\leq \overline u$ by means of the Kato's Inequality (Lemma \ref{kato}) as it is done in \cite{ponce},
which yields that $u$ is a solution of \eqref{098}.

The proof of the existence of the minimal and a maximal solution $u_1,u_2\in L^1(\Omega,\delta(x)dx)$ can be performed in an analogous way as in \cite{ponce}, as the only needed tool is the Kato's Inequality.

As for the uniqueness, suppose $f$ is decreasing in the second variable and consider two solutions $u,v\in L^1(\Omega,\delta(x)dx)$ to \eqref{098}. By the Kato's Inequality Lemma \ref{kato}, we have
\[
\As(u-v)^+\leq \chi_{\{u>v\}}[f(x,u)-f(x,v)]\leq 0
\qquad\hbox{in }\Omega
\]
which implies $(u-v)^+\leq 0$ by the Maximum Principle Lemma \ref{maxprinc-weak}. Reversing the roles of $u$ and $v$, we get
also $(v-u)^+\leq 0$, thus $u\equiv v$ in $\Omega$.
\end{proof}

\subsection{Proof of Theorem \ref{nonhom-cor}}

Problem \eqref{nonhom-prob} is equivalent to
\begin{equation}
\left\lbrace\begin{aligned}
\As v &= g(x,\mathbb P_{|\Omega}^s\zeta-v) & \hbox{ in }\Omega \\
\frac{v}{h_1} &= 0 & \hbox{ on }\partial\Omega
\end{aligned}\right.
\end{equation}
that possesses $\overline u=\mathbb P_{|\Omega}^s\zeta$ 
as a supersolution and $\underline u=0$ as a subsolution.
Indeed, by equation \eqref{h1-behav} we have
\[
0\leq\mathbb{P}_{|\Omega}^s\zeta\leq \|\zeta\|_{L^\infty(\Omega)}h_1\leq C\|\zeta\|_{L^\infty(\Omega)}\delta^{-(2-2s)}.
\]
Thus any $v\in L^1(\Omega,\delta(x)dx)$ such that $0\leq v\leq\mathbb{P}_{\Omega}^s\zeta$ satisfies
\[
g(x,v)\leq h(v)\leq h(c\delta^{-(2-2s)})\in L^1(\Omega,\delta(x)dx).
\]
So, all hypotheses of Theorem \ref{exist-nonlinspec} are satisfied and the result follows.
\hfill$\square$

\section{large solutions}\label{extra-sect}

Consider the sequence ${\{u_j\}}_{j\in\N}$ built by solving
\begin{equation}\label{approx}
\left\lbrace\begin{aligned}
\As u_j &=\; -u_j^p & \hbox{ in }\Omega \\
\frac{u_j}{h_1} &=\; j & \hbox{ on }\partial\Omega.
\end{aligned}\right.
\end{equation}
Theorem \ref{nonhom-cor} guarantees the existence of such a sequence if
$\delta^{-(2-2s)p}\in L^1(\Omega,\delta(x)dx)$, i.e. $p<1/(1-s)$.
We claim that the sequence is increasing in $\Omega$: indeed the solution
to problem \eqref{approx} is a subsolution for the same problem with boundary datum $j+1$.
In view of this, the sequence $\{u_j\}_{j\in\N}$ admits a pointwise limit, possibly infinite.


\subsection{Construction of a supersolution}

\begin{lem} There exist $\delta_0,C>0$ such that 
\[
\As\delta^{-\alpha}\ \geq\ -C\,\delta^{-\alpha p},
\qquad\hbox{for }\delta<\delta_0\hbox{ and }\alpha=\frac{2s}{p-1}.
\]
\end{lem}
\begin{proof}
We use the expression in equation \eqref{As2}. Obviously,
\[
\As\delta^{-\alpha}(x)=\int_\Omega[\delta(x)^{-\alpha}-\delta(y)^{-\alpha}]J(x,y)\;dy
+\delta(x)^{-\alpha}\kappa(x) \geq \int_\Omega[\delta(x)^{-\alpha}-\delta(y)^{-\alpha}]J(x,y)\;dy.
\]
For any fixed $x\in\Omega$ close to $\partial\Omega$, split the domain $\Omega$ into three parts:
\begin{align*}
& \Omega_1=\left\lbrace y\in\Omega:\delta(y)\geq\frac32\,\delta(x)\right\rbrace,\\
& \Omega_2=\left\lbrace y\in\Omega:\frac12\,\delta(x)<\delta(y)<\frac32\,\delta(x)\right\rbrace,\\
& \Omega_3=\left\lbrace y\in\Omega:\delta(y)\leq\frac12\,\delta(x)\right\rbrace.
\end{align*}
For $y\in\Omega_1$, since $\delta(y)>\delta(x)$, it holds $\delta(x)^{-\alpha}-\delta(y)^{-\alpha}>0$ and we can drop the integral on $\Omega_1$.
Also, since it holds by equation \eqref{Jbound}
\[
J(x,y)\leq \frac{C}{{|x-y|}^{N+2s}},
\]
the integration on $\Omega_2$ can be performed as in {\it Second step} in Proposition \ref{impo} providing
\[
PV\int_{\Omega_2}[\delta(x)^{-\alpha}-\delta(y)^{-\alpha}]\,
J(x,y)\:dy\geq -C\,\delta(x)^{-\alpha-2s}=-C\,\delta(x)^{-\alpha p}.
\]
To integrate on $\Omega_3$ we exploit once again \eqref{Jbound} under the form
\[
J(x,y)\leq C\cdot\frac{\delta(x)\,\delta(y)}{{|x-y|}^{N+2+2s}}
\]
to deduce
\[
\int_{\Omega_3}[\delta(x)^{-\alpha}-\delta(y)^{-\alpha}]\,
J(x,y)\:dy\geq -C\,\delta(x)\int_{\Omega_3}\frac{\delta(y)^{1-\alpha}}{{|x-y|}^{N+2+2s}}.
\]
Again, a direct computation as in {\it Third step} in Proposition \ref{impo} yields
\[
\int_{\Omega_3}[\delta(x)^{-\alpha}-\delta(y)^{-\alpha}]\,
J(x,y)\:dy\geq -C\,\delta(x)^{-\alpha-2s}=-C\,\delta(x)^{-\alpha p}.
\]
\end{proof}

\begin{lem}\label{prel-supersol}
If a function $v\in L^1(\Omega,\delta(x)dx)$ satisfies
\begin{equation}
\As v\in L^\infty_{loc}(\Omega),\qquad
\As v(x)\ \geq\ -C\,v(x)^p,\quad \hbox{when }\delta(x)<\delta_0,
\end{equation}
for some $C,\delta_0>0$, then there exists $\overline u\in L^1(\Omega,\delta(x)dx)$
such that
\begin{equation}
\As\overline u(x)\ \geq\ -\overline u(x)^p,\qquad \hbox{throughout }\Omega.
\end{equation}
\end{lem}
\begin{proof} Let $\lambda:=C^{1/(p-1)}\vee 1$ and $\Omega_0=\{x\in\Omega:\delta(x)<\delta_0\}$, then
\[
\As\left(\lambda v\right)\ \geq\ -\left(\lambda v\right)^p,\qquad \hbox{in }\Omega_0.
\]
Let also $\mu:=\lambda\|\As v\|_{L^\infty(\Omega\setminus\Omega_0)}$ and define $\overline u=\mu\mathbb{G}_{|\Omega}^s 1+\lambda v$. On $\overline u$ we have
\[
\As\overline u =\mu+\lambda \As v\geq \lambda|\As v|+\lambda\As v\geq -\overline{u}^p
\qquad\hbox{throughout }\Omega.
\]
\end{proof}

\begin{cor}\label{supersol} There exists a function $\overline{u}\in L^1(\Omega,\delta(x)dx)$
such that the inequality
\[
\As\overline u\ \geq\ -\overline{u}^p,\qquad \hbox{in }\Omega,
\]
holds in a pointwise sense.
Moreover, $\overline u\asymp\delta^{-2s/(p-1)}$.
\end{cor}
\begin{proof}
Apply Lemma \ref{prel-supersol} with $v=\delta^{-2s/(p-1)}$:
the corresponding $\overline u$ will be of the form
\[
\overline u=\mu\mathbb{G}_{|\Omega}^s1+\lambda\delta^{-2s/(p-1)}.
\]
\end{proof}

\subsection{Existence}

\begin{lem}\label{ujlequ} 
For any $j\in\N$, the solution $u_j$ to problem \eqref{approx} satisfies the upper bound
\[
u_j\ \leq\ \overline{u},\qquad\hbox{in }\Omega,
\]
where $\overline{u}$ is provided by Corollary \ref{supersol}.
\end{lem}
\begin{proof} Write $u_j=jh_1-v_j$ where
\[
\left\lbrace\begin{aligned}
\As v_j &=\ \left(jh_1-v_j\right)^p & \hbox{ in }\Omega \\
\frac{v_j}{h_1} &=\ 0 & \hbox{ on }\partial\Omega.
\end{aligned}\right.
\]
and $0\leq v_j\leq jh_1$. Since $\left(jh_1-v_j\right)^p\in L^\infty_{loc}(\Omega)$,
we deduce that $v_j\in C^\alpha_{loc}(\Omega)$ for any $\alpha\in(0,2s)$. 
By bootstrapping $v_j\in C^\infty(\Omega)$ and, by Lemma \ref{clas-harm}, also $u_j\in C^\infty(\Omega)$.
This says that $u_j$ is a classical solution to problem \eqref{approx}.
Now, we have that, 
by the boundary behaviour of $\overline{u}$ stated in Corollary \ref{supersol},
$u_j\leq \overline{u}$ close enough to $\partial\Omega$ (depending on the value of $j$)
and
\[
\As\left(\overline{u}-u_j\right)\ \geq\ u_j^p-\overline{u}^p,\qquad\hbox{in }\Omega.
\]
Since $u_j^p-\overline{u}^p\in C(\Omega)$ and
$\lim_{x\to\partial\Omega}u_j^p-\overline{u}^p=-\infty$, there exists $x_0\in\Omega$
such that $u_j(x_0)^p-\overline{u}(x_0)^p=m=:\max_{x\in\Omega}\left(u_j(x)^p-\overline{u}(x)^p\right)$.
If $m>0$ then $\As(\overline{u}-u_j)(x_0)\geq m > 0$: this is a contradiction,
as Definition \ref{As2} implies. 
Thus $m\leq 0$ and $u_j\leq \overline{u}$ throughout $\Omega$.
\end{proof}

\begin{theo} For any $\displaystyle p\in\left(1+s,\frac1{1-s}\right)$ 
there exists a function $u\in L^1(\Omega,\delta(x)dx)$
solving 
\[
\left\lbrace\begin{aligned}
\As u &=-u^p & \hbox{in }\Omega \\
\delta^{2-2s}u &=+\infty & \hbox{on }\partial\Omega
\end{aligned}\right.
\]
both in a distributional and pointwise sense.
\end{theo}
\begin{proof} 
Consider the sequence $\{u_j\}_{j\in\N}$ provided by problem \eqref{approx}:
it is increasing and locally bounded by Lemma \ref{ujlequ},
so it has a pointwise limit $u\leq\overline{u}$, 
where $\overline u$ is the function provided by Corollary \ref{supersol}. 
Since $p>1+s$ and $\overline u \leq C\delta^{-2s/(p-1)}$, then $u\in L^1(\Omega,\delta(x)dx)$. 
Pick now $\psi\in C^\infty_c(\Omega)$, and recall that $\delta^{-1}\As\psi\in L^\infty(\Omega)$: 
we have, by dominated convergence,
\[
\int_\Omega u_j\As\psi\xrightarrow[j\uparrow\infty]{}\int_\Omega u\As\psi,
\qquad \int_\Omega u_j^p\psi\xrightarrow[j\uparrow\infty]{}\int_\Omega u^p\psi
\]
so we deduce
\[
\int_\Omega u\As\psi\ =\ -\int_\Omega u^p\psi.
\]
Note now that for any compact $K\subset\subset\Omega$, 
applying Lemma \ref{clas-harm} we get for any $\alpha\in(0,2s)$
\[
\|u_j\|_{C^\alpha(K)}\leq C\left(
\|u_j\|^p_{L^\infty(K)}+\|u_j\|_{L^1(\Omega,\delta(x)dx)}\right)
\leq C\left(
\|\overline u\|^p_{L^\infty(K)}+\|\overline u\|_{L^1(\Omega,\delta(x)dx)}\right)
\]
which means that ${\{u_j\}}_{j\in\N}$ is equibounded and equicontinuous in $C(K)$. By the Ascoli-Arzel\`a Theorem, its pointwise limit $u$ will be in $C(K)$ too. Now, since 
\[
\As u\ =\ -u^p\qquad \hbox{in }\mathcal D'(\Omega),
\]
by bootstrapping the interior regularity in Lemma \ref{clas-harm}, we deduce $u\in C^\infty(\Omega)$. So, its spectral fractional Laplacian is pointwise well-defined and the equation is satisfied in a pointwise sense. Also,
\[
\liminf_{x\to\partial\Omega}\frac{u(x)}{h_1(x)}\geq
\liminf_{x\to\partial\Omega}\frac{u_j(x)}{h_1(x)}=j
\]
and we obtain the desired boundary datum.
\end{proof}


\section{appendix}

\subsection{Another representation for the spectral fractional Laplacian}\label{app1}

\begin{lem} For any $u\in H(2s)$ and almost every $x\in\Omega$, there holds
\[\As u(x)\ =  PV\int_\Omega[u(x)-u(y)]J(x,y)\;dy+\kappa(x)u(x),\]
where $J(x,y)$ and $\kappa(x)$ are given by \eqref{jkappa}.
\end{lem}
\begin{proof}
Assume that $u=\varphi_j$ is an eigenfunction of the Dirichlet Laplacian associated to the eigenvalue $\lambda_j$. Then, $\As u =\lambda_j^s u$,
$e^{t\lapl|_\Omega}u=\int_\Omega p_\Omega(t,\cdot,y)\,u(y)\:dy=e^{-t\lambda_j}u$
and for all $x\in\Omega$
\begin{align}
& \frac{\Gamma(1-s)}s\As u(x)\ = \nonumber \\
& =\ \int_0^\infty\left(u(x)-e^{t\left.\lapl\right|_\Omega}u(x)\right)\frac{dt}{t^{1+s}} \label{def2} \\
& =\ \int_0^\infty\left(u(x)-\int_\Omega p_\Omega(t,x,y)u(y)\;dy\right)\frac{dt}{t^{1+s}} \nonumber \\
& =\ \lim_{\epsilon\to0}\int_0^\infty\int_{\Omega\setminus B(x,\epsilon)} p_\Omega(t,x,y)[u(x)-u(y)]\;dy\;\frac{dt}{t^{1+s}}
+\int_0^\infty u(x)\left(1-\int_\Omega p_\Omega(t,x,y)\;dy\right)\;\frac{dt}{t^{1+s}} \nonumber \\
& =\ PV\int_\Omega[u(x)-u(y)]J(x,y)\;dy+\kappa(x)u(x). 
\end{align}
By linearity, equality holds on the linear span of the eigenvectors. Now, if $u\in H(2s)$, a sequence ${\{u_n\}}_{n\in\N}$ of functions belonging to that span  converges to $u$ in $H(2s)$. In particular, $\As u_n$ converges to $\As u$ in $L^2(\Omega)$. Note also that for $v\in L^2(\Omega)$,
\[
\frac{s}{\Gamma(1-s)}\left\vert\int_\Omega\int_0^\infty\left(u(x)-e^{t\left.\lapl\right|_\Omega}u(x)\right)\frac{dt}{t^{1+s}}\cdot v(x)\:dx\right\vert
=\left\vert\sum_{k=1}^{+\infty}\lambda_k^s\widehat u_k \widehat v_k\right\vert
\le\Vert u\Vert_{H(2s)}\Vert v\Vert_{L^2(\Omega)}
\]
so that we may also pass to the limit in $L^2(\Omega)$ 
when computing \eqref{def2} along the sequence ${\{u_n\}}_{n\in\N}$. 
By the Fubini's theorem, for almost every $x\in\Omega$, 
all subsequent integrals are convergent 
and the identities remain valid. 
\end{proof}

\subsection{The reduction to the flat case}

In this paragraph we are going to justify the computation
of the asymptotic behaviour of integrals of the type
\[
\int_{A\cap\Omega} F(\delta(x),\delta(y),|x-y|)\;dy \qquad\hbox{as }\ \delta(x)\downarrow 0,
\]
where $A$ is a fixed neighbourhood of $x$ with $A\cap\partial\Omega\neq\emptyset$, by just looking at
\[
\int_0^1dt\int_Bdy'\; F\!\left(\delta(x),t,\sqrt{|y'|^2+|t-\delta(x)|^2}\right).
\]
The first thing to be proved is that
\[
|x-y|^2\asymp|x_0-y_0|^2+|\delta(x)-\delta(y)|^2,
\]
where $x_0,y_0$ are respectively the projections of $x,y$ on $\partial\Omega$.
\begin{lem}
There exists $\eps=\eps(\Omega)>0$ such that
for any $x\in\Omega$, $x=x_0+\delta(x)\grad\delta(x_0)$, $x_0\in\partial\Omega$, with $\delta(x)<\eps$ and any $y\in\Omega$
with $\delta(y)<\eps$ and $|y_0-x_0|<\eps$
\[
\frac12\left(|x_0-y_0|^2+|\delta(x)-\delta(y)|^2\right)\leq
|x-y|^2\leq \frac32\left(|x_0-y_0|^2+|\delta(x)-\delta(y)|^2\right).
\]
\end{lem}
\begin{proof} Call $\Omega_\eps=\{x\in\Omega:\delta(x)<\eps\}$.
Write $x=x_0+\delta(x)\grad\delta(x_0),y=y_0+\delta(y)\grad\delta(y_0)$.
Then
\begin{multline}\label{444}
 |x-y|^2=|x_0-y_0|^2+|\delta(x)-\delta(y)|^2+\delta(y)^2|\grad\delta(x_0)-\grad\delta(y_0)|^2
+2[\delta(x)-\delta(y)]\langle x_0-y_0,\grad\delta(x_0)\rangle\ + \\
+\ 2\delta(y)\langle x_0-y_0,\grad\delta(x_0)-\grad\delta(y_0)\rangle
+2\delta(y)[\delta(x)-\delta(y)]\langle\grad\delta(x_0),\grad\delta(x_0)-\grad\delta(y_0)\rangle.
\end{multline}
Since, for $\eps>0$ small, $\delta\in C^{1,1}(\Omega_\eps)$ and
\begin{align*}
& |\grad\delta(x)-\grad\delta(y)|^2\leq \|\delta\|_{C^{1,1}(\Omega_\eps)}^2|x-y|^2 \\
& \langle x_0-y_0,\grad\delta(x_0)\rangle = O(|x_0-y_0|^2) \\
& |\langle x_0-y_0,\grad\delta(x_0)-\grad\delta(y_0)\rangle| \leq \|\delta\|_{C^{1,1}(\Omega_\eps)}|x_0-y_0|^2 \\
& |\delta(x)-\delta(y)|=\|\delta\|_{C^{1,1}(\Omega_\eps)}|x-y| \\
& |\langle\grad\delta(x_0),\grad\delta(x_0)-\grad\delta(y_0)\rangle|\leq \|\delta\|_{C^{1,1}(\Omega_\eps)}|x_0-y_0|
\end{align*}
The error term we obtained in \eqref{444} can be reabsorbed in the other ones by choosing $\eps>0$ small enough to have
\begin{multline*}
\delta(y)^2|\grad\delta(x_0)-\grad\delta(y_0)|^2+
2|\delta(x)-\delta(y)|\cdot|\langle x_0-y_0,\grad\delta(x_0)\rangle|+2\delta(y)|\langle x_0-y_0,\grad\delta(x_0)-\grad\delta(y_0)\rangle|\ +\\
+\ 2\delta(y)|\delta(x)-\delta(y)|\cdot|\langle\grad\delta(x_0),\grad\delta(x_0)-\grad\delta(y_0)\rangle|\leq\frac12\left(|x_0-y_0|^2+|\delta(x)-\delta(y)|^2\right).
\end{multline*}
\end{proof}

\begin{lem} \label{lemma39} Let $F:(0,+\infty)^3\to(0,+\infty)$ be a continuous function, decreasing in the third variable
and $\Omega_\eps=\{x\in\Omega:\delta(x)<\eps\}$,
with $\eps=\eps(\Omega)>0$ provided by the previous lemma.
Consider $x=x_0+\delta(x)\grad\delta(x_0)$, $x_0\in\partial\Omega$, and
the neighbourhood $A$ of the point $x$, defined by
$A=\{y\in \Omega_\eps: y=y_0+\delta(y)\grad\delta(y_0),|x_0-y_0|<\eps\}$.
Then there exist
constants $0<c_1<c_2$, $c_1=c_1(\Omega),c_2=c_2(\Omega)$ such that
\begin{multline*}
c_1\int_0^\eps dt\int_{B'_\eps}dy'\; F\!\left(\delta(x),t,c_2\sqrt{|y'|^2+|t-\delta(x)|^2}\right)
\leq \int_A F(\delta(x),\delta(y),|x-y|)\;dy\ \leq \\
\leq\ c_2\int_0^\eps dt\int_{B'_\eps}dy'\; F\!\left(\delta(x),t,c_1\sqrt{|y'|^2+|t-\delta(x)|^2}\right)
\end{multline*}
where the ' superscript denotes objects that live in $\R^{N-1}$.
\end{lem}
\begin{proof} By writing $y=y_0+\delta(y)\grad\delta(y_0),\ y_0\in\partial\Omega$
and using the Fubini's Theorem, we can split the integration into
the variables $y_0$ and $t=\delta(y)$:
\[
\int_A F(\delta(x),\delta(y),|x-y|)\;dy
\leq\int_{B_\eps(x_0)\cap\partial\Omega}\left(\int_0^{\eps} F(\delta(x),t,|x-y_0-t\grad\delta(y_0)|)\;dt \right)d\HH(y_0).
\]
Using the monotony of $F$ and the above lemma, we get
\[
\int_A F(\delta(x),\delta(y),|x-y|)\;dy\leq
\int_{B_\eps(x_0)\cap\partial\Omega}\left(\int_0^{\eps} F\!\left(\delta(x),t,c\sqrt{|x_0-y_0|^2+|\delta(x)-t|^2}\right)\;dt \right)d\HH(y_0)
\]
where $c$ is a universal constant.
Representing $B_\eps(x_0)\cap\partial\Omega$ via a diffeomorphism $\gamma$ with a ball $B'_\eps\subset\R^{N-1}$ centered at 0,
we can transform the integration in the $y_0$ variable into the 
integration onto $B'_\eps$. The volume element $|D\gamma|$ will be bounded above and below by
\[
0<c_1\leq |D\gamma|\leq c_2,
\]
in view of the smoothness assumptions on $\partial\Omega$.
\end{proof}

\part{Nonlocal curvatures}

\chapter{Nonlocal directional curvatures}

In this Chapter we study a nonlocal notion of curvature of a surface. 
For some bibliography dealing with the motion by mean curvature and minimal surfaces, 
see e.g.~\cite{deph, barrios, caputo, crs, meancurv,
regularity, asymptotics, dipierro, fusco, imbert, savin_cone}
as well as~\cite{milan} for a recent review.

The Chapter is organized as follows: 
the introductory Section \ref{classicsec} recalls some basic definitions and 
facts on classical curvatures of smooth surfaces (of course, this part can be
skipped by the expert reader but we included it in order to make a clear
comparison between the classical setting and the nonlocal one); 
in Section \ref{nonlocalsec} we introduce a definition of {\it nonlocal directional curvature} and
give some ideas of the context in which it arises;
finally, we state some theorems which compare
similarities and differences between the local and the nonlocal setting.
The remaining sections are devoted to proofs and explicit computations.

Though the motivation of the present Chapter arises
in the framework of nonlocal minimal surfaces
and integro-differential operators of fractional type, which
are subjects that involve a very advanced technology,
the Chapter itself is completely self-contained
and no prior knowledge on the topic is required
to follow the involved proofs.
Also, we put an effort in keeping all the arguments
as elementary as possible and accessible to a wide audience.

\subsection*{Notation}\label{NN}

In the following we will always use:
\begin{itemize}
\item $N$ to denote the dimension of the Euclidean space $\R^N$, with $N\ge3$,
whose points are sometimes written in the form $x=(x',x_N)\in \R^{N-1}\times\R$,
\item $\C E$ to denote the complementary set of~$E\subseteq\R^N$, i.e. $\C E:=\R^N\setminus E$,
\item $\HH^{N-2}(E)$ to denote the $(N-2)$-dimensional Hausdorff measure of a set $E$,
\item $S^{N-2}$ to denote the $(N-2)$-dimensional unit sphere in~$\R^{N-1}$, namely
\[
S^{N-2}:=\{e\in\R^{N-1}:|e|=1\};
\]
with a slight abuse of notation, we
will also identify $S^{N-2}$ and
the set
\[
\{(x',x_N)\in\R^{N-1}\times\R: |x'|=1,x_N=0\}\subseteq\R^N:
\]
note that the latter set is simply an $(N-2)$-dimensional sphere
lying in an $(N-1)$-dimensional subspace of~$\R^N$ and this justifies
our notation,
\item $\omega_{N-2}$ to denote the $(N-2)$-dimensional
Hausdorff measure of the $(N-2)$-dimensional sphere, that is
$\omega_{N-2}:=\HH^{N-2}(S^{N-2})$,
\item $\chi_E$, where $E\subseteq\R^N$, for the characteristic function of $E$, i.e. 
\[
\chi_E(x):=
\left\{\begin{matrix}
1 & {\mbox{if $x\in E$,}}\\
0 & {\mbox{if $x\in \C E$,}}
\end{matrix}\right.
\]
\item $\widetilde{\chi}_E$ for the difference $\chi_E-\chi_{\C E}$,
namely
\[
\widetilde{\chi}_E(x):=
\left\{\begin{matrix}
1 & {\mbox{if $x\in E$,}}\\
-1 & {\mbox{if $x\in \C E$,}}
\end{matrix}\right.
\]
\item $\langle Ax,x\rangle$, when $A$ is an $N\times N$ real symmetric matrix and $x\in\R^N$, 
to denote the quadratic form on $\R^N$ represented by $A$ and evaluated at~$x$, i.e.
if~$A=\{A_{ij}\}_{i,j=1,\dots,N}$, then
\[
\langle Ax,x\rangle:=\sum_{i,j=1}^N A_{ij} x_i x_j,
\]
\item $\displaystyle-\!\!\!\!\!\!\int_\Omega$, when $\Omega\subseteq\R^N$ has finite measure, 
for the integral operator $\displaystyle\frac{1}{|\Omega|}\int_\Omega$.
\end{itemize}
Also, we will sometimes
write multiple integrals by
putting in evidence the integration variables according to the expression
\[
\int_X dx \int_Y dy \int_Z dz \; f(x,y,z) \qquad\hbox{ in place of } \qquad\int_X \left[ \int_Y \left[ \int_Z f(x,y,z)\;dz
\right]\;dy\right]\;dx.
\]
Moreover, we will reserve the name $s$ for a fractional
parameter that, in our scaling, is taken in $(0,1/2)$.

\section{summary on classical curvatures}\label{classicsec}

In order to make a clear comparison between some
classical facts and their corresponding nonlocal counterparts, we recall here a few basic results.
Namely,
some well-known facts on the classical concept of curvature show a nice and deep interplay between geometry, analysis
and algebra that can risk to be not evident from the beginning.
In particular, the mean curvature, which is a geometrical object, can be described in normal coordinates by the Laplacian,
which comes from analysis, and also can be seen as the trace of a linear map, and here an algebraic notion shows up.
The interplay between these disciplines has some striking consequences: let us recall two of them.

First of all, we recall that, given a $C^2$ surface~$\Sigma$, a point~$p\in\Sigma$
and a vector~$e$ in the tangent space of~$\Sigma$ at~$p$, one may define
the classical notion of {\it directional curvature of~$\Sigma$ at~$p$ in direction $e$}
by the curvature at~$p$ of the curve~$C$ lying in the intersection between~$\Sigma$ and the two-dimensional
plane spanned by~$e$ and the normal vector of~$\Sigma$ at~$p$.
We denote by $K_e$ the directional curvature in direction $e$.

It is well-known that this directional curvature may be easily computed in normal coordinates. Namely,
suppose we are given a set $E\subseteq\R^N$ such that $0\in\partial E$,
and suppose that~$\Sigma=\partial E$ is described as a graph in normal coordinates, meaning that,
in an open ball $B_r\subseteq\R^{N}$, $\partial E$ coincides with the graph of a $C^2$ function
$\varphi:B_r\cap\R^{N-1}\rightarrow\R$ with $\varphi(0)=0$ and~$\grad\varphi(0)=0$.
Then the directional curvature in direction $e$ is given by 
\[
K_e=\langle D^2\varphi(0)\,e,e\rangle = D^2_e\varphi(0),\quad e\in\R^{N-1},|e|=1
\]
where $D^2\varphi(0)$ is the Hessian matrix of $\varphi$ evaluated at $0$.

Since $D^2\varphi(0)$ is a real symmetric matrix, 
it will admit $N-1$ real eigenvalues $\lambda_1\leq\ldots\leq\lambda_{N-1}$ called {\it principal curvatures}.
Moreover, associated to these eigenvalues, there is an orthonormal basis of eigenvectors $v_1,\ldots,v_{N-1}$
called {\it principal directions}.

The arithmetic mean of the principal curvatures is called
{\it mean curvature} and we denote it by $H$, namely
\[
H:=\frac{\lambda_1+\dots+\lambda_{N-1}}{N-1}.
\]
The above mentioned algebraic formulation implies that the principal directions
$v_1,\ldots,v_{N-1}$ can be always chosen orthogonally, 
which is a somehow surprising geometric outcome that allows to easily compute any directional
curvature once the principal curvatures are known:

\begin{theo}\label{locdirects} Let $\Sigma$ be a $C^2$ surface described as the graph 
of the $C^2$ function 
$\varphi:B_r\cap\R^{N-1}\rightarrow\R$ with $\varphi(0)=0$ and~$\grad\varphi(0)=0$. 
Every directional curvature can be calculated using principal curvatures,
i.e. the eigenvalues $\lambda_1,\ldots,\lambda_{N-1}$ of the matrix $D^2\varphi(0)$,
and principal curvatures, i.e. an orthonormal basis $v_1,\ldots,v_{N-1}$ of eigenvectors;
given a vector $e=\alpha_1 v_1+\ldots+\alpha_{N-1}v_{N-1}$, with $\alpha_1^2+\ldots+\alpha_{N-1}^2=1$, then
\[
K_e=\langle D^2\varphi(0)\,e,e\rangle=\lambda_1\alpha_1^2+\ldots+\lambda_{N-1}\alpha_{N-1}^2.
\]
\end{theo}

\begin{rmk}\label{remlocdirects}\rm We point out that Theorem \ref{locdirects} implies also that 
all directional curvatures are bounded below by $\lambda_1$ and above by $\lambda_{N-1}$ 
and that $\lambda_1$ and $\lambda_{N-1}$ are attained along orthogonal directions.
In particular, when~$N=3$, the two principal curvatures are the minimum and            
the maximum of the directional curvature~$K_e$ for~$e\in S^1$.
\end{rmk}

\begin{rmk}\rm In a sense,
Theorem \ref{locdirects}
shows a sort of linear phenomenon that
drives the classical directional curvatures.
As we will see in forthcoming
Remark \ref{R:8},
this linear feature cannot be recovered in the case of nonlocal
directional curvatures, that are somewhat
intrinsically nonlinear in nature.
\end{rmk}

Furthermore, the spherical
average of directional curvatures may be reconstructed by the arithmetic
mean of the principal curvatures,
that is the normalized integral of~$K_e$ over~$e\in S^{N-2}$
coincides with the normalized
trace of the Hessian matrix, thus reducing the (difficult, in general) computation
of an integral on the sphere to a (simple, in general) sum of
finitely many terms 
(that are the eigenvalues of the Hessian matrix) 
and this clearly provides an important computational simplification:

\begin{theo}\label{loccomput} Let $\Sigma$ be a $C^2$ surface described as the graph 
of the $C^2$ function 
$\varphi:B_r\cap\R^{N-1}\rightarrow\R$ with $\varphi(0)=0$ and~$\grad\varphi(0)=0$.
The mean curvature can be defined
\[
H:=\frac{\lambda_1+\ldots+\lambda_{N-1}}{N-1}
\qquad\hbox{ or equivalently }\qquad
H:=-\!\!\!\!\!\!\!\,\int_{S^{N-2}} K_e\;d\HH^{N-2}(e),
\]
since
\[
-\!\!\!\!\!\!\!\,\int_{S^{N-2}} K_e\;d\HH^{N-2}(e) =
-\!\!\!\!\!\!\!\,\int_{S^{N-2}}\langle D^2\varphi(0)\,e,e\rangle\;d\HH^{N-2}(e) =
\frac{\lambda_1+\ldots+\lambda_{N-1}}{N-1}.
\]
\end{theo}
\begin{proof}
By symmetry
\begin{equation}\label{i}
\int_{S^{N-2}}\alpha_i^2\;d\HH^{N-2}(\alpha)=
\int_{S^{N-2}}\alpha_1^2\;d\HH^{N-2}(\alpha)
\end{equation}
for any~$i\in\{1,\dots,N-1\}$. By summing up in~\eqref{i} we obtain
\begin{eqnarray*}
\omega_{N-2}=\int_{S^{N-2}} 1\;d\HH^{N-2}(\alpha)
=\int_{S^{N-2}}\sum_{i=1}^{N-1} \alpha_i^2\;d\HH^{N-2}(\alpha)
=(N-1)\int_{S^{N-2}}\alpha_1^2\;d\HH^{N-2}(\alpha).
\end{eqnarray*}
Using again~\eqref{i} we deduce from the identity above that
\[
\omega_{N-2}=
(N-1)\int_{S^{N-2}}\alpha_i^2\;d\HH^{N-2}(\alpha)
\]
for any~$i\in\{1,\dots,N-1\}$. As a consequence
\begin{multline*}
\frac{1}{\omega_{N-2}}\int_{S^{N-2}}\langle D^2\varphi(0)\,e,e\rangle\;d\HH^{N-2}(e)=
\frac{1}{\omega_{N-2}}\int_{S^{N-2}}\left(\lambda_1\alpha_1^2+\ldots+\lambda_{N-1}\alpha_{N-1}^2\right)\;d\HH^{N-2}(\alpha)\ =\\
=\ \frac{1}{\omega_{N-2}}\sum_{i=1}^{N-1}\lambda_i\int_{S^{N-2}}\alpha_i^2\;d\HH^{N-2}(\alpha)
=\frac{\lambda_1+\ldots+\lambda_{N-1}}{\omega_{N-2}}\cdot\frac{\omega_{N-2}}{N-1}
\end{multline*}
as desired.
\end{proof}

Associated with these concepts, there is also a theory of {\it motion by mean curvature}.
Let us think of a bounded set $\Omega\subseteq\R^N$ whose shape changes in time according to local features of its boundary, 
i.e. each point $x_0$ of the boundary moves along the normal direction to $\partial\Omega$ at $x_0$
and with a speed given by the mean curvature of $\partial\Omega$ at $x_0$.
In \cite{MBO}, with the aid of \cite{Ishii} and \cite{evans}, it is possible to find the following approximation 
of this motion. Let evolve the function $\widetilde{\chi}_\Omega$ according to the heat equation
\begin{equation}\label{Heat}
\left\{\begin{aligned}
\partial_tu(x,t)&=\lapl u(x,t), \\
u(x,0)&=\widetilde{\chi}_\Omega(x).
\end{aligned}\right.
\end{equation}
Then the set $\Omega_\varepsilon=\{u(x,\varepsilon)>0\}$, for small $\varepsilon>0$,
has a boundary close to the evolution of $\partial\Omega$ by mean curvature. 
Before passing to the nonlocal case, we would like to bring to the reader's attention two facts:
\begin{description}
\item[\it 1.] we underline how the evolution of a point $x_0\in\partial\Omega$ depends only on the shape of $\Omega$ in
a neighborhood of $x_0$,
\item[\it 2.] we recall that if a set $E$ has minimal perimeter in a region $U$, then it has zero mean curvature 
at each point of $\partial E\cap U$, see \cite{giusti}, and we can say that this is the Euler-Lagrange equation
associated to the minimization of the perimeter of a set;
therefore a set with minimal perimeter will be a stationary solution to the motion by mean curvature.
\end{description}

\section{nonlocal directional curvatures}\label{nonlocalsec}

From now on we take $s\in(0,1/2)$
and a set $E\subseteq\R^N$, with $C^2$ boundary $\partial E$.

\subsection{Nonlocal definitions}

We introduce here the nonlocal objects that will play the
role of directional and mean curvatures
(for details, heuristics and justifications of our definitions
see Section \ref{JJ}).

\begin{defi}[nonlocal mean curvature]\index{nonlocal mean curvature}\label{JD}
The nonlocal mean curvature of $\partial E$ at the point $p\in\partial E$ is
\begin{equation}\label{nonmeancurv}
H_s:=\frac{1}{\omega_{N-2}}\int_{\R^N}\frac{\widetilde{\chi}_E(x)}{|x-p|^{N+2s}}\;dx.
\end{equation}
\end{defi}

Denote now by $\nu$ be a normal unit vector for $\partial E$ at $p$.
Let also $e$ be any unit vector in the tangent space of $\partial E$
at $p$ and\footnote{Notice that $\pi(e)$
is simply the portion of the two-dimensional plane spanned by $e$ and $\nu$
given by the vectors with positive scalar product with respect to $e$.
We point out that a change of the orientation of $\nu$ does
not change $\pi(e)$ which is therefore uniquely defined.
Needless to say, such two-dimensional plane plays an important role
even in the classical setting.}
let $\pi(e)$ the two-dimensional open half-plane 
\[
\pi(e):=\{y\in\R^N:y=\rho e+h \nu,\ \rho>0,\ h\in\R\}.
\]
We endow $\pi(e)$ with the induced two-dimensional Lebesgue measure,
that is we define the integration over $\pi(e)$ by the formula
\begin{equation}\label{dy}
\int_{\pi(e)} g(y)\,dy:=\int_0^{+\infty}d\rho \int_\R dh\; g(\rho e +h\nu).
\end{equation} 

\begin{defi}[nonlocal directional curvature]\index{nonlocal directional curvature}
We define the nonlocal directional curvature of $\partial E$ at the point $p\in\partial E$ 
in direction $e$ the quantity
\begin{equation}\label{D2}
K_{s,e}:=\int_{\pi(e)}\frac{|y'-p'|^{N-2}\,
\widetilde{\chi}_E(y)}{|y-p|^{N+2s}}\;dy.
\end{equation}
\end{defi}

Without loss of generality, we can consider
now a normal frame of coordinates
in which
$p$ coincides with the origin $0$ of $\R^N$,
and the tangent space of $S$ at $0$ is the horizontal hyperplane
$\{x_N=0\}$. In this way we can take also 
\begin{equation}\label{NU}
\nu=(0,\dots,0,1).
\end{equation}
With this choice, \eqref{nonmeancurv}
and \eqref{D2} 
becomes
\begin{equation}\label{nonmeancurv.0}
H_s=\frac{1}{\omega_{N-2}}\int_{\R^N}\frac{\widetilde{\chi}_E(x)}{|x|^{N+2s}}\;dx
\end{equation}
and
\begin{equation}\label{D2.0}
K_{s,e}=\int_{\pi(e)}\frac{|y'|^{N-2}\,
\widetilde{\chi}_E(y)}{|y|^{N+2s}}\;dy.
\end{equation}
As a matter of fact, 
since the function ${\widetilde{\chi}_E(x)}/{|x|^{N+2s}}$ is not in the space $L^1(\R^N)$, 
the integral in \eqref{nonmeancurv.0} has to be taken in the principal value sense, that is
\begin{equation}\label{7.1}
\lim_{\eps\downarrow0}\int_{\C B_\eps}\frac{\widetilde\chi_E(x)}{|x|^{N+2s}}
\,dx.\end{equation}
Similarly, the integral in \eqref{D2.0} may be taken in the principal
value sense as
\begin{equation}\label{7.2}
\lim_{\eps\downarrow0}\int_{\pi(e)\setminus B_\eps}\frac{|y'|^{N-2}
\widetilde\chi_E(y)}{|y|^{N+2s}}
\,dy.\end{equation}
Next observation points out that these definitions are well-posed, in view of
the smoothness of $\partial E$:

\begin{lem}\label{goog}
The limits\footnote{Similar statements can be found in
\cite[Theorem 5.1]{crs}, where less regularity is asked on $E$
and only the $\limsup$ is calculated, and \cite[Lemma 1]{imbert}; we add here 
the finiteness of the nonlocal directional curvature.} 
in \eqref{7.1} and \eqref{7.2} exist and are finite.
\end{lem}

We postpone the proof of Lemma \ref{goog} to Section \ref{pf:goog}.
Though the definition of the nonlocal direction curvature
may look rather mysterious
at a first glance, it finds a concrete justification
thanks to the following result:

\begin{theo}\label{NLL} In the setting above
\[
H_s=-\!\!\!\!\!\!\int_{S^{N-2}}K_{s,e}\; d\HH^{N-2}(e).
\]
\end{theo}

Namely, Theorem \ref{NLL} states that the nonlocal mean curvature
is the average of the nonlocal directional curvatures, thus
providing a nonlocal counterpart of Theorem \ref{loccomput}.
See Section \ref{NLL:pf} for the proof of
Theorem \ref{NLL}.

In the particular case when the set $E$ is characterized as the subgraph of a function 
$f\in C^2(\R^{N-1})$ (that, due to our normalization
setting, satisfies $f(0)=0$ and $\grad f(0)=0$), namely if $E=\{x_N<f(x')\}$,
then formula (\ref{D2.0}) can be written directly in terms of $f$,
according to the expression
\begin{equation}\label{curv}
K_{s,e}=2\int_0^{+\infty}d\rho\;\rho^{N-2}\int_0^{f(\rho e)}\,\frac{dh}{\left(\rho^2+h^2\right)^{s+N/2}}.
\end{equation}
The proof of \eqref{curv} is deferred to Section \ref{radial:pf}.

\begin{rmk}\label{R:8}\rm
We point out that if $f$ is a radial function, i.e. $f(\rho e)=\phi(\rho)$ for some $\phi:[0,+\infty)\rightarrow\R$, 
equation (\ref{curv}) becomes
\begin{multline*}
K_{s,e}=2\int_0^{+\infty}d\rho\;\frac{1}{\rho^{2+2s}}\int_0^{\phi(\rho)}\frac{dh}{\left(1+\rho^{-2}h^2\right)^{s+N/2}}
=2\int_0^{+\infty}d\rho\;\frac{1}{\rho^{1+2s}}\int_0^{\frac{\phi(\rho)}{\rho}}\frac{d\tau}{\left(1+\tau^2\right)^{s+N/2}}\ = \\
=\int_0^{+\infty}\frac{1}{\rho^{1+2s}}\cdot F\left(\frac{\phi(\rho)}{\rho}\right)\;d\rho, 
\end{multline*}
where
\[
F(t):=\int_0^t \left(1+\tau^2\right)^{-s-N/2} d\tau.
\]
We observe that the function $F$ is nonlinear,
thus the nonlocal directional curvature
depends on the graph of the set in a nonlinear
fashion. Comparing this remark with Theorem \ref{locdirects},
we notice that this phenomenon is in sharp contrast with
the classical case.
\end{rmk}

\begin{rmk}\rm In our setting~$K_{s,-e}$ 
is, in general, not equal to~$K_{s,e}$,
differently from the classical case in which~$K_{-e}=K_{e}$.
For a notion of fractional directional curvature that is even on~$S^{N-2}$
one can consider~$\tilde K_{s,e}:=(K_{s,e}+K_{s,-e})/2$. Of course the results
presented in this paper hold for~$\tilde K_{s,e}$ too (with obvious minor
modifications).
\end{rmk}

\subsection{The context in which nonlocal curvatures naturally come forth}\label{JJ}

We now give some further motivation for the study of curvatures
of nonlocal type.
A few years ago the notion of {\it $s$-minimal set} has been introduced, see \cite{crs}. 
Roughly speaking, one can think of the problem of minimizing functionals which have a strong nonlocal flavor, 
meaning that these functionals take into account power-like
interactions between distant objects.
In particular, one is interested in the functional
\[
\J(A,B)=\frac{1}{\omega_{N-1}}\int_A\int_B\frac{dx\;dy}{|x-y|^{N+2s}}
\]
for every 
measurable sets\footnote{We refer here to sets as in \cite[footnote 1]{regularity}.}
$A,B\subseteq\R^N$,
and in the minimization of
the functional
\[
\hbox{Per}_s(E,U):=\J(E\cap U, U\setminus E)+\J(E\cap U,\C E \cap \C U)+\J(E\setminus U,U\setminus E),
\]
called the {\it $s$-perimeter of $E$ in $U$}, where $E,U\subseteq\R^N$ are measurable sets and $U$ is bounded. 

A set $E_\star\subseteq \R^N$ that minimizes
$\hbox{Per}_s(E,U)$ among all the measurable sets $E\subseteq\R^N$ such that $E\setminus U=E_\star\setminus U$ is called
$s$-minimal.
In this framework, $U$ can be viewed as an ambient space, meaning the space in which one is free to modify the set $E$, 
while the shape of $E$ is fixed outside $U$ and $E\setminus U$ plays the role of a boundary datum.

As the reader may have noticed, the notation Per$_s(E,U)$ and the name ``$s$-perimeter'' strongly remind the notation 
Per$(E,U)$ for the perimeter of a set $E$ in $U$, see \cite{giusti}
(and indeed $s$-minimal sets are the natural
nonlocal generalizations of sets with minimal perimeter). For instance, 
it is proved in \cite{asymptotics, deph} that, as
$s\uparrow\frac{1}{2}$, the $s$-perimeter reduces to the classical perimeter,
namely
\begin{equation}\label{XC} \lim_{s\uparrow\frac{1}{2}}\,(1-2s)\,\hbox{Per}_s(E,B_r)=
\hbox{Per}(E,B_r)\quad\hbox{ for a.e. }r>0.
\end{equation}
Also, asymptotics of the $s$-perimeter as $s\downarrow0$ are studied in \cite{dipierro}.

While, in the classical setting, sets with minimal perimeter
satisfy the zero mean curvature equation, it is proved in
\cite{crs} that if $E_\star$ is an $s$-minimal set and
$p\in\partial E$, then
\begin{equation}\label{els}
\int_{\R^N}\frac{\widetilde{\chi}_{E_\star}(x)}{|x-p|^{N+2s}}\;dx=0.
\end{equation}
Of course this equation makes sense if $\partial E_\star$
is smooth enough near $p$, so in general
\cite{crs} has to deal with \eqref{els} in a suitable
weak (and in fact viscosity) sense. In this setting,
one can say that (\ref{els}) is the Euler-Lagrange
equation of the functional
Per$_s$ and so, by analogy with the classical case,
it is natural to consider 
the left hand side of \eqref{els} as a nonlocal mean curvature.

This justifies Definition \ref{JD}.
Furthermore,
in \cite{imbert, meancurv} 
a nonlocal approximation scheme 
of motion by mean curvature has been developed.
This scheme differs from the classical
one recalled in Section \ref{classicsec}
since it substitutes the standard
heat equation in \eqref{Heat} with its nonlocal counterpart
\[
\left\{\begin{array}{l}
\partial_tu(x,t)=-\Ds u(x,t) \\
u(x,0)=\widetilde{\chi}_\Omega(x).
\end{array}\right.
\]
With this modification it has been proved that the counterpart of the normal velocity at a point $x_0\in\partial\Omega$
is given by the quantity (see \cite{crs} and references therein)
\begin{equation}\label{nmc}
\int_{\R^N}\frac{\widetilde{\chi}_\Omega(x)}{|x-x_0|^{N+2s}}\;dx.
\end{equation}

\subsection{Some comparisons between classical and nonlocal directional curvatures}\label{99}

Now we turn to the study of the objects that
we have introduced in the last paragraph, by stating some properties.
Our goal is threefold: 
first we study the directions in which maximal curvatures are attained, 
then we are interested in asymptotics for $s\uparrow1/2$, 
finally we present an example dealing with the relation between the
nonlocal mean curvature and 
the average of extremal nonlocal directional curvatures.

First of all, we establish that the counterparts of Theorem \ref{locdirects} and 
Remark \ref{remlocdirects} 
do not hold in the nonlocal framework. Indeed,
the direction that maximizes the nonlocal directional curvature
is not, in general, orthogonal to the
one that minimizes it. Even more, one can prescribe arbitrarily
the set of directions that maximize and minimize the nonlocal directional
curvature, according to the following result:

\begin{theo}[directions of extremal nonlocal curvatures]\label{teoprinc}
For any two disjoint, nonempty, closed subsets 
$$
\Sigma_-,\Sigma_+
\subseteq S^{N-2},
$$
there exists a set $E\subseteq\R^N$ such that $\partial E$ is $C^2$, $0\in\partial E$ and
$$
K_{s,e_-}< K_{s,e}< K_{s,e_+},\qquad \hbox{for any }e_-\in\Sigma_-,\
e\in S^{N-2}\setminus(\Sigma_+\cup\Sigma_-),\
e_+\in\Sigma_+,
$$
and the minimum and maximum of the nonlocal directional curvatures are attained 
at any point of $\Sigma_-$ and $\Sigma_+$ respectively. 
\end{theo}

We remark that, in the statement above,
it is not necessary to assume any smoothness on the boundary of
the sets~$\Sigma_-$ and $\Sigma_+$ in~$S^{N-2}$.

Next result points out that the definition of
nonlocal directional curvature is consistent with the classical concept
of directional curvature and reduces to it in the limit:

\begin{theo}[asymptotics to $1/2$]\label{lemasintotico}
For any $e\in S^{N-2}$
\begin{equation}\label{23}
\lim_{s\uparrow\frac{1}{2}}(1-2s)K_{s,e}=K_e
\end{equation}
and
\begin{equation*}
\lim_{s\uparrow\frac{1}{2}}(1-2s)H_s=\,H,
\end{equation*}
where $K_e$ (resp., $H$) is the directional curvature of $E$ in direction $e$
(resp., the mean curvature of $E$) at 0.
\end{theo}

Notice that Theorem \ref{lemasintotico} can be seen
as an extension of the asymptotics in \eqref{XC}
for the directional and mean curvatures.

A further remark is that, differently from the local case,
in the nonlocal one it is not possible to calculate the mean curvature 
simply by taking the arithmetic mean of the principal curvatures (this in
dimension $N=3$ reduces to
the half of the sum between the maximal and the minimal directional
curvatures). This phenomenon is a consequence of
Theorem~\ref{lemasintotico} and it
may also be detected by an explicit example:

\begin{ex}\label{es2}\rm Let $E=\{(x,y,z)\in\R^3:z\leq 8x^2y^2\}$.
Let $H_s$ be
the nonlocal mean curvature at $0\in\partial E$.
Let also $K_{s,e}$ be the nonlocal curvature at $0$ in direction $e$,
\[
\lambda_-:=\min_{e\in S^1} K_{s,e} \ {\mbox{ and }} \
\lambda_+:=\max_{e\in S^1} K_{s,e}.
\]
Then $\lambda_-$ is attained at $(0,1)$, $\lambda_+$ is attained at $(\sqrt{2}/2,
\sqrt{2}/2)$, and $H_s\ne(\lambda_- +\lambda_+)/2$.
\end{ex}
Of course, Example \ref{es2}
is in sharp contrast with the classical case, recall Remark \ref{remlocdirects}.
For the proof of the claims related to Example \ref{es2}
see Section \ref{remarks}.

\section{proof of theorem \ref{NLL}}\label{NLL:pf}

Given a function~$G$, we apply~\eqref{dy} to the function~$g(y):=|y'|^{N-2} G(y)$.
For this, we recall the normal coordinates in~\eqref{NU}
and, with a slight abuse of notation we identify the vector~$e=(e_1,\dots,e_{N-1},0)\in\R^N$
with~$(e_1,\dots,e_{N-1})\in\R^{N-1}$, so that we write
\begin{equation}\label{PO}
\pi(e)\ni y=\rho e+h\nu=(\rho e,h).\end{equation}
Then \eqref{dy} reads
$$ \int_{\pi(e)} dy \ |y'|^{N-2} G(y)=\int_0^{+\infty}d\rho\int_\R dh\ \rho^{N-2} G(\rho e,h).$$
We integrate this identity over~$e\in S^{N-2}$:
by recognizing the polar coordinates in~$\R^{N-1}$ we obtain
\begin{multline*}
\int_{S^{N-2}} d{\HH}^{N-2}(e)
\int_{\pi(e)} dy \ |y'|^{N-2} G(y)=
\int_\R dh
\int_{S^{N-2}} d{\HH}^{N-2}(e)
\int_0^{+\infty}d\rho\;\rho^{N-2} G(\rho e,h)\ =\\
=\ \int_\R dh\int_{\R^{N-1}}dx' \
G(x',h)=\int_{\R^N} dx \ G(x).
\end{multline*}
We apply this formula to~$G(y):=\widetilde\chi_E(y)/|y|^{N+2s}$ and we recall
\eqref{nonmeancurv.0}           
and~\eqref{D2.0}, so to conclude that
$$
\int_{S^{N-2}} d{\HH}^{N-2}(e) \ K_{s,e}=
\int_{S^{N-2}} d{\HH}^{N-2}(e)
\int_{\pi(e)} dy \ \frac{|y'|^{N-2} \widetilde\chi_E(y)}{|y|^{N+2s}}=
\int_{\R^N} dx \ \frac{\widetilde\chi_E(x)}{|x|^{N+2s}}
=\omega_{N-2}\,H_s,$$
establishing Theorem~\ref{NLL}.
\hfill$\square$

\section{proof of \texorpdfstring{\eqref{curv}}{eq:curv}}\label{radial:pf}

We exploit again the notation in~\eqref{PO}
and~\eqref{dy} applied to~$g(y):={|y'|^{N-2} \widetilde\chi_E(y)}/{|y|^{N+2s}}$,
to see that
\begin{equation}\label{P89}
\int_{\pi(e)} dy \ \frac{|y'|^{N-2} \widetilde\chi_E(y)}{|y|^{N+2s}}
=\int_0^{+\infty} d\rho\int_\R dh \
\frac{\rho^{N-2} \widetilde\chi_E(\rho e,h)}{(\rho^2+h^2)^{s+N/2}}.
\end{equation}
Now we observe that~$\widetilde\chi_E(\rho e,h)=1$ if~$h<-|f(\rho e)|$
and~$\widetilde\chi_E(\rho e,h)=-1$ if~$h>|f(\rho e)|$, being~$E$ the subgraph of~$f$.
Therefore, for any fixed~$e\in S^{N-2}$ the map
\[ 
h\longmapsto \frac{\rho^{N-2} \widetilde\chi_E(\rho e ,h)}{(\rho^2+h^2)^{s+N/2}}
\]
is odd for~$h\in\R\setminus [-|f(\rho e)|,|f(\rho e)|]$ and therefore
\begin{equation}\label{P899}
\int_{\R\setminus [-|f(\rho e)|,|f(\rho e)|]} dh \
\frac{\rho^{N-2} \widetilde\chi_E(\rho e,h)}{(\rho^2+h^2)^{s+N/2}}\,=\,0.
\end{equation}
The subgraph property also gives that
\begin{eqnarray*}
\int_{-|f(\rho e)|}^{|f(\rho e)|} dh \
\frac{\rho^{N-2} \widetilde\chi_E(\rho e,h)}{(\rho^2+h^2)^{s+N/2}}
&=&\left\{
\begin{matrix}
2\displaystyle\int_{0}^{|f(\rho e)|} dh \            
\frac{\rho^{N-2} }{(\rho^2+h^2)^{s+N/2}}
&
{\ \mbox{ if $f(\rho e)\ge0$,}}\\ & \\
-2\displaystyle\int_{-|f(\rho e)|}^{0} dh \            
\frac{\rho^{N-2} }{(\rho^2+h^2)^{s+N/2}}
&
{\ \mbox{ if $f(\rho e)\le0$}}
\end{matrix}
\right.\\ &=&
2\int_{0}^{f(\rho e)} dh \             
\frac{\rho^{N-2} }{(\rho^2+h^2)^{s+N/2}}.
\end{eqnarray*}
This, \eqref{P89} and~\eqref{P899} give that
\[
\int_{\pi(e)} dy \ \frac{|y'|^{N-2} \widetilde\chi_E(y)}{|y|^{N+2s}}=
2\int_0^{+\infty}d\rho
\int_{0}^{f(\rho e)} dh \ 
\frac{\rho^{N-2} }{(\rho^2+h^2)^{s+N/2}},
\]
and so~\eqref{curv} follows now from~\eqref{D2.0}.
\hfill$\square$

\section{proof of theorem \ref{teoprinc}}\label{S3}
In $\R^N$ define the set $E=\{x=(x',x_N)\in\R^{N-1}\times\R:x_N\leq f(x')\}$. 
We will construct $f$ in such a way to make it a regular function, say at least $C^2$. For this, we 
fix two closed and disjoint sets $\Sigma_-$ and $\Sigma_+$
in $S^{N-2}$, and we take
$a\in C^\infty(S^{N-2},[0,1])$ 
in such a way that 
\begin{equation}\label{llll}
\begin{array}{ll}
a(e)=0 & \hbox{ for any }e\in \Sigma_-, \\
a(e)\in(0,1) & \hbox{ for any }e\in S^{N-2}\setminus(\Sigma_-\cup\Sigma_+), \\
a(e)=1 & \hbox{ for any }e\in \Sigma_+.
\end{array}
\end{equation}
The existence of such an $a$ is warranted by a strong version of
the smooth Urysohn Lemma (notice that $a\in C^\infty(S^{N-2})$
in spite of the fact that no regularity assumption has been
taken on $\Sigma_-$ and $\Sigma_+$ and that~$a$ takes values~$0$ and~$1$
only in~$\Sigma_-\cup\Sigma_+$).
We provide the details of the construction of~$a$ for the
facility of the reader.
For this we observe that~$\Sigma_-$ is a closed set in~$\R^{N-1}$.
So, by Theorem~1.1.4 in~\cite{BR}, there exists~$f_-\in C^\infty(\R^{N-1})$
such that~$f_-(p)=0$ for any~$p\in\Sigma_-$ and~$f_-(p)\ne0$ for any~$p\in\R^{N-1}\setminus\Sigma_-$.
Then the function~$g_-(p):=\big(f_-(p)\big)^2$ satisfies
that~$g_-(p)=0$ for any~$p\in\Sigma_-$ and~$g_-(p)>0$ for any~$p\in\R^{N-1}\setminus\Sigma_-$.
Similarly, there exists~$g_+\in C^\infty(\R^{N-1})$
such that~$g_+(p)=0$ for any~$p\in\Sigma_+$ and~$g_+(p)>0$ for any~$p\in\R^{N-1}\setminus\Sigma_+$.
Then the function
$$ \R^{N-1}\ni p\longmapsto a(p):=\frac{g_-(p)}{g_+(p)+g_-(p)}$$
satisfies \eqref{llll} as desired.

Now we take an even function $\phi\in C^\infty_0(\R,[0,1])$
with $\phi(\rho)>0$ if $\rho\in (1,2)$,
\[
\lim_{\rho\rightarrow 0^+}\phi''(\rho)=\lim_{\rho\rightarrow 0^+}\phi'(\rho)=\lim_{\rho\rightarrow 0^+}\phi(\rho)=0.
\]
Then we define $f$ using the polar coordinates of $\R^{N-1}$, namely we set
\[
f(x'):=a(e)\phi(\rho),\qquad\hbox{where }\rho=|x'|\hbox{ and }e=\frac{x'}{|x'|}.
\]
By construction $f$ is $C^2$ in the whole of $\R^{N-1}$ (in particular, in a neighborhood of $0$).
Also
\begin{equation}\label{stepth}
0=a(e_-)\phi(\rho)\leq\phi(\rho)a(e)\leq\phi(\rho)a(e_+)\hbox{ for all } \rho>0,e\in S^{N-2}
\end{equation}
whenever we choose $e_-\in\Sigma_-$, $e_+\in\Sigma_+$ 
and $e\in S^{N-2}\setminus(\Sigma_-\cup\Sigma_+)$,
and strict inequalities occur whenever $\rho\in(1,2)$.
Therefore, by (\ref{curv}) and (\ref{stepth}),
\begin{multline*}
K_{s,e_-}=0=2\int_0^{+\infty}d\rho\;\rho^{N-2}\int_0^{\phi(\rho)a(e_-)}\frac{dh}{\left(\rho^2+h^2\right)^{s+N/2}}\ \leq \\
\leq\ 2\int_0^{+\infty}d\rho\;\rho^{N-2}\int_0^{\phi(\rho)a(e)}\frac{dh}{\left(\rho^2+h^2\right)^{s+N/2}}=
K_{s,e}
\end{multline*}
and
\begin{multline*}
K_{s,e}=2\int_0^{+\infty}d\rho\;\rho^{N-2}\int_0^{\phi(\rho)a(e)}\frac{dh}{\left(\rho^2+h^2\right)^{s+N/2}}\ \leq \\
\leq\ 2\int_0^{+\infty}d\rho\;\rho^{N-2}\int_0^{\phi(\rho)a(e_+)}\frac{dh}{\left(\rho^2+h^2\right)^{s+N/2}}=
K_{s,e_+},
\end{multline*}
hence $K_{s,e}$ attains its 
minimum at any point of $\Sigma_-$ 
and its maximum at any point of $\Sigma_+$ (and only there).
\hfill$\square$

\section{proof of theorem \ref{lemasintotico}}\label{S4}

For simplicity, we consider here the case in which $E$ is a subgraph, namely
that there exists $f\in C^2(\R^{N-1})$,
such that $f(0)=0$, $\grad f(0)=0$ and $E=\{(x',x_N)\in\R^{N-1}\times\R:x_N\leq f(x')\}\subseteq\R^N$. 
Such assumption can be easily dropped
a posteriori just working in local coordinates
and observing that the contribution to $K_{s,e}$ coming from far
is bounded uniformly\footnote{In further detail, recalling \eqref{dy},
\begin{multline*}
\left|\int_{\pi(e)\setminus B_1}\frac{|y'|^{N-2}\,
\widetilde{\chi}_E(y)}{|y|^{N+2s}}\,dy\right|\le
\int_{\pi(e)}\frac{|y|^{N-2}\,{\chi}_{\C B_1}(y)}{|y|^{N+2s}}\,dy \le \\ 
\leq\ \int_0^{+\infty}d\rho \int_{-\infty}^{+\infty} dh \, \chi_{\R\setminus(-1,1)}(\rho^2+h^2)
\;(\rho^2+h^2)^{-1-s}
\le \int_{\R^2\setminus B_1}dx |x|^{-2-2s}=2\pi
\int_1^{+\infty}dr\, r^{-1-2s}=\frac{\pi}{s}
\end{multline*}
that is uniformly bounded as $s\uparrow1/2$. This means that we can suppose that~$\partial E$
is a graph in, say, $B_1$ and replace it outside $B_1$ without affecting
the statement of Theorem \ref{lemasintotico}.}
when $s\uparrow1/2$ and so it does
not contribute to the limit in \eqref{23}.

So, to prove Theorem \ref{lemasintotico},
we take equation (\ref{curv}) and split the integral in $\eta>0$
\begin{equation}\label{add}
K_{s,e}=2\int_0^\eta d\rho\;\rho^{N-2}\int_0^{f(\rho e)}\frac{dh}{\left(\rho^2+h^2\right)^{s+N/2}}+
2\int_\eta^{+\infty}d\rho\;\rho^{N-2}\int_0^{f(\rho e)}\frac{dh}{\left(\rho^2+h^2\right)^{s+N/2}}.
\end{equation}
Let us start from the second addendum:
\begin{equation}\label{2ndaddend}\begin{split}
&
\left\arrowvert\int_\eta^{+\infty}d\rho\;\rho^{N-2}\int_0^{f(\rho e)}
\frac{dh}{\left(\rho^2+h^2\right)^{s+N/2}}\right\arrowvert 
\leq \displaystyle\int_\eta^{+\infty}d\rho\;\int_0^{+\infty}dh\,
\frac{\rho^{N-2}}{\left(\rho^2+h^2\right)^{s+N/2}}\ \leq \\
& \leq\ \int_\eta^{+\infty}d\rho\;\int_0^{+\infty}dh\,
\frac{(\rho^2+h^2)^{(N-2)/2}}{\left(\rho^2+h^2\right)^{s+N/2}}
\le \int_{\R^2\setminus B_\eta} |x|^{-2-2s}\;dx
=2\pi\displaystyle\int_\eta^{+\infty} r^{-1-2s}\,dr=\frac\pi{s}\,\eta^{-2s}.\end{split}
\end{equation}
Now we look at the first addendum in \eqref{add}: for this we write, for $\eta\in(0,1)$ sufficiently small,
\begin{equation}\label{75}
D^2_ef(0)\frac{\rho^2}{2}-\eps(\rho) \; \leq \; f(\rho e) \; \leq 
\; D^2_ef(0)\frac{\rho^2}{2}+\eps(\rho),\qquad\rho\in(0,\eta)
\end{equation}
where $\eps:(0,\eta)\rightarrow(0,+\infty)$ is defined as
\[
\eps(\rho):=\sup_{|\xi'|\le\rho}| D^2_ef(\xi')- D^2_ef(0)|\,\frac{\rho^2}{2}.
\]
Let also
\[
E(\eta):=\sup_{0<\rho\le\eta}\frac{\eps(\rho)}{\rho^2}=\frac12\sup_{|\xi'|\le\eta}| D^2_ef(\xi')- D^2_ef(0)|.
\]
and observe that
\begin{equation}\label{as}
E(\eta)\downarrow 0\;\hbox{ as }\eta\downarrow 0.
\end{equation}
Then, if we denote by $D:=\frac{1}{2}D^2_ef(0)$, we have that 
\begin{align}
& \left\arrowvert\int_0^\eta d\rho\;\rho^{N-2}\int_0^{f(\rho e)}\frac{dh}{\left(\rho^2+h^2\right)^{s+N/2}}-\frac{D}{1-2s}\right\arrowvert
 \ \leq\ \left\arrowvert\int_0^1d\rho\;\rho^{N-2}\int_0^{D\rho^2}\frac{dh}{\left(\rho^2+h^2\right)^{s+N/2}}-\frac{D}{1-2s}\right\arrowvert \nonumber \\
&+\ \left\arrowvert\int_0^\eta d\rho\;\rho^{N-2}\int_0^{f(\rho e)}\frac{dh}{\left(\rho^2+h^2\right)^{s+N/2}}-\int_0^\eta d\rho\;\rho^{N-2}\int_0^{D\rho^2}\frac{dh}{\left(\rho^2+h^2\right)^{s+N/2}}\right\arrowvert \label{triang} \\
&+\ \left|\int_\eta^1 d\rho\;\rho^{N-2}\int_0^{D\rho^2}\frac{dh}{\left(\rho^2+h^2\right)^{s+N/2}}\right|. \nonumber
\end{align}
The latter term is uniformly bounded as $s\uparrow1/2$: indeed
\begin{eqnarray}
\displaystyle \left|\int_\eta^1 d\rho\;\rho^{N-2}\int_0^{D\rho^2}\frac{dh}{\left(\rho^2+h^2\right)^{s+N/2}}\right|\;
\leq\; \int_\eta^1 d\rho\;\rho^{N-2}\cdot\frac{|D|\rho^2}{\rho^{N+2s}} \;
=\; \int_\eta^1 d\rho\;\frac{|D|}{\rho^{2s}} \nonumber \\
=\; |D|\,\frac{1-\eta^{1-2s}}{1-2s}\;\xrightarrow[s\uparrow\frac{1}{2}]{}\;-|D|\,\ln\eta,
\end{eqnarray}
and therefore, for~$s$ close to~$1/2$,
\begin{equation}\label{latterm}
\left|\int_\eta^1 d\rho\;\rho^{N-2}\int_0^{D\rho^2}\frac{dh}{\left(\rho^2+h^2\right)^{s+N/2}}\right|
\;\leq\;-\widetilde D\,\ln\eta,
\end{equation}
where $\widetilde D$ is a positive constant depending only on $D$.

The central term in \eqref{triang} may be estimated using \eqref{75}: indeed
\begin{eqnarray}
\displaystyle\left\arrowvert\int_0^\eta d\rho\;\rho^{N-2}\int_0^{f(\rho e)}\frac{dh}{\left(\rho^2+h^2\right)^{s+N/2}}-\int_0^\eta d\rho\;\rho^{N-2}\int_0^{D\rho^2}\frac{dh}{\left(\rho^2+h^2\right)^{s+N/2}}\right\arrowvert \nonumber \\
\nonumber \\
\leq \quad \displaystyle\int_0^\eta d\rho\;\rho^{N-2}\left\arrowvert\int_{D\rho^2}^{f(\rho e)}\frac{dh}{\left(\rho^2+h^2\right)^{s+N/2}}\right\arrowvert \quad \leq \quad \int_0^\eta\rho^{N-2}\frac{\arrowvert f(\rho e)-D\rho^2\arrowvert}{\rho^{N+2s}}\;d\rho \label{centaddend} \\
 \leq \quad \int_0^\eta\frac{\eps(\rho)}{\rho^{2+2s}}\;d\rho\quad \leq \quad
E(\eta) \int_0^\eta\rho^{-2s}\;d\rho\quad=\quad E(\eta)\cdot\frac{\eta^{1-2s}}{1-2s}. \nonumber 
\end{eqnarray}

It remains now to estimate the first term on the right hand side of
\eqref{triang}. For this we 
apply the change of variable $t=h/\rho$ and we see that
\begin{eqnarray*}
&& \left\arrowvert\rho^{N-2}
\int_0^{D\rho^2}
\frac{dh}{\left(\rho^2+h^2\right)^{s+N/2}}-D\rho^{-2s}\right\arrowvert=
|D|\rho^{-2s}\left\arrowvert\frac1{D\rho^2}\int_0^{D\rho^2}
\frac{dh}{\left(1+\frac{h^2}{\rho^2}\right)^{s+N/2}}-1\right\arrowvert 
\\ &&\qquad 
=|D|\rho^{-2s}\left\arrowvert\frac{1}{D\rho}\int_0^{D\rho}\frac{dt}{\left(1+t^2\right)^{s+N/2}}-1
\right\arrowvert.\end{eqnarray*}
Now, since the map $t\longmapsto 1/\left(1+t^2\right)^{s+N/2}$ is decreasing, we have that
\[
\frac{1}{\left(1+D^2\rho^2\right)^{s+N/2}} \ \le\
-\!\!\!\!\!\!\int_0^{|D|\rho}\frac{dt}{\left(1+t^2\right)^{s+N/2}}
\ \leq \ 1.
\]
Using this and a Taylor expansion, we obtain
\begin{eqnarray*}
&& \left\arrowvert\rho^{N-2}
\int_0^{D\rho^2}
\frac{dh}{\left(\rho^2+h^2\right)^{s+N/2}}-D\rho^{-2s}\right\arrowvert
=|D|\rho^{-2s}\left(
1-\frac{1}{|D|\rho}\int_0^{|D|\rho}\frac{dt}{\left(1+t^2\right)^{s+N/2}}
\right)\\
&&\quad\le |D|\rho^{-2s}\left(
1-\frac{1}{(1+D^2\rho^2)^{s+N/2}}
\right) \le |D|\rho^{-2s}\cdot (\beta D^2\rho^2)=\beta |D|^3 \rho^{2-2s}\end{eqnarray*}
for any $\rho\in(0,1)$, where $\beta$ is a suitable positive constant only depending on $N$.
By integrating this estimate over the interval $(0,1)$ we conclude that
\begin{equation}\label{lastone}
\begin{split}
\left\arrowvert\int_0^1 d\rho\;\rho^{N-2}\int_0^{D\rho^2}\frac{dh}{\left(\rho^2+h^2\right)^{s+N/2}}
-\frac{D}{1-2s}\right\arrowvert\ &=\\
=\ &\left\arrowvert\int_0^1 d\rho\;\left[
\rho^{N-2}\int_0^{D\rho^2}\frac{dh}{\left(\rho^2+h^2\right)^{s+N/2}}
-D\rho^{-2s}\right]\right\arrowvert\leq \widetilde\beta,
\end{split}
\end{equation}
for a suitable $\widetilde\beta$ possibly depending on $D$ and $N$,
but independent of $s$.

Resuming all this calculations in one formula and putting together the information in equations (\ref{add}), (\ref{2ndaddend}),
(\ref{triang}), (\ref{latterm}), (\ref{centaddend}) and (\ref{lastone}),
we obtain that
\begin{equation}\label{final}
\left\arrowvert K_{s,e}-\frac{2\,D}{1-2s}\right\arrowvert\;\;\leq\;\;
\frac{2\pi}{s}\eta^{-2s}
-\widetilde D\,\ln\eta
+2E(\eta)\cdot\frac{\eta^{1-2s}}{1-2s}
+2\widetilde\beta,
\end{equation}
hence
\[
\limsup_{s\uparrow\frac{1}{2}}(1-2s)\left\arrowvert K_{s,e}-\frac{2\,D}{1-2s}\right\arrowvert\:\leq\:
2\,E(\eta).
\]
But, by \eqref{as}, $E(\eta)$ is arbitrarily small. Hence we can conclude
\[
\lim_{s\uparrow\frac{1}{2}}(1-2s)\left\arrowvert K_{s,e}-\frac{2\,D}{1-2s}\right\arrowvert\:=\:0
\]
which also implies
\[
\lim_{s\uparrow\frac{1}{2}}(1-2s)K_{s,e}=2\,D=D^2_ef(0),
\]
that is the desired claim.
Moreover, since estimate (\ref{final}) is uniform on $S^{N-2}$, we have also the convergence
\[
(1-2s)H_s=\frac{1-2s}{\omega_{N-2}}\int_{S^{N-2}}K_{s,e}\;d\HH^{N-2}(e)\;\xrightarrow[s\uparrow\frac{1}{2}]{} 
\; \frac{1}{\omega_{N-2}}\int_{S^{N-2}}D^2_ef(0)\;d\HH^{N-2}(e)=H.
\]
This completes the proof of Theorem \ref{lemasintotico}.
\hfill$\square$

\section{proof of the claims in example \ref{es2}}\label{remarks}

In $\R^3$ define 
\[
E=\{ (x,y,z)\in\R^3: z\leq f(x,y)=8x^2y^2 \}
\]
and denote by $(\rho,\theta,z)$ the cylindrical coordinates, i.e. 
\[
\rho=\sqrt{x^2+y^2},\qquad
\cos\theta=\frac{x}{\sqrt{x^2+y^2}},
\qquad\sin\theta=\frac{y}{\sqrt{x^2+y^2}}.
\]
In this way, we set
$e:=(\cos\theta,\sin\theta)\in S^1$ and we have that $(x,y)=\rho e$.
Notice that
\[
1-\cos4\theta = 1-\cos^22\theta+\sin^22\theta = 2\sin^22\theta = 8\sin^2\theta\cos^2\theta\cdot\frac{\rho^4}{\rho^4} =
 \frac{8x^2y^2}{\left(x^2+y^2\right)^2}
\]
therefore
the function $f$, written in terms of $(\rho,\theta)$, becomes
\[ 
f(x,y)=f(\rho e)=
\widetilde{f}(\rho,\theta)=(1-\cos4\theta)\rho^4.
\]
With a slight abuse of notation, we denote by $K_{s,\theta}$
the nonlocal directional curvature of $E$ at $0$ in direction $e=(\cos\theta,\sin\theta)$
(i.e. $K_{s,\theta}$ is short for $K_{s,(\cos\theta,\sin\theta)}$).
We recall that, in this case, we can compute nonlocal curvatures exploiting identity (\ref{curv}), i.e.
\begin{equation}\label{X8}
K_{s,\theta}  = 2\int_0^{+\infty}d\rho\int_0^{\widetilde{f}(\rho,\theta)} dz\;
\frac{\rho}{\left(\rho^2+z^2\right)^{s+3/2}}.
\end{equation}
In the above domain of integration it holds that
\[
0\leq z\leq\widetilde{f}(\rho,\theta)=(1-\cos4\theta)\rho^4
\]
and so 
\[ 
\rho\geq\sqrt[4]{\frac{z}{1-\cos4\theta}},\qquad\hbox{when }\cos4\theta\neq1.
\]
Therefore, integrating first in the $\rho$ variable, we deduce from \eqref{X8}
that
\begin{equation}\label{GG}\begin{split}
K_{s,\theta} \,& =  
\int_0^{+\infty}dz\int_{\sqrt[4]{\frac{z}{1-\cos4\theta}}}^{+\infty}d\rho\;
\frac{2\rho}{\left(\rho^2+z^2\right)^{s+3/2}}
  = \int_0^{+\infty}dz \left.\frac{\left(\rho^2+z^2\right)^{-s-1/2}}{-s-\frac{1}{2}}\right\arrowvert_{\rho=
\sqrt[4]{\frac{z}{1-\cos4\theta}}}^{\rho=+\infty} \\
& = \frac{2}{2s+1} \int_0^{+\infty}\frac{dz}{\left(z^2+\sqrt{\frac{z}{1-\cos4\theta}}\right)^{s+1/2}} 
  = \frac{2}{2s+1} \int_0^{+\infty}
     \frac{(1-\cos4\theta)^{s/2+1/4}}{\left(z^2\sqrt{1-\cos4\theta}+\sqrt{z}\right)^{s+1/2}}\;dz.
\end{split}\end{equation}
We concentrate now on maximal and minimal nonlocal directional curvatures. 
Since $K_{s,\theta}$ is a nonnegative quantity (because $\widetilde{f}(\rho,\theta)$ is a nonnegative function) 
and since $K_{s,0}=K_{s,\pi/2}=0$, then $K_{s,\theta}$ attains its minimum in $0$ and $\pi/2$. 
Also, $\widetilde{f}(\rho,\pi/4)\geq\widetilde{f}(\rho,\theta)$ for every positive $\rho$ and $\theta\in[0,\pi)$, 
thus $K_{s,\theta}$ attains its maximum for $\theta=\pi/4$.
On the one hand we have that the arithmetic mean of the maximal and minimal nonlocal principal curvatures is given by 
\begin{equation}\label{7df87f7}
\frac{K_{s,0}+K_{s,\pi/4}}2=
\frac{1}{2}\,K_{s,\pi/4} = \frac{2^{s/2+1/4}}{2s+1}\int_0^{+\infty}\frac{dz}{\left(\sqrt{2}z^2+\sqrt{z}\right)^{s+1/2}},
\end{equation}
thanks to \eqref{GG}.

On the other hand, the nonlocal mean curvature is
\[
H_s=\frac{1}{2\pi}\int_0^{2\pi}K_{s,\theta}\;d\theta
\]
and we are going to estimate this quantity, in order to show that
it is not equal to the arithmetic mean of the principal nonlocal curvatures
(i.e. to the quantity in \eqref{7df87f7}).
Note that	
\begin{equation}\label{d66d6d}
\begin{split}
\frac{1}{2\pi}\int_0^{2\pi}K_{s,\theta}\;d\theta & 
=\frac{1}{(2s+1)\pi}\int_0^{2\pi}d\theta\int_0^{+\infty}dz\;
\frac{(1-\cos4\theta)^{s/2+1/4}}{\left(z^2\sqrt{1-\cos4\theta}+\sqrt{z}\right)^{s+1/2}} \\
& \geq \frac{1}{(2s+1)\pi}\int_0^{2\pi}(1-\cos4\theta)^{s/2+1/4}d\theta\cdot\int_0^{+\infty}\frac{dz}{\left(\sqrt{2}z^2+\sqrt{z}\right)^{s+1/2}} \\
& =\frac{1}{\pi\,2^{s/2+1/4}}\int_0^{2\pi}(1-\cos4\theta)^{s/2+1/4}d\theta\cdot
\underbrace{\frac{2^{s/2+1/4}}{2s+1}\int_0^{+\infty}\frac{dz}{\left(\sqrt{2}z^2+\sqrt{z}\right)^{s+1/2}}}_{=
K_{s,\pi/4}/2} \\
& =\frac{1}{\pi\,2^{s/2+1/4}}\int_0^{2\pi}(1-\cos4\theta)^{s/2+1/4}d\theta\cdot\frac{1}{2}\,
K_{s,\pi/4}.
\end{split}
\end{equation}
On the other hand, with the substitution $\varphi:=4\theta$
and recalling that $s+1/2\in(1/2,1)$, we see that
\begin{eqnarray*}
\frac{1}{\pi\,2^{s/2+1/4}}\int_0^{2\pi}(1-\cos4\theta)^{s/2+1/4}d\theta & = & 
\frac{1}{4\pi}\int_0^{8\pi}\left(\sqrt{\frac{1-\cos\varphi}{2}}\right)^{s+1/2}d\varphi 
=\ \frac{1}{\pi}\int_0^{2\pi}\left(\sin\frac{\varphi}{2}\right)^{s+1/2}d\varphi \\
& \geq & \frac{1}{\pi}\int_0^{2\pi}\sin\frac{\varphi}{2}\;d\varphi
=\left.-\frac{2}{\pi}\cos\frac{\varphi}{2}\right\arrowvert_0^{2\pi}=\frac{4}{\pi}> 1.
\end{eqnarray*}
By combining this with \eqref{7df87f7} and \eqref{d66d6d}
we conclude that
\[
H_s \ \gneq\ \frac{1}{2}K_{s,\pi/4}\ =\ \frac{K_{s,0}+K_{s,\pi/4}}2.
\]
This establishes the claims in Example \ref{es2}.

\section{proof of lemma \ref{goog}}\label{pf:goog}

We prove now \eqref{7.2}, while \eqref{7.1} can be deduced 
with minor modifications with respect to the proof below.
Let
\[ 
\sigma_\eps:=
\int_{\pi(e)\setminus B_\eps}\frac{|y'|^{N-2}
\widetilde\chi_E(y)}{|y|^{N+2s}}
\,dy.
\]
Notice that, for any $\eps>\eps'>0$,
\begin{equation}\label{7.7}
|\sigma_{\eps'}-\sigma_\eps|=
\left|\int_{\pi(e)\cap (B_{\eps}\setminus B_{\eps'})}\frac{|y'|^{N-2}
\widetilde\chi_E(y)}{|y|^{N+2s}}\,dy\right|.
\end{equation}
Since $\partial E$ is $C^2$, in normal coordinates we
may suppose that, for small $\eps>0$,
\[
E\cap B_\eps \subseteq \{ x=(x',x_N)\in\R^N\ : \ x_N\le M |x'|^2\}
\ {\mbox{ and }} \
(\C E)\cap B_\eps \subseteq \{ x\in\R^N\ : \ x_N\ge -M |x'|^2\},
\]
for a suitable $M>0$.
This provides a cancellation of the contributions outside the set
$E_{\eps,\eps'}:=\{ x\in B_\eps\setminus B_{\eps'} : \ x_N\le M |x'|^2\}$, namely
\begin{multline*}
\left|\int_{\pi(e)\cap (B_\eps\setminus B_{\eps'})}\frac{|y'|^{N-2}
\widetilde\chi_E(y)}{|y|^{N+2s}}
\;dy\right|=
\left|\int_{\pi(e)\cap E_{\eps,\eps'}}\frac{|y'|^{N-2}
\widetilde\chi_E(y)}{|y|^{N+2s}}
\;dy\right|\ \le\\ \leq\ 
\int_{\pi(e)\cap E_{\eps,\eps'}}\frac{|y'|^{N-2}}{|y|^{N+2s}}
\;dy\le \int_{\pi(e)}\frac{\chi_{E_{\eps,\eps'}}(y)}{|y'|^{2+2s}}
\;dy,
\end{multline*}
since $|y|\ge |y'|$. 
Note now that in $E_{\eps,\eps'}$ we have
\[
|x'|^2\geq |\eps'|^2-|x_N|^2\geq|\eps'|^2-M^2|x'|^4\geq |\eps'|^2-M^2\eps^2|x'^2|
\]
and
therefore
\[
|x'|^2\geq \frac{|\eps'|^2}{1+M^2\eps^2}=:|\eps^*|^2
\]
That is, by \eqref{dy},
\[
\left|\int_{\pi(e)\cap (B_\eps\setminus B_{\eps'})}\frac{|y'|^{N-2}
\widetilde\chi_E(y)}{|y|^{N+2s}}
\,dy\right|\le\int_{\eps^*}^\eps d\rho \int_{-M\rho^2}^{M\rho^2} dh \;\rho^{-2-2s}\le
2M \int_{\eps^*}^\eps d\rho \;\rho^{-2s}=\frac{2M\left(\eps-\eps^*\right)^{1-2s}}{1-2s}
\]
which is infinitesimal with $\eps$ (recall that $s\in(0,1/2)$ by assumption).
This and \eqref{7.7}
imply that for any $\eps_N \downarrow0$, $\sigma_{\eps_N}$
is a Cauchy sequence, as desired.
\hfill$\square$

\nocite{*}
\bibliographystyle{plain}
\bibliography{bibliothese}
\addcontentsline{toc}{chapter}{Bibliography}


\end{document}